\documentclass[leqno,12pt]{article} 
\usepackage{amsmath, amssymb}
\usepackage{amsthm} 
\theoremstyle{plain} 
\newtheorem{theorem}{\indent\sc Theorem}[section] 
\newtheorem{lemma}[theorem]{\indent\sc Lemma}
\newtheorem{corollary}[theorem]{\indent\sc Corollary}
\newtheorem{proposition}[theorem]{\indent\sc Proposition}
\newtheorem{claim}[theorem]{\indent\sc Claim}
\theoremstyle{definition} 
\newtheorem{definition}[theorem]{\indent\sc Definition}
\newtheorem{remark}[theorem]{\indent\sc Remark}
\newtheorem{example}[theorem]{\indent\sc Example}

\makeatletter
\def\address#1#2{\begingroup
\noindent\parbox[t]{7.8cm}{%
\small{\scshape\ignorespaces#1}\par\vskip1ex
\noindent\small{\itshape E-mail address}%
\/: #2\par\vskip4ex}\hfill%
\endgroup}%
\makeatother
\title{\uppercase{Ricci curvature and convergence of Lipschitz functions}} 
\author{
\textsc{Shouhei Honda} 
}
\date{} 
\begin{document}

\maketitle
\footnote{ 
Supported by GCOE `Fostering top leaders in mathematics', Kyoto University.
}

\begin{abstract}
We give a definition of convergence of differential of Lipschitz functions with respect to measured Gromov-Hausdorff topology. As their applications, we give a characterization of harmonic functions with polynomial growth on asymptotic cones of manifolds with nonnegative Ricci curvature and Euclidean volume growth, and distributional Laplacian comparison theorem on limit spaces of Riemannian manifolds.
\end{abstract}
\tableofcontents
\section{Introduction} 
Let $\{(M_i, m_i)\}$ be a sequence of pointed $n$-dimensional complete Riemannian manifolds $(n \ge 2)$ with $\mathrm{Ric}_{M_i} \ge -(n-1)$ and $(Y, y, \upsilon)$ a pointed proper metric space (i.e. every bounded subset of $Y$ is relatively compact) with Radon measure $\upsilon$  on $Y$ satisfying 
$(M_i, m_i, \underline{\mathrm{vol}})$ converges to $(Y, y, \upsilon)$
in the sense of measured Gromov-Hasdorff topology.
Here $\underline{\mathrm{vol}}$ is the renormalized Riemannian volume of $(M_i, m_i)$: $\underline{\mathrm{vol}}= \mathrm{vol}/\mathrm{vol}\,B_1(m_i)$.
We fix $R>0$, a sequence of Lipschitz functions $f_i$ on $B_R(m_i)=\{w \in M_i; \overline{w, m_i}< R\}$ and a Lipschitz function 
$f_{\infty}$ on $B_R(y)$ satisfying $\sup_{i} \mathbf{Lip}f_i < \infty$.
Here $\overline{w, m_i}$ is the distance between $w$ and $m_i$, $\mathbf{Lip}f_i$ is the Lipschitz constant of $f_i$.
Then we say that \textit{$f_i$ converges to $f_{\infty}$} if $f_i(x_i) \rightarrow f_{\infty}(x_{\infty})$ for every $x_i \in B_R(m_i)$ and $x_{\infty} \in B_R(y)$ satisfying that $x_i$ converges to $x_{\infty}$. 
See section $2$ for these precise definitions.
Assume $\{f_i\}$ converges to $f_{\infty}$ below.

The purpose of this paper is to give a definition: \textit{differential $df_i$ of $f_i$ converges to differential $df_{\infty}$ of $f_{\infty}$} in this setting.
To give the definition below,
we shall recall celebrated works for limit spaces of Riemannian manifolds by Cheeger-Colding.
By \cite{ch2} and \cite{ch-co3}, we can construct the cotangent bundle $T^*Y$ of $Y$,  a fiber $T^*_wY$ is a finite dimensional real vector space with canonical inner product $\langle \cdot, \cdot \rangle (w)$ for a.e. $w \in Y$.
Moreover,  every Lipschitz function $g$ on $B_R(y)$ have canonical differential section: $dg(w) \in T^*_wY$ for a.e. $w \in B_R(y)$.
See section $4$ in \cite{ch2} and section $6$ in \cite{ch-co3} for the details.

We shall give a definition of convergence of differential of Lipschitz functions (see Definition \ref{Lip}):
\begin{definition}[Convergence of differential of Lipschitz functions]\label{063}
We say that \textit{$df_i$ converges to $df_{\infty}$ on $B_R(y)$} if for every $\epsilon > 0$,  $x_{\infty} \in B_R(y)$ $z_{\infty} \in Y$, $x_i \in B_R(m_i)$ and $z_i \in M_i$ satisfying that $x_i$ converges to $x_{\infty}$ and that $z_i$ converges to $z_{\infty}$, there exists $r>0$ such that 
\[\limsup_{i \rightarrow \infty}\left|\frac{1}{\underline{\mathrm{vol}}\,B_t(x_i)}\int_{B_t(x_i)}\langle dr_{z_i}, df_i\rangle d\underline{\mathrm{vol}}-\frac{1}{\upsilon(B_t(x_{\infty}))}\int_{B_t(x_{\infty})}\langle dr_{z_{\infty}}, df_{\infty}\rangle d\upsilon \right|< \epsilon\]
and 
\[\limsup_{i \rightarrow \infty}\frac{1}{\underline{\mathrm{vol}}\,B_t(x_i)}\int_{B_t(x_i)}|df_i|^2d\underline{\mathrm{vol}}\le \frac{1}{\upsilon(B_t(x_{\infty}))}\int_{B_t(x_{\infty})}|df_{\infty}|^2d\upsilon + \epsilon\]
for every $0 < t < r$. 
\end{definition}
If $df_i$ converges to $df_{\infty}$ on $B_R(y)$, then we denote it by $(f_i, df_i) \rightarrow (f_{\infty}, df_{\infty})$ on $B_R(y)$.
Assume $(f_i, df_i) \rightarrow (f_{\infty}, df_{\infty})$ and $(g_i, dg_i) \rightarrow (g_{\infty}, dg_{\infty})$ on $B_R(y)$ below.

In the paper, we will study several properties of the convergence and give their applications.
For example, we will give the following in section $4$:
\begin{theorem}\label{473}
We have 
\[\lim_{i \rightarrow \infty}\int_{B_R(m_i)}F_i(\langle df_i, dg_i\rangle)d\underline{\mathrm{vol}}=\int_{B_R(y)}F_{\infty}(\langle df_{\infty}, dg_{\infty}\rangle )d\upsilon\]
for every sequence of continuous functions $\{F_i\}_{i=1, 2, _{\cdots}, \infty}$ on $\mathbf{R}^k$ satisfying that 
$F_i$ converges to $F_{\infty}$ uniformly on each compact subsets of $\mathbf{R}$.
Especially, if $f_{\infty}=g_{\infty}$, then we have
\[\lim_{i \rightarrow \infty}\frac{1}{\underline{\mathrm{vol}}\,B_R(m_i)}\int_{B_R(m_i)}F_i(|df_i-dg_i|)d\underline{\mathrm{vol}}=F_{\infty}(0).\]
\end{theorem}
See Proposition \ref{10101} and Theorem \ref{10105} for the proof.
We will also give the following in the section:
\begin{theorem}\label{har2}
Let $h_i$ be a harmonic function on $B_R(m_i)$ and $h_{\infty}$ a Lipschitz function on $B_R(y)$ satisfying that
$\sup_i \mathbf{Lip}h_i<\infty$ and that $h_i$ converges to $h_{\infty}$ on $B_R(y)$.
Then $h_{\infty}$ is harmonic function on $B_R(y)$, $(h_i, dh_i) \rightarrow (h_{\infty}, dh_{\infty})$ on $B_R(y)$.
\end{theorem} 
We remark that the harmonicity of $h_{\infty}$ in Theorem \ref{har2} is given already in \cite{di2} by Ding.
We will give an alternative proof of it in section $4$ (see Corollary \ref{har}).

The organization of this paper is as follows:

In the next section, we will give several important notions and propeties for metric spaces and manifolds to understand this paper.
Most of statements in the section do not have the proof,
we will give a reference for them only.

In section $3$, we will give results of rectifiability for limit spaces of Riemannian manifolds (Theorem \ref{7} and Theorem \ref{24}).
It is important that we can take functions which give a rectitfiability of limit spaces, by \textit{distance functions} in these theorem.
As a corollary, we will give an explicit geometric formula of radial derivative for Lipschitz functions (Theorem \ref{14}).
These results are used in section $4$ essentially.
In \cite{ho5}, we will also give a geometric application of results in this section $3$ to limit spaces of Riemannian manifolds
with Ricci curvature bounded below.

In section $4$, we will give a definition of convergences of $L^{\infty}$-functions associated to measured Gromov-Hausdorff convergence
and give the definition of convergence of differential of Lipschitz functions again via the definition of convergence of $L^{\infty}$-functions.
After that, we will give several properties of the convergence. 
Main properties of them are Theorem \ref{10105}, Theorem \ref{app} and Corollary \ref{suffconv}.

In section $5$, as an application of results in section $4$, we will study harmonic functions on asymptotic cones of manifolds with nonnegative Ricci curvature and Euclidean volume growth via Colding-Minicozzi big theory (\cite{co-mi1, co-mi2, co-mi3, co-mi4, co-mi5, co-mi6}).
See Definition \ref{asymp} for the definition of asymptotic cones.
It is important that we can replace most of statements for harmonic functions on manifolds in \cite{co-mi2} with one on asymptotic cones via Ding's important works \cite{di1, di2} and Theorem \ref{10105}.
For instance, we will prove that the space of harmonic functions with polynomial growth of a fixed rate is finite dimensional vector space (Theorem \ref{023}).
We can regard it as \textit{asymptotic cones version} of finite dimensionality conjecture on manifolds by Yau (see for instance Conjecture $0.1$ in \cite{co-mi1}).
We remark that most of important essential ideas to prove these statements given in \cite{co-mi2, co-mi6}.
Roughly speaking, we can get these results by ``taking limit of most of results in \cite{co-mi2} via Theorem \ref{10105}''.
As an application of them to manifolds, we will prove the following Liouville type theorem:
\begin{theorem}
Let $M$ be an $n$-dimensional $(n \ge 3)$ complete Riemannian manifold with nonnegative Ricci curvature and Euclidean volume growth.
Then, there exists unique $d_1 \ge 1$ satisfying the following properties:
\begin{enumerate}
\item For every asymptotic cone $M_{\infty}$ of $M$ and $0 < d < d_1$, we have
\[H^d(M_{\infty})=\{\mathrm{Constant \ functions}\}.\]
Here $H^d(M_{\infty})$ is the linear space of harmonic functions on $M_{\infty}$ with order of growth at most $d$.
\item There exists an asymptotic cone $M_{\infty}$ of $M$ such that 
\[H^{d_1}(M_{\infty}) \neq \{\mathrm{Constant \ functions}\}.\]
\item For every  $0 < d < d_1$, we have
\[H^d(M)=\{\mathrm{Constant \ functions}\}.\]
\end{enumerate}
\end{theorem}
See Corollary \ref{Liouville} for the proof.

In section $6$, as another application of results in section $4$, we will give (distributional) Laplacian comparison theorem on limit spaces of Riemannian manifolds by using several results in \cite{ho}.
See Theorem \ref{26}.
This formulation is given in \cite{KS1} by Kuwae-Shioya on weighted Alexandrov spaces.
Roughly speaking, this Laplacian comparison theorem implies that limit spaces of Riemannian manifolds have ``definite lower bound of Ricci curvature in some sense.''
In fact, we can get a stability result of lower bound of Ricci curvature with respect to Gromov-Hausdorff topology (Corollary \ref{87593}).
The corollary is well known in the setting of metric measure spaces. See for instance \cite{lo, lo-vi, oh, St1, St2, Vi1, Vi2}.
We will give an alternative proof of it via the Laplacian comparison theorem.

In section $7$, we will give  proofs of several propositions used in previous sections. 

\textbf{Acknowledgments.}
The author would like to express his deep gratitude to Professor Kenji Fukaya and Professor Tobias Holck Colding for warm encouragement and their numerous suggestions and advice.
He is grateful to Professor Takashi Shioya for his suggestion about Theorem \ref{26} and giving many valuable suggestions.
This work was done during the stay at MIT,
he also thanks to them and all members of Informal Geometry Seminar in MIT for warm hospitality and for giving nice environment.
\section{Preliminaries}
Our aim in this section is to introduce important notions and properties for metric spaces and manifolds to understand statements in this paper.
\subsection{Metric measure spaces}
For a positive number $\epsilon >0$, we use following notation:
\[a=b \pm \epsilon \Longleftrightarrow |a-b|<\epsilon.\]
We denote by $ \Psi (\epsilon_1, \epsilon _2 , _{\cdots} ,\epsilon_k  ; c_1, c_2, _{\cdots}, c_l )$ (more simply, $\Psi$)
some positive function on $\mathbf{R}_{>0}^k \times \mathbf{R}^l $ satisfying 
\[\lim_{\epsilon_1, \epsilon_2, _{\cdots} ,\epsilon_k \to 0}\Psi (\epsilon_1, \epsilon_2, _{\cdots} ,\epsilon_k  ; c_1, c_2 , _{\cdots} ,c_l)=0\] 
for each fixed real numbers $c_1, c_2, _{\cdots} ,c_l$.
We often denote by $C(c_1, c_2, _{\cdots} ,c_l)$ some (positive) constant depending only on fixed real numbers $c_1, c_2, _{\cdots} ,c_l$.

For a metric space $Z$, a point $z \in Z$ and a positive number $r>0$, we use the following notation:
\[B_r(z) = \{x \in Z; \overline{z,x} < r\}, \overline{B}_r(z)=\{x \in Z; \overline{z,x} \le r\},
\partial B_r(z)=\{x \in Z; \overline{z,x} = r\}.\]
Here $\overline{y, x}$ is the distance between $y$ and $x$,
we often denote the distance by $d_Z(y, x)$.
For $r < R$, we put $A_{r, R}(z) = \overline{B}_R(z) \setminus B_r(z)$.
For every $A \subset Z$, we also put $B_r(A)=\{x \in Z; \overline{A,w}<r\}$ and 
$\overline{B}_r(A) = \{x \in Z; \overline{A, x} \le r\}$.
For  an open subset $U$ of $Z$ and $\eta > 0$, we put $U_{\eta}=\{w \in U; B_{\eta}(w) \subset U\}$.
It is easy to check that $U_{\eta}$ is closed subset of $Z$.
For $z \in Z$, we define $1$-Lipschitz function $r_z$ on $Z$ by $r_z(w)=\overline{z,w}$. 

For a Lipschitz function $f$ on $Z$ and a point $z \in Z$, we will use the following notations:
\begin{enumerate}
\item If $z$ is not an isolated point in $Z$, then we put
\[lipf(z)= \liminf_{r \rightarrow 0}\left( \sup _{x \in B_r(z) \setminus \{z\}}\frac{|f(x)-f(z)|}{\overline{x,z}}\right),\]
if $z$ is an isolated point in $Z$, then we put $lipf(z)=0$.
\item If $z$ is not an isolated point in $Z$, then we put
\[\mathrm{Lip}f(z)= \limsup_{r \rightarrow 0}\left( \sup _{x \in B_r(z) \setminus \{z\}}\frac{|f(x)-f(z)|}{\overline{x,z}}\right),\]
if $z$ is an isolated point in $Z$, then we put $\mathrm{Lip}f(z)=0$.
\item If $Z$ is not single point, then we put
\[\mathbf{Lip}f= \sup_{w_1 \neq w_2} \frac{|f(w_1)-f(w_2)|}{\overline{w_1, w_2}} < \infty,\]
if $Z$ is a single point, then we put $\mathbf{Lip}f=0$.
\end{enumerate}

We shall remark that for every subset $A \subset Z$ and Lipschitz function $f$ on $A$, there exists a Lipschitz function $f^*$ on $Z$ such that 
$f^*|_A=f$ and $\mathbf{Lip}f^*=\mathbf{Lip}f$.
In fact, if we define a function $f^*$ on $Z$ by $f^*(z)=\inf_{a \in A}(f(a) + \mathbf{Lip}f \overline{z, a})$, then it is easy to check that $f^*|_A=f$ and $\mathbf{Lip}f^*=\mathbf{Lip}f$.

For a Borel subset $A$ of $Z$,  an extended real valued Borel function $f$ on $A$ and an extended nonnegative real valued Borel function $g$ on $A$, we say that $g$ is an upper gradient for $f$ if for every $a_1, a_2 \in A$ and continuous rectifiable curve $\gamma: [0, l] \rightarrow A$ parametrized by arclength with $\gamma(0)=a_1, \gamma(l)=a_2$, we have 
\[|f(a_1)-f(a_2)|\le \int_{0}^lg(\gamma(s))ds.\]
For an open subset $U \subset Z$ and a Lipschitz function $f$ on $U$, $lipf$ is an upper gradient for $f$ on $U$.
See  \cite[Proposition $1. 11$]{ch2}.
\

We say that $Z$ is \textit{proper} if every bounded subbsets of $Z$ are relatively compact.
We also say that $Z$ is a \textit{geodesic space} if for every $x_1, x_2 \in Z$, 
there exists an isometric embedding $\gamma$ from $[0, \overline{x_1, x_2}]$ to $Z$ such that 
$\gamma(0)=x_1, \gamma(\overline{x_1, x_2})=x_2$. We say that $\gamma$ is a \textit{minimal geodesic from} $x_1$ \textit{to} $x_2$.
For a proper geodesic space $W$ and $w \in W$, we put $C_w = \{z \in W;$ For every $x \in W \setminus \{z\},$ we have 
$\overline{w, z}+ \overline{z,x}>\overline{w, x} \}$ (if $W$ is a single point, then we put $C_w=\emptyset$).
We call $C_w$ \textit{cut locus of $W$ at $w$}.
\

For a proper metric space $Z$ and a Borel measure $\upsilon$ on $Z$, we say that $\upsilon$ \textit{is Radon measure} if $\upsilon(K) < \infty$ for every compact set $K$,
\[\upsilon(A) = \sup_{K \subset A: \mathrm{compact}}\upsilon(K) = \inf_{A \subset O: \mathrm{open}}\upsilon(O)\]
for every Borel subset $A$ of $Z$.  
Then we say that a pair $(Z, \upsilon)$ is a \textit{metric measure space} in this paper.
For a metric measure space $(Z, \upsilon)$, a point $z \in Z$ and $k \in \mathbf{R}_{\ge 0}$, we say that $\upsilon$ \textit{is Ahlfors $k$-regular at $z$} if there exist $r>0$ and $C \ge 1$ such that $C^{-1}\le \upsilon(B_t(z))/t^k \le C$ for every $0 < t < r$.
We shall introduce the notion of \textit{$\upsilon$-rectifiability} for metric measure spaces by Cheeger-Colding.
See \cite[Definition $5. 3$]{ch-co3} and \cite[Theorem $5. 7$]{ch-co3}.
For metric spaces $X_1, X_2$, $0 < \delta < 1$ and a bijection map $f$ from $X_1$ to $X_2$, 
we say that \textit{$f$ gives $(1 \pm \delta)$-bi-Lipschitz equivalent to $X_2$} if $f$ and $f^{-1}$ are $(1+\delta)$-Lipschitz map. 
\begin{definition}[Rectifiability for metric measure spaces]
For a metric measure space $(Z, \upsilon)$ and a Borel subset $A \subset Z$, we say that \textit{$A$ is $\upsilon$-rectifiable} if there exists a positive integer $m$, a collection of Borel subset $\{C_{k,i}\}_{1 \le k \le m, i \in \mathbf{N}}$ of $A$ and a collection of bi-Lipschitz embedding map $\{ \phi_{k,i}: C_{k,i} \rightarrow \mathbf{R}^k \}$ satisfying the following properties:
\begin{enumerate}
\item $\upsilon(A \setminus \bigcup_{k,i}C_{k,i})=0$
\item $\upsilon$ is Ahlfors $k$-regular at each $x \in C_{k,i}$.
\item For every $k$, $x \in \bigcup_{i \in \mathbf{N}}C_{k,i}$ and $0 < \delta < 1$, there exists $C_{k,i}$ such that $x \in C_{k,i}$ and that 
the map $\phi_{k,i}$ gives $(1 \pm \delta)$-bi-Lipschitz equivalent to the image $\phi_{k,i}(C_{k,i})$.
\end{enumerate}
\end{definition}
We shall recall the definition of Sobolev spaces on metric measure spaces (see \cite{ch1} and \cite{he-ko2}).
We fix a metric measure space $(Z, \upsilon)$ satisfying that $Z$ is a geodesic space and that $(Z, \upsilon)$ satisfies doubling condition below: For every $r > 0$, there exists $K=K(r)\ge 1$ such that $0 < \upsilon(B_{2s}(x))\le 2^K\upsilon(B_s(x))$ for every $x \in Z$ and $0<s <r$.
We fix an open set $U \subset Z$. 
For functions $f, g \in L^2(U)$,
we say that \textit{$g$ is a generalized upper gradient for $f$} if there exists sequences of extended real valued functions $f_i$ on $U$ and 
upper gradient $g_i$ for $f_i$ on $U$ such that $f_i \rightarrow f$ and $g_i \rightarrow g$ in $L^2(U)$.
Let $H_{1,2}(U)$ be the subspace of $L^2(U)$ consisting functions $f$ satisfying that there exists a generalized upper gradient $g$ for $f$ on $U$.
By \cite[Theorem $2.10$]{ch2}, for every $f \in H_{1, 2}(U)$, there exists unique $g_f \in L^2(U)$ satisfying that 
 $|g_f|_{L^2(U)}\le |g|_{L^2(U)}$ for every generalized upper gradient $g$ for $f$. 
We define a norm $|\cdot|_{1,2}$ on $H_{1,2}(U)$ by $|f|_{1,2}=|f|_{L^2(U)}+|g_f|_{L^2(U)}$.
We call $(H_{1,2}(U), |\cdot|_{1,2})$ \textit{the Sobolev space}.
We put $K(U)= \{k \in H_{1,2}(U);$ There exists $\eta >0$ such that $\upsilon(\{k \neq 0\} \cap (U \setminus U_{\eta}))=0 \}$.
\

We recall the definition of ($2$-)harmonic function on metric measure spaces by Cheeger.
For a Borel function $f$ on $U$, we say that \textit{$f$ is harmonic on $U$} if $f|_V \in H_{1,2}(V)$ for every bounded subset $V \subset U$ and 
$|g_{f+k}|_{L^2(V)}\ge |g_f|_{L^2(V)}$ for every $k \in K(U)$.
\

We shall recall the definition of \textit{weak Poincar\'{e} inequality of type $(1,2)$} for metric measure spaces.
We say that $(Z, \upsilon)$ satisfies \textit{a weak Poincar\'{e} inequality of type $(1, 2)$} if 
for every $R>0$, there exist $\tau \ge 1$ and $C \ge 1$ such that 
\[\frac{1}{\upsilon(B_r(x))}\int_{B_r(x)}\left|f-\frac{1}{\upsilon(B_r(x))}\int_{B_r(x)}fd\upsilon \right|d\upsilon \le Cr \sqrt{\frac{1}{\upsilon(B_{\tau r}(x))}\int_{B_{\tau r}(x)}g_f^2d\upsilon }\]
for every $x \in Z$, $0 < r < R$ and $f \in H_{1,2}(B_{\tau r}(x))$.
We remark that if  $(Z, \upsilon)$ satisfies a weak Poincar\'{e} inequality of type $(1, 2)$, then for every $R>0$, there exist
$C_1 \ge 1$ such that 
\[\frac{1}{\upsilon(B_r(x))}\int_{B_r(x)}\left|f-\frac{1}{\upsilon(B_r(x))}\int_{B_r(x)}fd\upsilon \right|d\upsilon \le C_1r \sqrt{\frac{1}{\upsilon(B_{r}(x))}\int_{B_{r}(x)}g_f^2d\upsilon}\]
for every $x \in Z$, $0 < r < R$ and $f \in H_{1,2}(B_r(x))$.
See for instance $(4.4)$ in \cite{ch2} or \cite{ha-ko}.
\

We shall give a short review of important results about differentiability of Lipschitz functions on metric measure spaces by Cheeger.
We assume that $(Z, \upsilon)$ satisfies weak Poincar\'{e} inequality of type $(1,2)$ below.
Then, by section $4$ in \cite{ch2}, we can construct the cotangent bundle $T^*Z$ of $Z$.
See  \cite[Definition $4.42$]{ch2} for the construction.
We will give several fundamental properties of the cotangent bundle only:
\begin{enumerate}
\item $T^*Z$ is a topological space.
\item There exists a Borel map $\pi:T^*Z \rightarrow Z$ such that $\upsilon(Z \setminus \pi (T^*Z))=0$.
\item For every $w \in \pi(T^*Z)$, $\pi^{-1}(w)$ is finite dimensional real vector space with canonical norm $|\cdot|(w)$.
\item For every open set $U \subset Z$ and $f \in H_{1, 2}(U)$, there exists a Borel set $V \subset U$ and a Borel map $df$ (called  differential section of $f$) from $V$ to $T^*Z$ such that $\upsilon(U \setminus V)=0$
and that $\pi \circ df(w)=w$, $|df|(w)=g_f(w)$ for every $w \in V$.
Moreover, if $f$ is Lipschitz, then $|df|(w)=\mathrm{Lip}f(w)=lipf(w)$.
\item For every open set $U \subset Z$ and Lipschitz functions $f_1, f_2$ on $U$, Leibnitz rule hold:
\[d(f_1f_2)(w)=f_2(w)df_1(w)+f_1(w)df_2(w)\]
for a.e. $w \in U$.
\end{enumerate} 
See section $4$ and $5$ in \cite{ch2} for the details.
\

In addition, we assume that $Z$ is $\upsilon$-rectifiable below.
Then, by section $6$ in \cite{ch-co3}, for a.e. $w \in Z$, each norms $|\cdot|(w)$ defines the inner product $\langle \cdot, \cdot \rangle (w)$, i.e. $|v|(w)= \sqrt{\langle v, v\rangle (w)}$ for every $v \in \pi^{-1}(w)$. We call $\{\langle \cdot, \cdot \rangle (w)\}_{w \in Y}$ \textit{Riemannian metric of $Y$} and denote it by $\langle \cdot, \cdot \rangle$.
Moreover, the following bilinear form
\[\int_Z \langle df_1, df_2 \rangle d\upsilon \]
on $H_{1, 2}(Z)$ is closable (see \cite[Theorem $6. 25$]{ch-co3}).
Therefore this bilinear form determines a canonical (positive definite) self-adjoint operator $\Delta_Z$ on $L^2(Z)$.
We call $\Delta_Z$ \textit{Laplace operator of $(Z, \upsilon)$ or Laplacian of $(Z, \upsilon)$}
Moreover, if $Z$ is compact, then $(1 + \Delta_Z)^{-1}$ is compact operator (see \cite[Theorem $6. 27$]{ch-co3}).
\subsection{Gromov-Hausdorff convergence}
For compact metric spaces $X_1, X_2$, we define \textit{Gromov-Hausdorff distance between $X_1$ and $X_2$} by
\begin{align*}
d_{GH}(X_1, X_2)&=\inf \{d_H^W(\phi_1(X_1), \phi_2(X_2));\textrm{There exist a metric space } W \ \textrm{and} \\
&\textrm{isometric embeddings 
}\phi_i \ \textrm{from } X_i \ \textrm{to } W (i=1, 2)\}.
\end{align*}
Here $d_H^W$ is the Hausdroff distance and the infimum above runs over all $W, \phi_i$ satisfying conditions above.
We remark that $d_{GH}$ is a distance on the set of isometry class of compact metric spaces.
On the other hand, for compact metric spaces $X_1$, $X_2$, a positive number $\epsilon >0$ and a map $\phi$ from $X_1$ to $X_2$, we say that $\phi$ is an 
\textit{$\epsilon$-Gromov-Hausdorff approximation} if $B_{\epsilon}(\mathrm{Image}\phi)=X_1$ and $|\overline{x,y}-\overline{\phi(x), \phi(y)}|<\epsilon$ for every $x, y \in X_1$.
It is easy to check that if $d_{GH}(X_1, X_2)\le \epsilon$, then there exists an $3\epsilon$-Gromov-Hausdorff approximation from $X_1$ to $X_2$
and that if there exists an $\epsilon$-Gromov-Hausdorff approximation from $X_1$ to $X_2$, then $d_{GH}(X_1, X_2) \le 9\epsilon$.
For a sequence of compact metric spaces $X_i$, we say that $X_i$ \textit{converges to $X_{\infty}$} if $d_{GH}(X_i, X_{\infty})$ converges to $0$.
Then we denote it by $X_i \rightarrow X_{\infty}$.
Similarly, for pointed compact metric spaces $(X_1, x_1), (X_2, x_2)$, we can define the \textit{pointed} Gromov-Hausdorff distance $d_{GH}((X_1, x_1), (X_2, x_2))$.
\

Moreover, for a sequence of pointed proper geodesic spaces, $(Z_i, z_i)$, we say that 
$(Z_i, z_i)$ converges to $(Z_{\infty}, z_{\infty})$ if there exist  sequences of positive numbers
$\epsilon_i$, $R_i$ and a (Borel) map $\phi_i$ from $(B_{R_i}(z_i), z_i)$ to $(B_{R_i}(z_{\infty}), z_{\infty})$ such that
$\epsilon_i \rightarrow 0$, $R_i \rightarrow \infty$ as $i \rightarrow \infty$,
$B_{R_i}(z_{\infty}) \subset B_{\epsilon_i}(\mathrm{Image} \phi_i)$ and 
$|\overline{x_1, x_2}-\overline{\phi_i(x_1), \phi_i(x_2)}|\le \epsilon_i$ for every $x_1, x_2 \in B_{R_i}(x_i)$.
We denote it by $(Z_i, z_i) \stackrel{(\phi_1, R_i, \epsilon_i)}{\rightarrow} (Z_{\infty}, z_{\infty})$, or
more simply $(Z_i, z_i) \rightarrow (Z_{\infty}, z_{\infty})$.
For every $x_{\infty} \in Z_{\infty}$ and $x_i \in Z_i$, we say that $x_i$ converges to $x_{\infty}$ if
$\overline{\phi_i(x_i), x_{\infty}} \rightarrow 0$. 
Then, we denote it by $x_i \rightarrow x_{\infty}$.
\

Let $(Z_i ,z_i) \rightarrow (Z_{\infty}, z_{\infty})$. 
For a sequence of sets $A_i \subset Z_i$ satisfying that there exists $R>0$ such that $A_i \subset B_R(z_i)$ for every $i$, we say that \textit{$A_i$ is included by $A_{\infty}$ asymptotically} if 
for every $\epsilon > 0$, there exists $i_0$ such that for every $i \ge i_0$, $\phi_i(A_i) \subset B_{\epsilon}(A_{\infty})$.
Then we denote it by $\limsup_{i \rightarrow \infty}A_i \subset A_{\infty}$.
(If $A_{\infty} = \emptyset$, then  $\limsup_{i \rightarrow \infty}A_i \subset A_{\infty}$ implies $A_i= \emptyset$ for every sufficiently large $i$.)
Similarly, 
we also say that \textit{$A_{\infty}$ is included by $A_{i}$ asymptotically} if for every $\epsilon > 0$,  there exists $i_0$ such that for every $i \ge i_0$, $A_{\infty} \subset A_{\epsilon}(\phi_i(A_i))$.
Then we denote it by $A_{\infty} \subset \liminf_{i \rightarrow \infty}A_i$.
Let $C_{\infty} \subset \liminf_{i \rightarrow \infty} C_i$.
For a sequence of Lipschitz function $f_i$ on $C_i$ satisfying $\sup_i \mathbf{Lip}f_i < \infty$,  we say that \textit{$f_{\infty}$ is a restriction of $f_i$ asymptotically} if 
for every $w \in C_{\infty}$, subsequence $\{n(i)\}$ of $\mathbf{N}$ and $w_{n(i)} \in C_{n(i)}$ satisfying $\overline{\phi_{n(i)}(w_{n(i)}), w} \rightarrow 0$, we have
\[\lim_{i \rightarrow \infty}f_{n(i)}(w_{n(i)})=f_{\infty}(w).\]
Let $\limsup_{i \rightarrow \infty} D_i \subset D_{\infty}$ and $D_{\infty}$ be compact.
For  a sequence of Lipschitz function $g_i$ on $D_i$ satisfying  $\sup_i \mathbf{Lip}g_i < \infty$,  we say that \textit{$g_{\infty}$ is an extension of $g_i$ asymptotically} if 
for every $w \in D_{\infty}$, subsequence $\{n(i)\}$ of $\mathbf{N}$ and $w_{n(i)} \in D_{n(i)}$ satisfying $\overline{\phi_{n(i)}(w_{n(i)}), w}
\rightarrow 0$, we have
\[\lim_{i \rightarrow \infty}g_{n(i)}(w_{n(i)})=g_{\infty}(w).\]

For a sequence of compact set $K_i \subset Z_i$, we say that 
$(Z_i, z_i, K_i)$ converges to $(Z_{\infty}, z_{\infty}, K_{\infty})$ if there exists $\tau_i > 0$ such that $\tau_i \rightarrow 0$, 
$\phi_i(K_i) \subset B_{\epsilon_i+ \tau_i}(K_{\infty})$ and $K_{\infty} \subset B_{\epsilon_i+ \tau_i}(\phi_i(K_i))$.
Then we denote it by $(Z_i, z_i, K_i) \stackrel{(\phi_1, R_i, \epsilon_i)}{\rightarrow} (Z_{\infty}, z_{\infty}, K_{\infty})$ or,
more simply, $(Z_i, z_i, K_i) \rightarrow (Z_{\infty}, z_{\infty}, K_{\infty})$ or $K_i \rightarrow K_{\infty}$.
It is easy to check that  $(Z_i, z_i, K_i) \rightarrow (Z_{\infty}, z_{\infty}, K_{\infty})$ holds if and only if 
$\limsup_{i \rightarrow \infty}K_i  \subset K_{\infty}$  and $K_{\infty} \subset \liminf_{i \rightarrow \infty}K_i$ hold.
\

Let $(Z_i, z_i, K_i) \rightarrow (Z_{\infty}, z_{\infty}, K_{\infty})$. 
For a sequence of Lipschitz functions, $f_i^1, f_i^2, _{\cdots}, f_i^k$ on $K_i$ satisfying $\sup_{i, l}(\mathbf{Lip}f_i^l + |f_i^l|_{L^{\infty}}) < \infty$, we say that 
$(Z_i, z_i, K_i, f_i^1, _{\cdots}, f_i^k)$ converges to $(Z_{\infty}, z_{\infty}, K_{\infty}, f_{\infty}^1, _{\cdots}, f_{\infty}^k)$ if 
\[\lim_{i \rightarrow \infty}f_i^l(x_i)=f_{\infty}^l(x_{\infty})\]
for every $x_i \in K_i$ and $x_{\infty} \in K_{\infty}$ satisfying  $x_i \rightarrow x_{\infty}$.
It is easy to check that this condition holds if and only if $f_{\infty}^l$ is an extension (or a restriction) of $\{f_i^l\}$ asymptotically for every $l$.
We denote it by $(Z_i, z_i, K_i, f_i^1, _{\cdots}, f_i^k) \rightarrow (Z_{\infty}, z_{\infty}, K_{\infty}, f_{\infty}^1, _{\cdots}, f_{\infty}^k)$, or more simply, $f_i^l \rightarrow f_{\infty}^l$ for every $l$.
Then we can also check that 
\[\lim_{i \rightarrow \infty}|f_i^l-f_{\infty}^l\circ \phi_i|_{L^{\infty}(K_i)}=0\]
easily. 
\begin{example}\label{d1}
Let $(Z_i, z_i) \rightarrow (Z_{\infty}, z_{\infty})$. Then it is easy to check that $\limsup_{i \rightarrow \infty}\overline{B}_R(z_i)
\subset B_R(z_{\infty})$ and $\overline{B}_R(z_{\infty}) \subset \liminf_{i \rightarrow \infty}B_R(z_i)$.
\end{example}
\begin{example}\label{d2}
Let $(Z_i, z_i) \rightarrow (Z_{\infty}, z_{\infty})$. Then for every $A \subset Z_{\infty}$ and $\tau_i \rightarrow 0$, we have $\limsup_{i \rightarrow \infty}B_{\tau_i}((\phi_i)^{-1}(A_i)) \subset A$ and $A \subset \liminf_{i \rightarrow \infty}(\phi_i)^{-1}(A_i)$. 
\end{example}
It is not difficult to check the following proposition:
\begin{proposition}
Let $(Z_i, z_i) \rightarrow (Z_{\infty}, z_{\infty})$, $A_i^1, A_i^2$ bounded subsets of $Z_i$.
Then we have the following:
\begin{enumerate}
\item If $\limsup_{i \rightarrow \infty}A_i^j \subset A_{\infty}^j$ for $j = 1, 2$, then  $\limsup_{i \rightarrow \infty}(A_i^1 \cup A_i^2)  \subset A_{\infty}^1 \cup A_{\infty}^2$.
\item If $A_{\infty}^j \subset \liminf_{i \rightarrow \infty}A_i^j$ for $j =1, 2$, then $\liminf_{i \rightarrow \infty}(A_i^1 \cup A_i^2) \subset A_{\infty}^1 \cup A_{\infty}^2$.
\item If $X, Y \subset Z_{\infty}$ satisfies $\limsup_{i \rightarrow \infty}A_i^1 \subset X$, 
$\limsup_{i \rightarrow \infty}A_i^1 \subset Y$ and $X \cup Y \subset  \liminf_{i \rightarrow \infty}A_i^1$, then $\overline{X}=\overline{Y}$.
Here, $\overline{X}$ is the closure of $X$ in $Z_{\infty}$.
\end{enumerate}
\end{proposition}
We shall give a proof of the next proposition:
\begin{proposition}\label{compact2}
Let $(Z_i, z_i)$ be a sequence of proper geodesic spaces, $\Lambda$ a set  and $\{A_i^{\lambda}\}_{\lambda \in \Lambda}$ a collection of bounded subsets of $Z_i$. 
We assume that $(Z_i, z_i)$ converges to $(Z_{\infty}, z_{\infty})$, $A_{\infty}^{\lambda}$ is compact for every $\lambda \in \Lambda$ and that $\limsup_{i \rightarrow \infty}A_i^{\lambda} \subset A_{\infty}^{\lambda}$ for every $\lambda \in \Lambda$.
Then, $\limsup_{i \rightarrow \infty}\bigcap_{\lambda \in \Lambda}A_i^{\lambda} \subset \bigcap_{\lambda \in \Lambda}A_{\infty}^{\lambda}$.
\end{proposition}
\begin{proof}
The proof is done by a contradiction.
We assume that the assertion is false.
Then, there exists $\tau > 0$ such that for every $i$, there exist $N_i \ge i$ and $w_i \in  \phi_{N_i}(\bigcap_{\lambda \in \Lambda}A_{N_i}^{\lambda})
\setminus B_{\tau}(\bigcap_{\lambda \in \Lambda}A_{\infty}^{\lambda})$.
Without loss of generality, we can assume  that there exists $w_{\infty} \in Z_{\infty}$ such that $w_i \rightarrow w_{\infty}$.
By the assumption, we have $w_{\infty} \in \overline{A_{\infty}^{\lambda}}=A_{\infty}^{\lambda}$ for every $\lambda \in \Lambda$.
Thus, $w_{\infty} \in \bigcap_{\lambda \in \Lambda}A_{\infty}^{\lambda}$. Especially we have $w_i \in B_{\tau}(\bigcap_{\lambda \in \Lambda}A_{\infty}^{\lambda})$ for every sufficiently large $i$.
This is a contradiction.
\end{proof}
We shall consider convergence of a sequence of complement of open balls:
\begin{proposition}\label{compact3}
Let $(Z_i, z_i)$ be a sequence of proper geodesic spaces and $A_i$ a bounded subset of $Z_i$. 
We assume that $(Z_i, z_i)$ converges to $(Z_{\infty}, z_{\infty})$,  $A_{\infty}$ is compact and that $\limsup_{i \rightarrow \infty}A_i \subset A_{\infty}$.
Then for every $r > 0$ and $x_i \rightarrow x_{\infty} \in Z_{\infty}$, we have 
$\limsup_{i \rightarrow \infty}(A_i \setminus B_r(x_i)) \subset A_{\infty} \setminus B_r(x_{\infty})$.
\end{proposition}
\begin{proof}
We assume that the assertion is false.
Then there exists $\tau >0$ such that for every $i$, there exist $N_i \ge i$ and $w_i \in \phi_{N_i}(A_{N_i} \setminus B_r(x_{N_i}))
\setminus B_{\tau}(A_{\infty} \setminus B_r(x_{\infty}))$. 
Without loss of generality, we can assume that there exists $w_{\infty} \in Z_{\infty}$ such that $w_i \rightarrow w_{\infty}$.
By the assumption, we have $w_{\infty} \in \overline{A_{\infty}}=A_{\infty}$.
We take $\alpha_i \in A_{N_i} \setminus B_r(x_{N_i})$ satisfying $w_i = \phi_{N_i}(\alpha_i)$.
Then, since $\overline{\alpha_i, x_{N_i}} \ge r$, we have $\overline{w_{\infty}, x_{\infty}} \ge r$.
Therefore, $w_{\infty} \in A_{\infty} \setminus B_r(x_{\infty})$.
Thus, we have $w_i \in B_{\tau}(A_{\infty} \setminus B_r(x_{\infty}))$ for every sufficiently large $i$.
This is a contradiction.
\end{proof}
\begin{example}\label{3f}
Let $(Z_i, z_i) \rightarrow (Z_{\infty}, z_{\infty})$.
Then, for every $r>0$, we have $\limsup_{i \rightarrow \infty}\partial B_r(z_i) \subset \partial B_r(z_{\infty})$.
\end{example}
The proof of next proposition is done by a contradiction similar to the proof of Proposition \ref{compact2} or \ref{compact3}.
\begin{proposition}\label{compact4}
Let $(Z_i, z_i)$ be a sequence of proper geodesic spaces and $\eta_i$ a positive numbers. 
We assume that $(Z_i, z_i)$ converges to $(Z_{\infty}, z_{\infty})$ and $\eta_i \rightarrow \eta_{\infty}$.
Then for every $r > 0$, we have 
$\limsup_{i \rightarrow \infty}(B_r(z_i))_{\eta_i} \subset (B_r(z_{\infty}))_{\eta_{\infty}}$.
\end{proposition}
We will give the following fundamental result by Gromov for precompactness of Gromov-Hausdorff topology.
See \cite{gr} for the proof.
\begin{proposition}\label{090}
Let $\{(Z_i, z_i)\}_i$ be a sequence of pointed proper geodesic spaces.
We assume that for every $\epsilon>0$ and $R \ge 1$, there exists $N$ such that 
for every $i$, there exists a finite covering $\{B_{\epsilon}(x_j)\}_{j=1, _{\cdots}, N}$ of $B_R(z_i)$.
Then, there exist a subsequence $\{(Z_{n(i)}, z_{n(i)})\}$ and a pointed proper geodesic space $(Z_{\infty}, z_{\infty})$ such that 
$(Z_{n(i)}, z_{n(i)})$ converges to $(Z_{\infty}, z_{\infty})$.
\end{proposition}
We will give a result of precompactness for a sequence of compact sets;
\begin{proposition}\label{compact}
Let $(Z_i, z_i)$ be a sequence of proper geodesic spaces and $K_i$ a sequence of compact subset of $Z_i$. 
We assume that $(Z_i, z_i)$ converges to $(Z_{\infty}, z_{\infty})$ and that there exists $R>0$ such that $K_i \subset B_R(z_i)$ for every $i$.
Then, there exist a subsequence $\{n(i)\}$ and a compact subset $K_{\infty}$ of $Z_{\infty}$ such that $(Z_{n(i)}, z_{n(i)}, K_{n(i)})$
converges to $(Z_{\infty}, z_{\infty}, K_{\infty})$.
\end{proposition}
\begin{proof}
By the assumption, for every $k$, there exists $N_k$ such that for every $i$, there exists $x_1(i, k), _{\cdots}, x_{N_k}(i, k)
\in B_R(z_i)$
such that 
$K_i \subset B_R(z_i) \subset \bigcup_{j=1}^{N_k}B_{k^{-1}}(x_j(i,k)).$
Since $Z_{\infty}$ is proper, by diagonal argument, there exists  a subsequence $\{n(i)\}$
such that $\{\phi_{n(i)}(x_j(n(i), k))\}$ is Cauchy sequence for every $j, k$.
We put $x_j(k)=\lim_{i \rightarrow \infty}\phi_{n(i)}(x_j(n(i), k))$ and 
$K_{\infty}= \overline{\{x_j(k)\}}$.
It is easy to check that $(Z_{n(i)}, z_{n(i)}, K_{n(i)})$ converges to $(Z_{\infty}, z_{\infty}, K_{\infty})$.
\end{proof}
We will give a result of precompactness for a sequence of Lipschitz functions.
\begin{proposition}\label{Lips}
Let $(Z_i, z_i)$ be a sequence of proper geodesic spaces, $K_i$ a sequence of compact subset of $Z_i$ and
$f_i$ a sequence of Lipschitz function on $K_i$.
We assume that $(Z_i, z_i, K_i)$ converges to $(Z_{\infty}, z_{\infty}, K_{\infty})$ and that 
$\sup_i(\mathbf{Lip}f_i + |f_i|_{L^{\infty}}) < \infty$.
Then there exist a Lipschitz function $f_{\infty}$ on $K_{\infty}$ and a subsequence $\{n(i)\}$ such that 
$(Z_{n(i)}, z_{n(i)}, K_{n(i)}, f_{n(i)})$ converges to $(Z_{\infty}, z_{\infty},  K_{\infty}, f_{\infty})$.
\end{proposition}
\begin{proof}
We take a countable dense subset $\{x_j\}$ of $K_{\infty}$.
For every $x_j$, we take $x_j(i) \in K_i$ satisfying that $x_j(i)$ converges to $x_j$.
Then, there exists a subsequence $\{n(i)\}$ of $\mathbf{N}$ such that 
the sequence $\{f_{n(i)}(x_j(n(i)))\}$ is Cauchy sequence.
We define a function $F_{\infty}$ on $\{x_j\}$ by 
\[ F_{\infty}(x_j)= \lim_{i \rightarrow \infty}f_{n(i)}(x_j(n(i))).\]
It is easy to check that the function $F_{\infty}$ is $\sup_i \mathbf{Lip}f_i$-Lipschitz function.
Therefore there exists unique Lipschitz function $f_{\infty}$ on $K_{\infty}$ such that $F_{\infty}(x_j)=f_{\infty}(x_j)$.
It is easy to check that $f_{\infty}$ satisfies the assertion.
\end{proof}
We shall give a fundamental covering lemma (for proper metric spaces).
See chapter $1$ in \cite{Si} for the proof.
\begin{proposition}\label{cov}
Let $X$ be a proper metric space, $A$ a subset of $X$, $\Lambda$ a set, $\{x_{\lambda}\}_{\lambda \in \Lambda}$ a collection of points in $X$ and 
$\{r_{\lambda}\}_{\lambda \in \Lambda}$ a collection of positive numbers.
We assume that for every $x \in A$ and $\epsilon > 0$, there exists $\lambda \in \Lambda$ such that $x \in \overline{B}_{r_{\lambda}}(x_{\lambda})$ and $\mathrm{diam}\overline{B}_{r_{\lambda}}(x_{\lambda}) < \epsilon$.
Then, there exists a countable subset $\Lambda_1 \subset \Lambda$ satisfying the following properties:
\begin{enumerate}
\item $\{\overline{B}_{r_{\lambda_1}}(x_{\lambda_1})\}_{\lambda_1 \in \Lambda_1}$ are pairwise disjoint collection.
\item For every finite subset $\Lambda_2 \subset \Lambda_1$, we have 
\[A \setminus \bigcup_{\lambda_2 \in \Lambda_2}\overline{B}_{r_{\lambda_2}}(x_{\lambda_2}) \subset \bigcup_{\lambda \in \Lambda_1 \setminus \Lambda_2}\overline{B}_{5r_{\lambda}}(x_{\lambda}).\]
\end{enumerate} 
\end{proposition}
We shall recall the definition of measured Gromov-Hausdorff convergence by Fukaya, first.
Let $(Z, z_i) \rightarrow (Z_{\infty}, z_{\infty})$.
For a sequence of Radon measure $\upsilon_i$ on $Z_i$, we say that
$(Z_i, z_i, \upsilon_i)$ \textit{converges to $(Z_{\infty}, z_{\infty}, \upsilon_{\infty})$ in the sense of measured Gromov-Hausdorff topology} if  
\[\lim_{i \rightarrow \infty}\upsilon_i(B_r(x_i))= \upsilon_{\infty}(B_r(x_{\infty}))\]
for every $r>0$, $x_{\infty} \in Z_{\infty}$ and sequence $x_i \in Z_i$ satisfying $x_i \rightarrow x_{\infty}$.
Then we denote it by $(Z_i, z_i, \upsilon_i) \rightarrow (Z_{\infty}, z_{\infty}, \upsilon_{\infty})$.
We introduce a following fundamental result for precompactness of measured Gromov-Hausdorff topology.
See  \cite[Theorem $1.6$]{ch-co1} or \cite{fu}.
\begin{proposition}\label{measure}
Let $\{(Z_i, z_i, \upsilon_i)\}_i$ be a sequence of pointed proper geodesic spaces with Radon measure $\upsilon_i$.
We assume that $\upsilon_i(B_1(z_i))=1$ and that for every $R > 0$ there exists $K=K(R)\ge1$ such that 
$\upsilon_i(B_{2r}(x_i)) \le 2^K\upsilon_i(B_r(x_i))$
for every $0 < r < R$, $i \in \mathbf{N}$
and $x_i \in Z_i$.
Then, there exist a subsequence $\{(Z_{n(i)}, z_{n(i)}, \upsilon_i)\}$ and a pointed proper geodesic space with Radon measure $(Z_{\infty}, z_{\infty}, \upsilon_{\infty})$ such that 
$(Z_{n(i)}, z_{n(i)}, \upsilon_i)$ converges to $(Z_{\infty}, z_{\infty}, \upsilon_{\infty})$.
\end{proposition}
Next, we will give a relation between the measure of limit set and the limit of measures of sets: 
\begin{proposition}\label{sup}
Let $\{(Z_i, z_i, \upsilon_i)\}_i$ be a sequence of pointed proper geodesic spaces with Radon measure $\upsilon_i$ and $A_i$ a Borel subset of $Z_i$. 
We assume that $\upsilon_i(B_1(z_i))=1$, $A_{\infty}$ is compact, 
$(Z_i, z_i, \upsilon_i) \rightarrow (Z_{\infty}, z_{\infty}, \upsilon_{\infty})$,
$\limsup_{i \rightarrow \infty}A_i \subset A_{\infty}$ and that for every $R>0$
there exist $K=K(R)\ge1$ such that 
$\upsilon_i(B_{2r}(x_i)) \le 2^K\upsilon_i(B_r(x_i))$ 
for every $0 < r < R$, $i \in \mathbf{N}$ and $x_i \in Z_i$. 
Then we have 
\[\limsup_{i \rightarrow \infty}\upsilon_i(A_i) \le \upsilon_{\infty}(A_{\infty}).\]
\end{proposition}
\begin{proof}
By Proposition \ref{cov}, there exists a pairwise disjoint collection $\{\overline{B}_{r_j}(x_j)\}_{j \in \mathbf{N}}$ such that 
$x_j \in A_{\infty}$, $0 < r_j << 1$ and 
$A_{\infty}\setminus \bigcup _{i=1}^N\overline{B}_{r_i}(x_i) \subset \bigcup_{i=N+1}^{\infty}\overline{B}_{5r_i}(x_i)$
for every $N$.
We fix $\epsilon >0$.
We take $N$ satisfying 
$\sum_{i=N+1}^{\infty}\upsilon_{\infty}(B_{r_i}(x_i)) < \epsilon.$
By the assumption, we have 
$\sum_{i=N+1}^{\infty}\upsilon(B_{5r_i}(x_i)) < 2^{5K(1)}\epsilon.$
We consider an open covering $\{B_{s_i}(y_i)\} = \{B_{(1+\epsilon)r_i}(x_i)\}_{i=1, _{\cdots}, N} \cup \{B_{5(1+\epsilon)r_i}(x_i)\}_{i\ge N+1}$ of $A_{\infty}$.
By compactness of $A_{\infty}$, there exists a finite subcollection  $\{B_{t_i}(z_i)\}_{i=1, _{\cdots}, l}$ of $\{B_{s_i}(y_i)\}$,
such that $A_{\infty} \subset \bigcup_{i=1}^lB_{t_i}(z_i)$ and 
$|\upsilon_{\infty}(A_{\infty})-\sum_{i=1}^l\upsilon_{\infty}(B_{t_i}(z_i))|< \Psi(\epsilon;K).$
There exists $\tau_0>0$ such that $\tau_0 << \min \{t_j; 1 \le j \le l\}$ and that $B_{\tau_0}(A_{\infty}) \subset \bigcup_{i=1}^lB_{t_i}(z_i)$.
We take $\tau > 0$ and a sequence $z_i(j) \in Z_j$ satisfying that $\tau < \tau_0$ and 
that $z_i(j) \rightarrow z_i$.
Then since $\phi_i(A_i) \subset B_{\tau_0}(A_{\infty})$ for every sufficiently large $i$, it is easy to check that $A_i \subset \bigcup_{j=1}^lB_{t_j+\tau}(z_j(i))$ for every sufficiently large $i$.
Therefore we have 
$\upsilon_i(A_i)\le \sum_{j=1}^l\upsilon_i(B_{t_j+\tau}(z_j(i))).$
Thus,
\[\limsup_{i \rightarrow \infty}\upsilon_{\infty}(A_i) \le \sum_{j=1}^l\upsilon_{\infty}(B_{t_j+\tau}(z_j)).\]
By letting $\tau \rightarrow 0$ and $\epsilon \rightarrow 0$, we have the assertion. 
\end{proof}
\begin{proposition}\label{99990}
Let $\{(Z_i, z_i, \upsilon_i)\}_i$ be a sequence of pointed proper geodesic spaces with Radon measure $\upsilon_i$.
We assume that $\upsilon_i(B_1(z_i))=1$, $\mathrm{diam}Z_{\infty} > 0$, 
$(Z_i, z_i, \upsilon_i) \stackrel{(\phi_i, R_i, \epsilon_i)}{\rightarrow} (Z_{\infty}, z_{\infty}, \upsilon_{\infty})$
and that for every $R>0$, there exist $K=K(R) \ge1$ such that 
$\upsilon_i(B_{2r}(x_i)) \le 2^K\upsilon_i(B_r(x_i))$
for every $0 < r < R$, $i \in \mathbf{N}$ and $x_i \in Z_i$.
Then, we have 
\[\lim_{i \rightarrow \infty}\sup_{x_i \in B_{R}(z_i), 0 < r < R}\left|\upsilon_{i}(B_r(x_i))-\upsilon_{\infty}(B_r(\phi_i(x_i)))\right|=0\]
for every $R \ge 1$.
\end{proposition}
\begin{proof}
By the assumption, it is easy to check that $\mathrm{rad}Z_{\infty}>0$.
Here $\mathrm{rad}X=\inf _{x_2 \in X}(\sup _{x_1 \in X}\overline{x_1, x_2})$ for metric space $X$.
We put $K=K(100R)$.
We take $0<\tau << \mathrm{rad} Z_{\infty}$.
Then, by the definition,  there exists $N$ satisfying that
for every $N \le i \le \infty$ and $w \in Z_{i}$, there exists $\hat{w} \in Z_{i}$ such that 
$\overline{w, \hat{w}}=\tau$.
Since $B_{\hat{\tau}}(w) \subset B_{\tau + \hat{\tau}}(\hat{w})\setminus B_{\tau-\hat{\tau}}(\hat{w})$, by \cite[Lemma $3. 3$]{co-mi5} (or  \cite[Proposition $6.12$]{ch1}), 
there exists $0 < \hat{\tau}<<\tau$ such that
for every $N \le i \le \infty$, $w \in Z_i$ and $0 < t < \hat{\tau}$, we have
\[\upsilon_i(B_t(w))\le \Psi(t;K, R)\upsilon_i(B_{10{\tau}}(w)).\]
Therefore, for every $\epsilon >0$, there exists
$N_1 \in \mathbf{N}$ and $0 < r_1 << \min \{R, \hat{\tau}, \epsilon, 1\}$
such that for every $N_1 \le i \le \infty$,  $0 < s <r_1$ and  $z \in B_{R}(z_i)$, we have
$\upsilon_i(B_{s}(z))\le \epsilon.$
We take $\{x_j\}_{j=1, _{\cdots}, l} \subset B_{R}(z_{\infty})$ and $\{t_j\}_{j=1, _{\cdots}, \hat{l}} \subset
[0, R]$ satisfying
$B_{\hat{R}}(z_{\infty}) \subset \bigcup_{j=1}^lB_{\epsilon r_1}(x_j)$ and
 $[0, R] \subset \bigcup _{j=1}^{\hat{l}}B_{\epsilon r_1}(t_j).$
We take $x_j(i) \in B_{R}(z_i)$ satisfying that $x_j(i) \rightarrow x_j$.
There exists $N_2 \ge N_1$ such that  
$|\upsilon_i(B_{t_{\hat{j}}}(x_j(i)))-\upsilon_{\infty}(B_{t_{\hat{j}}}(x_j))|<\epsilon .$
for every $i \ge N_2$, $j=1, _{\cdots}, l$ and $\hat{j}=1, _{\cdots}, \hat{l}$.
Then, for every $z \in B_{R}(z_{\infty})$ and $s \in [r_1, R]$,
we take $j \in \{1, _{\cdots}, l\}$ and $\hat{j} \in \{1, _{\cdots}, \hat{l}\}$ satisfying 
$\overline{z, z_j} < \epsilon r_1$ and $|s-t_{\hat{j}}|<\epsilon r_1$.
Then by \cite[Lemma $3. 3$]{co-mi5},
\begin{align}
|\upsilon_{\infty}(B_s(z))-\upsilon_{\infty}(B_{t_{\hat{j}}}(x_j))|&\le \upsilon_{\infty}(B_{s+5\epsilon r_1}(z))-
\upsilon_{\infty}(B_{s-5\epsilon r_1}(z)) \\
&\le \Psi (\epsilon; K, R, \tau)\upsilon_{\infty}(B_{R}(z_{\infty})).
\end{align}
On the other hand, for a sequence $z(i) \in B_{R}(z_i)$ satisfying $z(i) \rightarrow z$,
\begin{align}
|\upsilon_i(B_s(z(i)))-\upsilon_i(B_{t_{\hat{j}}}(x_j(i)))|&\le
\upsilon_i(B_{s+10\epsilon r_1}(z(i)))-\upsilon_i(B_{s-10\epsilon r_1}(z(i))) \\
&\le \Psi(\epsilon; K, R, \tau)\upsilon_i(B_R(z_i)) \\
&\le \Psi(\epsilon; K, R, \tau)\upsilon_{\infty}(B_R(z_{\infty}))
\end{align}
for every $i \ge N_2$.
Thus, we have 
\[|\upsilon_i(B_s(z(i)))-\upsilon_{\infty}(B_s(z))|<\Psi(\epsilon ;K, R, \tau)\upsilon_{\infty}(B_R(z_{\infty})).\]
Therefore, we have the assertion.
\end{proof}
We remark that an assumption $\mathrm{diam}Z_{\infty}>0$ of Proposition \ref{99990} is necessary.
For example, consider a sequence $\mathbf{S}^n(r) \rightarrow \{p\}$ as $r \rightarrow 0$.
Here $\mathbf{S}^n(r)=\{x \in \mathbf{R}^{n+1}; |x|=r\}$.
\begin{proposition}\label{lap}
Let $\{(Z_i, z_i, \upsilon_i)\}_i$ be a sequence of pointed proper geodesic spaces with Radon measure $\upsilon_i$.
We assume that $\upsilon_i(B_1(z_i))=1$,
$(Z_i, z_i, \upsilon_i) \stackrel{(\phi_i, R_i, \epsilon_i)}{\rightarrow} (Z_{\infty}, z_{\infty}, \upsilon_{\infty})$
and that for every $R>0$, there exist $K=K(R) \ge1$ such that 
$\upsilon_i(B_{2r}(x_i)) \le 2^K\upsilon_i(B_r(x_i))$
for every $0 < r < R$, $i \in \mathbf{N}$
and $x_i \in Z_i$ 
Then  we have
\[\lim_{i \rightarrow \infty}\int_{Z_i}f \circ \phi_id\upsilon_i=\int_{Z_{\infty}}fd\upsilon_{\infty}\]
for every $f \in C_c^0(Z_{\infty})$.
\end{proposition}
\begin{proof}
We put $A=\mathrm{supp} f$ and fix $\epsilon >0$.
We take $R \ge1$ satisfying $A \subset B_R(z_{\infty})$ and put $K=K(100R)$.
For every $x \in Z_{\infty}$, we take $r(x) >0$ satisfying that for every $w \in B_{r(x)}(x)$, we have $f(w)=f(x) \pm \epsilon$.
By Proposition \ref{cov}, there exists a pairwise disjoint collection $\{\overline{B}_{r_i}(x_i)\}_i$ such that 
$x_i \in A$, $0<r_i<< \min \{r(x_i), \epsilon\}$ and  
$K \setminus \bigcup_{i=1}^N\overline{B}_{r_i}(x_i) \subset \bigcup_{i=N+1}^{\infty}\overline{B}_{5r_i}(x_i)$
for every $N$.
We take $N$ satisfying 
$\sum_{i=N+1}^{\infty}\upsilon_{\infty}(B_{r_i}(x_i)) < \epsilon.$
By the assumption, we have
$\sum_{i=N+1}^{\infty}\upsilon_{\infty}(B_{5r_i}(x_i)) < \Psi(\epsilon; K).$
We take $x_j(i) \in Z_i$ satisfying that $x_j(i) \rightarrow x_j$.
Then we have
\[ \int_{Z_i}f \circ \phi_i d\upsilon_i =\sum_{j=1}^N\int_{B_{r_j}(x_j(i))}f \circ \phi_i d\upsilon_i \pm \left|\int_{Z_i \setminus \bigcup
_{j=1}^NB_{r_j}(x_j(i))}f \circ \phi_i d\upsilon_i\right|.\]
We also have
\begin{align}
\left|\int_{Z_i \setminus \bigcup_{j=1}^NB_{r_j}(x_j(i))}f \circ \phi_i d\upsilon_i\right| &=\left|\int_{\phi_i^{-1}(A) \setminus \bigcup
_{j=1}^NB_{r_j}(x_j(i))}f \circ \phi_i d\upsilon_i\right| \\
&\le \sup |f| \upsilon_i(\phi_i^{-1}(A) \setminus \bigcup
_{j=1}^NB_{r_j}(x_j(i)) \\
&\le \sup |f| \upsilon_i(\overline{\phi_i^{-1}(A)} \setminus \bigcup
_{j=1}^NB_{r_j}(x_j(i))).
\end{align}
By Proposition \ref{sup}, we have 
\begin{align}
\limsup_{i \rightarrow \infty}\upsilon_i(\overline{\phi_i^{-1}(A)} \setminus \bigcup_{j=1}^NB_{r_j}(x_j(i)))
&\le \upsilon_{\infty}(A \setminus \bigcup
_{j=1}^NB_{r_j}(x_j)) \\
& \le \sum_{i=N+1}^{\infty}\upsilon_{\infty}(\overline{B}_{5r_j}(x_j)) \le \Psi(\epsilon;K).
\end{align}
Therefore for every sufficiently large $i$, we have 
\begin{align}
\int_{Z_i}f\circ \phi_i d\upsilon_i &= \sum_{j=1}^N(f(x_j)\pm \epsilon)\upsilon_i(B_{r_j}(x_j(i))) \pm \Psi(\epsilon;  K, \sup |f|) \\
& = \sum_{j=1}^N(f(x_j)\pm \epsilon)\upsilon_{\infty}(B_{r_j}(x_j))\pm \Psi(\epsilon;  K, \sup |f|, R) \\
& = \int_{\bigcup_{j=1}^NB_{r_j}(x_j)}fd\upsilon_{\infty} \pm \Psi(\epsilon;  K, \sup |f|) \\
& = \int_{Z_i}fd\upsilon_{\infty} \pm \left(\int_{A \setminus \bigcup_{j=1}^NB_{r_j}(x_j)} |f|d\upsilon_{\infty}+\Psi(\epsilon;  K, \sup |f|) \right)\\
& = \int_{Z_i}fd\upsilon_{\infty} \pm \Psi(\epsilon;  K, \sup |f|).
\end{align}
Therefore we have the assertion.
\end{proof}
In section $4$, we will generalize Proposition \ref{lap}.
See Proposition \ref{10103}.
\subsection{Riemannian manifolds and its limit space}
First, we shall introduce a very powerful gradient estimates for harmonic functions on manifolds by Cheng-Yau.
This estimate is used in this paper many times. We fix $n \ge 2$.
\begin{theorem}[Cheng-Yau, \cite{ch-yau}]\label{gradient1}
Let $K$ be a real number, $R$ a positive number, $M$ a complete $n$-dimensional Riemannian manifold with $\mathrm{Ric}_M \ge K(n-1)$, $m$ a point in $M$ and $f$ a nonnegative valued harmonic function on $B_R(m)$.
Then, we have
\[|\nabla f|(x)\le C(n)f(x)\frac{R(R|K(n-1)|+1)}{R^2-\overline{m,x}^2}\]
for every $x \in B_R(m)$.
\end{theorem}
Next theorem is a fundamental result for the study of Gromov-Hausdorff convergence of Riemannian manifolds:
\begin{theorem}[Bishop-Gromov, \cite{gr}]\label{BG}
Let $K$ be a real number, $M$ a complete $n$-dimensional Riemannian manifold with $\mathrm{Ric}_M \ge K(n-1)$ and $m$ a point in $M$.
Then we have 
\[\frac{\mathrm{vol}\,B_r(m)}{\mathrm{vol}\,B_r(\underline{p})}\ge \frac{\mathrm{vol}\,B_s(m)}{\mathrm{vol}\,B_s(\underline{p})}\]
for every $0 < r < s$.
Here, $\underline{p}$ is a point in the $n$-dimensional space form $\underline{M}_K^n$ whose sectional curvature is equal to $K$.
\end{theorem}
As a corollary of Theorem \ref{BG}, if a sequence of pointed $n$-dimensional complete Riemannian manifolds with renormalized volume $\{(M_i, m_i, \underline{\mathrm{vol}})\}$ 
 satisfy $\mathrm{Ric}_{M_i}\ge K(n-1)$, then the sequence satisfies the assumption of Proposition \ref{measure}.
Here renormalized volume means 
\[\underline{\mathrm{vol}}=\frac{\mathrm{vol}}{\mathrm{vol}\,B_1(m_i)}.\]
For a real number $K$ and a pointed proper geodesic space $(Y, y)$, in this paper, we say that $(Y, y)$ \textit{is $(n, K)$-Ricci limit space} if 
there exist a sequence of real numbers $\{K_i\}$ and a sequence of pointed $n$-dimensional complete Riemannian manifolds $\{(M_i, m_i)\}$ with $\mathrm{Ric}_{M_i}\ge K_i(n-1)$ such that $K_i \rightarrow K$ and $(M_i, m_i) \rightarrow (Y,y)$.
Then, we often call $(Y, y)$ a Ricci limit space of $\{(M_i, m_i)\}$.
Similarly, for a pointed proper geodesic space with Radon measure $(Y, y, \upsilon)$, we also say that 
$(Y, y, \upsilon)$ is $(n, K)$-Ricci limit space if 
there exist a sequence of real numbers $\{K_i\}$ and a sequence of pointed $n$-dimensional complete Riemannian manifolds $\{(M_i, m_i)\}$ with $\mathrm{Ric}_{M_i}\ge K_i(n-1)$ such that $K_i \rightarrow K$ and $(M_i, m_i, \underline{\mathrm{vol}}) \rightarrow (Y,y, \upsilon)$.
More simply, for $(n, -1)$-Ricci limit space $(Y, y)$ (or $(Y, y, \upsilon)$), we say that $(Y, y)$ is Ricci limit space.
See section $4.1$ in \cite{lo}.
We shall fix a Ricci limit space $(Y, y, \upsilon)$ in this subsection and give a very short review of structure theory of Ricci limit spaces developed by Cheeger-Colding below. See \cite{ch-co1, ch-co2, ch-co3} for the details.
\

We shall give an important notion called \textit{tangent cone} to study Ricci limit spaces:
For pointed proper geodesic spaces $(Z, z)$ and $(X, x)$, we say that \textit{$(Z, z)$ is a tangent cone of $X$ at $x$} if 
there exists a sequence of positive numbers $\{r_i\}$ such that $r_i \rightarrow 0$ and $(X, x, r_i^{-1}d_X) \rightarrow (Z, z)$. 
For $k \ge 1$, we put $\mathcal{R}_k(Y)=\{x \in Y;$ All tangent cones at $x$ are isometric to $\mathbf{R}^k\}$ and call it \textit{$k$-dimensional regular set}.
More simply, we shall denote it by $\mathcal{R}_k$.
We also put $\mathcal{R}=\bigcup_{1\le k \le n}\mathcal{R}_k$ and call it \textit{regular set}.
Next theorem is an important properties for Ricci limit spaces:
\begin{theorem}[Cheeger-Colding, \cite{ch-co1}]\label{regular}
We have $\upsilon(Y \setminus \mathcal{R})=0$.
\end{theorem}
For $\delta, r>0$ and $ 0 < \alpha < 1$, we put $(\mathcal{R}_k)_{\delta, r}=\{x \in Y; d_{GH}((\overline{B}_s(x), x), (\overline{B}_s(0_k), 0_k))\le \delta s$
for every $0 < s \le r\}$ and $(\mathcal{R}_{k;\alpha})_r=\{x \in Y; d_{GH}((\overline{B}_s(x), x), (\overline{B}_s(0_k), 0_k))\le s^{1+\alpha}$ for every $0 < s \le r\}$.
Here $0_k \in \mathbf{R}^k$.
By the definition, we remark that these set are closed.
It is easy to check that $\bigcap_{\delta > 0}(\bigcup_{r>0}(\mathcal{R}_k)_{\delta, r} )= \mathcal{R}_k$.
We also put $\mathcal{R}_{k;\alpha}= \bigcup_{r>0}(\mathcal{R}_{k;\alpha})_r$.
By \cite[Theorem $3. 23$]{ch-co1} and  \cite[Theorem $4.6$]{ch-co1}, there exists $0< \alpha(n)<1$ such that  $\upsilon(\mathcal{R}_{k} \setminus \mathcal{R}_{k;\alpha(n)})=0$, $\upsilon$ is Ahlfors $k$-regular at each point in $\mathcal{R}_{k;\alpha(n)}$ for every $k$.
Next, we shall introduce an important result for rectifiability and Poincar\'{e} inequality on Ricci limit spaces:
\begin{theorem}[Cheeegr-Colding, \cite{ch-co3}]\label{rectifiability}
$Y$ is $\upsilon$-rectifiable, $(Y, \upsilon)$ satisfies weak $(1, 2)$-Poincar\'{e} inequality.
\end{theorem}
More strongly, they proved that \textit{segment inequality} on Ricci limit spaces holds.
(We do not give the definition here. See  \cite[Theorem $2. 15$]{ch-co3}.)
Therefore we can construct the cotangent bundle $T^*Y$ of $Y$.
Finally, for cut loci on Ricci limit spaces, we also remark that $\upsilon(C_x)=0$ for every $x \in Y$.
See \cite[Theorem $3.2$]{ho}.
These results above are used in section $3$, essentially.
\section{Rectifiability on limit spaces}
In this section, we shall study a rectifiability of Ricci limit spaces.
These results given in this section are used in section $4$, essentially. 
\subsection{Radial rectifiability}
The main result in this subsection is Theorem \ref{7}.
\begin{lemma}\label{1}
Let $Z$ be a proper geodesic space, $z$ a point in $Z$, 
$s, \delta$ positive numbers, $\upsilon$ a Radon measure on $Z$ and
$F$ a nonnegative valued Borel function on $B_s(m)$.
We assume that there exists $K \ge 1$ such that for every
$w \in B_s(z)$ and $0<t\le s$, we have $0<\upsilon(B_{2t}(w))\le 2^K\upsilon(B_t(w))$,
\[\frac{1}{\upsilon(B_s(z))}\int_{B_s(z)}Fd\upsilon \le \delta.\]
Then, there exists a compact set $K \subset \overline{B}_{s/10^2}(z)$ such that 
$\upsilon (K)/\upsilon (B_{s/10^2}(z)) \ge 1- \Psi(\delta;K)$ and that 
for every $x \in K$ and $0 < t \le s/10^2$, 
\[\frac{1}{\upsilon(B_t(x))}\int_{B_t(x)}Fd\upsilon \le \Psi(\delta;K).\]
\end{lemma}
\begin{proof}
Without loss of generality, we can assume that $F$ is a nonnegative valued Borel function on $Z$ by
$F \equiv 0$ on $Z \setminus B_s(z)$.
We fix $C > 0$. We put $A_1(C) = \{w \in B_s(z); \int_{B_{s/10^2}(w)}F d \upsilon \ge C\upsilon(B_{s/10^2}(w))\}$ and 
take $x_1^1, _{\cdots}, x_{k_1}^1 \in A_1(C)$ which are an $s/10$-maximal separated subset of $A_1(C)$.
We also put $A_2(C) = \{w \in B_s(m) \setminus \bigcup_{i=1}^{k_1}B_s(x_i^1); \int_{B_{s/10^3}(w)}F d \upsilon \ge C\upsilon(B_{s/10^3}(w))\}$ and take $x_1^2, _{\cdots}, x_{k_2}^2 \in A_2(C)$ which are $s/10^2$-maximal separated subset of $A_1(C)$.
By iterating this argument, we put $A_l(C) = \{w \in B_s(m) \setminus \bigcup_{1 \le j \le l-1, \ 1 \le i \le k_j}B_{s/10^{l-2}}(x_i^{l-1}); \int_{B_{s/10^{l+1}}(w)}F d \mathrm{vol} \ge C\upsilon(B_{s/10^{l+1}}(w))\}$
and take $x_1^l, _{\cdots}, x_{k_l}^l \in A_l(C)$ which are  $s/10^l$-maximal separated subset of $A_l(C)$.
\begin{claim}\label{1.1}The collection $\{\overline{B}_{s/10^{l+1}}(x_i^l)\}$ are pairwise disjoint.
\end{claim}
We take $w \in \overline{B}_{s/10^{\hat{l}+1}}(x_{\hat{i}}^{\hat{l}}) \cap \overline{B}_{s/10^{l+1}}(x_i^l)$.
We assume that $l < \hat{l}$. Then, by the definition, we have $x_{\hat{i}}^{\hat{l}} \in M \setminus \bigcup_{j=1}^{k_l}
B_{s/10^{l-1}}(x_j^l)$.
Especially, we have $\overline{x_{\hat{i}}^{\hat{l}}, x_i^l} \ge s/10^{l-1}$.
Therefore, we have $\overline{B}_{s/10^{\hat{l}+1}}(x_{\hat{i}}^{\hat{l}}) \cap \overline{B}_{s/10^{l+1}}(x_i^l) = \emptyset$.
This is a contradiction.
Therefore, we have $l=\hat{l}$. By the definition, we have $i=\hat{i}$.
Thus, we have Claim \ref{1.1}.
\

It is easy to check the following claim.
\begin{claim}\label{1.2}
We have $\bigcup_{i \in \mathbf{N}}A_i(C) \subset \bigcup_{l \in \mathbf{N}, 1 \le i \le k_l}\overline{B}_{s/10^{l-2}}(x_i^l)$
\end{claim}
We have 
\begin{align}
\sum_{l \in \mathbf{N}, 1 \le i \le k_l} \int_{B_{\frac{s}{10^{l+1}}}(x_i^l)}Fd \upsilon &\ge 
C\sum_{l \in \mathbf{N}, 1 \le i \le k_l}\upsilon (B_{\frac{s}{10^{l+1}}}(x_i^l)) \\
&\ge C C(n)\sum_{l \in \mathbf{N}, 1 \le i \le k_l}\upsilon (B_{\frac{s}{10^{l-2}}}(x_i^l))
\ge C C(n)\upsilon (\bigcup_{l \in \mathbf{N}, 1 \le i \le k_l}B_{\frac{s}{10^{l-2}}}(x_i^l)).
\end{align}
On the other hand, 
\begin{align}
\sum_{l \in \mathbf{N}, 1 \le i \le k_l} \int_{B_{\frac{s}{10^{l+1}}}(x_i^l)}Fd \upsilon &=
\int _{\bigcup_{l \in \mathbf{N}, 1 \le i \le k_l}B_{\frac{s}{10^{l+1}}}(x_i^l)}F d \upsilon \\
&\le \int_{B_s(z)}Fd \upsilon \le C(n) \upsilon (B_s(z)) \delta.
\end{align}
Therefore, we have 
\[\frac{\upsilon(\bigcup_{l \in \mathbf{N}, 1 \le i \le k_l}B_{\frac{s}{10^{l-2}}}(x_i^l))}{\upsilon (B_s(m))} \le 
\frac{\delta}{C}C(n).\]
By taking $C=\sqrt{\delta}, K= \overline{B}_{s/10^2}(z) \setminus \bigcup_{l \in \mathbf{N}, 1 \le i \le k_l}B_{\frac{s}{10^{l-2}}}(x_i^l)$,
we have the assertion.
\end{proof}

\begin{definition}\label{2}
let $(Y, y, \upsilon)$ be a Ricci limit space, $k$ an integer satisfying $k \le n$ and 
$r, \delta$ positive numbers satisfying $r<1$ and $\delta < 1$. Let $(\mathcal{R}_k)_{\delta, r}^y$,
denote the set of points, $w \in Y$ such that for every $0 < s \le r$, there exists a map 
$\Phi$ from $\overline{B}_s(w)$ to $\mathbf{R}^k$ such that $\pi _1 \circ \Phi = r_y$ 
and that $\Phi$ gives an $\delta s$-Gromov-Hausdorff approximation to $\overline{B}_s(\Phi(w))$. 
Here, $\pi _1$ is the projection from $\mathbf{R}^k = \mathbf{R}\times \mathbf{R}^{k-1}$ to $\mathbf{R}$.
\end{definition}

\begin{lemma}\label{3}
We have
\[\bigcap_{\delta > 0}\left(\bigcup_{r>0}\left( (\mathcal{R}_k)_{\delta, r}^x \setminus C_x\right) \right) = \mathcal{R}_k
\setminus C_x.\]
\end{lemma}
\begin{proof}
It is easy to check that 
\[\bigcap_{\delta > 0}\left(\bigcup_{r>0}\left((\mathcal{R}_k)_{\delta, r}^x \setminus C_x\right)\right) \subset \mathcal{R}_k
\setminus C_x.\]
We take $w \in \mathcal{R}_k \setminus C_x$.
By the definition, for every $\delta > 0$, there exists $r > 0$ such that for every $0 < s < r$, there exists an 
$\delta s$-Gromov-Hausdorff approximation from $(\overline{B}_s(0_k), 0_k)$ to $(\overline{B}_s(w), w)$.
Here, $\overline{B}_s(0_k) \subset \mathbf{R}^k$.
On the other hand, by splitting theorem (see \cite[Theorem $9. 27$]{ch1}), there exist a pointed proper geodesic space $(W_s, w_s)$ and a map $\hat{\Phi}$
from $(\overline{B}_s(w), w)$ to $(\overline{B}_s(0, w_s), (0, w_s))$ such that 
$\pi_{\mathbf{R}} \circ \hat{\Phi} = r_x - \overline{x, w}$ and that $\hat{\Phi}$ gives an $\delta s$-Gromov-Hausdorff approximation. 
Here, $\overline{B}_s(0, w_s) \subset \mathbf{R} \times W_s$ with the product metric $\sqrt{d_{\mathbf{R}}^2 + d_{W_s}^2}$, 
$\pi_{\mathbf{R}}$ is the projection from $\mathbf{R} \times W_s$ to 
$\mathbf{R}$.
By rescaling $s^{-1}d_{\mathbf{R}^k}$ and \cite[Claim $4.4$]{ho2}, 
there exists an $\Psi (\delta; n)s$-Gromov-Hausdorff approximation $f$ from $(\overline{B}_s(w_s), w_s)$ to 
$(\overline{B}_s(0_{k-1}), 0_{k-1})$.
We define a map $g$ from $\overline{B}_s(w)$ to $\mathbf{R}^k$ by $g(z) = (\overline{x,z}, f \circ \hat{\Phi})$.
Let $\pi_s$ be the canonical retraction from $\mathbf{R}^k$ to $\overline{B}_s(g(w))$.
We put $\hat{g}= \pi_s \circ g$.
Then, it is easy to check that $\hat{g}$ gives a $\Psi(\delta; n)s$-Gromov-Hausdorff approximation to $(\overline{B}_s(\hat{g}(w)), \overline{g}(w))$. 
Since $\delta$ is arbitrary, we have the assertion.
\end{proof}
For every proper geodesic space $X$, a point $x \in X$ and a positive number $\tau > 0$, we put
\[\mathcal{D}_x^{\tau} = \{w \in X; \mathrm{\ There\ exists\ } \alpha \in X \mathrm{\ such \ that\ } \overline{\alpha, w} \ge \tau \mathrm{\ and\ } \overline{x, w} + \overline{w, \alpha} = \overline{x, \alpha}\}.\]
It is easy to check that $\mathcal{D}_x^{\tau}$ is a closed set.
By the definition, we have
\[\bigcup_{ \tau > 0} \mathcal{D}_x^{\tau} = X \setminus C_x.\]
\begin{lemma}\label{4}
Let $(Y, y, \upsilon)$ be a Ricci limit space, $k$ an integer satisfying $k \le n$, 
$\delta, r$ positive numbers satisfying $\delta < 1, r < 1$, $x$ a point in $Y$ and 
$w$ a point in $(\mathcal{R}_k)_{\delta, r}^x \cap \mathrm{Leb}((\mathcal{R}_k)_{\delta, r}) \setminus (C_x \cup \{x\})$.
Then, there exists $\eta(w) > 0$ satisfying the following property:
For every $0 < s \le \eta(w)$, there exist a compact set 
$L \subset \overline{B}_s(w) \cap (\mathcal{R}_k)_{\delta, r}$ and points 
$x_2, x_3, _{\cdots}, x_k \in Y$ such that $\upsilon (L)/\upsilon (B_s(w)) \ge 1 - \Psi(\delta;n)$ and that the map $\Phi = (r_x, r_{x_2}, _{\cdots}, r_{x_k})$ from $L$ to $\mathbf{R}^k$, 
gives $(1 \pm \Psi(\delta; n))$-bi-Lipschitz equivalent to the image $\Phi(L)$.
\end{lemma}
\begin{proof}
There exists $0 < \tau < r$ such that $\upsilon(B_s(w) \cap (\mathcal{R}_k)_{\delta, r})/\upsilon(B_s(w)) \ge 1-\delta$ for every $0 < s < \tau$ and $w \in \mathcal{D}_x^{\tau} \setminus B_{\tau}(x)$. 
Let $(M_i, m_i, \underline{\mathrm{vol}}) \rightarrow (Y, y, \upsilon)$.
We take $x_i, w_i \in M_i$ satisfying $w_i \rightarrow w$, $x_i \rightarrow x$.
We fix $0 < s << \min \{\delta, \tau\}$.
Then, for every sufficiently large $i$, there exists an $\delta s$-Gromov-Hausdorff approximation 
$\Phi^i = (\Phi_1^i, _{\cdots}, \Phi_k^i)$ from $(\overline{B}_{s}(w_i), w_i)$ to $(\overline{B}_{s}(0_k), 0_k)$
such that $\Phi_1^i = r_{x_i} - r_{x_i}(w_i)$.
We put $s_0 = \sqrt{\delta}s$.
For convenience, we shall use the following notations for rescaled metrics $s_0^{-1}d_{M_i}, s_0^{-1}d_{Y}$:
$\hat{\mathrm{vol}}=\mathrm{vol}^{s_0^{-1}d_{M_i}}$, $\hat{r}_w(\alpha)=s_0^{-1}r_w(\alpha)$,
$\hat{B}_t(\alpha)= B_t^{s_0^{-1}d_{M_i}}(\alpha)=B_{s_0t}(\alpha)$, $\hat{\upsilon}=\upsilon/\upsilon(B_{s_0}(y))$,
$\hat{g}=s_0^{-1}g$ for a Lipschitz function $g$ and so on.
We also denote the differential section of $g$ as rescaled manifolds $(M_i, s_0^{-1}d_{M_i})$ by $\hat{d}g : M_i \rightarrow T^*M_i$ 
and denote the Riemannian metric of $(M_i, s_0^{-1}d_{M_i})$ by $\langle \cdot, \cdot \rangle_{s_0}=s_0^{-2}\langle \cdot, \cdot \rangle$.
We remark that $(M_i, m_i, s_0^{-1}d_{M_i}, \underline{\mathrm{vol}}^{s_0^{-1}d_{M_i}}) \rightarrow (Y, y, s_0^{-1}d_{Y}, \hat{\upsilon})$.
The following claim follows from the proof of splitting theorem (see for instance \cite[Lemma $9. 8$]{ch1}, \cite[Lemma $9. 10$]{ch1} and \cite[Lemma $9. 13$]{ch1}).
\begin{claim}
For every sufficiently large $i$, there exist harmonic functions $\hat{\mathbf{b}}^i_j$ on $\hat{B}_{100^2}(w_i) (j=1, _{\cdots}, k)$, and 
points $x_j^i \in \hat{B}_{\sqrt{\delta}^{-1}}(w_i), (j=2, _{\cdots}, k))$ such that
$|\hat{\mathbf{b}}_j^i-\hat{r}_{x^i_j}|_{L^{\infty}(\hat{B}_{100^2}(w_i))} \le \Psi(\delta;n)$,
\[\frac{1}{\hat{\mathrm{vol}}\, \hat{B}_{100^2}(w_i)}\int_{\hat{B}_{100^2}(w_i)}
|\hat{d}\hat{\mathbf{b}}_j^i-\hat{d}
\hat{r}_{x_j^i}|_{s_0}^2d \hat{\mathrm{vol}} \le \Psi(\delta ; n),\]
\[\frac{1}{\hat{\mathrm{vol}}\,\hat{B}_{100^2}(w_i)}\int_{\hat{B}_{100^2}(w_i)}
|\langle \hat{d}\hat{\mathbf{b}}_j^i, 
\hat{\mathbf{b}}_l^i\rangle_{s_0}|d \hat{\mathrm{vol}} = \delta_{jl} \pm \Psi(\delta ; n)\]
and
\[\frac{1}{\hat{\mathrm{vol}}\,\hat{B}_{100^2}(w_i)}\int_{\hat{B}_{100^2}(w_i)}
|\mathrm{Hess}_{\hat{\mathbf{b}}_j^i}|_{s_0}^2d \hat{\mathrm{vol}} \le \Psi(\delta; n)\]
for $1 \le j, l \le k$.
Here $x=x_i^1$.
\end{claim}
We define a nonnegative Borel function $F_i$ on $\hat{B}_{100^2}(w_i)$ by
\[F_i =\sum_{l=1}^k\hat{\mathrm{Lip}}(\hat{\mathbf{b}}_l^i-\hat{r}_{x_l^i})^2
+\sum_{l \neq j}|\langle \hat{d}\hat{\mathbf{b}}_l^i, \hat{d}\hat{\mathbf{b}}_j^i\rangle_{s_0}|
+\sum_{l=1}^k(|\mathrm{Hess}_{\hat{\mathbf{b}}_l^i}|_{s_0})^2.\]
By Lemma \ref{1}, there exists a compact set $K_i \subset \hat{\overline{B}}_{100}(w_i)$ such that 
$\hat{\mathrm{vol}}\, K_i/\hat{\mathrm{vol}}\, \hat{B}_{100}(w_i) \ge 1- \Psi (\delta; n)$
and that for every $\alpha \in K_i$ and $0 < t < 100$, we have
\[\frac{1}{\hat{\mathrm{vol}}\,\hat{B}_t(\alpha)} \int _{\hat{B}_t(\alpha)}F_id \hat{\mathrm{vol}} \le \Psi (\delta;n).\]
\begin{claim}\label{101}
For every sufficiently large $i$, $\alpha \in K_i \cap \hat{B}_{50}(w_i)$ and $0 < t < 50$, there exists a 
constant $C_j^i (j=1, _{\cdots}, k)$ such that 
$\hat{\mathbf{b}}_j^i=\hat{r}_{x_j^i}+ C_j^i \pm \Psi(\delta;n)t$ on $\hat{B}_t(\alpha)$ for $j=1, _{\cdots}, k$.
\end{claim}
The proof is as follows.
By Poincar\'{e} inequality, we have 
\begin{align}
&\frac{1}{\hat{\mathrm{vol}}\,\hat{B}_t(\alpha)}\int_{\hat{B}_t(\alpha)}
\left|(\hat{\mathbf{b}}_j^i-\hat{r}_{x_j^i})- \frac{1}{\hat{\mathrm{vol}}\,\hat{B}_t
(\alpha)}\int_{\hat{B}_t(\alpha)}
(\hat{\mathbf{b}}_j^i-\hat{r}_{x_j^i})d\hat{\mathrm{vol}}\right|d \hat{\mathrm{vol}} \\
&\le t C(n) \sqrt{\frac{1}{\hat{\mathrm{vol}}\,\hat{B}_t(\alpha)}
\int _{\hat{B}_t(\alpha)} (\hat{\mathrm{Lip}}
(\hat{\mathbf{b}}_1^i-\hat{r}_{x_i}))^2d \hat{\mathrm{vol}}} \\
&\le t \Psi(\delta; n).
\end{align}
For $C>0$, Let $A_j(C)$, denote the set of points $\beta \in \hat{B}_t(\alpha)$, such that 
\[\left|(\hat{\mathbf{b}}_j^i(\beta)-\hat{r}_{x_j^i}(\beta))-\frac{1}{\hat{\mathrm{vol}}\,\hat{B}_t(\alpha)}
\int_{\hat{B}_t(\alpha)}(\hat{\mathbf{b}}_j^i-\hat{r}_{x_j^i})d
\hat{\mathrm{vol}}\right|\ge C.\]
Then, we have
\begin{align}
\Psi (\delta; n)t &\ge \frac{1}{\hat{\mathrm{vol}}\,\hat{B}_t(\alpha)}\int _{\hat{B}_t(\alpha)}\left|(\hat{\mathbf{b}}_j^i-\hat{r}_{x_j^i})-\frac{1}{\hat{\mathrm{vol}}\,\hat{B}_t(\alpha)}
\int_{\hat{B}_t(\alpha)}(\hat{\mathbf{b}}_j^i-\hat{r}_{x_j^i})d
\mathrm{vol}\right|d \hat{\mathrm{vol}} \\
&\ge \frac{1}{\hat{\mathrm{vol}}\,\hat{B}_t(\alpha)} \int 
_{A_j(C)}\left|(\hat{\mathbf{b}}_j^i-\hat{r}_{x_j^i})-\frac{1}{\hat{\mathrm{vol}}\,\hat{B}_t(\alpha)}
\int_{\hat{B}_t(\alpha)}(\hat{\mathbf{b}}_j^i-\hat{r}_{x_j^i})d
\hat{\mathrm{vol}}\right|d \hat{\mathrm{vol}} \\
&\ge C\frac{\hat{\mathrm{vol}}\,A_j(C)}{\hat{\mathrm{vol}}\,\hat{B}_t(\alpha)}.
\end{align}
Therefore, for above $\Psi(\delta;n)$, if we put $C= \sqrt{\Psi(\delta;n)}t$, then we have
\[\frac{\hat{\mathrm{vol}}\,A_j(C)}{\hat{\mathrm{vol}}\,\hat{B}_t(\alpha)}\le \sqrt{\Psi(\delta;n)}.\]
Here, we assume that $\hat{B}_{\epsilon t}(\beta) \subset A_j(C)$ for some $\beta \in \hat{B}_t(\alpha)$ and $\epsilon > 0$.
Then, by Bishop-Gromov volume comparison theorem, we have 
\[C(n) \epsilon^n \le \frac{\hat{\mathrm{vol}}\,B_{\epsilon t}(\beta)}{\hat{\mathrm{vol}}\,\hat{B}_t(\alpha)} 
\le \frac{\hat{\mathrm{vol}}\,A_j(C)}{\hat{\mathrm{vol}}\,\hat{B}_t(\alpha)} \le \sqrt{\Psi(\delta;n)}.\]
Therefore, for $C(n)$ above, if we take $\epsilon = \left(2C(n)^{-1}\sqrt{\Psi(\delta;n)}\right)^{1/n}$, then we have a contradiction.
We put $\epsilon=\left(2C(n)^{-1}\sqrt{\Psi(\delta;n)}\right)^{1/n}$.
We take $\beta \in \hat{B}_t(\alpha)$.
We also take $\hat{\beta} \in \hat{B}_{(1-\epsilon)t}(\alpha)$ satisfying $\hat{r}_{\beta}(\hat{\beta}) < \epsilon t$.
Then, there exists $\gamma \in \hat{B}_{\epsilon t}(\hat{\beta}) \setminus A_j(C)$.
Thus, we have $\gamma \in \hat{B}_t(\alpha)$.
By the definition of $A_j(C)$, we have 
\[\hat{\mathbf{b}}^i_j(\gamma) = \hat{r}_{x_j^i}(\gamma) + 
\frac{1}{\hat{\mathrm{vol}}\,\hat{B}_{100}(\alpha)}\int_{\hat{B}_{100}(\alpha)}
(\hat{\mathbf{b}}^i_j-\hat{r}_{x_j^i})d \hat{\mathrm{vol}} \pm \sqrt{\Psi(\delta;n)}t.\]
By Cheng-Yau's gradient estimate, we have $|\hat{\nabla} \hat{\mathbf{b}}^i_j|_{s_0}\le C(n)$.
Thus, we have 
\[\hat{\mathbf{b}}^i_j(\beta)= \hat{r}_{x_j^i}(\beta) + 
\frac{1}{\hat{\mathrm{vol}}\,\hat{B}_{100}(\alpha)}\int_{\hat{B}_{100}(\alpha)}
(\hat{\mathbf{b}}^i_j-\hat{r}_{x_j^i})d \hat{\mathrm{vol}} \pm \Psi(\epsilon;n)t.\]
Therefore we have Claim \ref{101}.
\

By an argument similar to the proof of \cite[Theorem $3.3$]{ch-co3}, we have the following:
\begin{claim}\label{102}
For every sufficiently large $i$,
$\alpha \in K_i \cap \hat{B}_{50}(w_i)$ and 
$0<t\le10^{-5}$, there exist a compact set $Z_t \subset M_i$, a point $z_t \in Z_t$ and a map $\phi$ from 
$(\hat{\overline{B}}_t(\alpha), \alpha)$ to $(\hat{\overline{B}}_t(z_t), z_t)$ such that the map $\Phi=(\hat{\mathbf{b}}^i_1, _{\cdots}, \hat{\mathbf{b}}^i_k, \phi)$ from 
$\hat{\overline{B}}_t(\alpha)$ to $\hat{\overline{B}}_{t + \Psi (\delta;n)t)}(\Phi(\alpha)) \subset (\mathbf{R}^k \times Z_t, \sqrt{d_{\mathbf{R}^k}^2 + ({s_0}^{-1}d_{M_i})^{2}})$, gives $\Psi(\delta;n)t$-Gromov-Hausdorff
approximation.
\end{claim}
We put $\hat{K}_i = K_i \cap \hat{\overline{B}}_{40}(w_i)$.
Then, we have $\hat{\mathrm{vol}}\,K_i / 
\hat{\mathrm{vol}}\,\hat{B}_{40}(w_i) \ge 1 - \Psi(\delta; n).$
By Proposition \ref{compact}, without loss of generality, we can assume that 
there exist a compact set $K_{\infty} \subset \hat{\overline{B}}_{40}(w)$ 
and points $x_j^{\infty} \in Y (2 \le j \le k)$ such that 
$x_j^i \rightarrow x_j^{\infty}$ and $K_i \rightarrow K_{\infty}$.
By Proposition \ref{sup}, we have $\hat{\upsilon}(K_{\infty})/ \hat{\upsilon}(\hat{B}_{40}(w)) \ge 1- \Psi(\delta;n).$

On the other hand, by Proposition \ref{090}, Claim \ref{101} and \ref{102}, 
for every $\alpha \in K_{\infty}$ and $0< t \le 10^{-5}$, there exist a compact metric space $Z_{\infty}$, 
a point $z_{\infty} \in Z_{\infty}$ and a map $\phi$ from 
$(\hat{\overline{B}}_t(\alpha), \alpha)$ to $(\overline{B}_t(z_{\infty}), z_{\infty})$
such that the map $\hat{\phi}=(\hat{r}_{x}, \hat{r}_{x_2^{\infty}}, _{\cdots}, \hat{r}_{x_k^{\infty}}, \phi)$  from $\hat{\overline{B}}_t(\alpha)$ to 
$\hat{\overline{B}}_{t + \Psi(\delta;n)t}(\hat{\phi}(\alpha))$, 
gives an
$\Psi(\delta; n)t$-Gromov-Hausdorff approximation.
\

We put $\hat{K_{\infty}} = K_{\infty} \cap (\mathcal{R}_k)_{\delta, r} \cap 
\overline{B}_{10^{-10}s_0}(w)$.
Then, we have 
$\upsilon(\hat{K}_{\infty})/\upsilon(\overline{B}_{10^{-10}s_0}(w)) \ge 1- \Psi (\delta;n)$.
On the other hand, for every $\alpha \in \hat{K}_{\infty}$ and $0<t\le10^{-5}$, if we take $\phi, Z_{\infty}, z_{\infty}$ as above, then, since $\alpha \in (\mathcal{R}_k)_{\delta, r}$,
we have $\mathrm{diam}Z_{\infty} \le \Psi(\delta;n)t$.
Especially, the map $f=(\hat{r}_x, \hat{r}_{x_2^{\infty}}, _{\cdots}, \hat{r}_{x_k^{\infty}})$ from $\hat{\overline{B}}_t(\alpha)$ to 
$\overline{B}_{t+\Psi(\delta;n)t}(f(\alpha))$, 
gives an $\Psi(\delta;n)t$-Gromov-Hausdorff approximation.
Especially, for every $\alpha, \beta \in \hat{K}_{\infty}$ satisfying $\alpha \neq \beta$, if we put $0 < t = 
\hat{r}_{\alpha}(\beta) \le 10^{-5}$, then we have
\begin{align*}
\sqrt{(\overline{x, \alpha}^{s_0^{-1}d_Y}-\overline{x, \beta}^{s_0^{-1}d_Y})^2+ \sum_{l=2}^k(\overline{x_l^{\infty}, \alpha}^{s_0^{-1}d_Y}-\overline{x_l^{\infty}, \beta}^{s_0^{-1}d_Y})^2}
&= \overline{\alpha, \beta}^{s_0^{-1}d_Y} \pm \Psi(\delta;n)t \\
&= (1 \pm \Psi(\delta;n))\overline{\alpha, \beta}^{s_0^{-1}d_Y}.
\end{align*} 
Therefore, we have the assertion.
\end{proof}

\begin{lemma}\label{5}
Let $(Y, y, \upsilon)$ be a Ricci limit space and $x$ a point in $Y$.
Then, there exist a collection of compact subsets $\{C_{k, i}^x\}_{1 \le k \le n, i \in \mathbf{N}}$ of $Y$ and a collection of points $\{x_{k, i}^l\}_{2 \le l \le k \le n, i \in \mathbf{N}} \in Y$ 
satisfying the following properties:
\begin{enumerate}
\item $\bigcup_{i \in \mathbf{N}}C_{k, i}^x \subset \mathcal{R}_k$ for every $k$.
\item $\upsilon(\mathcal{R}_k \setminus \bigcup_{i \in \mathbf{N}}C_{k, i}^x) = 0$ for every $k$.
\item For every $z \in \bigcup_{i \in \mathbf{N}}C_{k, i}^x$ and $0 < \delta < 1$, there exists $C_{k, i}^x$ such that $z \in C_{k, i}^x$ 
and the map $\Phi_{k, i}^x = (r_x, r_{x_{k, i}^2}, _{\cdots}, r_{x_{k, i}^k})$ from $C_{k, i}^x$ to $\mathbf{R}^k$ gives
$(1 \pm \delta)$-bi-Lipschitz equivalent to the image $\Phi_{k, i}^x(C_{k, i}^x)$.
\end{enumerate}
\end{lemma}
\begin{proof}
We put 
\[A_k = \bigcap _{m_1 \in \mathbf{N}}\left(\bigcup _{m_2 \in \mathbf{N}}
(\mathcal{R}_k)_{1/m_1, 1/m_2}^x \cap \mathrm{Leb}((\mathcal{R}_k)_{1/m_1, 1/m_2}) \setminus (C_x \cup \{x\})\right).\]
\begin{claim}\label{a1}We have $A_k \subset \mathcal{R}_k$ and  $\upsilon (\mathcal{R}_k \setminus A_k) = 0$.
\end{claim}
The proof is as follows.
For 
\[B_k= \bigcap _{m_1 \in \mathbf{N}}\left(\bigcup_{m_2 \in \mathbf{N}}(\mathcal{R}_k)_{1/m_1, 1/m_2}^x
\cap (\mathcal{R}_k)_{1/m_1, 1/m_2} \setminus (C_x \cup \{x\})\right),\]
by Proposition \ref{130}, we have, $A_k \subset B_k$, $\upsilon(B_k \setminus A_k)=0$.
On the other hand, by Lemma \ref{3}, we have $B_k = \mathcal{R}_k \setminus \{C_x \cup \{x\} \}$.
Since $\upsilon (C_x)=0$, we have Claim \ref{a1}. 
\

For every $z \in A_k$ and $N \in \mathbf{N}$, we take $m_2=m_2(z, N)$ satisfying
$z \in (\mathcal{R}_k)_{1/N, 1/m_2}^x \cap \mathrm{Leb}((\mathcal{R}_k)_{1/N, 1/m_2}) \setminus (C_x \cup \{x\})$.
By Lemma \ref{4}, there exists $\eta(z, N) >0$ such that for every $0 < s \le \eta(z, N)$, there exist a compact set
$L(z, s, N) \subset \overline{B}_s(z) \cap (\mathcal{R}_k)_{1/N, 1/m_2}$
and points $x_2(z, s, N), _{\cdots}, x_k(z, s, N) \in Y$ such that 
$\upsilon(L(z, s, N))/\upsilon(\overline{B}_s(z)) \ge 1- \Psi(N^{-1}; n)$ and that 
the map $\Phi_{z, s, N}(w)=(\overline{x, w}, \overline{x_2(z, s, N), w}, _{\cdots}, \overline{x_k(z, s, N),w})$
 from $L(z, s, N)$ to $\mathbf{R}^k$, gives  $(1\pm \Psi(N^{-1};n))$-bi-Lipschitz equivalent to the image $\Phi_{z, s, N}(L(z, s, N))$.
We fix $R > 1$.
By Lemma \ref{cov}, there exists pairwise disjoint collection $\{\overline{B}_{s_i^{N. R}}(z_i^{N, R})\}_{i \in \mathbf{N}}$ such that
$z_i^{N, R} \in A_k \cap \overline{B}_R(y)$, $0<s_i^{N, R} \le \eta(z_i^{N, R}, N)/100$ and that 
$A_k \cap \overline{B}_R(y) \setminus \bigcup_{i=1}^m\overline{B}_{s_i^{N, R}}(z_i^{N, R}) \subset
\bigcup_{i=m+1}^{\infty}\overline{B}_{5s_i^{N, R}}(z_i^{N, R})$ for every $m$.
We put $\hat{L}(i, N, R) = L(z_i^{N, R}, 5s_i^{N, R}, N) \cap A_k \cap \overline{B}_R(y) \subset A_k \cap \overline{B}_R(y)$.
\begin{claim}\label{103}
$\upsilon(A_k \cap \overline{B}_R(y) \setminus \bigcup_{N \ge N_0, i \in \mathbf{N}}\hat{L}(i, N, R)) = 0$
for every $N_0 \in \mathbf{N}$.
\end{claim}
Because, for every $N \ge N_0$, we have
\begin{align}
&\upsilon \left( A_k \cap \overline{B}_R(y) \setminus \bigcup _{i \in \mathbf{N}}\hat{L}(i, N, R)\right) \\
&\le \upsilon \left( \bigcup_{i \in \mathbf{N}}\left( \overline{B}_{5s_i^{N, R}}(z_i^{N, R}) \cap A_k \cap \overline{B}_R(y)\right)
\setminus \bigcup_{i \in \mathbf{N}}\left(L(z_i^{N, R}, 5s_i^{N, R}, N) \cap A_k \cap \overline{B}_R(y)\right)\right) \\
&\le \sum_{i \in \mathbf{N}}\upsilon(\overline{B}_{5s_i^{N, R}}(z_i^{N, R}) \setminus L(z_i ^{N, R}, 5s_i^{N, R}, N)) \\
&\le \Psi (N^{-1};n)\sum_{i \in \mathbf{N}}\upsilon (\overline{B}_{5s_i^{N, R}}(z_i^{N,R})) \\
&\le \Psi( N^{-1};n) \sum_{i \in \mathbf{N}}\upsilon(B_{s_i^{N,R}}(z_i^{N, R})) \\
&\le \Psi (N^{-1};n)\upsilon(B_{2R}(y)).
\end{align}
Therefore, by letting $N \rightarrow \infty$, we have Claim \ref{103}.

\

By Claim \ref{103}, we have $\upsilon(A_k \cap \overline{B}_R(y) \setminus \bigcap_{N_0}(\bigcup_{N \ge N_0, i \in \mathbf{N}}\hat{L}(i, N, R))) = 0$.
We put $E(i, N, R) = \hat{L}(i, N, R) \cap \bigcap _{N_0 \in \mathbf{N}}(\bigcup_{N \ge N_0, j \in \mathbf{N}}\hat{L}(j, N, R))$.
Then, we have $\upsilon(A_k \cap \overline{B}_R(y) \setminus \bigcup_{i, N \in \mathbf{N}}E(i, N, R))=0$.
For every $z \in \bigcup_{i, N \in \mathbf{N}}E(i, N, R)$ and $0 < \delta < 1$, we take $i, N \in \mathbf{N}$ satisfying
$z \in E(i, N, R)$.
We also take $N_0 \in \mathbf{N}$ satisfying $N_0^{-1}<<\delta$.
Then there exist $\hat{N} \ge N_0$ and $\hat{i} \in \mathbf{N}$ such that
$z \in \hat{L}(\hat{i}, \hat{N}, R)$.
Then, the map $\phi(w) = (\overline{x, w}, \overline{x_2(z_{\hat{i}}^{\hat{N}, R}, s_{\hat{i}}^{\hat{N}, R}), w}, _{\cdots}, \overline{x_k(z_{\hat{i}}^{\hat{N}, R}, s_{\hat{i}}^{\hat{N}, R}), w})$ from $L(z_{\hat{i}}^{\hat{N}, R}, s_{\hat{i}}^{\hat{N}, R}, \hat{N})$ to $\mathbf{R}^k$, gives 
$\Psi (N^{-1}, n)$-bi-Lipschitz equivalent to the image.
Especially, the map gives $(1 \pm \delta)$-bi-Lipschitz equivalent to the image.
We remark that $\hat{L}(\hat{i}, \hat{N}, R) \subset L(z_{\hat{i}}^{\hat{N}, R}, s_{\hat{i}}^{\hat{N}, R}, \hat{N})$ and 
$z \in \hat{L}(\hat{i}, \hat{N}, R) \cap \bigcap_{l \in \mathbf{N}}(\bigcup_{j \ge l, p \in \mathbf{N}}\hat{L}(p, j, R)) =
E(\hat{i}, \hat{N}, R)$.
Therefore, if we put $x_2(i, N, R) = x_2(z_i^{N, R}, s_i^{N, R}, R), _{\cdots}, x_k(i, N, R) = x_k(z_i^{N, R}, s_i^{N, R}, R)$, then
we have the following claim:
\begin{claim}\label{104}
For every $z \in \bigcup_{i, N \in \mathbf{N}}E(i, N, R)$ and $0 < \delta < 1$, there exists $E(i, N, R)$ such that
$z \in E(i, N, R)$ and that the map $\phi(w)= (\overline{x, w}, \overline{x_2(i, N, R), w}, _{\cdots}, \overline{x_k(i, N, R), w})$ from $E(i, N, R)$ to $\mathbf{R}^k$, gives $(1 \pm \delta)$-bi-Lipschitz equivalent to the image.
\end{claim}
By Claim \ref{104}, it is easy to check the assertion.
\end{proof}

\begin{lemma}\label{6}
With same notaion as in Lemma \ref{5}, let $\{\mathcal{D}_{k, i, j}^x\}_{j \in \mathbf{N}}$ be a collection of Borel subsets of $C_{k, i}^x$ satisfying 
 $\upsilon(C_{k, i}^x \setminus \bigcup_{j \in \mathbf{N}}\mathcal{D}_{k, i, j}^x)=0$.
Then, there exists a collection of Borel subsets $\{\mathcal{E}_{k, i, j}^x\}$ such that 
$\mathcal{E}_{k, i, j}^x \subset \mathcal{D}_{k, i, j}^x$, 
$\upsilon (\mathcal{D}_{k, i, j}^x \setminus \mathcal{E}_{k, i, j}^x) = 0$ and that for every $k$, 
$z \in \bigcup_{i, j \in \mathbf{N}}\mathcal{E}_{k, i, j}^x$ and $0 < \delta < 1$, 
there exists $\mathcal{E}_{k, i, j}^x$ such that $z \in \mathcal{E}_{k, i, j}^x$ and 
that the map $\Phi_{k, i, j}^x = (r_x, r_{x_{k, i}^2}, _{\cdots}, r_{x_{k, i}^k})$ from 
$\mathcal{E}_{k, i, j}^x$ to $\mathbf{R}^k$ gives $(1 \pm \delta)$-bi-Lipschitz equivalent to the image 
$\Phi_{k, i, j}^x(\mathcal{E}_{k, i, j}^x)$.
\end{lemma}
\begin{proof}
We fix $1\le k \le n$.
For every $M \in \mathbf{N}$, we put 
$\mathcal{B}_M=\{i  \in \mathbf{N}; $ the map $\phi = (r_x, r_{x_{k, i}^2}, 
_{\cdots}, r_{x_{k, i}^k})$ from $C_{k, i}^x$ to $\mathbf{R}^k$, gives $(1 \pm M^{-1})$-bi-Lipschitz equivalent to the image $\}$
and $\mathcal{E}_{k, i, j}^x= \mathcal{D}_{k,i,j}^x \cap \bigcap_{M \in \mathbf{N}}(\bigcup_{i \in \mathcal{B}_M, j  \in \mathbf{N}}\mathcal{D}_{k,i,j}^x).$
\begin{claim}\label{101010101}$\upsilon(\mathcal{D}^x_{k, i, j} \setminus \mathcal{E}_{k, i, j}^x)=0.$
\end{claim}
Because, by Lemma \ref{5}, we have
$\bigcup_{i \in \mathbf{N}}C_{k, i}^x \subset \bigcap_{M \in \mathbf{N}}(\bigcup_{i \in \mathcal{B}_M}C_{k, i}^x)).$
On the other hand,
it is easy to check that 
$\bigcap_{M \in \mathbf{N}}(\bigcup_{i \in \mathcal{B}_M}C_{k, i}^x)) \subset 
\bigcup_{i \in \mathbf{N}}C_{k, i}^x .$
Therefore, we have $\bigcap_{M \in \mathbf{N}}(\bigcup_{i \in \mathcal{B}_M}C_{k, i}^x)) = 
\bigcup_{i \in \mathbf{N}}C_{k, i}^x .$
Thus,
$\upsilon(\mathcal{D}_{k, i, j}^x \setminus \mathcal{E}_{k, i, j}^x)=\upsilon(\mathcal{D}_{k, i, j}^x \cap \bigcup_{l \in \mathbf{N}}C_{k, l}^x \setminus \mathcal{E}_{k, i, j}^x) = \upsilon (\mathcal{D}_{k, i, j}^x \cap \bigcap_{M \in \mathbf{N}}(\bigcup_{l \in \mathcal{B}_M}C_{k, l}^x) \setminus \mathcal{E}_{k, i, j}^x) = \upsilon (\mathcal{D}_{k, i, j}^x \cap \bigcap_{M \in \mathbf{N}}(\bigcup_{l \in \mathcal{B}_M, j \in \mathbf{N}}\mathcal{D}_{k, l, j}^x) \setminus \mathcal{E}_{k, i, j}^x)=0$.
Therefore we have Claim \ref{101010101}.
\
\begin{claim}\label{p0}
For every $z \in \bigcup_{i, j \in \mathbf{N}}\mathcal{E}_{k, i, j}^x$ and $0 < \delta <1$, 
there exists $\mathcal{E}_{k, i, j}^x$ such that $z \in \mathcal{E}_{k, i, j}^x$ and that the map
$\phi$ from $\mathcal{E}_{k, i, j}^x$ to $\mathbf{R}^k$ defined by $\phi = (r_x, r_{x_{k, i}^2}, _{\cdots}, r_{x_{k, i}^k})$ gives $(1 \pm \delta)$-bi-Lipschitz equivalent to the image.
\end{claim}
Because, we take $M \in \mathbf{N}$ and $i, j \in \mathbf{N}$ satisfying $ M^{-1} << \delta$ and  
$z \in \mathcal{E}_{k, i, j}^x$.
By the definition, there exist $N_0 \in \mathcal{B}_M$ and $N_1 \in \mathbf{N}$ such that $z \in \mathcal{D}_{k, N_0, N_1}^x$.
Therefore, we have 
$z \in \mathcal{D}_{k, N_0, N_1}^x \cap \bigcap_{\hat{M} \in \mathbf{N}}(\bigcup_{\hat{i} \in \mathcal{B}_{\hat{M}}, 
\hat{j} \in \mathbf{N}}\mathcal{D}_{k, \hat{i}, \hat{j}}^x) = \mathcal{E}_{k, N_0, N_1}^x$
and the map $\phi = (r_x, r_{x_{k, j}^2}, _{\cdots}, r_{x_{k, j}^k})$ from $\mathcal{E}_{k, N_0, N_1}^x$ to $\mathbf{R}^k$,
gives $(1 \pm M^{-1})$-bi-Lipschitz equivalent to the image.
Therefore, we have Claim \ref{p0}.
\

Thus, we have the assertion.
\end{proof}
The following theorem is the main result in this subsection.
See Appendix $7.4$ or $(2.2)$ in \cite{ch-co2} or \cite[Definition $4.1$]{ho} for the definition of the measure $\upsilon_{-1}$.
\begin{theorem}[Radial rectifiability]\label{7}
Let $(Y, y, \upsilon)$ be a Ricci limit space satisfying $Y \neq \{y\}$ and $x$ a point in $Y$.
Then, there exist a collection of Borel subsets $\{C_{k, i}^x\}_{1 \le k \le n, i \in \mathbf{N}}$ of $Y$, a collection of points $\{x_{k, i}^l\}_{2\le l \le k \le n, i \in \mathbf{N}}$ of $Y$, 
a positive number $0 < \alpha(n) < 1$ 
and a Borel subset $A$ of $[0, \mathrm{diam}Y)$ such that  the following properties hold:
\begin{enumerate}
\item $\bigcup_{i \in \mathbf{N}}C_{k, i}^x \subset \mathcal{R}_{k, \alpha(n)} \setminus C_x$.
\item $\upsilon (\mathcal{R}_k \setminus \bigcup_{i \in \mathbf{N}}C_{k, i}^x) = 0$.
\item For every $C_{k, i}^x$ and $z \in C_{k, i}^x$, we have $\lim_{r \rightarrow 0} \upsilon(B_r(z) \cap C_{k, i}^x)
/\upsilon(B_r(z))=1$.
\item For every $C_{k, i}^x$, there exists $A_{k, i}^x > 1$ such that
$(A_{k, i}^x)^{-1} \le \upsilon(B_r(z))/r^k \le A_{k, i}^x$ holds  for every $z \in C_{k, i}^x$ and $0 < r < 1$.
\item The limit measure $\upsilon$ and $k$-dimensional Hausdorff measure $H^k$ are mutually absolutely continuous on $C_{k, i}^x$.
\item For every $z \in \bigcup_{i \in \mathbf{N}}C_{k, i}^x$ and $0 < \delta < 1$, there exists $C_{k, i}^x$ such that 
$z \in C_{k, i}^x$ and that the map $\Phi_{k, i}^x = (r_x, r_{x_{k, i}^2}, _{\cdots}, r_{x_{k, i}^k})$ from $C_{k, i}^x$ to
$\mathbf{R}^k$ gives $(1 \pm \delta)$-bi-Lipschitz equivalent to the image $\Phi_{k, i}^x(C_{k, i}^x)$.
\item $H^1([0, \mathrm{diam}Y) \setminus A)=0$.
\item For every $R \in A$, the collection $\{\partial B_R(x) \cap C_{k, i}^x\} \subset \partial B_R(x) \setminus C_x$ satisfies 
following properties:
\begin{enumerate}
\item $\upsilon_{-1}\left( (\partial B_R(x) \setminus C_x) \setminus \bigcup_{1 \le k \le n, i \in \mathbf{N}}C_{k, i}^x\right)=0$.
\item For every $\partial B_R(x) \cap C_{k, i}^x$, there exist $B_{k, i}^x > 1$ and $\tau_{k, i}^x > 0$ such that 
$(B_{k, i}^x)^{-1} \le 
\upsilon_{-1}(\partial B_R(x) \cap B_r(z) \setminus C_x)/r^{k-1} \le \upsilon_{-1}(\partial B_R(x) \cap \overline{B}_r(z))/
r^{k-1}\le B_{k, i}^x$ for every $z \in \partial B_R(x) \cap C_{k, i}^x$ and $0 < r < \tau_{k, i}^x$.
\item For every $z \in \bigcup_{i \in \mathbf{N}}(\partial B_R(x) \cap C_{k, i}^x)$ and $0 < \delta <1$, there exists
$\partial B_R(x) \cap C_{k, i}^x$ such that $z \in \partial B_R(x) \cap C_{k, i}^x$ and that the map 
$\hat{\Phi}_{k, i}^x = (r_{x_{k, i}^2}, _{\cdots}, r_{x_{k, i}^k})$ from 
$\partial B_R(x) \cap C_{k, i}^x$ to $\mathbf{R}^{k-1}$, gives $(1 \pm \delta)$-bi-Lipschitz equivalent to the image 
$\hat{\Phi}_{k, i}^x(\partial B_R(x) \cap C_{k, i}^x)$.
\end{enumerate}
Especially, $\partial B_R(x) \setminus C_x$ is $\upsilon_{-1}$-rectifiable.
\end{enumerate}
\end{theorem}
\begin{proof}
First, we shall prove the following claim:
\begin{claim}\label{105}
For every $R > 0$, $z \in \overline{B}_R(x)\setminus \{x\}$ and $0 < \epsilon < \min \{ \overline{z, x}/100, 1\}$, we have $\upsilon_{-1}(\partial B_{\overline{x, z}}(x) \cap \overline{B}_{\epsilon}(z))\le C(n)\upsilon(B_{\epsilon}(z))
/\epsilon$.
\end{claim}
Because, by \cite[Corollary $5.7$]{ho1}, we have 
\[\frac{\upsilon_{-1}(\partial B_{\overline{x, z}}(x) \cap \overline{B}_{\epsilon}(z))}{\mathrm{vol}\,\partial B_{\overline{x, z}}(\underline{p})} \le C(n) \frac{\upsilon(C_x(\partial B_{\overline{x, z}}(x) \cap \overline{B}_{\epsilon}(z))
\cap A_{\overline{x, z}-2\epsilon, \overline{x, z}}(x))}{\mathrm{vol}\,A_{\overline{x, z}-2\epsilon, \overline{x, z}}(\underline{p})}.\]
Here $C_x(A)=\{z \in Y; $ there exists $a \in A$ such that $\overline{x, z} + \overline{z, a} = \overline{z, a}\}$ for every subset $A$ of $Y$.
On the other hand, by triangle inequality, we have $C_x(\partial B_{\overline{x, z}}(x) \cap \overline{B}_{\epsilon}(z)) \cap 
A_{\overline{x, z}-2\epsilon, \overline{x, z}}(x) \subset \overline{B}_{100\epsilon}(z)$.
Thus, we have
\[\upsilon_{-1}(\partial B_{\overline{x,z}}(x) \cap \overline{B}_{\epsilon}(z))\le \frac{\mathrm{vol}\,\partial B_{\overline{x,z}}(\underline{p})}{\mathrm{vol}\,A_{\overline{x,z}-2\epsilon, \overline{x,z}}(\underline{p})}
\upsilon (B_{100\epsilon}(z)) C(n) \le C(n, R)\frac{1}{\epsilon}\upsilon(B_{\epsilon}(z)).\]
Therefore, we have Claim \ref{105}.

We take collections of Borel sets $\{C_{k, i}^x\}$ and of points $\{x_{k, i}^l\}$ as in Lemma \ref{5}.
By Lemma \ref{6}, without loss of generality, we can assume that for every $C_{k, i}^x$, there exists $\tau > 0$ such that
$C_{k, i}^x \subset \mathcal{D}_x^{\tau} \setminus B_{\tau}(x)$.
Moreover, by  \cite[Theorem $3. 23$]{ch-co3} and  \cite[Theorem $4.6$]{ch-co3}, we can assume that
for every $C_{k, i}^x$, there exists $A_{k, i}^x > 1$ such that for every $0 < r< 1$ and $z \in C_{k, i}^x$, we have
$(A_{k, i}^x)^{-1} \le \upsilon(B_r(z))/r^k \le A_{k,i}^x$.
By Proposition \ref{130}, we can also assume that for every $C_{k, i}^x$ and $z \in C_{k,i}^x$, we have 
$\lim_{r \rightarrow 0}\upsilon(B_r(z) \cap C_{k, i}^x)/\upsilon(B_r(z)) = 1.$ 
\begin{claim}\label{106}
Let $(Y, y, \upsilon)$ be a Ricci limit space, 
$x$ a point in $Y$ , $\tau, R $ positive numbers satisfying $0< \tau < 1$, $R>1$, and $z$ a point in 
$\mathcal{D}_x^{\tau} \cap B_R(x) \setminus B_{\tau}(x)$.  
Then, we have $\upsilon(\partial B_{\overline{x,z}}(x) \cap B_{\epsilon}(z) \setminus C_x) \ge C(n,R) \upsilon(B_{\epsilon}(z))/ \epsilon$
for every $0 < \epsilon < \tau/100$.
\end{claim}
The proof is as follows.
We take $w \in Y$ satisfying $\overline{z, w}=\epsilon/100$ and $\overline{x, z}+\overline{z,w}=\overline{x,w}$.
By \cite[Theorem $4. 6$ ]{ho1}, we have
\[\frac{\upsilon(B_{\frac{\epsilon}{1000}}(w))}{\mathrm{vol}\,A_{\overline{x,z}, \overline{x,z}+\epsilon}(\underline{p})}
\le C(n)\frac{\upsilon_{-1}\left(C_x(B_{\frac{\epsilon}{1000}}(w))\cap \partial B_{\overline{x,z}}(x)\right)}{\mathrm{vol}\,\partial B_{\overline{x,z}}(\underline{p})}. \]
By triangle inequality, we have $C_x(B_{\epsilon/1000}(w)) \cap \partial B_{\overline{x,z}}(x) \subset 
\partial B_{\overline{x,z}}(x) \cap B_{\epsilon}(z)$. 
Thus, by Bishop-Gromov volume comparison theorem for $\upsilon$,
\begin{align}
\upsilon_{-1}(\partial B_{\overline{x,z}}(x) \cap B_{\epsilon}(z) \setminus C_x) &\ge
C(n) \frac{\mathrm{vol}\,\partial B_{\overline{x,z}}(\underline{p})}{\mathrm{vol}\,A_{\overline{x,z}, \overline{x,z}+\epsilon}(\underline{p})}\upsilon(B_{\epsilon/1000}(w)) \\
&\ge C(n ,R) \frac{1}{\epsilon}\upsilon(B_{\frac{\epsilon}{1000}}(w)) \\
&\ge C(n, R) \frac{1}{\epsilon}\upsilon(B_{5\epsilon}(w)) \\
&\ge C(n, R) \frac{\upsilon(B_{\epsilon}(z))}{\epsilon}.
\end{align}
Therefore we have Claim \ref{106}.
\

By Claim \ref{105} and \ref{106}, for every $C_{k, i}^x$, there exist  $B_{k, i}^x >1$ and $\tau_{k, i}^x>0$ such that 
for every $z \in C_{k, i}^x$ and $0 < r <\tau_{k, i}^x$, we have $(B_{k, i}^x)^{-1} \le 
\upsilon(\partial B_{\overline{x,z}}(x) \cap B_r(z) \setminus C_x)/r^k \le B_{k, i}^x$.
We put $\hat{A} = \{ t \in [0, \mathrm{diam}Y); 
\upsilon_{-1}(\partial B_t(x) \setminus \bigcup C_{k,i}^x) = 0\}$. Since $\upsilon (Y \setminus \bigcup C_{k, i}^x)=0$, by \cite[Proposition $5.1$]{ho1} and  \cite[Theorem $5. 2$]{ho1}, we have,
$\hat{A}$ is $H^1$-Lebesgue measurable, $H^1([0, \mathrm{diam}Y) \setminus \hat{A}) = 0$.
Since $\upsilon$ is a Radon measure, there exists a Borel set $A \subset \hat{A}$ such that $H^1(\hat{A} \setminus A) =0$. Thus we have the assertion.
\end{proof}
\subsection{Calculation of radial derivative for Lipschitz functions}
The purpose in this subsection is to calculate the radial derivative of Lipschitz functions: $\langle dr_x, df \rangle$ explicitly.
The main result in this subsection is Theorem \ref{14}.
\begin{lemma}\label{8}
Let $(Y, y)$ be a Ricci limit space satisfying $Y \neq \{y\}$, $z$ a point in $Y \setminus C_y$, $f$ a Lipschitz function on 
$Y$, $\tau$ a positive number and $\gamma_i$ an isometric embedding from $[0, \overline{y, z}+\tau]$ to $Y$
satisfying $\gamma_i(0) = y$ and $\gamma_i(\overline{y, z})= z (i=1, 2)$.
We put $f_i = f \circ \gamma_i$. Then, we have 
$lipf_1(\overline{y, z}) =
lipf_2(\overline{y, z})$ and 
$\mathrm{Lip}f_1(\overline{y, z})=
\mathrm{Lip}f_2(\overline{y, z})$.
\end{lemma}
\begin{proof}
For every real number $\epsilon$ satisfying $0<|\epsilon|<<\tau$, by splitting theorem (see  \cite[Theorem $9.25$]{ch1} or \cite[Theorem $6.64$]{ch-co}), we have 
$\overline{\gamma_1(\overline{x,z} + \epsilon), \gamma_2(\overline{x,z} + \epsilon)}\le \Psi(|\epsilon|;n)|\epsilon|$.
Therefore, we have
\[\frac{|f\circ \gamma_1 (\overline{x,z} + \epsilon)- f \circ \gamma_1(\overline{x,z})|}{|\epsilon|} \le 
\frac{|f \circ \gamma_2(\overline{x,z} + \epsilon)-f \circ \gamma_2(\overline{x,z})|}{|\epsilon|} + \mathbf{Lip}f\Psi(|\epsilon|;n).\]
Thus, we have 
$\mathrm{Lip}f_1(\overline{y, z})\le
\mathrm{Lip}f_2(\overline{y, z})$ and $lipf_1(\overline{y, z})\le
lipf_2(\overline{y, z})$.
Therefore we have $\mathrm{Lip}f_1(\overline{y, z})=
\mathrm{Lip}f_2(\overline{y, z})$ and $lipf_1(\overline{y, z})=
lipf_2(\overline{y, z})$.
\end{proof}
We shall give the following definition:
\begin{definition}
Let $(Y, y)$ be a Ricci limit space, $z$ a point in $Y \setminus C_y$, $\tau$ a positive number, $\gamma$ an isometric embedding from
$[0, \overline{y, z}+\tau]$ to $Y$ satisfying $\gamma(0)=y$ and $\gamma(\overline{y,z})=z$.
We put $F=f \circ \gamma$.
Then, we put $lip_y^{\mathrm{rad}}f(z) = lipF(\overline{y,z})$ and $\mathrm{Lip}_y^{\mathrm{rad}}f(z)=\mathrm{Lip}F(\overline{y, z})$.
\end{definition}
\begin{theorem}\label{9}
Let $(Y, y, \upsilon)$ be a Ricci limit space, $x$ a point in $Y$ and $f$ a Lipschitz function on $Y$.
Then, we have the following:
\begin{enumerate}
\item $lipf(z)^2=lip_x^{\mathrm{rad}}f(z)^2+lip(f|_{\partial B_{\overline{x, z}}(x)})(z)^2$ for a.e. $z \in Y$.
\item $\mathrm{Lip}f(z)^2=\mathrm{Lip}_x^{\mathrm{rad}}f(z)^2+\mathrm{Lip}(f|_{\partial B_{\overline{x, z}}(x)})(z)^2$
for a.e. $z \in Y$.
\item $\mathrm{Lip}(f|_{\partial B_{\overline{x, z}}(x)})(z)=lip(f|_{\partial B_{\overline{x, z}}(x) \setminus C_x})(z)$
for a.e. $z \in Y \setminus C_x$.
\end{enumerate}
\end{theorem}
\begin{proof}
First we shall remark the following:
\begin{claim}\label{109}
Let $f$ be a Lipschitz function on $\mathbf{R}^k$.
Then,  we have $\mathrm{Lip}f(z)^2 = (\mathrm{Lip}(f|_{\mathbf{R} \times \{z_2, _{\cdots}, z_k\}})(z))^2$ $+$ $(\mathrm{Lip}(f|_{\{z_1\} \times \mathbf{R}^{k-1}})(z))^2 = 
(lip(f|_{\mathbf{R} \times \{z_2, _{\cdots}, z_k\}})(z))^2 + (lip(f|_{\{z_1\} \times \mathbf{R}^{k-1}})(z))^2
 = lipf(z)^2$ for a.e $z=(z_1, _{\cdots}, z_k) \in \mathbf{R}^k$.
\end{claim}
Because, by Rademacher's theorem for Lipschitz functions on $\mathbf{R}^k$, the function $f$ is totally differentiable at a.e $z \in \mathbf{R}^k$.
Therefore we have Claim \ref{109}.
\

The next claim is clear:
\begin{claim}\label{110a}
Let $Z_i$ be metric spaces $(i =1, 2)$, $\delta$ a positive number with $0 < \delta <1$, and $\Phi$ a map from $Z_1$ to $Z_2$ satisfying
$\Phi(Z_i) = Z_2$ and $(1-\delta)\overline{x_1, x_2} \le
\overline{\Phi(x_1), \Phi(x_2)} \le (1+\delta)\overline{x_1, x_2}$ for every $x_1, x_2 \in Z_1$.
Then, for every Lipschitz function $f$ on $Z_2$, we have, 
$(1-\Psi(\delta))\mathrm{Lip}f(\Phi (z_1))\le \mathrm{Lip}(f\circ \Phi)(z_1) \le (1 + \Psi(\delta))\mathrm{Lip}f(z_1)$, $(1-\Psi(\delta))lipf(\Phi(z_1)) \le lip(f \circ \Phi)(z_1) \le (1 + \Psi(\delta))lipf(\Phi(z_1))$  for 
every $z_1 \in Z_1$. 
\end{claim}
We will give a proof of the following claim in appendix.
\begin{claim}\label{111}
For every Lebesgue measurable $A \subset \mathbf{R}^k$, we put $sl_1-\mathrm{Leb}A=\{ a = (a_1, _{\cdots}, a_k) \in A;
\lim_{r \rightarrow 0} H^{k-1}(\{a_1\} \times \overline{B}_r(a_2, _{\cdots}, a_k) \cap A)/H^{k-1}(\{a_1\} \times 
\overline{B}_r(a_2, _{\cdots}, a_k))=1\}$.
Then we have the following:
\begin{enumerate}
\item The set $sl_1-\mathrm{Leb}A$ is a Lebesgue measurable set.
\item For every $t \in \mathbf{R}$, $H^{k-1}(A \cap \{t\} \times \mathbf{R}^{k-1} \setminus sl_1-\mathrm{Leb}A)=0$.
\item $H^k(A \setminus sl_1-\mathrm{Leb}A) = 0$.
\end{enumerate}
\end{claim}
We put $L= \mathbf{Lip}f$.
We take collections of Borel sets $\{C_{k,i}^x\}_{1 \le k \le n, i \in \mathbf{N}}$ and of points 
$\{x_{k, i}^l\}_{2 \le k \le n, i \in \mathbf{N}, 2\le l \le k}$ as in Theorem \ref{7}.
We fix a sufficiently small $\delta > 0 $ and
$C_{k, i}$ satisfying that the map $\Phi_{k,i}^x=(r_x, r_{x_{k,i}^2}, _{\cdots}, r_{x_{k,i}^k})$ from $C_{k,i}^x$ to $\mathbf{R}^k$, gives $(1 \pm \delta)$-bi-Lipschitz equivalent to the image.
Then we put a function $f_{k,i}^x = f \circ (\Phi_{k,i}^x)^{-1}$ on $\Phi_{k,i}^x(C_{k,i}^x)$.
and take a Lipschitz function $F_{k,i}^x$ on $\mathbf{R}^k$ satisfying 
$F_{k,i}^x|_{\Phi_{k,i}^x(C_{k,i}^x)} = f_{k,i}^x$ and $\mathbf{Lip}F_{k,i}^x=\mathbf{Lip}f_{k,i}^x$.
\begin{claim}\label{112}
With notation as above, we have the following:
\begin{enumerate}
\item $(1-\Psi(\delta;n))\mathrm{Lip}F_{k,i}^x(w) \le \mathrm{Lip}f((\Phi_{k,i}^x)^{-1}(w))\le(1+\Psi(\delta;n))\mathrm{Lip}
F_{k,i}^x(w)$ for a.e $w \in \Phi_{k,i}^x(C_{k,i}^x)$.
\item  $(1-\Psi(\delta;n))lipF_{k,i}^x(w) \le lipf((\Phi_{k,i}^x)^{-1}(w))\le(1+\Psi(\delta;n))lip
F_{k,i}^x(w)$ for a.e $w \in \Phi_{k,i}^x(C_{k,i}^x)$.
\item $\mathrm{Lip}(F_{k,i}^x|_{\mathbf{R}\times \{w_2, _{\cdots}, w_k\}})(w)-L\Psi(\delta;n)
\le \mathrm{Lip}_x^{\mathrm{rad}}f((\Phi_{k,i}^x)^{-1}(w)) \le 
\mathrm{Lip}(F_{k,i}^x|_{\mathbf{R} \times \{w_2, _{\cdots}, w_k\}})(w) + L\Psi(\delta;n)$
for a.e $w = (w_1, _{\cdots}, w_k) \in \Phi_{k,i}^x(C_{k,i}^x)$.
\item $lip(F_{k,i}^x|_{\mathbf{R}\times \{w_2, _{\cdots}, w_k\}})(w)-L\Psi(\delta;n)
\le lip_x^{\mathrm{rad}}f((\Phi_{k,i}^x)^{-1}(w)) \le 
lip(F_{k,i}^x|_{\mathbf{R} \times \{w_2, _{\cdots}, w_k\}})(w) + L\Psi(\delta;n)$
for a.e $w = (w_1, _{\cdots}, w_k) \in \Phi_{k,i}^x(C_{k,i}^x)$.
\item $(1-\Psi(\delta;n))\mathrm{Lip}(F_{k,i}^x|_{\{w_1\} \times \mathbf{R}^{k-1}})(w) \le 
\mathrm{Lip}(f|_{\partial B_{\overline{x, (\Phi_{k,i}^x)^{-1}(w)}}(x) \cap C_{k,i}^x})((\Phi_{k,i}^x)^{-1}(w))
\le (1 + \Psi(\delta;n))\mathrm{Lip}(F_{k,i}^x|_{\{w_1\}\times \mathbf{R}^{k-1}})(w)$ for a.e.
$w=(w_1, _{\cdots}, w_k) \in \Phi_{k,i}^x(C_{k,i}^x)$.
\item $(1-\Psi(\delta;n))lip(F_{k,i}^x|_{\{w_1\} \times \mathbf{R}^{k-1}})(w) \le 
lip(f|_{\partial B_{\overline{x, (\Phi_{k,i}^x)^{-1}(w)}}(x) \cap C_{k,i}^x})((\Phi_{k,i}^x)^{-1}(w))
\le (1 + \Psi(\delta;n))lip(F_{k,i}^x|_{\{w_1\}\times \mathbf{R}^{k-1}})(w)$ for a.e.
$w=(w_1, _{\cdots}, w_k) \in \Phi_{k,i}^x(C_{k,i}^x)$.
\end{enumerate}
\end{claim} 
The proof is as follows.
First, we shall check the statement $1$.
We put $\mathbf{C}_{k, i}^x = \mathrm{Leb}(\Phi_{k,i}^x(C_{k, i}^x)) \cap \Phi_{k,i}^x(\mathrm{Leb}C_{k,i}^x)$.
Then, we have $H^k(\Phi_{k,i}^x(C_{k,i}^x)\setminus \mathbf{C}_{k,i}^x)=0$.
By Claim \ref{110a} and Proposition \ref{130}, we have 
$(1-\Psi(\delta))\mathrm{Lip}(F_{k,i}^x|_{\Phi_{k,i}(C_{k,i}^x)})(w) \le 
\mathrm{Lip}(f|_{C_{k,i}^x})((\Phi_{k,i}^x)^{-1}(w))
\le (1 + \Psi(\delta)) \mathrm{Lip}(F_{k,i}^x|_{\Phi_{k,i}^x(C_{k,i}^x)})(w)$,
$\mathrm{Lip}(F_{k,i}^x|_{\Phi_{k,i}^x(C_{k,i}^x)})(w) = \mathrm{Lip}F_{k,i}^x(w)$ and 
$\mathrm{Lip}(f|_{C_{k,i}^x})((\Phi_{k,i}^x)^{-1}(w))= \mathrm{Lip}f((\Phi_{k,i}^x)^{-1}(w))$
for every $w \in \mathbf{C}_{k, i}^x$.
Therefore we have the statement $1$.
Similarly, we have the statement $2$.
\

Next, we shall give a proof of statement $3$.
We put $\mathbf{C}_{k,i}^{x,f} = sl_1-\mathrm{Leb}\mathbf{C}_{k,i}^x \cap \{ w \in \mathbf{R}^k; F_{k,i}^x$ is  
totally differentiable at $w. \}$.
Then, by Claim \ref{111}, we have $H^k(\mathbf{C}_{k,i}^x \setminus \mathbf{C}_{k,i}^{x, f})=0$.

We take a point $w \in \mathbf{C}_{k,i}^{x, f}$ and put 
$w_{\epsilon}=w +(\epsilon, 0, _{\cdots}, 0)$ for every $ \epsilon > 0$.
Since $w \in \mathrm{Leb}\mathbf{C}_{k,i}^x$, there exists $\hat{w}_{\epsilon} \in \mathbf{C}_{k,i}^x$ such that
$\overline{w_{\epsilon}, \hat{w}_{\epsilon}} \le a(\epsilon)\epsilon (a(\tau) \rightarrow 0$ as $\tau \rightarrow 0 )$.
Clearly, $(1-\delta)(\epsilon-a(\epsilon)\epsilon) \le 
(1-\delta)\overline{w,\hat{w}_{\epsilon}} \le \overline{(\Phi_{k,i}^x)^{-1}(w), (\Phi_{k,i}^x)^{-1}(\hat{w}_{\epsilon})}
\le (1 + \delta)\overline{w, \hat{w}_{\epsilon}} \le (1 + \delta)(\epsilon + a(\epsilon)\epsilon).$
We define the projection $\pi_1$ from $\mathbf{R}^k$ to $\mathbf{R}$ by $\pi_1(w) = w_1$.
Then we have 
$\overline{x, (\Phi_{k,i}^x)^{-1}(\hat{w}_{\epsilon})} = \pi_1(\hat{w}_{\epsilon}) = \pi_1(w_{\epsilon}) \pm a(\epsilon)\epsilon = \pi_1(w) + \epsilon \pm a(\epsilon)\epsilon
= \overline{x, (\Phi_{k,i}^x)^{-1}(w)} + 
\overline{(\Phi_{k,i}^x)^{-1}(w), (\Phi_{k,i}^x)^{-1}(\hat{w}_{\epsilon})} \pm
(\delta + a(\epsilon))\epsilon.$
By Lemma \ref{6}, without loss of generality, we can assume that 
there exists $\tau_0 > 0$ such that $C_{k,i} \subset \mathcal{D}_x^{\tau_0}$.
We take an isometric embedding $\gamma$ from 
$[0, \overline{x, (\Phi_{k,i}^x)^{-1}(w)}+\tau_0]$ to $Y$ satisfying 
$\gamma(0) = x$ and $\gamma(\overline{x, (\Phi_{k,i}^x)^{-1}(w)})= (\Phi_{k,i}^x)^{-1}(w)$.
Then, by rescaling $\epsilon^{-1}d_Y$ and splitting theorem, we have
$\overline{(\Phi_{k,i}^x)^{-1}(\hat{w}_{\epsilon}), \gamma(\overline{x, (\Phi_{k,i}^x)^{-1}(w) + \epsilon})} 
\le \Psi (a(\epsilon), \delta;n)\epsilon$.
For $\epsilon << \tau_0$,
we have 
\begin{align*}
\frac{|F_{k,i}^x(w)-F_{k,i}^x(w_{\epsilon})|}{\epsilon} &\le \frac{|F_{k,i}^x(w)-F_{k,i}^x(\hat{w}_{\epsilon})|}{\epsilon}
+ La(\epsilon) \\
&\le \frac{|f((\Phi_{k,i}^x)^{-1}(w))-f(\gamma(\overline{x, (\Phi_{k,i}^x)^{-1}(w)}+\epsilon))|}{\epsilon}
+L\Psi(a(\epsilon), \delta;n).
\end{align*} 
By letting $\epsilon \rightarrow 0$, we have 
$\mathrm{Lip}(F_{k,i}^x|_{\mathbf{R} \times \{w_2, _{\cdots}, w_k\} })(w) \le 
\mathrm{Lip}_x^{\mathrm{rad}}f((\Phi_{k,i}^x)^{-1}(w)) + L\Psi(\delta;n).$
We take a sequence $\{\epsilon_i\}$ such that $\epsilon_j \rightarrow 0$ and that 
\[ \lim_{j \rightarrow \infty}\frac{|f \circ (\Phi_{k,i}^x)^{-1}(w)-f(\gamma(\overline{x, (\Phi_{k,i}^x)^{-1}(w)}+\epsilon_j))|}{|\epsilon_j|}=\mathrm{Lip}_x^{\mathrm{rad}}f((\Phi_{k,i}^x)^{-1}(w)).\]
We fix $j \in \mathbf{N}$. We assume that $\epsilon_j >0$.
Since $(\Phi_{k,i}^x)^{-1}(w) \in \mathrm{Leb}C_{k,i}^x$,  there exists $\hat{w}(j) \in C_{k,i}^x$ such that
$\overline{\hat{w}(j), \gamma(\overline{x, (\Phi_{k,i}^x)^{-1}(w)} + \epsilon_j)}\le \tau_j \epsilon_j (\tau_j \rightarrow 0$ as $j \rightarrow \infty)$.
Then, we have 
\begin{align}
\pi_1(\hat{w}(j))-\pi_1(w) &= \overline{x, \hat{w}(j)}-\overline{x, (\Phi_{k,i}^x)^{-1}(w)} \\
&= \overline{x, \gamma(\overline{x, (\Phi_{k,i}^x)^{-1}(w)} + \epsilon_j)} \pm \tau_j\epsilon_j \\
&=\epsilon_j \pm \tau_j\epsilon_j \\
&=\overline{\gamma(\overline{x, (\Phi_{k,i}^x)^{-1}(w)} + \epsilon_j), (\Phi_{k,i}^x)^{-1}(w)}\pm \tau_j\epsilon_j \\
&\ge (1-\delta)\overline{\Phi_{k,i}^x(\hat{w}(j)), w}-\tau_j\epsilon_j.
\end{align}
On the other hand, since
$\overline{\Phi_{k,i}^x(\hat{w}(j)), w} \le (1+\delta)\epsilon_j + \tau_j\epsilon_j,$
we have $\overline{w+(\epsilon_j, 0, _{\cdots}, , 0), \Phi_{k,i}^x(\hat{w}(j))} \le \Psi (|\epsilon_j|, \delta;n)|\epsilon_j|.$
Similarly, we have the inequality above in the case $\epsilon_j < 0$.
We put $w(j) = w + (\epsilon_j, 0 , _{\cdots}, 0)$.
Then, we have 
\begin{align*}
\frac{|f((\Phi_{k,i}^x)^{-1}(w))-f(\gamma(\overline{x, (\Phi_{k,i}^x)^{-1}(w))}+\epsilon_j))}{|\epsilon_j|} &\le
\frac{|F_{k,i}^x(w)-F_{k,i}^x(\Phi_{k,i}^x(\hat{w}(j)))|}{|\epsilon_j|}+L\tau_j \\
&\le\frac{|F_{k,i}^x(w)-F_{k,i}^x(w(j))|}{|\epsilon_j|}+L\Psi(|\epsilon_j|, \tau_j, \delta;n). 
\end{align*}
By letting $j \rightarrow \infty$, we have the statement $3$.
Similarly, we have the statement $4$.
\

We shall give a proof of the statement $5$.
we take $w \in \mathbf{C}_{k,i}^{x f}$. By Claim \ref{110a}, we have
\begin{align}
(1-\Psi(\delta))\mathrm{Lip}(F_{k,i}^x|_{\{w_1 \}\times \mathbf{R}^{k-1}\cap \mathbf{C}_{k,i}^x})(w)
&\le \mathrm{Lip}(f|_{(\Phi_{k,i}^x)^{-1}(\{w_1\} \times \mathbf{R}^{k-1} \cap \mathbf{C}_{k,i}^x)})(\Phi_{k,i}^x)^{-1}(w) \\
&\le (1 +\Psi(\delta))\mathrm{Lip}(F_{k,i}^x|_{\{w_1\}\times \mathbf{R}^{k-1}\cap \mathbf{C}_{k,i}^x})(w).
\end{align}
We remark that $(\Phi_{k,i}^x)^{-1}(\{w_1\} \times \mathbf{R}^{k-1} \cap \mathbf{C}_{k,i}^x)
=\partial B_{\overline{x, (\Phi_{k,i}^x)^{-1}(w)}}(x) \cap \mathbf{C}_{k,i}^x$.
By Proposition \ref{133}, we have 
$\mathrm{Lip}(F_{k,i}^x|_{\{w_1\}\times \mathbf{R}^{k-1} \cap \mathbf{C}_{k,i}^x})(w)= \mathrm{Lip}(F_{k,i}^x|_{\{w_1\} \times
\mathbf{R}^{k-1}})(w)$.
Therefore, by Claim \ref{110a}, we have 
\begin{align}
(1-\Psi(\delta))\mathrm{Lip}(F_{k,i}^x|_{\{w_1\}\times \mathbf{R}^{k-1}})(w) &\le
\mathrm{Lip} f|_{\partial B_{\overline{x, (\Phi_{k,i}^x)^{-1}(w)}}(x) \cap \mathbf{C}_{k,i}^x})((\Phi_{k,i}^x)^{-1}(w)) \\
&\le \mathrm{Lip}(f|_{\partial B_{\overline{x, (\Phi_{k,i}^x)^{-1}(w)}}(x) \cap C_{k,i}^x})(\Phi_{k,i}^x)^{-1}(w)) \\
&\le (1 +\Psi(\delta))\mathrm{Lip}(F_{k,i}^x|_{\{w_1\} \times \mathbf{R}^{k-1} \cap \Phi_{k,i}^x(C_{k,i}^x)})(w) \\
&\le (1 +\Psi(\delta))\mathrm{Lip}(F_{k,i}^x|_{\{w_1\} \times \mathbf{R}^{k-1}})(w).
\end{align}
Thus we have the statement $5$.
Similarly, we have the statement $6$.
\

Therefore we have Claim \ref{112}.

\
\begin{claim}\label{113}
With same notation as in Claim \ref{112}, we have
\[lip(f|_{\partial B_{\overline{x, (\Phi_{k,i}^x)^{-1}(w)}}(x) \cap C_{k,i}^x})((\Phi_{k,i}^x)^{-1}(w)) \ge 
\mathrm{Lip}(f_{\partial B_{\overline{x, (\Phi_{k,i}^x)^{-1}}}(x)})((\Phi_{k,i}^x)^{-1}(w))-\Psi(\delta;n, L)\]
for a.e $w \in \Phi_{k,i}^x(C_{k,i}^x)$.
\end{claim}
The proof is as follows.
We will use same notaion as in the proof of Claim \ref{112}.
We take $w \in \Phi_{k,i}^x(\mathrm{Leb}(\Phi_{k,i}^x)^{-1}(\mathbf{C}_{k,i}^{x, f}))$ and put 
$z = (\Phi_{k,i}^x)^{-1}(w)$.
First, we assume $k \ge 2$.
We shall prove that $z \in \partial B_{\overline{x,z}}(x)$ is not an isolated point in $\partial B_{\overline{x,z}}(x) \setminus C_x$.
Because, by the definition of $sl_1-\mathrm{Leb}(\mathbf{C}_{k,i}^x)$, there exists a sequence 
$\{\beta(j)\} \in \mathbf{C}_{k,i}^x$ such that $\pi_1(\beta(j))=\pi_1(w)$, $\beta(j) \neq w$ and
$\beta(j) \rightarrow w$.
Then, we have $(\Phi_{k,i}^x)^{-1}(\beta(j)) \neq z$, $(\Phi_{k,i}^x)^{-1}(\beta(j)) \in \partial B_{\overline{x,z}}(x) \setminus C_x$ and
$(\Phi_{k,i}^x)^{-1}(\beta(j)) \rightarrow z$.
Therefore, $z$ is not an isolated point in $\partial B_{\overline{x,z}}(x) \setminus C_x$.

\
We take a sequence $\{z(j)\} \in \partial B_{\overline{x,z}}(x) \setminus \{z\}$ such that $z(j) \rightarrow z$ and that 
$|f(z(j))-f(z)|/\overline{z(j), z} \rightarrow \mathrm{Lip}(f|_{\partial B_{\overline{x,z}}(x)})(z)$.
We put $\eta_j = \overline{z(j), z} > 0$.
Since $z \in \mathrm{Leb}(\Phi_{k,i}^x)^{-1}(\mathbf{C}_{k,i}^{x, f})$, there exists $\hat{z}(j) \in (\Phi_{k,i}^x)^{-1}(\mathbf{C}_{k,i}^{x, f})$ such that $\overline{z(j), \hat{z}(j)} \le \hat{\tau}_j\eta_j (\hat{\tau}_j \rightarrow 0$ as $j \rightarrow \infty )$.
We put $\alpha(j)= \Phi_{k,i}^x(\hat{z}(j))$.
Thus, we have $|\pi_1(\alpha(j))-\pi_1(w)|\le (1+\delta)\hat{\tau}_j\eta_j$.
Therefore, there exists $\hat{\alpha}(j) \in \{w_1\}\times \mathbf{R}^{k-1}$ such that 
$\overline{w(j), \hat{\alpha}(j)}\le \Psi (\hat{\tau}_j;n)\eta_j$.
Then, we have 
\begin{align}
\frac{|f(z(j))-f(z)|}{\overline{z(j) ,z}} &\le \frac{|f(\hat{z}(j))- f(z)|}{\eta_j} + L\hat{\tau}_j\\
&\le \frac{|F_{k,i}^x(w(j))-F_{k,i}^x(w)|}{\eta_j}+\Psi(\hat{\tau}_j;n, L) \\
&\le \frac{|F_{k,i}^x(\hat{\alpha}(j))-F_{k,i}^x(w)|}{\overline{\hat{\alpha}(j), w}}\frac{\overline{\hat{\alpha}(j), w}}{\eta_j} + L\Psi(\hat{\tau}_j;n, L).
\end{align}
By letting $j \rightarrow \infty$, we have Claim \ref{113} for the case $k \ge 2$.
Next, we assume $k =1$.
It suffices to check that $z$ is an isolated point in $\partial B_{\overline{x,z}}(x)$.
We assume that $z$ is not an isolated point in $\partial B_{\overline{x,z}}(x)$.
Then, there exists a sequence $\{z(i)\} \in \partial B_{\overline{x,z}}(x) \setminus \{z\}$ such that
$z(i) \rightarrow z$.
We take an isometric embedding $\gamma$ from $[0, \overline{x,z} + \tau_0]$ to $Y$ such that
$\gamma(0) = x,  \gamma(\overline{x,z}) =z$. Here $\tau_0$ is a positive constant.
We put $\epsilon(i) = \overline{z, z(i)}$.
Then we have 
$\overline{z(i), \gamma(\overline{x,z}-\epsilon_i)} \ge 
\overline{x, z(i)}-\overline{x, \gamma(\overline{x,z}-\epsilon_i)} = \epsilon_i$,
$\overline{z(i), \gamma(\overline{x,z} + \epsilon_i)} \ge \overline{x, \gamma(\overline{x, z}+\epsilon_i)}-\overline{x, z(i)}
=\epsilon_i$.
On the other hand, by Proposition \ref{090}, without loss of generality, we can assume that 
$(Y, \epsilon_i^{-1}d_Y, z)$ converges to some tangent cone $(T_zY, 0_z)$ at $z$.
By the argument above and splitting theorem, there exists a pointed proper geodesic space $(W, w)$ such that $T_zY = \mathbf{R} \times W$ and that 
$W\neq \{w\}$.
On the other hand, $z \in C_{1, i} \subset \mathcal{R}_1$.
This is a contradiction. Therefore we have the Claim \ref{113} for the case $k=1$.
\

By Claim \ref{109}, \ref{112} and \ref{113}, 
for every $N \in \mathbf{N}$, we have $\mathrm{Lip}f(z)^2 = \mathrm{Lip}_x^{\mathrm{rad}}f(z)^2 + \mathrm{Lip}(f|_{\partial B_{\overline{x,z}(x)}})(z)^2 \pm N^{-1} = lip_x^{\mathrm{rad}}f(z)^2 + lip(f|_{\partial B_{\overline{x,z}}(x) \setminus C_x})(z)^2 \pm N^{-1} = lipf(z)^2 \pm N^{-1}$ for a.e. $z \in Y \setminus C_x$.
Therefore, we have the assertion. 
\end{proof}
\begin{remark}
For every Ricci limit space $(Y, y, \upsilon)$ and Lipschitz function $f$ on $Y$, by \cite[Corollary $6. 36$]{ch1}, we have 
$lipf(x)=\mathrm{Lip}f(x)$ for a.e. $x \in Y$.
\end{remark}
By an argument similar to the proof of Lemma \ref{8}, we have the following:
\begin{lemma}\label{11}
Let $(Y, y)$ be a Ricci limit space satisfying $Y \neq \{y\}$, $z$ a point in $Y \setminus C_y$, $f$ a Lipschitz function on
$Y$, $\tau$ a positive number and $\gamma$ an isometric embedding from $[0, \overline{y, z}+\tau]$ to $Y$ 
satisfying $\gamma(0)=y$ and $\gamma(\overline{y,z})= z$.
We assume that the limit $\lim_{r \rightarrow 0}(f\circ \gamma(\overline{y, z}+ r)-f(z))/r$ exists.
Then, for every isometric embedding $\hat{\gamma}: [0, \overline{y, z}+\tau] \rightarrow Y$ such that $\gamma(0) = y$ and that 
$\gamma(\overline{y, z})=z$, we have $\lim_{r \rightarrow 0}(f\circ \hat{\gamma}(\overline{y, z}+ r)-f(z))/r = 
\lim_{r \rightarrow 0}(f\circ \gamma(\overline{y, z}+ r)-f(z))/r$.
\end{lemma}
We shall give the following definition:
\begin{definition}\label{090a}
Let $(Y, y)$ be a Ricci limit space satisfying $Y\neq \{y\}$, $f$ a Lipschitz function on $Y$.
We put
\[A_y=\left\{x \in Y \setminus C_y; \mathrm{The \ limit} \ \lim_{r \rightarrow 0}\frac{f\circ \gamma(\overline{x,y}+r)-f(x)}{r} \ \mathrm{exists}\right\}.\]
Here $\gamma$ is an isometric embedding from $[0, \overline{y, x} + \tau] \  (\tau >0)$ to $Y$ satisfying $\gamma(0)=y$ and  $\gamma(\overline{y, x})=x$. For $x \in A_y$, we put
\[\frac{df}{dr_y}(x)=\lim_{r \rightarrow 0}\frac{f\circ \gamma(\overline{x,y}+r)-f(x)}{r}.\]
\end{definition}
Similarly, we have the following lemma: 
\begin{lemma}\label{12}
Let $(Y, y)$ be a Ricci limit space satisfying $Y \neq \{y\}$, $z$ a point in $Y \setminus C_y$, $f$ a Lipschitz function on
$Y$, $\tau$ a positive number and $\gamma_i  (i = 1, 2)$ isometric embeddings from $[0, \overline{y, z}+\tau]$ to $Y$ 
satisfying $\gamma(0)=y$ and $\gamma(\overline{y,z})= z$.
Then, we have $\liminf_{r \rightarrow 0}|f\circ \gamma_1(\overline{y, z}+ r)-f(z)|/|r| = \liminf_{r \rightarrow 0}|f\circ \gamma_2(\overline{y, z}+ r)-f(z)|/|r|$.
\end{lemma} 
With same notaion as in Lemma \ref{12}, we put $\underline{\mathrm{Lip}}_x^{\mathrm{rad}}f(z) = \liminf_{r \rightarrow 0}|f\circ \gamma_1(\overline{y, z}+ r)-f(z)|/|r|$.

\begin{lemma}\label{13}
Let $(Y, y, \upsilon)$ be a Ricci limit space, $x$ a point in $Y$ and $f$ a Lipschitz function on $Y$.
Then, we have $\underline{\mathrm{Lip}}^{\mathrm{rad}}_xf(z) = \mathrm{Lip}_x^{\mathrm{rad}}f(z)$ for a.e. $z \in Y$.
\end{lemma}
\begin{proof}
We will use same notaion as in the proof of Claim \ref{112}.
We put $L=\mathbf{Lip}f$.
We take a sufficiently small $ 0 < \delta < 1$ and a Borel set $C_{k,i}^x$ such that the map
$\Phi_{k,i}^x=(r_x, r_{x_{k,i}^2}, _{\cdots}, r_{x_{k,i}^k})$ from $C_{k,i}^x$ to $\mathbf{R}^k$,
gives a $(1 \pm \delta)$-bi-Lipschitz equivalent to the image.
We take $w \in \mathbf{C}_{k,i}^{x, f}$ and put $z = (\Phi_{k,i}^x)^{-1}(w)$.
We choose an isometric embedding $\gamma$ from $[0, \overline{x,z}+ \tau]$ to $Y$ such that
$\gamma(0)=x, \gamma(\overline{x,z})=z$.
Here, $\tau $ is a positive constant.
We take a sequence of real number, $\{\epsilon_i\}$ such that $\epsilon_i\rightarrow 0$ and $\lim _{i \rightarrow \infty}
|f\circ \gamma (\overline{x,z} + \epsilon _i)-f(z)|/|\epsilon_i|= \underline{\mathrm{Lip}}_x^{\mathrm{rad}}f(z)$.
By an argument similar to the proof of Claim \ref{104}, there exists $\hat{w}(j) \in C_{k, i}^x$ such that $\overline{\hat{w}(j), \gamma(\overline{x, z}+\epsilon_j)} \le \tau_j |\epsilon_j| (\tau_j \rightarrow 0$ as $ j \rightarrow \infty)$ and that
\begin{align}
\frac{|f(z)-f(\gamma(\overline{x,z}+\epsilon_j))|}{|\epsilon_j|} &= 
\frac{|F_{k,i}^x(w)-F_{k,i}^x(\Phi_{k,i}^x(\hat{w}(j)))|}{|\epsilon_j|} - 2L\tau_j \\
&\ge \frac{|F_{k,i}^x(w)-F_{k,i}^x(w_j)|}{|\epsilon_j|} - \Psi (\tau_j, \delta;n, L).
\end{align}
By letting $j \rightarrow \infty$, we have 
$ \underline{\mathrm{Lip}}_x^{\mathrm{rad}}f(z) \ge \mathrm{Lip}(F_{k,i}^x|_{\mathbf{R}\times \{w_2, _{\cdots},w_k\}})(w)-
\Psi(\delta;n, L) \ge \mathrm{Lip}_x^{\mathrm{rad}}f(z)- \Psi(\delta;n, L)$.
Therefore, we have the assertion.
\end{proof}
Thus, we have 
\[\mathrm{Lip}_x^{\mathrm{rad}}f(z) = \lim_{h \rightarrow 0}\frac{|f \circ \gamma (\overline{x,z} + h)-f(z)|}{|h|}\]
for a.e. $ z \in Y \setminus C_x$.
\begin{theorem}[Radial derivative for Lipschitz functions]\label{14}
Let $(Y, y, \upsilon)$ be a Ricci limit space satisfying $Y \neq \{y\}$, $x$ a point in $Y$ and $f$ a
Lipschitz function on $Y$. 
Then, we have $\upsilon (Y \setminus A_x)=0$ and 
\[\frac{df}{dr_x}(z)=\langle df, dr_x\rangle (z)\]
for a.e. $z \in A_x$. 
\end{theorem}
\begin{proof}
For every $w \in Y \setminus C_x$, there exist $\tau > 0$ and an isometric embedding $\gamma$ from $[0, \overline{x,z} + \tau]$ to 
$Y$ such that $\gamma(0) = x$ and  $\gamma(\overline{x,w}) =w$.
Then, by Theorem \ref{9} and Lemma \ref{13}, for a.e. $w \in Y \setminus C_x$, we have 
\begin{align*}
\langle dr_x, df\rangle (w) &= \frac{1}{2}(\mathrm{Lip}(r_x + f)(w)^2 - \mathrm{Lip}f(w)^2-\mathrm{Lip}r_x(w)^2) \\
&=\frac{1}{2}(\mathrm{Lip}_x^{\mathrm{rad}}(r_x+f)(w)^2 + \mathrm{Lip}((r_x + f)|_{\partial B _{\overline{x,z}}(x)\setminus C_x})(w)^2 \\
& \ \ -\mathrm{Lip}_x^{\mathrm{rad}}f(w)^2-\mathrm{Lip}(f|_{\partial B_{\overline{x,z}} \setminus C_x})(w)^2-1) \\
&= \frac{1}{2}(\mathrm{Lip}_x^{\mathrm{rad}}(r_x + f)(w)^2 + \mathrm{Lip}(f|_{\partial B_{\overline{x,z}}(x)\setminus C_x})(w)^2 \\
& \ \ -\mathrm{Lip}_x^{\mathrm{rad}}f(w)^2-\mathrm{Lip}(f|_{\partial B_{\overline{x,z}} \setminus C_x})(w)^2 -1) \\
&=\frac{1}{2}(\mathrm{Lip}_x^{\mathrm{rad}}(r_x + f)(w)^2-\mathrm{Lip}_x^{\mathrm{rad}}f(w)^2-1) \\
&=\frac{1}{2}\left(\lim_{h \rightarrow 0}\frac{|(r_x + f)\circ \gamma (\overline{x,w}+ h)-(r_x +f)(w)|^2}{|h|^2}-\lim_{h \rightarrow 0}\frac{|f\circ \gamma (\overline{x,w}+h)-f(w)|^2}{|h|^2}-1\right) \\
&= \frac{1}{2}\left(\lim_{h \rightarrow 0}\left|1 + \frac{f \circ \gamma(\overline{x,w} + h)-f(w)}{h}\right|^2-\lim_{h \rightarrow 0}
\frac{|f\circ \gamma(\overline{x,w}+h)-f(w)|^2}{|h|^2}-1\right) \\
& \left( \mathrm{Here,\  we \ have \ the  \ existence \ of \ the \ limit \  }\lim_{h \rightarrow 0}\frac{f \circ \gamma(\overline{x,w}+h)-f(w)}{h}.\right) \\
&=\frac{1}{2}\Biggl(1 + 2\lim_{h \rightarrow 0}\frac{f \circ \gamma(\overline{x,w}+h)-f(w)}{h}+ \lim_{h \rightarrow 0}
\frac{|f \circ \gamma(\overline{x,w}+h)- f(w)|^2}{|h|^2} \\
& \ \ -\lim_{h \rightarrow 0}\frac{|f \circ \gamma(\overline{x,w}+h)- f(w)|^2}{|h|^2} -1 \Biggl) \\
&= \lim_{h \rightarrow 0} \frac{f \circ \gamma(\overline{x,w}+h)- f(w)}{h} = \frac{df}{dr_x}(w).
\end{align*}
\end{proof}
\subsection{Rectifiability associated with Lipschitz functions}
In this section, we will give a generalization of Theorem \ref{7}.
First, we shall state the following lemma:
\begin{lemma}\label{15}
Let $\delta$ be a positive number, $\{(M_i, m_i)\}_i$ a sequence of $n$-dimensional complete Riemannian manifolds with 
$\mathrm{Ric}_{M_i} \ge -\delta(n-1)$, $(Y, y, \upsilon)$ an $(n, -\delta)$-Ricci limit space of $\{(M_i, m_i, \underline{\mathrm{vol}})\}_i$,  
$x, x_1, x_2$ points in $Y$, $x(i), x_1(i), x_2(i)$ points in $M_i$, $\mathbf{b}_1^i$ a harmonic function on $B_{100}(x(i))$ and 
$\mathbf{b}_1^{\infty}$ a Lipschitz function on $B_{100}(x)$.
We assume that $\overline{x,x_1} \ge \delta^{-1}$, $\overline{x,x_2}\ge 
\delta^{-1}$, $\overline{x, x_1} + \overline{x, x_2} - \overline{x_1, x_2} \le \delta$, 
$x(i) \rightarrow x$, $x_j(i) \rightarrow x_j(i) (j=1, 2)$, $\sup_i\mathbf{Lip}\mathbf{b}_1^i < \infty$, $\mathbf{b}_1^i \rightarrow \mathbf{b}_1^{\infty}$ on $B_{100}(x)$, 
 \[|\mathbf{b}_1^i-r_{x_1(i)}|_{L_{\infty}(B_{100}(x(i)))} \le \delta, \] 
\[\frac{1}{\mathrm{vol}\,B_{100}(x(i))}\int _{B_{100}(x(i))}|\nabla \mathbf{b}_1^i-\nabla r_{x_1(i)}|^2d\mathrm{vol} \le \delta\]
and
\[\frac{1}{\mathrm{vol}\,B_{100}(x(i))}\int_{B_{100}(x(i))}|\mathrm{Hess}_{\mathbf{b}_1^i}|^2d\mathrm{vol} \le \delta.\]
Then, we have 
\[\frac{1}{\upsilon(B_1(x))}\int_{B_1(x)}|d\mathbf{b}_1^{\infty}-dr_{x_1}|^2d \upsilon < \Psi(\delta;n).\]
\end{lemma}
We remark that Lemma \ref{15} does \textit{not} follows from   \cite[Lemma $9. 10$]{ch1} directly.
We shall give a proof of Lemma \ref{15} in the proof of the following Lemma \ref{16}.
\begin{lemma}\label{16}
Let $\delta$ be a positive number, $\{(M_i, m_i)\}_i$ a sequence of $n$-dimensional complete Riemannian manifolds with 
$\mathrm{Ric}_{M_i} \ge -\delta(n-1)$, $(Y, y, \upsilon)$ an $(n, -\delta)$-Ricci limit space of $\{(M_i, m_i, \underline{\mathrm{vol}})\}_i$, 
$x, x_j (j = 1, 2, 3, 4)$ points in $Y$ and $x(i), x_j(i) (j=1, 2, 3, 4)$ points in $M_i$.
We assume that $x(i) \rightarrow x$, $x_j(i) \rightarrow x_j (j=1, 2, 3, 4))$,  $\overline{x,x_j} \ge \delta^{-1}$, 
$\overline{x, x_1} + \overline{x, x_2} - \overline{x_1, x_2} \le \delta$ and 
$\overline{x, x_3} + \overline{x, x_4} - \overline{x_3, x_4} \le \delta$.
Then, for every sufficiently large $i$, we have 
\[\frac{1}{\upsilon(B_1(x))}\int_{B_1(x)}\langle dr_{x_1}, dr_{x_2}\rangle d \upsilon = 
\frac{1}{\mathrm{vol}\,B_1(x(i))}\int_{B_1(x(i))}\langle dr_{x_1(i)}, dr_{x_2(i)}\rangle d \mathrm{vol} \pm \Psi (\delta; n)
\]
\[\frac{1}{\upsilon(B_1(x))}\int_{B_1(x)}\left| \langle dr_{x_1}, dr_{x_2}\rangle d \upsilon - 
\frac{1}{\upsilon(B_1(x))}\int_{B_1(x)}\langle dr_{x_1}, dr_{x_2}\rangle d \upsilon \right|d \upsilon < \Psi (\delta; n)\]
and
\[\frac{1}{\mathrm{vol}\,B_1(x(i))}\int_{B_1(x(i))}\left| \langle dr_{x_1(i)}, dr_{x_2(i)}\rangle -\frac{1}{\upsilon(B_1(x))}\int_{B_1(x)}\langle dr_{x_1}, dr_{x_2}\rangle d \upsilon \right|d \mathrm{vol} < \Psi (\delta; n)\]
\end{lemma}
\begin{proof}
First, we remark  the following claim:
\begin{claim}\label{114}
For every sufficiently large $i$, there exist harmonic functions $\mathbf{b}_1^i, \mathbf{b}_3^i$ on 
$B_{100}(x(i))$ such that $\mathbf{Lip}\mathbf{b}_j^i \le C(n)$, $|\mathbf{b}_j^i-r_{x_j(i)}|_{L^{\infty}(B_{100}(x(i)))} \le \Psi(\delta;n)$,
\[ \frac{1}{\mathrm{vol}\,B_{100}(x(i))} \int_{B_{100}(x(i))}|d \mathbf{b}_j^i-d r_{x_j(i)}|^2 d\mathrm{vol} \le \Psi (\delta;n)\]
and
\[ \frac{1}{\mathrm{vol}\,B_{100}(x(i))}\int_{B_{100}(x(i))}|\mathrm{Hess}_{\mathbf{b}_j^i}|^2 d\mathrm{vol} \le \Psi(\delta;n)\]
for $j = 1, 3$.
\end{claim}
See \cite[Lemma $9. 8$, Lemma $9. 10$, Lemma $9. 13$]{ch1} or \cite[Lemma $6.15$, Lemma $6.22$, Proposition $6.60$]{ch-co} for the proof of Claim \ref{114}.
\

Since $C(n)(|\mathrm{Hess}_{\mathbf{b}_1^i}|^2+|\mathrm{Hess}_{\mathbf{b}_3^i}|^2)$ is an upper gradient of 
$\langle d \mathbf{b}_1^i, d \mathbf{b}_3^i\rangle$,
by Poincar\'{e} inequality, we have
\begin{align*}
&\frac{1}{\mathrm{vol}\,B_{100}(x(i))}\int_{B_{100}(x(i))}\left| \langle d\mathbf{b}_1^i, d\mathbf{b}_3^i \rangle -\frac{1}{\mathrm{vol}\,B_{100}(x(i))}\int_{B_{100}(x(i))}\langle d \mathbf{b}_1^i, d \mathbf{b}_3^i \rangle d\mathrm{vol}\right|d\mathrm{vol} \\
&\le C(n)\frac{1}{\mathrm{vol}\,B_{100}(x(i))}\int_{B_{100}(x(i))}\left(|\mathrm{Hess}_{\mathbf{b}_1^i}|^2 + 
|\mathrm{Hess}_{\mathbf{b}_3^i}|^2\right)d\mathrm{vol} \le \Psi(\delta;n).
\end{align*}
Therefore, we have 
\begin{align*}
&\frac{1}{\mathrm{vol}\,B_{100}(x(i))}\int_{B_{100}(x(i))}\left| \langle d \mathbf{b}_3^i, dr_{x_1(i)}\rangle - 
\frac{1}{\mathrm{vol}\,B_{100}(x(i))}\int_{B_{100}(x(i))}\langle d \mathbf{b}_3^i, dr_{x_1(i)}\rangle d\mathrm{vol}\right|d\mathrm{vol} \\
&\le \Psi(\delta;n).
\end{align*}
By Proposition \ref{Lips}, without loss of generality, we can assume that 
there exists Lipschitz functions $\mathbf{b}_1^{\infty}, \mathbf{b}_3^{\infty}$ on $B_{100}(x)$ such that 
$\mathbf{b}_j^i \rightarrow \mathbf{b}_j^{\infty}$ on $B_{100}(x)$.
By Theorem \ref{14}, there exists a Borel set $A \subset B_{100}(x) \setminus C_{x_1}$ such that $\upsilon (B_{100}(x) \setminus A)=0$ and that 
\[\lim_{h \rightarrow 0}\frac{f \circ \gamma (\overline{x_1, a}+h)-f(a)}{h}=\langle dr_{x_1}, d\mathbf{b}_3^{\infty}\rangle (a)\]
for every $a \in A$ and minimal geodesic $\gamma$ from $x_1$ to $a$. 
By Lusin's theorem, there exists a Borel set $A(\delta) \subset A$ such that $\upsilon (A \setminus A(\delta))< \delta \upsilon(B_1(x))$ and that the function $\langle dr_{x_1}, df\rangle |_{A(\delta)}$ is continuous.
For every $0 < \eta < \delta$, we put a function $f_{\eta}^{\delta}$ on $A(\delta) \setminus B_{2\delta}(x)$ by
\[f_{\eta}^{\delta}(z)=\sup_{w \in C_{z}(\{x_1\}) \cap \overline{B}_{\eta}(z)}\left| \frac{f(z)-f(w)}{\overline{z,w}}-\langle dr_{x_1}, df\rangle (z)\right|.\]
It is easy to check that $f_{\eta}^{\delta}$ is an upper semi-continuous function.
Especially, $f_{\eta}^{\delta}$ is a Borel function.
By the definition of $A$, for every $a \in A$, we have $\lim_{\eta \rightarrow 0}f_{\eta}^{\delta}(a)=0$.
Thus, by Egoroff's theorem, there exists a Borel set $X=X(\delta) \subset A(\delta)$ such that $\upsilon(A(\delta) \setminus X(\delta))< \delta \upsilon(B_1(x))$ and that 
\[\lim_{\eta \rightarrow 0}(\sup_{a \in X}f_{\eta}^{\delta}(a))=0.\]
We take $\eta=\eta(\delta) << \delta$ satisfying $\sup_{a \in X}f_{\eta_0}^{\delta}(a) < \delta$ for every $\eta_0 \le \eta$.
For every $i$, let $X_i$ denote the set of points $w \in B_1(x(i))$ such that
\[\left| \langle d\mathbf{b}_3^i, dr_{x_1(i)}\rangle (w)-\frac{1}{\mathrm{vol}\,B_{100}(x(i))}\int_{B_{100}(x(i))}\langle d\mathbf{b}_3^i, dr_{x_1(i)}\rangle d\mathrm{vol}\right|
\le \Psi(\delta;n)\].
Then, we have 
$\mathrm{vol}(B_1(x(i))\setminus X_i)/\mathrm{vol}\,B_1(x(i)) \le \Psi(\delta;n)$ for every sufficiently large $i$.
For every $i$, we define a Borel function $F_i$ on $B_{100}(x(i)) \setminus C_{x_1(i)}$, 
\[F_i (w)=\frac{\mathbf{b}_3^i(\gamma(\overline{x_1(i), w}-\eta^2))-\mathbf{b}_3^i(w)}{-\eta^2}.\]
Here, $\gamma$ is the minimal geodesic from $x_1(i)$ to $w$.
\begin{claim}\label{115}
For every sufficiently large $i$, we have 
\[\frac{1}{\mathrm{vol}\,B_{10}(x(i))} \int_{B_{10}(x(i)) \setminus C_{x_1(i)}}|\langle d\mathbf{b}_3^i, dr_{x_1(i)}\rangle -
F_i(w)|d\mathrm{vol} \le \Psi(\delta;n).\]
\end{claim}
The proof is as follows.
It is easy to check that for every $a<b$, smooth function $f$ on $(a, b)$ and $c \in (a, b)$, we have 
\[f(t)=f(c) + f'(t)(t-c) - \int_c^t(s-c)f''(s)ds.\]
Therefore, we have
\[\frac{\mathbf{b}_3^i(\gamma(\overline{x_1(i), w}- \eta^2))-\mathbf{b}_3^i(w)}{-\eta^2}=
\frac{d\mathbf{b}_3^i}{dr_{x_1(i)}}(w)-\frac{1}{\eta^2}\int_{\overline{x_1(i), w}-\eta^2}^{\overline{x_1(i), w}}
\left(s-(\overline{x_1(i), w}-\eta^2)\right)\frac{d^2\mathbf{b}_3^i}{dr_{x_1(i)}^2}(\gamma(s))ds.\]
Thus,  by an argument similar to the proof of \cite[Estimate $2.6$]{ho2}, we have
\begin{align}
&\frac{1}{\mathrm{vol}\,B_{10}(x(i))} \int_{B_{10}(x(i)) \setminus C_{x_1(i)}}\left|\langle d\mathbf{b}_3^i, dr_{x_1(i)}\rangle -
F_i(w)\right|d\mathrm{vol} \\
&\le \frac{1}{\eta^2}\frac{1}{\mathrm{vol}\,B_{10}(x(i))}\int_{B_{10}(x(i))} \int_{\overline{x_1(i), w}-\eta^2}^{\overline{x_1(i), w}}\eta^2|\mathrm{Hess}_{\mathbf{b}_3^i}|(\gamma(s))ds d\underline{\mathrm{vol}} \\
&\le \eta^2 C(n) \frac{1}{\mathrm{vol}\,B_{100}(x(i))}\int_{B_{100}(x(i))}|\mathrm{Hess}_{\mathbf{b}_3^i}|d\mathrm{vol} \\
&\le \eta^2 C(n) \sqrt{\frac{1}{\mathrm{vol}\,B_{100}(x(i))}\int_{B_{100}(x(i))}|\mathrm{Hess}_{\mathbf{b}_3^i}|^2d\mathrm{vol}} \\
&\le \eta^2C(n)\Psi(\delta;n). 
\end{align}
Therefore, we have Claim \ref{115}
\

\begin{claim}\label{116}
For sufficiently large $i$, we have 
\[\frac{1}{\upsilon(B_1(x))} \int _{B_1(x)}\left|\langle d\mathbf{b}_3^{\infty}, dr_{x_1}\rangle -
\frac{1}{\mathrm{vol}\,B_1(x(i))} \int _{B_1(x(i))}\langle d\mathbf{b}_3^i, dr_{x_1(i)}\rangle d\mathrm{vol}
\right|d\upsilon \le \Psi(\delta;n).\]
\end{claim}
The proof is as follows.
Let $Y_i=\{w \in \overline{B}_1(x(i)) \setminus C_{x_1(i)};
|\langle d\mathbf{b}_3^i, dr_{x_1(i)}\rangle (w)-F_i(w) |\le \Psi(\delta;n)\}$.
By Claim \ref{115}, we have
$\mathrm{vol}\,(\overline{B}_1(x(i)) \setminus Y_i)/\mathrm{vol}\,\overline{B}_1(x(i))\le \Psi(\delta;n)$
for every sufficiently large $i$.
We put $Z_i = X_i \cap Y_i$.
We take a compact set $W_i \subset Z_i$ satisfying 
$\mathrm{vol}(Z_i \setminus W_i)/\mathrm{vol}\,\overline{B}_1(x(i)) \le \Psi(\delta;n)$.
Thus, we have
$\mathrm{vol}(\overline{B}_1(x(i)) \setminus W_i)/\mathrm{vol}\,\overline{B}_1(x(i)) \le \Psi(\delta;n)$
for every sufficiently large $i$.
By Proposition \ref{compact}, without loss of generality, we can assume that 
there exists a compact set $W_{\infty} \subset \overline{B}_1(x)$ such that 
$W_j \rightarrow W_{\infty}$.
By Lemma \ref{sup}, we have 
$\upsilon(W_{\infty})/\upsilon(\overline{B}_1(x))\ge 1-\Psi(\delta;n)$.
We put $E=W_{\infty} \cap X$, then $\upsilon(\overline{B}_1(x) \setminus E) \le 
\Psi(\delta;n)\upsilon(\overline{B}_1(x))$.
For every $w_i \in W_i$ and $w \in E$, we take the minimal geodesic $\gamma_{w_i}$ from $x_1(i)$ to
$w_i$ and a minimal geodesic $\gamma_w$ from $x_1$ to $w$.
Then, there exists $i_0$ such that for every $i \ge i_0$, $w \in E$ and $w_i \in W_i$ satisfying $w_i \rightarrow w$, we have $\epsilon_i << \eta$,
\[\left|\langle d\mathbf{b}_3^i, dr_{x_1(i)}\rangle (w)-\frac{\mathbf{b}_3^{i}(\gamma_i(\overline{x_1(i), w_i}-\eta^2))-\mathbf{b}_3^{i}(w_i)}{-\eta^2}\right| \le \Psi(\delta;n)\]
and
\[\left|\langle d\mathbf{b}_3^i, dr_{x_1(i)}\rangle (w_i)- \frac{1}{\mathrm{vol}\,B_{100}(x(i))}\int_{B_{100}(x(i))}\langle d\mathbf{b}_3^i, dr_{x_1(i)}\rangle d\mathrm{vol}\right| \le \Psi(\delta;n).\]
On the other hand, by rescaling $\eta^{-2}d_Y$, since
\[\overline{x_1, \phi_i(\gamma_i(\overline{x_1(i), w_i}-\eta^2))}^{\eta^{-2}d_Y} \ge \eta^{-1}, \ 
\overline{\phi_i(\gamma_i(\overline{x_1(i), w_i}-\eta^2)), w}^{\eta^{-2}d_Y} \ge \eta^{-1}\]
and
\[\overline{x_1, \phi_i(\gamma_i(\overline{x_1(i), w_i}-\eta^2))}^{\eta^{-2}d_Y}+\overline{\phi_i(\gamma_i(\overline{x_1(i), w_i}-\eta^2)), w}^{\eta^{-2}d_Y} - \overline{x_1, w}^{\eta^{-2}d_Y}\le \eta,\]
by splitting theorem, we have
\[\overline{\phi_i(\gamma_i(\overline{x_1(i), w_i}-\eta^2)), \gamma(\overline{x_1, w}-\eta^2)}^{\eta^{-2}d_Y} \le \Psi(\delta;n).\]
Therefore, we have 
\[\left|\frac{\mathbf{b}_3^{i}(\gamma_i(\overline{x_1(i), w_i}-\eta^2))-\mathbf{b}_3^{i}(w_i)}{-\eta^2}-
\frac{\mathbf{b}_3^{\infty}(\gamma(\overline{x_1, w}-\eta^2))-\mathbf{b}_3^{\infty}(w)}{-\eta^2}\right| \le \Psi(\delta;n).\]
Thus,  we have 
\[\left|\langle d\mathbf{b}_3^{\infty}, dr_{x_1}\rangle (w)-\frac{1}{\mathrm{vol}\,B_{100}(x(i))}\int_{B_{100}(x(i))}\langle d\mathbf{b}_3^i, dr_{x_1(i)}\rangle d\mathrm{vol}\right| \le \Psi(\delta;n).\]
We put 
\[C_i = \frac{1}{\mathrm{vol}\,B_{100}(x(i))}\int_{B_{100}(x(i))}\langle d\mathbf{b}_3^i, dr_{x_1(i)}\rangle d\mathrm{vol}.\]
Then 
\begin{align}
&\frac{1}{\upsilon(B_1(x))}\int_{B_1(x)}\left| \langle d\mathbf{b}_3^{\infty}, dr_{x_1}\rangle -C_i\right|d\upsilon \\
&= \frac{1}{\upsilon(B_1(x))} \int _{B_1(x) \setminus E}\left| \langle d\mathbf{b}_3^{\infty}, dr_{x_1}\rangle -C_i\right|d\upsilon + 
\frac{1}{\upsilon(B_1(x))}\int_E \left| \langle d\mathbf{b}_3^{\infty}, dr_{x_1}\rangle -C_i\right|d\upsilon \\
&\le \frac{C(n)\upsilon(B_1(x) \setminus E)}{\upsilon(B_1(x))} + \frac{\upsilon(E)}{\upsilon(B_1(x))}\Psi(\delta;n) \le \Psi(\delta;n).
\end{align}
Therefore, we have Claim \ref{116}.
\

\begin{claim}\label{117}
We have 
\[\frac{1}{\upsilon(B_1(x))}\int_{B_1(x)}|d\mathbf{b}_3^{\infty}|^2d\upsilon \le 1 + \Psi(\delta;n).\]
\end{claim}
This proof is as follows.
Since
\[\frac{1}{\mathrm{vol}\,B_1(x(i)}\int_{B_1(x(i))}||d\mathbf{b}_3^i|-1|d\mathrm{vol} \le \Psi(\delta;n)\]
for every sufficiently large $i$, 
by \cite[Lemma $16.2$]{ch2}, there exists a compact set $K_i \subset \overline{B}_1(x(i))$ such that 
$\underline{\mathrm{vol}}(B_1(x(i))\setminus K_i)/\underline{\mathrm{vol}}B_1(x(i)) \le \Psi(\delta;n)$ and that
$\mathbf{Lip}(\mathbf{b}_3^i|_{K_i})\le 1 + \Psi(\delta;n)$.
By Proposition \ref{compact}, without loss of generality, we can assume that 
there exists a compact set $K_{\infty} \subset \overline{B}_1(x)$ such that 
$K_i \rightarrow K_{\infty}$.
By Lemma \ref{sup}, we have $\upsilon(K_{\infty})/\upsilon(B_1(x)) \ge 1- \Psi(\delta;n)$.
By the definition, we have $\mathbf{Lip}(\mathbf{b}_3^{\infty}|_{K_{\infty}}) \le 1 + \Psi(\delta;n)$.
We put $\hat{K}_{\infty} = \mathrm{Leb}K_{\infty}$.
Then by Proposition \ref{133}, we have 
\begin{align*}
\frac{1}{\upsilon(B_1(x))}\int_{B_1(x)}|d\mathbf{b}_3^{\infty}|^2 d \upsilon &=\frac{1}{\upsilon(B_1(x))}\int_{\hat{K}_{\infty}}
|d\mathbf{b}_3^{\infty}|^2d\upsilon + \frac{1}{\upsilon(B_1(x))}\int_{B_1(x)\setminus K_{\infty}}|d\mathbf{b}_3^{\infty}|^2
d\upsilon \\
&\le \frac{1}{\upsilon(B_1(x))}\int_{\hat{K}_{\infty}}
(\mathrm{Lip}\mathbf{b}_3^{\infty})^2 d\upsilon + C(n)\frac{\upsilon(B_1(x)\setminus K_{\infty})}{\upsilon(B_1(x))} \\
&\le\frac{1}{\upsilon(B_1(x))} \int_{\hat{K}_{\infty}}(\mathrm{Lip}(\mathbf{b}_3^{\infty}|_{K_{\infty}}))^2d\upsilon + \Psi(\delta;n) \\
&\le \frac{1}{\upsilon(B_1(x))}\int_{\hat{K}_{\infty}}(1+\Psi(\delta;n))d\upsilon +\Psi(\delta;n) \le 1+ \Psi(\delta;n).
\end{align*}
Therefore, we have Claim \ref{117}.
\

If we consider the case $x_1 = x_3, x_2 = x_4$, then, by Claim \ref{114}, \ref{116} and \ref{117}, we have 
\begin{align}
&\frac{1}{\upsilon(B_1(x))}\int_{B_1(x)}|d\mathbf{b}_3^{\infty}-dr_{x_3}|^2d\upsilon \\
&=\frac{1}{\upsilon(B_1(x))}\int_{B_1(x)}|d\mathbf{b}_3^{\infty}|^2d\upsilon -2\frac{1}{\upsilon(B_1(x))}\int_{B_1(x)}\langle d\mathbf{b}_3^{\infty}, dr_{x_3}\rangle d\upsilon + \frac{1}{\upsilon(B_1(x))}\int_{B_1(x)}|dr_{x_3}|^2d\upsilon \\
&\le1 + \Psi(\delta;n)-2(1-\Psi(\delta;n))+1 \le \Psi(\delta;n)
\end{align}
for every sufficiently large $i$.
Therefore, we have Lemma \ref{15}.
On the other hand, 
Lemma \ref{16} follows from Lemma \ref{15} and Claim \ref{116}.
\end{proof}
\begin{lemma}\label{17}
Let $\{(M_i, m_i)\}_i$ be a sequence of $n$-dimensional complete Riemannian manifolds with 
$\mathrm{Ric}_{M_i} \ge -(n-1)$, 
$(Y, y, \upsilon)$ a Ricci limit space of $\{(M_i, m_i, \underline{\mathrm{vol}})\}_i$, $\tau$ a positive number, $x, x_1, x_2$ points in $Y$, $x(i), x_1(i), x_2(i)$ points in $M_i$.
We assume that 
$x \in \bigcap_{j=1, 2} (\mathcal{D}_{x_j(i)}^{\tau} \setminus B_{\tau}(x_j))$,
$x(i) \rightarrow x$ and $x_j(i) \rightarrow x_j (j=1, 2)$.
Then, for every sufficiently large $i$, we have 
\[\frac{1}{\upsilon(B_r(x))}\int_{B_r(x)}\langle dr_{x_1}, dr_{x_2}\rangle d \upsilon = 
\frac{1}{\mathrm{vol}\,B_r(x(i))}\int_{B_r(x)}\langle dr_{x_1(i)}, dr_{x_2(i)}\rangle d \mathrm{vol} \pm \Psi (r, \frac{r}{\tau}; n),\]
\[\frac{1}{\upsilon(B_r(x))}\int_{B_r(x)}\left| \langle dr_{x_1}, dr_{x_2}\rangle d \upsilon - 
\frac{1}{\upsilon(B_r(x))}\int_{B_r(x)}\langle dr_{x_1}, dr_{x_2}\rangle d \upsilon \right|d \upsilon \le \Psi (r, \frac{r}{\tau}; n)\]
and
\[\frac{1}{\mathrm{vol}\,B_r(x(i))}\int_{B_r(x)}\left| \langle dr_{x_1(i)}, dr_{x_2(i)}\rangle -\frac{1}{\mathrm{vol}\,B_r(x(i))}\int_{B_r(x)}\langle dr_{x_1(i)}, dr_{x_2(i)}\rangle d\mathrm{vol} \right|d \mathrm{vol}\le \Psi (r, \frac{r}{\tau}; n)\]
\end{lemma}
\begin{proof}
By rescaling $r^{-1}d_Y$ and Lemma \ref{16}, it is easy to check the assertion.
\end{proof}
Next corollary follows from Lemma \ref{16}, \ref{17} directly:
\begin{corollary}\label{19}
Let $\{(M_i, m_i)\}_i$ be a sequence of $n$-dimensional complete Riemannian manifolds with 
$\mathrm{Ric}_{M_i} \ge -(n-1)$,
$(Y, y)$ a Ricci limit space of $\{(M_i, m_i, \underline{\mathrm{vol}})\}_i$, $\tau, L$ positive numbers, 
$x, x_1, _{\cdots}, x_k, z_1, _{\cdots}, z_l$ points in $Y$,  
$x(i), x_1(i), _{\cdots}, x_k(i), z_1(i), _{\cdots}, z_l(i)$ points in $M_i$ and $a_1, _{\cdots}, a_k, b_1, _{\cdots}, b_l$ real numbers.
We assume that $x \in \bigcap_{i=1}^k(\mathcal{D}_{x_i}^{\tau} \setminus B_{\tau}(x_i)) \cap 
\bigcap_{i=1}^l(\mathcal{D}_{z_i}^{\tau} \setminus B_{\tau}(z_i))$, $x(i) \rightarrow x$, $x_j(i) \rightarrow x_j (j=1, _{\cdots}, k)$, 
$z_m(i) \rightarrow z_m (m=1, _{\cdots}, l)$ and
$\sum_{i=1}^k a_i^2  + \sum _{i=1}^l b_i^2\le L$.
We put $f = \sum _{j=1}^k a_j r_{x_j}, g = \sum _{j=1}^l b_j r_{z_j}, 
f_i = \sum _{j=1}^k a_j r_{x_j(i)}$ and $g_i = \sum _{j=1}^l b_j r_{z_j(i)}$.
Then, for every sufficiently large $i$, we have,
\[\frac{1}{\mathrm{vol}\,B_r(x(i))} \int _{B_r(x(i))}\left| \langle df_i, dg_i\rangle -\frac{1}{\upsilon(B_r(x))}\int _{B_r(x)}\langle df, dg\rangle d \upsilon \right| d \mathrm{vol} \le \Psi (r, \frac{r}{\tau};n, L),\]
\[\frac{1}{\upsilon(B_r(x))}\int _{B_r(x)}\left| \langle df, dg\rangle - \frac{1}{\mathrm{vol}\,B_r(x(i))} \int _{B_r(x(i))} \langle df_i, dg_i\rangle d \mathrm{vol} \right| d \upsilon \le \Psi (r, \frac{r}{\tau};n, L).\]
\end{corollary}

\begin{lemma}\label{20}
Let $\{(M_i, m_i)\}_i$ be a sequence of $n$-dimensional complete Riemannian manifolds with 
$\mathrm{Ric}_{M_i} \ge -(n-1)$,
$(Y, y, \upsilon)$ a Ricci limit space of $\{(M_i, m_i, \underline{\mathrm{vol}})\}_i$, $l, k_{\alpha} (1 \le \alpha \le l)$ positive integers,  $r, \epsilon, \tau, L$ positive numbers,
$x, x_t^s (1 \le s \le l, 1 \le t \le k_l)$ points in $Y$,  
$x(i), x_t^s(i)$ points in $M_i$ and $a_t^s (1 \le s \le l, 1 \le t \le k_l)$ real numbers.
We put $f_j= \sum _{m=1}^{k_j}a_m^jr_{x_m^j}, f_j^i = \sum _{m=1}^{k_j}a_m^jr_{x_m^j(i)}$.
We assume that $l \le n$, $k_i \le n(1 \le i \le l)$, $x \in \bigcap_{1 \le i \le l, 1 \le j \le k_i}^k(\mathcal{D}_{x_j^i}^{\tau} \setminus B_{\tau}(x_j^i))$, 
$x(i) \rightarrow x$, $x_t^s(i) \rightarrow x_t^s$,
$\sum_{i, j} (a_j^i)^2 \le L$ and
\[\frac{1}{\upsilon(B_r(x))}\int _{B_r(x)}\langle df_j, df_i\rangle d \upsilon = \delta_{ij} \pm \epsilon.\]
Then, for every sufficiently large $i$, there exists a compact set $K_r^i \subset \overline{B}_{r/10}(x(i))$ satisfying the following properties:
\begin{enumerate}
\item $\mathrm{vol}(B_{r/10}(x(i)) \setminus K^i_r)/ \mathrm{vol}\,B_{r/10}(x(i)) \le \Psi (r, r/\tau, \epsilon ;n,  L)$
\item For every $w \in K_r^i$ and $0 < s < r/10^6$, there exist a compact set $Z \subset \overline{B}_s(w)$, 
$z \in Z$ and a map $\phi : (\overline{B}_s(w), w) \rightarrow (Z, z)$ such that 
the map $\Phi = (f_1^i, f_2^i, _{\cdots}, f_l^i, \phi)$ from $\overline{B}_s(w)$ to $\overline{B}_{s+\Psi (r, r/\tau, \epsilon ;n,  L)s}(f_1^i(w), _{\cdots}, f_l^i(w), \phi(w))$, gives 
$\Psi (r, r/\tau, \epsilon ;n, L)s$-Gromov-Hausdorff approximation.
\item For every $w \in K_r^i$ and $0 < s < r/10^6$, we have 
\[\frac{1}{\mathrm{vol}\,B_s(w)}\int_{B_s(w)}|\langle df_{\alpha}^i, df_{\beta}^i\rangle -\delta_{\alpha \beta}|d\underline{\mathrm{vol}} < \Psi (r, \frac{r}{\tau}, \epsilon; n, L).\]
\end{enumerate}
\end{lemma}
\begin{proof}
By Lemma \ref{20}, we have 
\[ \frac{1}{\mathrm{vol}\,B_r(x(i))} \int _{B_r(x(i))}|\langle df^i_j, df_{\hat{l}}^i\rangle -\delta_{j, \hat{l}}|d\mathrm{vol}
\le  \Psi(r, \frac{r}{\tau}, \epsilon; n, L)\]
for every sufficiently large $i$.
We consider rescaled distances $r^{-1}d_Y$ and $r^{-1}d_{M_i}$.
For convenience, we shall use the following notations:
$\hat{\mathrm{vol}} = \mathrm{vol}^{r^{-1}d_{M_i}}$, 
$\hat{\upsilon}=\upsilon/\upsilon(B_r(y))$, 
$\hat{r}_z(w) = r^{-1}\overline{w,z}^{d_Y}$,
$\hat{B}_s(w)=B_s^{r^{-1}d_Y}(w)=B_{sr}(w)$,
$\hat{g}=r^{-1}g$ for Lipschitz function $g$ and so on. 
We remark that $(M_i, m_i, r^{-1}d_{M_i}, \underline{\mathrm{vol}}^{r^{-1}d_{M_i}}) \rightarrow (Y, y, r^{-1}d_Y, \hat{\upsilon})$.
We also denote the differential section of Lipschitz function $f$ on $Y$ as metric measure space $(Y, \hat{\upsilon})$ 
by $\hat{d}f: Y \rightarrow T^*Y$ and denote the Riemannian metric of rescaled Ricci limit space $(Y, y, r^{-1}d_Y, \hat{\upsilon})$ by $\langle \cdot, \cdot \rangle_r$. 
By the definition, we have $\langle \cdot, \cdot \rangle_r=r^{-2}\langle \cdot, \cdot \rangle$.
Then we have 
\[ \frac{1}{\hat{\mathrm{vol}}\,\hat{B}_1(x(i))} \int_{\hat{B}_1(x(i))}|\langle \hat{d}\hat{f}^i_j, \hat{d}\hat{f}_{\hat{l}}^i\rangle_r-\delta_{j, \hat{l}}|d\hat{\mathrm{vol}}
\le \Psi(r, \frac{r}{\tau}, \epsilon ; n, L).\]
On the other hand, by \cite[Lemma $9. 8$, Lemma $9. 10$, Lemma $9. 13$]{ch1} (or \cite[Lemma $6.15$, Lemma $6.22$, Proposition $6.60$]{ch-co}), there exist harmonic functions $\hat{\mathbf{b}}_j^{m,i}
(1 \le m \le l, 1\le j \le k_m)$ on 
$\hat{B}_{100}(x(i))$ such that 
$|\hat{\mathbf{b}}_j^{m,i}-\hat{r}_{x_j^m(i)}|_{L^{\infty}(\hat{B}_{100}(x(i)))} \le \Psi(r, r/\tau ;n)$,
\[\frac{1}{\hat{\mathrm{vol}}\,\hat{B}_{100}(x(i))} \int _{\hat{B}_{100}(x(i))}
|\hat{d}\hat{\mathbf{b}}_j^{m,i}-\hat{d}\hat{r}_{x_j^m(i)}|_r^2d\hat{\mathrm{vol}} \le \Psi(r, \frac{r}{\tau}; n), \]
and 
\[\frac{1}{\hat{\mathrm{vol}}\,\hat{B}_{100}(x(i))} \int _{\hat{B}_{100}(x(i))}
|\mathrm{Hess}_{\hat{\mathbf{b}}_j^{m,i}}|_r^2 d \hat{\mathrm{vol}} \le \Psi (r, \frac{r}{\tau};n).\]
We put 
$\hat{\mathbf{b}}_j^i = \sum_{m=1}^{k_j}a_m^j \hat{\mathbf{b}}_j^{m, i}$.
Then, we have 
\[|\hat{f}_j^i-\hat{\mathbf{b}}_j^{i}|_{L^{\infty}(\hat{B}_{100}(x(i)))} \le \Psi(r, \frac{r}{\tau};n, L), \]
\[\frac{1}{\hat{\mathrm{vol}}\,\hat{B}_{100}(x(i))} \int _{\hat{B}_{100}(x(i))}
|\hat{d}\hat{\mathbf{b}}_j^{i}-\hat{d}\hat{f}_j^i|_r^2d\hat{\mathrm{vol}} \le \Psi(r, \frac{r}{\tau};n, L)\]
and 
\[\frac{1}{\hat{\mathrm{vol}}\,\hat{B}_{100}(x(i))} \int _{\hat{B}_{100}(x(i))}
|\mathrm{Hess}_{\hat{\mathbf{b}}_j^{i}}|_r^2 d \hat{\mathrm{vol}} \le \Psi(r, \frac{r}{\tau};n, L).\]
Especially, we have 
\[\frac{1}{\hat{\mathrm{vol}}\,\hat{B}_{100}(x(i))} \int _{\hat{B}_{100}(x(i))}
|\langle \hat{d}\hat{\mathbf{b}}_j^i, \hat{d}\hat{\mathbf{b}}_l^i\rangle_r-\delta_{j,l}|d\hat{\mathrm{vol}} \le 
\Psi (r, \frac{r}{\tau}, \epsilon; n, L)\]
We put 
\[\hat{F}_i = \sum_{j=1}^l|\hat{d}\hat{\mathbf{b}}_j^i-\hat{d}{\hat{f}}_j^i|_r^2
+ \sum_{j =1}^l||\hat{d}\hat{\mathbf{b}}_j^i|_r^2-1|
+ \sum _{j < \hat{l}}|\langle \hat{d}\mathbf{b}_j^i, \hat{d}\mathbf{b}_{\hat{l}}^i\rangle_r|
+ \sum_{j=1}^l|\mathrm{Hess}_{\hat{\mathbf{b}}_j^i}|_r^2.\]
By Lemma \ref{1}, we have the following:
\begin{claim}\label{118}
For every sufficiently large $i$, there exists a compact set $K_r^i \subset \hat{B}_{1/10}(x(i))$ such that 
\[\frac{\hat{\mathrm{vol}}(\hat{B}_{\frac{1}{10}}(x(i))\setminus K_r^i)}{\hat{\mathrm{vol}}\,\hat{B}_{\frac{1}{10}}(x(i))} \le \Psi(r, \frac{r}{\tau}, \epsilon; n, L),\]
and that  
\[\frac{1}{\hat{\mathrm{vol}}\,\hat{B}_{5s}(w)} \int _{\hat{B}_{5s}(w)}\hat{F}_id\hat{\mathrm{vol}} \le \Psi( r, \frac{r}{\tau }, \epsilon; n , L)\]
for every $w \in K_r^i$ and $0 < s < 1/10$.
\end{claim}
We fix $w \in K_r^i$ and $0 < s \le 1/10$.  
By an argument same to the proof of \cite[Theorem $3. 3$]{ch-co3}, we have the following:
\begin{claim}\label{119}
There exist a compact set $Z \subset \hat{\overline{B}}_s(w)$, a point $z \in Z$ and a map $\phi$ from 
$\hat{\overline{B}}_{s/10^5}(w)$ to $Z$, such that the map
$\Phi(\alpha)=(\hat{\mathbf{b}}_1^i(\alpha), _{\cdots}, \hat{\mathbf{b}}_l^i(\alpha), \phi(\alpha))$ from $\hat{\overline{B}}_{s/10^5}(w)$ to $\overline{B}_{s/10^5+\Psi s}(\hat{\mathbf{b}}_1^i(w), _{\cdots}, 
\hat{\mathbf{b}}_l^i(w), \phi(w)) \subset \mathbf{R}^k \times Z$, gives $\Psi s$-Gromov-Hausdorff approximation.
Here, $\Psi= \Psi(r, r/\tau, \epsilon ;n, L)$.
\end{claim}
Since
\[\frac{1}{\hat{\mathrm{vol}}\,\hat{B}_{5s}(w)}\int_{\hat{B}_{5s}(w)}|\hat{d}\hat{\mathbf{b}}_j^i-\hat{d}\hat{f}_j^i|_r^2d\hat{\mathrm{vol}} \le \Psi(r, \frac{r}{\tau}, \epsilon;n, L),\]
by segment inequality (see \cite[Theorem $2. 15$]{ch-co3}), for every $z_1 \in \hat{\overline{B}}_s(w)$, there exist $\hat{z}_1 \in \hat{\overline{B}}_{5s}(w)$, $\hat{w} \in \hat{\overline{B}}_{5s}(w)$ and a minimal geodesic $\gamma$ from $\hat{z}_1$ to $\hat{w}$ such that 
$\overline{z_1, \hat{z}_1} \le \Psi(r, r/\tau, \epsilon;n, L)$,
$\overline{w, \hat{w}} \le \Psi(r, r/\tau, \epsilon;n, L)$, and that 
\[\int_0^{\overline{\hat{z}_1, \hat{w}}}\hat{\mathrm{Lip}}(\hat{\mathbf{b}}_j^i-\hat{f}_j^i)(\gamma(t))dt \le 
\Psi(r, \frac{r}{\tau }, \epsilon;n,L)s.\]
Therefore, we have
\[|\hat{\mathbf{b}}_j^i(\hat{z}_1)-\hat{f}_j^i(\hat{z}_1)-(\hat{\mathbf{b}}_j^i(\hat{w})-\hat{f}_j^i(\hat{w}))|
\le \int_0^{\overline{\hat{z}_1, \hat{w}}} \hat{\mathrm{Lip}}(\hat{\mathbf{b}}_j^i-\hat{f}_j^i)(\gamma(t))dt
\le \Psi(r, \frac{r}{\tau}, \epsilon;n, L)s.\]
By Cheng-Yau's gradient estimate, we have $\hat{\mathbf{Lip}}(\hat{\mathbf{b}}_j^i|_{\hat{B}_{2s}(w)}) \le C(n, L)$.
Thus, we have
$|\hat{\mathbf{b}}_j^i(z_1)-\hat{f}_j^i(z_1)-(\hat{\mathbf{b}}_j^i(w)-\hat{f}_j^i(w))|
\le \Psi(r, r/\tau, \epsilon;n,L)s. $
Therefore, if we put $C= \hat{\mathbf{b}}_j^i(w)-\hat{f}_j^i(w)$, then we have 
\[\hat{\mathbf{b}}_j^i=\hat{f}_j^i + C \pm \Psi(r, \frac{r}{\tau}, \epsilon;n,L)s\]
on $\hat{\overline{B}}_s(w)$.

Thus, the map $\hat{\Phi}(\alpha)=(\hat{f}_1^i(\alpha), _{\cdots}, \hat{f}_l^i(\alpha), \phi(\alpha))$ from $\hat{\overline{B}}_{s/10^5}(w)$ to 
$\overline{B}_{s/10^5+\Psi s}(\hat{f}_1^i(w), _{\cdots}, \hat{f}_l^i(w), \phi(w))$, gives  $\Psi s $-Gromov-Hausdorff approximation.
Therefore we have the assertion.
\end{proof}

\begin{lemma}\label{21}
Let
$(Y, y, \upsilon)$ be a Ricci limit space, $\tau, \epsilon, \delta, L$ positive numbers, $l, m, k_s (1 \le s \le l \le m)$ positive integers,
$x, x_t^s(1 \le s \le l, 1 \le t \le k_s)$ points in $Y$ 
and
$a_t^s$ real numbers.
We put $f_j= \sum _{m=1}^{k_j}a_m^jr_{x_m^j}$.
We assume that
$x \in \mathrm{Leb}\left(\bigcap_{1 \le i \le l, 1 \le j \le k_i}(\mathcal{D}_{x_j^i}^{\tau} \setminus \{x_j^i \}) \cap (\mathcal{R}_m)_{\delta, \tau}\right)$, 
 $\sum_{i, j} (a_j^i)^2\le L$  
and
\[\limsup_{r \rightarrow 0}\frac{1}{\upsilon(B_r(x))}\int _{B_r(x)}|\langle df_j, df_i\rangle-\delta_{ij}|d \upsilon \le \epsilon.\]
Then, 
for every sufficiently small $s > 0$, there exists a compact set $K_s \subset \overline{B}_s(x)$ 
satisfying the following properties:
\begin{enumerate}
\item $\upsilon(K_s)/\upsilon(B_s(x)) \ge 1-\Psi(\epsilon, \delta;n, L)$.
\item For every $\alpha \in K_s$ and every sufficiently small $t > 0$, there exist points $w_1^t(\alpha), _{\cdots}, w_{m-l}^t(\alpha) \in Y$
 and a compact set $U_t \subset \overline{B}_t(\alpha)$ such that $\upsilon(U_t)/\upsilon(B_t(\alpha)) \ge 1- \Psi(\epsilon, \delta;n, L)$ and that 
the map $\Phi_t = (f_1, _{\cdots}, f_l, r_{w_1^t(\alpha)}, _{\cdots}, r_{w_{m-l}^t(\alpha)})$ from $U_t$ to $\mathbf{R}^m$,
gives $(1\pm \Psi (\epsilon, \delta; n, L))$-bi-Lipschitz equivalent to the image $\Phi_t(U_t)$.
\end{enumerate}
\end{lemma}
\begin{proof}
Let $(M_i, m_i, \underline{\mathrm{vol}}) \rightarrow (Y, y, \upsilon)$.
We take $x_t^s(i) \in M_i$ satisfying $x_t^s(i) \rightarrow x_t^s$ and put $f_j^i= \sum _{m=1}^{k_j}a_m^jr_{x_m^j(i)}$.
There exists $s_1 > 0$ such that $s_1 << \tau$,
\[\frac{1}{\upsilon(B_{10^{10}s}(x))}\int _{B_{10^{10}s}(x)}|\langle df_j, df_i\rangle -\delta_{ij}|d \upsilon + \frac{\upsilon \left(B_{10^{10}s}(x) \cap \bigcap_{1 \le i \le l, 1 \le j \le k_i} (\mathcal{D}_{x_j^i}^{\tau} \cap (\mathcal{R}_m)_{\delta, r})\right)}{\upsilon (B_{10^{10}s}(x))}\le 3\epsilon \]
for every $0 < s < s_1$.
By Proposition \ref{compact} and Lemma \ref{20},
for every $0 < s < s_1$, there exists a compact set $K_s \subset \overline{B}_{10^9s}(x)$ satisfying the following properties:
\begin{enumerate}
\item $\upsilon (K_s)/\upsilon (B_{10^9s}(x)) \ge 1- \Psi (\epsilon; n, L)$.
\item For every $w \in K_s$ and $0 < t < 10^4s$, 
there exist a compact set $Z_t^w \subset \overline{B}_t(w)$ and a map $\phi_t^w$ from $\overline{B}_t(w)$ to $Z_t^w$ such that 
the map $\Phi_t^w = (f_1, _{\cdots}, f_l, \phi_t^w)$ from $\overline{B}_t(w)$ to $\overline{B}_{10^9(t+\Psi t)}(f_1(w), _{\cdots}, f_l(w), \phi_t^w(w))$, gives $\Psi t$-Gromov-Hausdorff approximation. 
Here $\Psi = \Psi(\epsilon;n, L)$
\item For every $w \in K_s$ and $0 < t < 10^4s$, we have 
\[\frac{1}{\upsilon(B_{t}(w))}\int _{B_{t}(w)}|\langle df_j, df_i\rangle -\delta_{ij}|d \upsilon \le \Psi (\epsilon; n, L). \]
\end{enumerate}
Here, with same notation as in Lemma \ref{20}, we used Proposition \ref{10103} as
\[\lim_{k \rightarrow \infty}\frac{1}{\underline{\mathrm{vol}}\,B_{t}(w(k))}\int_{B_{t}(w(k))}|\langle df_j^k, df_i^k\rangle -\delta_{ij}|d \underline{\mathrm{vol}} =
\frac{1}{\upsilon(B_{t}(w))}\int _{B_{t}(w)}|\langle df_j, df_i\rangle -\delta_{ij}|d \upsilon. \]
for $w(k) \rightarrow w$.
We fix $0 < s < s_1$ and take $K_s$, $w \in K_s \cap \mathrm{Leb}(\bigcap_{1 \le i \le l, 1 \le j \le k_i}(\mathcal{D}_{x_j^i}^{\tau} \setminus \{x_j^i \}) \cap (\mathcal{R}_m)_{\delta, r}), 0 < t < 10^4s, Z_t^w, \phi_t^w, \Phi_t^w$ as above.
We remark that $\upsilon (K_s \cap \mathrm{Leb}(\bigcap_{1 \le i \le l, 1 \le j \le k_i}(\mathcal{D}_{x_j^i}^{\tau} \setminus \{x_j^i \}) \cap (\mathcal{R}_m)_{\delta, r}))/\upsilon (B_{10^9s}(x)) \ge 1- \Psi (\epsilon; n, L)$.
We assume that $t$ is sufficiently small and that  
\[\frac{\upsilon \left(B_{\hat{t}}(w) \cap \bigcap_{1 \le i \le l, 1 \le j \le k_i}(\mathcal{D}_{x_j^i}^{\tau} \setminus \{x_j^i \}) \cap (\mathcal{R}_m)_{\delta, r}\right)}{\upsilon(B_{\hat{t}}(w))} \ge 1- \epsilon\]
for every $0 < \hat{t} \le t$.
There exist points $y_i^+, y_j^- \in \overline{B}_{t}(w)(1 \le i, j \le l)$ such that 
$\overline{\Phi_t^w(y_i^+), (\underbrace{0, _{\cdots}, 0, t}_i, 0, _{\cdots}, 0, \phi_t^w(w))} \le \Psi t$ and $\overline{\Phi_t^w(y_j^-), (\underbrace{0, _{\cdots}, 0, -t}_j, 0, _{\cdots}, 0, \phi_t^w(w))} \le \Psi t.$
We also take an $\Psi t$-Gromov-Hausdorff approximation $\hat{\Phi}_t^w$ from $\overline{B}_{10^9(t+\Psi t)}(f_1(w), _{\cdots}, f_l(w), \phi_t^w(w))$ to $\overline{B}_t(w)$ satisfying $\overline{\Phi_t^w \circ \hat{\Phi}_t^w(\alpha), \alpha}\le \Psi t$ for every $\alpha \in 
\overline{B}_{10^9(t+\Psi t)}(f_1(w), _{\cdots}, f_l(w), \phi_t^w(w))$ and $\overline{\hat{\Phi}_t^w \circ \Phi_t^w(\beta), \beta} \le \Psi t$
for every $\beta \in \overline{B}_t(w)$.
On the other hand, we can take $\delta t$-Gromov-Hausdorff approximation $\psi_t^w$ from $(\overline{B}_{t}(w), w)$ to $(\overline{B}_{t}(0_m), 0_m)$ and
$\hat{\psi}_t^w$ from $(\overline{B}_{t}(0_m), 0_m)$ to $(\overline{B}_{t}(w), w)$ satisfying that 
$\overline{\psi_t^w \circ \hat{\psi}_t^w(\alpha), \alpha } \le 5\delta t$ for every $\alpha \in \overline{B}_{t}(0_m)$ and 
$\overline{\hat{\psi}_t^w \circ \psi_t^w(\beta), \beta } \le 5\delta t$ for every $\beta \in \overline{B}_{t}(w)$.
Especially, there exists an $\Psi t$-Gromov-Hausdorff approximation $\hat{h}_t^w$ from $(\overline{B}_t(0_{m-l}), 0_{m-l})$ to $(Z_t^w, \phi_t^w(w))$
such that $\overline{(0, _{\cdots}, 0, \alpha), \psi_t^w \circ \hat{\Phi}_t^w(f_1(w),_{\cdots}, f_l(w), \hat{h}_t^w(\alpha))}\le \Psi t$ for every $\alpha \in Z_t^w$.
Here $\Psi=\Psi(\epsilon, \delta;n, L)$.
Without loss of generality, we can assume that 
$\overline{\psi_t^w(y_i^+), (\underbrace{0, _{\cdots}, 0, t}_i, 0, _{\cdots}, 0)} \le \Psi t.$
There  exist points $z_{i}^+, z_{j}^- \in \overline{B}_{t}(w)(l+1 \le i, j \le m)$ such that 
$\overline{\psi_t^w(z_i^+), (\underbrace{0, _{\cdots}, 0, t}_i, 0, _{\cdots}, 0)} \le \Psi t$ and 
$\overline{\psi_t^w(z_j^-), (\underbrace{0, _{\cdots}, 0, -t}_j, 0, _{\cdots}, 0)} \le \Psi t$.
We put $F_i=f_i-f_i(w)$ and define a function $G_i$ on $(\overline{B}_t(0_m), 0_m)$ by $G_i = F_i \circ \psi_t^w$. 
Since $\overline{\pi_{\mathbf{R}^{m-l}}(\psi_t^w \circ \hat{\Phi}_t^w(f_1(w),_{\cdots}, f_l(w), \hat{h}_t^w(\alpha))),\alpha} \le \Psi t$, 
the map $G=(G_1, _{\cdots}, G_l, \pi_{l+1}, _{\cdots}, \pi_m)$ from $(\overline{B}_t(0_m), 0_m)$ to $(\overline{B}_{t+\Psi t}(0_m), 0_m)$
gives $\Psi t$-Gromov-Hausdorff approximation and satisfies $\overline{G((\underbrace{0, _{\cdots}, 0, \pm t}_i, 0, _{\cdots}, 0), (\underbrace{0, _{\cdots}, 0, \pm t}_i, 0, _{\cdots}, 0)} \le \Psi t$.
Here $\pi_{\mathbf{R}^{m-l}}$ is the canonical projection $\mathbf{R}^m = \mathbf{R}^l \times \mathbf{R}^{m-l} \rightarrow \mathbf{R}^{m-l}$
and $\pi_i$ is the $i$-th projection $\mathbf{R}^m \rightarrow \mathbf{R}$.
Thus, we have $\overline{\alpha, G(\alpha)} \le \Psi t$ for every $\alpha \in \overline{B}_{t}(0_m)$.
Especially, we have the following claim:
\begin{claim}\label{120}
We have 
\[|G_i - \pi_i| \le \Psi(\epsilon, \delta ;n, L)t\]
on $B_{t}(0_m)$.
\end{claim}
We fix $0 < \hat{t} < t$.
By rescaling $\hat{t}^{-1}d_Y$, $\hat{t}^{-1}d_{\mathbf{R}^m}$, Claim \ref{120} and the definition of Busemann function, we have the following:
\begin{claim}\label{121}
We have 
\[|F_i(\alpha)-(r_{y_i^-}(\alpha)-r_{y_i^-}(w))| \le \Psi \left(\epsilon, \delta, \frac{\hat{t}}{t}, \frac{\Psi(\epsilon, \delta ;n, L)t}{\hat{t}};n, L\right)\hat{t}\]
on $\overline{B}_{\hat{t}}(w)$.
\end{claim}
We take $y_j^-(k), z_j^-(k), w(k) \in M_k$ such that $y_j^-(k) \rightarrow y_j^-, z_j^-(k) \rightarrow z_j^-$ and $w(k) \rightarrow w$.
For $\Psi=\Psi (\epsilon, \delta ;n, L)$ in Claim \ref{121}, we put $r=\sqrt{\Psi} t$.

For convenience, for rescaled distances $r^{-1}d_Y$ and $r^{-1}d_{M_i}$, we shall use the same notation 
as in the proof of Lemma \ref{20} below: $f_i^k, \hat{d}f, \hat{\mathrm{vol}}$ and so on.
\begin{claim}\label{122}
For every sufficiently large $k$, we have 
\[\frac{1}{\hat{\mathrm{vol}}\,\hat{B}_{100}(w(k))}\int_{\hat{B}_{100}(w(k))}|\hat{d}\hat{f}_i^k-\hat{d}\hat{r}_{y_i^-(k)}|_r^2d
\hat{\mathrm{vol}} \le \Psi(\epsilon, \delta ;n, L).\]
\end{claim}
This proof is as follows.
By the assumption and Proposition \ref{10103}, for every sufficiently large $k$, we have 
\[\frac{1}{\hat{\mathrm{vol}}\,\hat{B}_{1000}(x(k))}\int_{\hat{B}_{1000}(x(k))}||\hat{d}\hat{f}_i^k|_r^2-1|d
\hat{\mathrm{vol}} \le \Psi (\epsilon, \delta;n, L).\]
By an argument similar to the proof of Lemma \ref{20}, for every sufficiently large $k$, there exist harmonic functions $\hat{\mathbf{b}}_i^k$
on $\hat{B}_{100}(w(k))$, such that $\mathbf{Lip}\hat{\mathbf{b}}_i^k \le C(n)$,  $|\hat{\mathbf{b}}_i^k-\hat{f}_i^k|_{L^{\infty}(\hat{B}_{100}(w(k)))} \le \Psi (r, r/\tau; n, L)$,
\[ \frac{1}{\hat{\mathrm{vol}}\,\hat{B}_{1000}(w(k))}\int_{\hat{B}_{1000}(w(k))}
|\hat{d}\hat{\mathbf{b}}_i^k- \hat{d}\hat{f}_i^k|_r^2 d\hat{\mathrm{vol}} \le \Psi(r, r/\tau; n, L)\]
and 
\[\frac{1}{\hat{\mathrm{vol}}\,\hat{B}_{1000}(w(k))}\int_{\hat{B}_{1000}(w(k))}|\mathrm{Hess}_{\hat{\mathbf{b}}_i^k}|_r^2 d \hat{\mathrm{vol}} \le \Psi (r, r/\tau; n, L).\]
For every $\alpha \in \hat{B}_{1000}(w(k)) \setminus C_{y_i^-(k)}$, we take the minimal geodesic $\gamma^{\alpha}_i$
from $y_i^-(k)$ to $\alpha$ on $(M_i, r^{-1}d_{M_i})$.
We fix $0 < h < 1$.
By Claim \ref{121}, there exists $k_0$ such that for every $k \ge k_0$ and $\alpha \in \hat{B}_{1000}(w(k)) \setminus C_{y_i^-(k)}$, 
we have 
\begin{align}
&\frac{\hat{\mathbf{b}}_i^k(\alpha)-\hat{\mathbf{b}}_i^k(\gamma_i^{\alpha}(\overline{y_i^-(k), \alpha}^{r^{-1}d_{M_k}}-h))}{h} \\
&=\frac{\hat{f}_i^k(\alpha)-\hat{f}_i^k(\gamma_i^{\alpha}(\overline{y_i^-(k), \alpha}^{r^{-1}d_{M_k}}-h))}{h}  \pm \frac{\Psi(\epsilon, \delta;n, L)}{h}\\
&=\frac{\hat{f}_i(\phi_k(\alpha))-\hat{f}_i(\phi_k(\gamma_i^{\alpha}(\overline{y_i^-(k), \alpha}^{r^{-1}d_{M_k}}-h)))}{h} \pm  \frac{\Psi(\epsilon, \delta;n, L)}{h} \\
&=\frac{\overline{y_i^-, \phi_k(\alpha)}^{r^{-1}d_Y}-\overline{y_i^-, \phi_k(\gamma_i^{\alpha}(\overline{y_i^-(k), \alpha}^{r^{-1}d_{M_k}}-h)}^{r^{-1}d_Y}}{h} \pm \frac{\Psi(\epsilon, \delta;n, L)}{h} \\
&=\frac{\overline{y_i^-(k), \alpha}^{r^{-1}d_{M_k}}-\overline{y_i^-(k), \gamma_i^{\alpha}(\overline{y_i^-(k), \alpha}^{r^{-1}d_{M_k}}-h)}^{r^{-1}d_{M_k}}}{h} \pm \frac{\Psi(\epsilon, \delta; n, L)}{h} \\
&=1 \pm \frac{\Psi(\epsilon, \delta;n,L)}{h}.
\end{align}
On the other hand, 
by an argument similar to the proof of Claim \ref{115}, we have
\begin{align}
&\left|\frac{1}{\hat{\mathrm{vol}}\,\hat{B}_{100}(w(k))}\int _{\hat{B}_{100}(w(k))}
\frac{1}{h} \int_{\overline{y_i^-(k), \alpha}^{r^{-1}d_{M_k}}-h}^{\overline{y_i^-(k), \alpha}^{r^{-1}d_{M_k}}}
\left(s-(\overline{y_i^-(k), \alpha}^{r^{-1}d_{M_k}}-h)\right)\frac{d^2\hat{\mathbf{b}}_i^k \circ \gamma_i^{\alpha}}{ds^2}ds
d\hat{\mathrm{vol}}\right|\\
&\le C(n) \frac{h}{\hat{\mathrm{vol}}\,\hat{B}_{1000}(w(k))} \int_{\hat{B}_{1000}(w(k))}
|\mathrm{Hess}_{\hat{\mathbf{b}}_i^k}|_r d\hat{\mathrm{vol}} \le \Psi(\epsilon, \delta; n, L).
\end{align}
Since
\begin{align}
\hat{\mathbf{b}}_i^k(\alpha)&=
\hat{\mathbf{b}}_i^k(\gamma_i^{\alpha}(\overline{y_i^-(k), \alpha}^{r^{-1}d_{M_k}}-h)) + \frac{\hat{d}\hat{\mathbf{b}}_i^k}
{\hat{d}\hat{r}_{y_i^-(k)}}(\alpha)h \\
& \ \ \ -\int_{\overline{y_i^-(k), \alpha}^{r^{-1}d_{M_k}}-h}^{\overline{y_i^-(k), \alpha}^{r^{-1}d_{M_k}}}
\left(s-(\overline{y_i^-(k), \alpha}^{r^{-1}d_{M_k}}-h)\right)\frac{d^2\hat{\mathbf{b}}_i^k \circ \gamma_i^{\alpha}}{ds^2}ds,
\end{align}
for every $\alpha \in \hat{B}_{100}(w(k)) \setminus C_{y_i^-(k)}$, we have
\[\frac{1}{\hat{\mathrm{vol}}\,\hat{B}_{100}(w(k))}\int _{\hat{B}_{100}(w(k))}\langle \hat{d}\hat{\mathbf{b}}_i^k, \hat{d}\hat{r}_{y_i^-(k)}\rangle_rd \hat{\mathrm{vol}}
= 1 \pm \frac{\Psi (\epsilon, \delta; n, L)}{h}.\]
Therefore, we have
\begin{align*}
&\frac{1}{\hat{\mathrm{vol}}\,\hat{B}_{100}(w(k))} \int_{\hat{B}_{100}(x(k))}|\hat{d}\hat{f}_i^k-\hat{d}\hat{r}_{y_i^-(k)}|_r^2
d\hat{\mathrm{vol}} \\
&= 
\frac{1}{\hat{\mathrm{vol}}\,\hat{B}_{100}(w(k))} \int_{\hat{B}_{100}(w(k))}|\hat{d}\hat{f}_i^k|_r^2d\hat{\mathrm{vol}}
-\frac{2}{\hat{\mathrm{vol}}\,\hat{B}_{100}(w(k))} \int_{\hat{B}_{100}(w(k))}\langle \hat{d}\hat{f}_i^k, \hat{d}\hat{r}_{y_i^-(k)}\rangle_rd\hat{\mathrm{vol}}
+1 \\
&=1-2\frac{1}{\hat{\mathrm{vol}}\,\hat{B}_{100}(w(k))} \int_{\hat{B}_{100}(w(k))}\langle \hat{d}\hat{\mathbf{b}}_i^k, \hat{d}\hat{r}_{y_i^-(k)}\rangle_rd\hat{\mathrm{vol}} + 1 \pm \Psi(\epsilon, \delta;n, L) \\
\\
&=2 - 2(1 \pm \frac{\Psi(\epsilon, \delta;n, L)}{h}) \pm \Psi(\epsilon, \delta;n, L) = \frac{\Psi (\epsilon, \delta;n, L)}{h}.
\end{align*} 
Therefore, we have Claim \ref{122}.
\

Next claim follows from Claim \ref{122} and \cite[Theorem $9.29$]{ch1} directly:
\begin{claim}\label{123}
For every sufficiently large $k$, we have 
\[\frac{1}{\hat{\mathrm{vol}}\,\hat{B}_{100}(w(k))}\int_{\hat{B}_1(w(k))}
|\langle \hat{d}\hat{f}_i^k, \hat{d}\hat{r}_{z_j^-(k)}\rangle_r|d\hat{\mathrm{vol}} \le \Psi(\epsilon, \delta;n, L)\]
for every $1 \le i \le l$ and $l+1 \le j \le m$. 
Moreover we have
\[\frac{1}{\hat{\mathrm{vol}}\,\hat{B}_{100}(w(k))}\int_{\hat{B}_1(w(k))}
|\langle \hat{d}\hat{f}_i^k, \hat{d}\hat{f}_{\hat{i}}^k\rangle_r|d\hat{\mathrm{vol}} \le \Psi(\epsilon, \delta;n, L)\]
for every $1 \le i < \hat{i} \le l$.
\end{claim}
There exist harmonic functions $\hat{\mathbf{b}}_i^k (l+1 \le i \le m)$ on $\hat{B}_{1000}(w(k))$ such that $|\hat{r}_{z_i^-}- \hat{\mathbf{b}}_i^k|_{L^{\infty}(\hat{B}_{1000}(w(k)))} \le \Psi(\epsilon, \delta;n, L)$, 
\[ \frac{1}{\hat{\mathrm{vol}}\,\hat{B}_{1000}(w(k))}\int_{\hat{B}_{1000}(w(k))}
|\hat{d}\hat{\mathbf{b}}_i^k- \hat{d}\hat{r}_{z_i^-(k)}|_r^2 d\hat{\mathrm{vol}} \le \Psi(\epsilon, \delta;n, L)\]
and 
\[\frac{1}{\hat{\mathrm{vol}}\,\hat{B}_{1000}(w(k))}\int_{\hat{B}_{1000}(w(k))}|\mathrm{Hess}_{\hat{\mathbf{b}}_i^k}|_r^2 d \hat{\mathrm{vol}} \le \Psi(\epsilon, \delta;n, L).\]
We put 
\[\hat{F}_k = \sum_{1 \le i, j \le m}|\langle \hat{d}\hat{\mathbf{b}}_i^k, \hat{d}\hat{\mathbf{b}}_j^k\rangle_r- \delta_{i, j}| + \sum_{1 \le i \le m} |\mathrm{Hess}_{\hat{\mathbf{b}}_i^k}|_r^2
+ \sum_{i=1}^l |\hat{d}\hat{\mathbf{b}}_i^k- \hat{d}\hat{f}_i^k|_r^2 + \sum_{i=l+1}^m
|\hat{d}\hat{\mathbf{b}}_i^k-\hat{d}\hat{r}_{z_i^-}|_r^2.\]
Then, by Lemma \ref{1}, for every sufficiently large $k$, there exists a compact set $C(k) \subset \hat{\overline{B}}_{1}(w(k))$ such that 
$\hat{\mathrm{vol}}(\hat{B}_{1}(w(k)) \setminus C(k))/ \hat{\mathrm{vol}}\,\hat{B}_{1}(w(k)) \le \Psi(\epsilon, \delta;n, L)$
and that for every $\alpha \in C(k)$ and $0 < \hat{s} < 10$, we have
\[\frac{1}{\hat{\mathrm{vol}}\,\hat{B}_{\hat{s}}(\alpha)}\int_{\hat{B}_{\hat{s}}(\alpha)}\hat{F}_k d\hat{\mathrm{vol}} \le \Psi(\epsilon, \delta; n, L).\]
Thus, by an argument similar to the proof of \cite[Theorem $3. 3$]{ch-co3}, for every $\alpha \in C(k)$ and $0 < \hat{s} < 1$, there exist
a compact set $P_s^{\alpha} \subset \hat{\overline{B}}_{\hat{s}}(\alpha)$, a point $p_{\hat{s}}^{\alpha} \in P_{\hat{s}}^{\alpha}$ and a map $q_{\hat{s}}^{\alpha}$ from 
$(\hat{\overline{B}}_{\hat{s}}(\alpha), \alpha)$ to $(\overline{B}_{\hat{s}}(p_{\hat{s}}^{\alpha}), p_{\hat{s}}^{\alpha})$ such that 
the map $Q_{\hat{s}}^{\alpha} = (\hat{\mathbf{b}}_1^k, _{\cdots}, \hat{\mathbf{b}}_m^k, q_{\hat{s}}^{\alpha})$ from $\hat{\overline{B}}_{\hat{s}}(\alpha)$ to $\hat{\overline{B}}_{\hat{s} + \Psi \hat{s}}(\hat{\mathbf{b}}_1^k(\alpha), _{\cdots}, \hat{\mathbf{b}}_m^k(\alpha), p_{\hat{s}}^{\alpha})$, gives $\Psi \hat{s}$-Gromov-Hausdorff approximation.
For every $\alpha \in C(k)$ and $0 < \hat{s} < 1$.
by an argument similar to the proof of Claim \ref{119}, we have 
\[\hat{\mathbf{b}}_i^k = \hat{f}_i^k + \mathrm{constant} \pm \Psi \hat{s}\]
on $\hat{B}_{\hat{s}}(\alpha)$ for every $1 \le i \le l$, and
\[\hat{\mathbf{b}}_i^k = \hat{r}_{z_i^-(k)} + \mathrm{constant} \pm \Psi \hat{s}\]
on $\hat{B}_{\hat{s}}(\alpha)$ for $l+1 \le i \le m$.
Therefore, the map $\hat{Q}_{\hat{s}}^{\alpha} = (\hat{f}_1^k, _{\cdots}, \hat{f}_l^k, \hat{r}_{z_{l+1}^-(k)}, _{\cdots}, \hat{r}_{z_m^-(k)}, q_{\hat{s}}^{\alpha})$ from $\hat{\overline{B}}_{\hat{s}}(\alpha)$ to $\hat{\overline{B}}_{\hat{s} + \Psi \hat{s}}(\hat{f}_1^k(\alpha), _{\cdots}, \hat{f}_l^k(\alpha), \hat{r}_{z_{l+1}^-(k)}(\alpha), _{\cdots},  \hat{r}_{z_m^-(k)}(\alpha), p_{\hat{s}}^{\alpha})$, gives $\Psi \hat{s}$-Gromov-Hausdorff approximation.
\

By Proposition \ref{compact}, without loss of generality, we can assume that there exists a compact set $C(\infty) \subset \hat{\overline{B}}_1(w)$ such that $C(k) \rightarrow C(\infty)$.
We put $U= C(\infty) \cap \bigcap_{1 \le i \le l, 1 \le j \le k_i}(\mathcal{D}_{x_j^i}^{\tau} \setminus \{x_j^i \}) \cap (\mathcal{R}_m)_{\delta, r}$.
By Proposition \ref{sup}, we have $\hat{\upsilon}(\hat{B}_1(w) \cap U)/\hat{\upsilon}(\hat{B}_1(w)) \ge 1-\Psi$.
Since $\alpha \in (\mathcal{R}_m)_{\tau, \delta}$, by the argument above, the map $T_{\hat{s}}^{\alpha}=(\hat{f}_1, _{\cdots}, \hat{f}_l, \hat{r}_{z_{l+1}^-}, _{\cdots}, \hat{r}_{z_m^-})$ from $\hat{\overline{B}}_{\hat{s}}(\alpha)$ to 
$\overline{B}_{\hat{s}}(T_{\hat{s}}^{\alpha}(\alpha))$, gives $\Psi \hat{s}$-Gromov-Hausdorff approximation for every $\alpha \in U$ and $0 < \hat{s} < 1$.
Therefore for every $\alpha, \beta \in U \cap \hat{B}_{1/2}(w)$ satisfying $\alpha \neq \beta$, if we put $\hat{s}=\overline{\alpha, \beta}^{r^{-1}d_Y}< 1$,
then we have 
\begin{align}
&\overline{(\hat{f}_1(\alpha), _{\cdots}, \hat{f}_l(\alpha), \hat{r}_{z_{l+1}^-}(\alpha), _{\cdots}, \hat{r}_{z_m^-}(\alpha)), (\hat{f}_1(\beta), _{\cdots}, \hat{f}_l(\beta), \hat{r}_{z_{l+1}^-}(\beta), _{\cdots}, \hat{r}_{z_m^-}(\beta))} \\
&=\overline{\alpha, \beta}^{r^{-1}d_Y} \pm \Psi \hat{s} \\
&=(1 \pm \Psi)\overline{\alpha, \beta}^{r^{-1}d_Y}.
\end{align}
Therefore we have the assertion.
\end{proof}

\begin{lemma}\label{22}
Let $(Y, y, \upsilon)$ be a Ricci limit space, $l, k, m (1 \le l \le m \le n)$ positive integers, $x $ a point
in $Y$, $h_i (1 \le i \le l)$ Lipschitz functions on 
$Y$, $\tau$ a positive number, $x_i (1 \le i \le k)$ points in $Y$ and $a_i^j (1 \le i \le k, 1 \le j \le l)$ real numbers
We put $f_j= \sum _{i=1}^{k}a_i^jr_{x_i}$.
We assume that 
\[\lim_{r \rightarrow 0}\frac{1}{\upsilon(B_r(x))}\int_{B_r(x)}|df_j-dh_j|d\upsilon=0\]
for every $j$,
\[x \in \bigcap_{\delta > 0} \left(\bigcup_{r>0}\mathrm{Leb}\left(\bigcap_{i, j}(\mathcal{D}_{x_i^j}^{\tau} \setminus \{x_i^j\}) \cap (\mathcal{R}_m)_{\delta, r}\right)\right), \]
the limit
\[\lim_{r \rightarrow 0}\frac{1}{\upsilon(B_r(x))}\int_{B_r(x)}\langle dh_i, dh_j\rangle d \upsilon\]
exists for every $i, j$, and
\[\det \left(\lim_{r \rightarrow 0}\frac{1}{\upsilon(B_r(x))}\int_{B_r(x)}\langle dh_i, dh_j\rangle d \upsilon \right)_{i, j} \neq 0.\]
Then, for every $0 < \delta < 1$, there exists $r_0 > 0$ such that for every $0 < s < r_0$, there exists 
compact set $K_s \subset \overline{B}_s(x)$ satisfying the following properties:
\begin{enumerate}
\item $\upsilon(K_s)/\upsilon(B_s(x)) \ge 1-\delta$.
\item For every $\alpha \in K_s$ and every sufficiently small $t > 0$, there exist points $w_1^t(\alpha), _{\cdots}, w_{m-l}^t(\alpha) \in Y$
 and a compact set $U_t \subset \overline{B}_t(\alpha)$ such that $\upsilon(U_t)/\upsilon(B_t(\alpha)) \ge 1- \delta$ and that 
the map $\Phi_t = ((h_1, _{\cdots}, h_l)A, r_{w_1^t(\alpha)}, _{\cdots}, r_{w_{m-l}^t(\alpha)})$ from $U_t$ to $\mathbf{R}^m$,
gives $(1\pm \delta)$-bi-Lipschitz equivalent to the image $\Phi_t(U_t)$. Here, 
\[A = \sqrt{\left(\lim_{r \rightarrow 0}\frac{1}{\upsilon(B_r(x))}\int_{B_r(x)}\langle dh_i, dh_j\rangle d \upsilon \right)_{i, j}}^{-1}. \]
\end{enumerate}
\end{lemma}
\begin{proof}
We define Lipschitz functions $g_i$ on $Y$ by
$(g_1, _{\cdots}, g_l) = (h_1, _{\cdots}, h_l)A$.
By the definition, we have 
\[\lim _{r \rightarrow 0} \frac{1}{\upsilon(B_r(x))}\int_{B_r(x)}\langle g_i, g_j\rangle d\upsilon = \delta_{i, j}.\]
By Corollary \ref{19}, we have 
\[\lim _{r \rightarrow 0} \frac{1}{\upsilon(B_r(x))}\int_{B_r(x)}|\langle g_i, g_j\rangle -\delta_{i, j}|d\upsilon = 0.\]
We put $(F_1, _{\cdots}, F_l)=(\sum_{i=1}^kb_i^1r_{x_i}, _{\cdots}, 
\sum_{i=1}^kb_i^lr_{x_i})=(\sum_{i=1}^ka_i^1r_{x_i}, _{\cdots}, 
\sum_{i=1}^ka_i^lr_{x_i})A$ and take $L \ge 1$ such that $|A| + \sum_{i, j}(b_i^j)^2 \le L.$
We fix $ 0 < \delta < 1$.
By Lemma \ref{21}, we have the following claim:
\begin{claim}\label{124}
There exists $r_1 >0$ such that for every $0 < s \le r_1$, there exist a compact set $K_s \subset \overline{B}_s(x)$ satisfying the following properties:
\begin{enumerate}
\item $\upsilon(K_s)/\upsilon(B_s(x)) \ge 1-\delta$.
\item For every $\alpha \in K_s$ and every sufficiently small $t > 0$, there exist points $w_1^t(\alpha), _{\cdots}, w_{m-l}^t(\alpha) \in Y$
 and a compact set $E_t \subset \overline{B}_t(\alpha)$ such that $\upsilon(E_t)/\upsilon(B_t(\alpha)) \ge 1- \delta$ and that 
the map $\Phi_t = (F_1, _{\cdots}, F_l, r_{w_1^t(\alpha)}, _{\cdots}, r_{w_{m-l}^t(\alpha)})$ from $E_t$ to $\mathbf{R}^m$,
gives $(1\pm \delta)$-bi-Lipschitz equivalent to the image.
\end{enumerate}
\end{claim}
On the other hand, there exists $r_0 > 0$ such that  
\[\frac{1}{\upsilon(B_{s}(x))}\int_{B_{s}(x)}\sum_j|dF_j-dg_j|d\upsilon \le \delta \]
for every $0 < s < r_0$.
Thus, by Lemma \ref{1}, we have the following;
\begin{claim}\label{125}
For every $0 < s < r_0/100$, there exists a compact set $X_s \subset \overline{B}_{s}(x)$ such that 
$\upsilon(X_s)/\upsilon(\overline{B}_{s}(x)) \ge 1-\Psi(\delta;n)$ and that  
\[\frac{1}{\upsilon(B_{5\hat{s}}(\alpha))}\int_{B_{5\hat{s}}(\alpha)}\sum_j|dF_j-dg_j|d\upsilon \le \Psi(\delta;n) \]
for every $\alpha \in X_s$ and $0 < \hat{s} \le s$.
\end{claim}
We put $V_s = K_s \cap X_s$ for $0 < s < \min \{r_0, r_1\}/1000$.
Then we have $\upsilon(V_s)/\upsilon (B_s(x)) \ge 1-\Psi (\delta;n)$.
We fix $0 < s < \min \{r_0, r_1\}/1000$.
We also take $\alpha \in V_s$ and sufficiently small $t > 0$.
By an argument similar to the proof of Claim \ref{119}, we have 
\[F_j= f_j + \mathrm{constant} \pm \Psi(\delta;n)t\]
on $\overline{B}_{t}(\alpha)$.
We put $U_t= B_{t/2}(\alpha) \cap E_t$.
Then we have $\upsilon(U_t)/\upsilon(B_{t/2}(\alpha)) \ge 1- \Psi(\delta;n)$.
For $p_1, p_2 \in B_{t/2}(\alpha) \cap E_t$ satisfying $p_1 \neq p_2$, if we put $\hat{t}=\overline{p_1, p_2} > 0$, then we have 
\begin{align}
&\overline{(f_1(p_1), _{\cdots}, f_l(p_1), r_{w_1^t(\alpha)}, _{\cdots}, r_{w_{m-l}^t(\alpha)}(p_1)), (f_1(p_2), _{\cdots}, f_l(p_2), r_{w_1^t(\alpha)}(p_2), _{\cdots}, r_{w_{m-l}^t(\alpha)}(p_2))} \\
&=\overline{(F_1(p_1), _{\cdots}, F_l(p_1), r_{w_1^t(\alpha)}, _{\cdots}, r_{w_{m-l}^t(\alpha)}(p_1)), (F_1(p_2), _{\cdots}, F_l(p_2), r_{w_1^t(\alpha)}(p_2), _{\cdots}, r_{w_{m-l}^t(\alpha)}(p_2))} \pm \Psi \hat{t} \\
&=(1 \pm \delta)\overline{p_1, p_2} \pm \Psi \hat{t} =(1 \pm \Psi)\overline{p_1, p_2}.
\end{align}
Therefore we have the assertion.
\end{proof}
\begin{lemma}\label{23}
Let $(Y, y, \upsilon)$ be a Ricci limit space, $l$ a positive integer integer, 
$f_i, f (1 \le i \le l)$ Lipschitz functions on $Y$ and $A$ a Borel subset of $Y$. 
We assume that for a.e. $x \in A$, $\mathrm{span}\{df_1(x), _{\cdots}, df_l(x)\}= T^*_xY$.
Then, for a.e. $x \in A$,
there exists $b_1(x), _{\cdots}, b_l(x) \in \mathbf{R}$ such that 
\[\lim_{r \rightarrow 0} \frac{1}{\upsilon(B_r(x))}\int _{B_r(x)}\left|df-\sum_{i=1}^lb_i(x)df_i\right|^2d \upsilon = 0.\]
\end{lemma}
\begin{proof}
Without loss of generality, we can assume that for every $x \in A$, $\{df_i(x)\}$ is a base of $T^*_xY$.
For every $x \in A$, we put 
\[(b_1(x), \dots, b_l(x))=(\langle df, df_1 \rangle (x), \dots, \langle df, df_l \rangle (x)) \sqrt{(\langle df_i, df_j \rangle (x))_{i, j}}^{-1}.\]
By Corollary \ref{131}, for a.e. $x \in A$, we have 
\[\lim_{r \rightarrow 0}\frac{1}{\upsilon(B_r(x))}\int_{B_r(x)}|df|^2d\upsilon = |df|^2(x),\]
\[\lim_{r \rightarrow 0}\frac{1}{\upsilon(B_r(x))}\int_{B_r(x)}\langle df, df_i\rangle d\upsilon = \langle df, df_i\rangle (x)\]
for every $i$ and
\[\lim_{r \rightarrow 0}\frac{1}{\upsilon(B_r(x))}\int_{B_r(x)}\langle df_i, df_j\rangle d\upsilon = \langle df_i, df_j\rangle (x)\]
for every $i, j$.
Therefore, for a.e. $x \in A$, 
since
\[\lim_{r \rightarrow 0}\frac{1}{\upsilon(B_r(x))}\int_{B_r(x)}|df|^2d\upsilon = |df|^2(x)= \left|\sum_{i=1}^lb_i(x)df_i(x)\right|^2,\]
\begin{align}
\lim_{r \rightarrow 0}\frac{1}{\upsilon(B_r(x))}\int_{B_r(x)}\left\langle df, \sum_{i=1}^lb_i(x)df_i\right\rangle d\upsilon &= 
\sum_{i=1}^lb_i(x)\langle df, df_i\rangle (x) \\
&= \sum_{i=1}^lb_i(x)\left\langle \sum_{j=1}^lb_j(x)df_j, df_i\right\rangle (x) \\
&= \left|\sum_{i=1}^lb_i(a)df_i(x)\right|^2
\end{align}
and
\begin{align*}
\lim_{r \rightarrow 0}\frac{1}{\upsilon(B_r(x))}\int_{B_r(x)}\left|\sum_{i=1}^lb_i(a)df_i\right|^2d\upsilon 
= \sum_{i, j}^lb_i(x)b_j(x)\langle df_i, df_j\rangle (x) =\left|\sum_{i=1}^lb_i(x)df_i(x)\right|^2,
\end{align*}
we have 
\begin{align}
&\lim_{r \rightarrow 0}\frac{1}{\upsilon(B_r(x))}\int_{B_r(x)}\left|df-\sum_{i=1}^lb_i(x)df_i\right|^2d\upsilon \\
&=\lim_{r \rightarrow 0}\frac{1}{\upsilon(B_r(x))}\int_{B_r(x)}|df|^2d\upsilon  -2 
\lim_{r \rightarrow 0}\frac{1}{\upsilon(B_r(x))}\int_{B_r(x)}\left\langle df, \sum_{i=1}^lb_i(x)df_i\right\rangle d\upsilon \\
& \ \ + \lim_{r \rightarrow 0}\frac{1}{\upsilon(B_r(x))}\int_{B_r(x)}\left|\sum_{i=1}^lb_i(a)df_i\right|^2d\upsilon =0.
\end{align}
\end{proof}
\begin{theorem}[Rectifiability associated with Lipschitz functions]\label{24}
Let $(Y, y, \upsilon)$ be a Ricci limit space, $l$ a positive  integer,
$f_i (1 \le i \le l)$ Lipschitz functions on $Y$, $A$ a Borel subset of $Y$. 
We assume that $\{f_1(x), _{\cdots}, f_l(x)\}$ are linearly independent for a.e. $x \in A$.
Then, there exist $ 0 < \alpha(n) < 1 $, a collection of compact sets $\{C_{k,i}\}_{l \le k \le n, i \in \mathbf{N}} \subset A$,
 points $\{x_{k, i}\} \in A$ and $\{x_{k, i}^s\}_{1 \le s \le k-l} \in Y$ satisfying 
the following properties:
\begin{enumerate}
\item $\upsilon(A \setminus \bigcup_{l \le k \le n, i \in \mathbf{N}}C_{k, i}) = 0$.
\item For every $l \le k \le n$, $x \in \bigcup_{i \in \mathbf{N}}C_{k, i}$ and $0<\delta<1$, 
there exists $i \in \mathbf{N}$ such that $x \in C_{k, i}$ and that the map $\phi_{k, i}=((f_1(z), _{\cdots}, f_l(z)) \sqrt{(\langle df_i, df_j\rangle )_{i, j}(x_{k, i})}^{-1}, r_{x_{k, i}^1}, _{\cdots}, r_{x_{k, i}^{k-l}})$ gives a $(1 \pm \delta)$-bi-Lipschitz equivalent to the image $\phi_{k, i}(C_{k,i})$.
\item $C_{k, i} \subset \mathcal{R}_{k, \alpha(n)} \cap \bigcap_{j=1}^{k-l}(Y \setminus (C_{x_{k, i}^j} \cup \{x_{k, i}^j\}))$.
\item The limit measure $\upsilon$ and $k$-dimensional Hausdorff measure $H^k$ are mutually absolutely continuous on  $C_{k, i}$.
Moreover, $\upsilon$ is Ahlfors $k$-regular at every $x \in C_{k, i}$. 
\end{enumerate}
\end{theorem}

\begin{proof}
We take a collection of Borel subset $\{C_{k, i}^y\}$ of $Y$ and a collection of points $\{x_{k,i}^{\hat{l}}\}$ in $Y$ as in Theorem \ref{7}.
For convenience, we put $x_{k, i}^1=y, C_{k, i}= C_{k,i}^y$.
By Lemma \ref{6}, we can assume that $C_{k, i}$ is bounded for every $i, k$.
By the definition of $T^*Y$ (see section $4$ in \cite{ch1} or section $6$ in \cite{ch-co3} for the detail), we have   $\mathrm{span}\{dr_{x_{k, i}^1}(x), _{\cdots}, dr_{x_{k, i}^k}(x)\}=T^*_xY$ for a.e. $x \in C_{k,i}^y$.
Therefore, by the assumption, we have $\upsilon(A \cap C_{k, i})=0$ for $k < l$.
Since 
\[\upsilon \left(\mathcal{R}_k \setminus \bigcup_{\tau >0} \left(\bigcap_{\delta > 0} \left(\bigcup_{r>0}\mathrm{Leb}\left(\bigcap_{i, j}(\mathcal{D}_{x_i^j}^{\tau} \setminus \{x_i^j\}) \cap (\mathcal{R}_k)_{\delta, r}\right)\right)\right)\right)=0,\]
by Lemma \ref{22} and Lemma \ref{23}, we have the following claim:
\begin{claim}\label{126}
For every $k \ge l$ and $i \in \mathbf{N}$, there exists a Borel set $A_{k,i} \subset A \cap C_{k, i}$ satisfying the following properties:
\begin{enumerate}
\item $\upsilon(A \cap C_{k, i} \setminus A_{k, i})=0$.
\item For every $x \in A_{k, i}$ and $ 0 < \delta < 1$, there exists $r_x^{\delta}>0$ such that for every $0 < s < r_x^{\delta}$, there exists a compact set $K(x, \delta, s) \subset \overline{B}_s(x)$ satisfying the following properties:
\begin{enumerate}
\item $\upsilon(K(x, \delta, s))/\upsilon(B_s(x)) \ge 1- \delta$.
\item For every $\alpha \in K(x, \delta, s)$ and every sufficiently small $t > 0$, 
there exist points $w(i, x, \delta, s, \alpha, t) \in Y (1 \le i \le k-l)$ and a compact set $U(x, \delta, s, \alpha, t) \subset \overline{B}_t(\alpha)$ such that the map 
\[\Phi^{x, \delta, s, \alpha, t}=((f_1, _{\cdots}, f_l)A(x), r_{w(1, x, \delta, s, \alpha, t)}, 
_{\cdots}, r_{w(k-l, x, \delta, s, \alpha, t)})\]
 from $U(x, \delta, s, \alpha, t)$ to $\mathbf{R}^k$, gives $(1 \pm \delta)$-bi-Lipschitz equivalent to the image.
Here,
\begin{align}
A(x)&=\sqrt{\left(\lim_{r \rightarrow 0}\frac{1}{\upsilon(B_r(x))}\int_{B_r(x)}\langle df_s, df_t\rangle d\upsilon \right)_{s, t}}^{-1} \\
&=\sqrt{(\langle df_s, df_t\rangle (x))_{s, t}}^{-1}.
\end{align}
\end{enumerate}
\end{enumerate}
\end{claim}
We put $\hat{A}_{k, i}= \mathrm{Leb}(A_{k, i})$.
For every $N \in \mathbf{N}$ and $x \in \hat{A}_{k,i}$, 
we take $0 < s_x^{N} < \min \{r_x^{1/N}, N^{-1}\}$ satisfying
\[\frac{\upsilon(B_{s_x^{N}}(x) \cap A_{k, i})}{\upsilon(B_{s_x^{N}}(x))} \ge 1- N^{-1}.\]
We take $K(x, N^{-1}, s_x^{N})$ as in Claim \ref{126}.
We put $\hat{K}(x, N^{-1}, s_x^{N})=K(x, N^{-1}, s_x^{N}) \cap \hat{A}_{k, i}.$
Thus, we have 
\[\frac{\upsilon\left(B_{s_x^{N}}(x) \cap \hat{K}(x, N^{-1}, s_x^{N})\right)}{\upsilon(B_{s_x^{N}}(x))} \ge 1-100N^{-1}.\]
For every $\alpha \in \hat{K}(x, N^{-1}, s_x^{N})$, there exists a sufficiently small $0 < t=t(\alpha) < N^{-1}$ such that 
\[\frac{\upsilon(B_{\hat{t}}(\alpha) \cap A_{k, i})}{\upsilon(B_{\hat{t}}(\alpha))} \ge 1-N^{-1}\]
for every $0 < \hat{t} < t$.
We take $w(i, x, N^{-1}, s_x^{N}, \alpha, \hat{t})$ and $U(x, N^{-1}, s_x^{N}, \alpha, \hat{t})$ as in Claim \ref{126}.
We put $\hat{U}(x, N^{-1}, s_x^{N}, \alpha, \hat{t}) = U(x, N^{-1}, s_x^{N}, \alpha, \hat{t}) \cap \hat{A}_{k,i}$.
Then we have 
\[\frac{\upsilon\left(B_{\hat{t}}(\alpha) \cap \hat{U}(x, N^{-1}, s_x^{N}, \alpha, \hat{t})\right)}{\upsilon(B_{\hat{t}}(\alpha))}\ge 1-1000N^{-1}.\]
By Lemma \ref{cov}, it is not difficult to check that the following claim:
\begin{claim}\label{14567}
With same notation as above, there exist $x_j^N \in \hat{A}_{k, i}$, $\alpha_j^N \in \hat{K}(x_j^N, N^{-1}, s_{x_j^N}^{N})$ and 
$0 < t_j^N < t(\alpha_j^N)$ such that 
\[\upsilon\left(A_{k,i} \setminus \bigcup_{j \in \mathbf{N}}\hat{U}(x_j^N, N^{-1}, s_{x_j^N}^{N}, \alpha_j^N, t_j^N)\right)\le \Psi(N^{-1};n)\upsilon(B_{10}(A_{k, i})).\]
\end{claim}
We put $\hat{U}(j, N)=\hat{U}(x_j^N, N^{-1}, s_{x_j^N}^{N}, \alpha_j^N, t(\alpha_j^N))$,
$w(i, j, N) = w(i, x_j^N, N^{-1}, s_{x_j^N}^{N}, \alpha_j^N, t(\alpha_j^N))$,  
$U(j)= \bigcap_{N_0 \in \mathbf{N}}\left( \bigcup_{N_1 \ge N_0} \hat{U}(j, N_1)\right)$ and 
$U(j, N)=\hat{U}(j, N) \cap U(j)$.
Then we have $\upsilon(A_{k, i} \setminus \bigcup_{j \in \mathbf{N}}U(j))=0$ and $\bigcup_{N \in \mathbf{N}}U(j, N)=U(j)$.
We fix $j$.
We take $w \in \bigcup_{N \in \mathbf{N}}U(j, N)$ and $0 < \delta < 1$.
There exists $N_0$ such that $w \in U(j, N_0)$.
We take $N_1$ satisfying $N_1^{-1} << \delta$.
Since $w \in \bigcup_{N_2 \ge N_1}\hat{U}(j, N_2)$, there exists $N_2 \ge N_1$ such that $w \in \hat{U}(j, N_2)$.
Especially we have $w \in U(j, N_2)$.
Thus the map $G_{j, N_2}=((f_1, _{\cdots}, f_l)A(x_j^{N_2}), r_{w(1, j, N_2)}, _{\cdots}, r_{w(k-l, j, N_2)})$ from $U(j, N_2)$ to $\mathbf{R}^k$, gives
$(1 \pm N_2^{-1})$-bi-Lipschitz equivalent to the image.
Especially, $G_{j, N_2}$ gives $(1 \pm \delta)$-bi-Lipschitz equivalent to the image.
Therefore, we have the assertion.
\end{proof}
\begin{remark}
Radial rectifiability theorem (Theorem \ref{7}) corresponds to Theorem \ref{24} for a distance function $r_x$.
\end{remark}
We shall give two corollaries of Theorem \ref{24}. 
For metric space $X$, we define a distance on $\mathbf{R}_{\ge 0}\times X/\{0\}\times X$ by 
\[\overline{(t_1, x_1), (t_2, x_2)}= \sqrt{t_1^2+t_2^2-2t_1t_2\cos \min \{\overline{x_1, x_2}, \pi \}}.\]
Let $C(X)$ denote this metric space and $p = [(0, x)] \in C(X)$. 
\begin{corollary}\label{25}
Let $X$ be a compact geodesic space, $l$ a nonnegative integer.
We assume that $l \le n$, $\mathrm{dim}_HX=n-l-1$, $(\mathbf{R}^l \times C(X), (0_l, p))$ is an $(n, -1)$-Ricci limit space.
Here $p \in C(X)$ is the pole.
Then, $X$ is $H^{n-l-1}$-rectifiable.
\end{corollary}
\begin{proof}
We define $1$-Lipschitz functions $\pi_j (1 \le j \le l)$ and $g$ on $\mathbf{R}^k \times C(X)$ by $\pi_j(t_1, _{\cdots}, t_l, w)=t_j$
and $g(t_1, _{\cdots}, t_l, w)=\overline{p, w}$.
By Theorem \ref{14}, we have 
$\langle d\pi_i, d\pi_j\rangle (\alpha)=\delta_{i,j}, \langle d\pi_i, dg\rangle (\alpha)=0, |dg|(\alpha)=1$
for a.e. $\alpha \in \mathbf{R}^k \times C(X)$.
Therefore, we can take a collection of $\{C_{k,i}\}_{l +1 \le k \le n}$ as in Theorem \ref{24} for Lipschitz functions
$\pi_1, _{\cdots}, \pi_l, g$ and $A=\mathbf{R}^l \times C(X)$.
By an argument similar to the proof of Lemma \ref{0004}, the product measure $H^l \times H^{n-l}$ on $\mathbf{R}^l \times C(X)$ is equal to $H^n$.
Therefore by Fubini's theorem, we have 
\[0=H^n(\mathbf{R}^l\times C(X) \setminus \bigcup_{k, i} C_{k,i})=\int_{\mathbf{R}^l}H^{n-l}(\{t_1, _{\cdots}, t_l\} \times C(X) \setminus \bigcup_{k, i} C_{k,i})dH^l.\]
Especially, we can take $(t_1, _{\cdots}, t_l) \in \mathbf{R}^l$ satisfying 
$H^{n-l}(\{t_1, _{\cdots}, t_l\} \times C(X) \setminus \bigcup C_{k,i})=0$.
We put $\hat{C}_{k,i}= \{t_1, _{\cdots}, t_l\} \times C(X) \cap C_{k,i}$ and regard it as a subset of $C(X)$.
By an argument similar to the proof of Proposition \ref{0005}, we have  
\[\int_{C(X)}fdH^{n-l}=\int_0^{\infty}\int_{\partial B_t(p)} fdH^{n-l-1}dt\] 
for every $f \in L^1(C(X))$. (This is \textit{co-area formula for distance function from the pole on} $C(X)$).
Especially, we have
\[H^{n-l-1}(\partial B_t(p) \cap C(X) \setminus \bigcup_{k,i}\hat{C}_{k,i})=0\]
for a.e. $t >0$.
Then it is not difficult to check the assertion.
\end{proof}
\begin{remark}\label{26}
With same notation as in Corollary \ref{25}, for every $x \in X$ and $r > 0$, we have $0 < H^{n-l-1}(B_r(x))< \infty$.
It follows from \cite[Theorem $5. 9$]{ch-co1}, \cite[Theorem $4. 6$]{ch-co3} and co-area formula for distance function from the pole on  $C(X)$. Since it is not difficult to check it, we skipped the proof.
\end{remark}
Similarly, we have the following:
\begin{corollary}\label{27}
Let $(X,x)$ be a pointed proper geodesic space, $l$ a nonnegative integer.
We assume that $l \le n$, $\mathrm{dim}_HX=n-l$, $(\mathbf{R}^l \times X, (0_l, x))$ is $(n, -1)$-Ricci limit space.
Then, $X$ is $H^{n-l}$-rectifiable.
\end{corollary}
\section{Convergence of Borel functions and Lipschitz functions}
In this section, 
we will give several notions of convergence of a sequence of Borel functions.
By using these notions, we will define a notion of convergence of differential of Lipschitz functions (see Definition \ref{Lip}).
Moreover, by using results in section $3$, we will discuss convergence of harmonic functions.
Throughout subsections $4.1$ and $4.2$, we shall consider the following situation:
Let $(Z_i, z_i)$ be a sequence of pointed proper geodesic spaces, $\upsilon_i$ a Radon measure on $Z_i$ satisfying $\upsilon_i(B_1(z_i))=1$, and 
for every $R \ge 1$, there exists $K=K(R) \ge 1$ such that for every $1 \le i \le \infty$, $z \in Z_i$ and $0 < s \le R$, we have 
$\upsilon_i(B_{2s}(z)) \le 2^K \upsilon_i(B_s(z))$.
We assume that $(Z_i, z_i, \upsilon_i) \stackrel{(\phi_i, R_i, \epsilon_i)}{\rightarrow} (Z_{\infty}, z_{\infty}, \upsilon_{\infty})$.
We fix $x_i \in Z_i$ satisfying $x_i \rightarrow x_{\infty}$.
\subsection{Infinitesimal constant convergence property}
Our aims in this subsection are to define the following  notion of \textit{infinitesimal constant convergence} and to give several fundamental properties of it:
\begin{definition}[Infinitesimal constant convergence property]
Let $R$ be a positive number, $w$ a point in $B_R(x_{\infty})$ and $f_i$ a Borel function on $B_R(x_i) (1 \le i \le \infty)$ satisfying
$\sup _i |f_i|_{L^{\infty}(B_R(x_i))} + |f_{\infty}|_{L^{\infty}(B_R(x_{\infty}))} < \infty$.
We say that $\{f_i\}_i$ has \textit{infinitesimal constant convergence property to $f_{\infty}$ at} $w$ if  
for every $\epsilon > 0$, there exists $r > 0$ such that 
\[\limsup _{i \rightarrow \infty} \frac{1}{\upsilon_i(B_t(w_i))}\int_{B_t(w_i)}\left| f_i -\frac{1}{\upsilon_{\infty}(B_t(w))}\int _{B_t(w)}f_{\infty}d\upsilon_{\infty}\right| d\upsilon_i \le \epsilon\]
and 
\[ \limsup_{i \rightarrow \infty}\frac{1}{\upsilon_{\infty}(B_t(w))}\int_{B_t(w)}\left| f -\frac{1}{\upsilon_{i}(B_t(w_i))}\int _{B_t(w_i)}f_{i}d\upsilon_{i}\right| d\upsilon_{\infty}  \le \epsilon\]
for every $0 < t < r$ and $w_i \rightarrow w$.
\end{definition}
\begin{example}\label{ex1}
It is easy to check that for every $f \in C^0(B_R(x_{\infty}))$, if we put $f_i = f \circ \phi_i$ and $f_{\infty}=f$, then,  
$\{f_i\}$ has infinitesimal constant convergence property to $f_{\infty}$ at every $w \in B_R(x_{\infty})$.
\end{example}
\begin{example}\label{ex2}
If $f_i$ is Lipschitz function with $\sup_i \mathbf{Lip}f_i < \infty$, and $f_i \rightarrow f_{\infty}$, then for every $w \in B_R(x_{\infty})$, 
$\{f_i\}_i$ has infinitesimal constant convergence property to $f_{\infty}$ at $w$.
\end{example}
\begin{example}\label{ex3}
Let $w_i \rightarrow w \in B_R(x_{\infty})$, $r>0$ satisfying $B_r(w) \subset B_R(x_{\infty})$.
Then,  $\{1_{B_R(x_i) \setminus \overline{B}_r(w_i)}\}_i$ has 
infinitesimal constant convergence property to $1_{B_R(x_{\infty}) \setminus \overline{B}_r(w_{\infty})}$ at every $\alpha \in B_R(x_{\infty}) \setminus \partial B_r(w)$.
\end{example}
We shall give a fundamental result for infinitesimal constant convergence property:
\begin{proposition}\label{10101}
Let $k$ be a positive integer, $R$ a positive number, $f_i^l$ Borel functions on $B_R(x_i) (1 \le l \le k, 1 \le i \le \infty)$ satisfying
$\sup_{i, l}(|f_i^l|_{L^{\infty}(B_R(x_i))} + |f_{\infty}^l|_{L^{\infty}(B_R(x_{\infty}))}) < \infty$, $w$ a point in $B_R(x_{\infty})$ and $\{F_i\}_{1 \le i \le \infty}$ a sequence of continuous functions on $\mathbf{R}^k$.
We assume that $\{f_i^l\}_{1 \le i \le \infty}$ has infinitesimal constant convergence property to $f_{\infty}^l$ at $w$ for every $l$ and that 
$F_i$ converges to $F_{\infty}$ in the sense of compact uniformly topology.
Then, the sequence $\{F_i(f_i^1, _{\cdots}, f_i^k)\}$ has infinitesimal constant convergence property 
to $F_{\infty}(f_{\infty}^1, _{\cdots}, f_{\infty}^k)$ at $w$.
\end{proposition}
\begin{proof}
We fix $\epsilon > 0$.
We take $\hat{R}, L \ge 1$ satisfying that $\bigcup_i \mathrm{Image}(f_i^1, _{\cdots}, f_i^l) \subset B_{\hat{R}}(0_k)$, 
$\sup_{i, l}(|f_i^l|_{L^{\infty}(B_R(x_i))} + |f_{\infty}^l|_{L^{\infty}(B_R(x_{\infty}))}) \le \hat{R}$ and $\sup_i|F_i|_{L^{\infty}(B_{\hat{R}}(0_k))} \le L$.
There exists a nonnegative valued function $b$ on $\mathbf{R}_{>0}$ such that $b(t) \rightarrow 0$ as $t \rightarrow 0$ and that 
for every $t > 0$, there exists $i_t$ such that  
$F_{\infty}(\alpha)=F_i(\beta) \pm b(t)$ for every $\alpha \in B_{\hat{R}}(0_k)$, $i \ge i_t$ and
 $\beta \in B_{t}(\alpha)$.
On the other hand, there exists $\tau _1 > 0$  satisfying the following properties: For every $0< s < \tau_1$, there exists $j_s$ such that

\[\frac{1}{\upsilon_i(B_s(w_i))}\int_{B_s(w_i)}\left|f_i^l-\frac{1}{\upsilon_{\infty}(B_s(w))}\int_{B_s(w)}f_{\infty}^ld\upsilon_{\infty} \right| \upsilon_i \le \epsilon \]
and
\[\frac{1}{\upsilon_{\infty}(B_s(w))}\int_{B_s(w)}\left| f_{\infty}^l-\frac{1}{\upsilon_{i}(B_s(w_i))}\int_{B_s(w_i)}f_{i}^ld\upsilon_{i} \right| \upsilon_{\infty} \le \epsilon \]
for every $1 \le l \le k$, $i \ge j_s$ and $w_i \rightarrow w$.
Especially, we have 
\[\frac{1}{\upsilon_{\infty}(B_s(w))}\int_{B_s(w)}f_{\infty}^ld\upsilon_{\infty}= \frac{1}{\upsilon_{i}(B_s(w_i))}\int_{B_s(w_i)}f_{i}^ld\upsilon_{i} \pm \epsilon. \]
We fix $0 < s < \tau_1$.
Therefore, there exist a sequence of compact sets $K_i \subset B_s(w_i)$ and a compact set $K_{\infty} \subset B_s(w)$ such that 
$\upsilon_i(K_i)/\upsilon_i(B_s(w_i)) \ge 1- \Psi(\epsilon; K(1))$, $\upsilon_{\infty}(K_{\infty})/\upsilon_{\infty}(B_s(w)) \ge 1- \Psi(\epsilon; K(1))$ and that  
\[\left| f_i^l(\alpha)- \frac{1}{\upsilon_{\infty}(B_s(w))}\int_{B_s(w)}f_{\infty}^ld\upsilon_{\infty}\right| < \Psi(\epsilon; K(1))\]
and
\[\left| f_{\infty}^l(\beta)- \frac{1}{\upsilon_{i}(B_s(w_i))}\int_{B_s(w_i)}f_{i}^ld\upsilon_{i}\right| < \Psi(\epsilon; K(1))\]
for every $j_s \le i < \infty$, $1 \le l \le k$, $\alpha \in K_i$ and 
$\beta \in K_{\infty}$.
Without loss of generality, we can assume that there exists a compact set $\hat{K} \subset \overline{B}_s(w)$
such that $K_i \rightarrow \hat{K}$.
We put $\hat{K}_{\infty} = \hat{K} \cap K_{\infty}$.
By Proposition \ref{sup}, we have $\upsilon_{\infty}(\hat{K}_{\infty})/\upsilon_{\infty}(B_s(w)) \ge 1- \Psi(\epsilon; K(1))$.
We put 
\[a_i^l = \frac{1}{\upsilon_i(B_s(w_i))}\int_{B_s(w_i)}f_i^ld\upsilon_i. \]
Then, there exists $k_s \ge j_s$ such that 
\begin{align}
F_{\infty}(f_{\infty}^{1}(\alpha), _{\cdots}, f_{\infty}^k(\alpha)) &= F_{\infty}(a_{\infty}^1, _{\cdots}, a_{\infty}^k) \pm b(\Psi(\epsilon; K(1))) \\
&= F_i(a_i^1, _{\cdots}, a_i^k) \pm 2b(\Psi(\epsilon; K(1)))\\
&= F_i(f_i^{1}(\alpha_i), _{\cdots}, f_i^{k}(\alpha_i)) \pm 3b(\Psi(\epsilon; K(1)))
\end{align}
for every $i \ge k_s$, $\alpha \in \hat{K}_{\infty}$ and $\alpha_i \in K_i$ with $\alpha_i \rightarrow \alpha$.
Thus, we have 
\begin{align}
&\frac{1}{\upsilon_{\infty}(B_s(w))}\int_{B_s(w)}\left| F_{\infty}(f_{\infty}^1, _{\cdots}, f_{\infty}^k)- F_{\infty}(a_{\infty}^1, _{\cdots}, a_{\infty}^k)\right| d\upsilon_{\infty} \\
&= \frac{1}{\upsilon_{\infty}(B_s(w))}\int_{\hat{K}_{\infty}}\left| F_{\infty}(f_{\infty}^1, _{\cdots}, f_{\infty}^k)- F_{\infty}(a_{\infty}^1, _{\cdots}, a_{\infty}^k)\right| d\upsilon_{\infty} \pm \Psi(\epsilon; K(1), L) \\
& < 3b(\Psi(\epsilon; K(1)))  + \Psi(\epsilon; K(1), L)
\end{align}
and 
\begin{align}
&\frac{1}{\upsilon_{i}(B_s(w_i))}\int_{B_s(w_i)}\left| F_{_i}(f_{_i}^1, _{\cdots}, f_{_i}^k)- F_{i}(a_{i}^1, _{\cdots}, a_{i}^k)\right| d\upsilon_{i} \\
&= \frac{1}{\upsilon_{i}(B_s(w_i))}\int_{K_{i}}\left| F_{i}(f_{i}^1, _{\cdots}, f_{i}^k)- F_{i}(a_{i}^1, _{\cdots}, a_{i}^k)\right| d\upsilon_{i} \pm \Psi(\epsilon; K(1), L) \\
& < 3b(\Psi(\epsilon; K(1))) + \Psi(\epsilon; K(1), L)
\end{align}
for $i \ge k_s$.
Moreover, we have 
\begin{align}
&\frac{1}{\upsilon_{\infty}(B_s(w))}\int_{B_s(w)}F_{\infty}(f_{\infty}^1, _{\cdots}, f_{\infty}^k)d\upsilon_{\infty} \\
&=\frac{1}{\upsilon_{\infty}(B_s(w))}\int_{\hat{K}_{\infty}}F_{\infty}(f_{\infty}^1, _{\cdots}, f_{\infty}^k)d\upsilon_{\infty} \pm \Psi(\epsilon; K(1), L)\\
&=(1 \pm \Psi (\epsilon; K(1)))(F_{\infty}(a_{\infty}^1, _{\cdots}, a_{\infty}^k) \pm b(\Psi(\epsilon; K(1))) \pm \Psi(\epsilon; K(1))\\
&=(1 \pm \Psi (\epsilon; K(1)))(F_{i}(a_{i}^1, _{\cdots}, a_{i}^k) \pm b(\Psi(\epsilon; K(1))) \pm \Psi(\epsilon; K(1), L)\\
&=(1 \pm \Psi (\epsilon; K(1)))\left(\frac{1}{\upsilon_{i}(B_s(w_i))}\int_{K_{i}}F_{i}(f_{i}^1, _{\cdots}, f_{i}^k)d\upsilon_{i} \pm 3b(\Psi(\epsilon; K(1)))\right) \pm \Psi(\epsilon; K(1), L)\\
&=(1 \pm \Psi (\epsilon; K(1)))\left(\frac{1}{\upsilon_{i}(B_s(w_i))}\int_{B_s(w_i)}F_{i}(f_{i}^1, _{\cdots}, f_{i}^k)d\upsilon_{i} \pm 3b(\Psi(\epsilon; K(1))) \right) \pm \Psi(\epsilon; K(1), L)
\end{align}
for $i \ge k_s$.
Therefore, we have the assertion.
\end{proof}
\begin{remark}
By the proof of Proposition \ref{10101}, we also have the following:
Let $k$ be a positive integer, $f_i^l$ Borel functions on $B_R(x_i) (1 \le l \le k, 1 \le i \le \infty)$ satisfying
$\sup_{i, l}(|f_i^l|_{L^{\infty}(B_R(x_i))} + |f_{\infty}^l|_{L^{\infty}(B_R(x_{\infty}))}) < \infty$, $w$ a point in $B_R(x_{\infty})$ and $\{F_i\}_{1 \le i \le \infty}$ a sequence of locally $L^{\infty}$ functions on $\mathbf{R}^k$.
Assume the following:
\begin{enumerate}
\item $\{f_i^l\}_{1 \le i \le \infty}$ has infinitesimal constant convergence property to $f_{\infty}^l$ at $w$ for every $l$.
\item The limits 
\[a^l=\lim_{r \rightarrow 0}\frac{1}{\upsilon_{\infty}(B_r(w))}\int_{B_r(w)}f_{\infty}^ld\upsilon_{\infty}\]
exist for every $l$.
\item There exists an open neighborhood $U$ at $(a^1, \cdots, a^k) \in \mathbf{R}^k$ such that $F_i$ is continuous on $U$ for every $1 \le i \le \infty$ and that 
$F_i$ converges to $F_{\infty}$ on $U$ uniformly.
\end{enumerate}
Then, the sequence $\{F_i(f_i^1, _{\cdots}, f_i^k)\}$ has infinitesimal constant convergence property 
to $F_{\infty}(f_{\infty}^1, _{\cdots}, f_{\infty}^k)$ at $w$.
\end{remark}
For Ricci limit spaces, we shall give a sufficient condition to satisfy infinitesimal constant convergence property for radial derivative of Lipschitz functions:
\begin{proposition}\label{Hesss}
Let $\{(M_i, m_i, \underline{\mathrm{vol}})\}_i$ be a sequence of pointed connected $n$-dimensional complete Riemannian manifolds with  $\mathrm{Ric}_{M_i} \ge -(n-1)$, $(Y, y, \upsilon)$ be a pointed proper geodesic space with Radon measure $\upsilon$, $R$ a positive number, $x_{\infty}$ a point in $Y$, $x_i$ a point in $M_i$, $f_i$ a $C^2$-function on $B_R(x_i)$ and $f_{\infty}$ a Lipschitz function on
$B_R(x)$. 
We assume that $\sup_i \mathbf{Lip}f_i < \infty$, $(M_i, m_i, x_i, f_i, \underline{\mathrm{vol}}) \stackrel{(\phi_i, R_i, \epsilon_i)}{\rightarrow} (Y, y, x_{\infty}, f_{\infty}, \upsilon)$ and that 
\[\sup_i \int_{B_R(x_i)}|\mathrm{Hess}_{f_i}|^2d\underline{\mathrm{vol}} < \infty.\]
Then, there exists a Borel subset $A \subset B_R(x_{\infty})$ such that $\upsilon(B_R(x_{\infty}) \setminus A)=0$ and that for every $z \in A$ and $w_i \rightarrow w \in Y$,  the sequence $\{\langle dr_{w_i}, df_i\rangle \}$ has infinitesimal constant convergence property to $\langle dr_{w}, df_{\infty}\rangle$ at $z$. 
\end{proposition}
\begin{proof}
We fix $\epsilon >0$ and take $L \ge 1$ satisfying 
\[\sup_i \left(\frac{1}{\underline{\mathrm{vol}}\,B_R(x_i)}\int_{B_R(x_i)}|\mathrm{Hess}_{f_i}|^2d\underline{\mathrm{vol}} + \mathbf{Lip}f_i \right)\le L.\]
By Theorem \ref{14}, there exist $0 < \eta << \epsilon$ and a Borel subset $X(\epsilon) \subset B_R(x_{\infty}) \cap \mathcal{D}_z^{\eta}
\setminus B_{\eta}(z)$ such that 
\[\frac{\upsilon (B_R(x_{\infty})\setminus X(\epsilon))}{\upsilon (B_R(x_{\infty}))} \le \epsilon\]
and that 
\[\left| \frac{f_{\infty}\circ \gamma(\overline{z, \alpha}+h)-f_{\infty}(\alpha)}{h}-\langle dr_z, df_{\infty}\rangle (\alpha)\right| \le \epsilon\]  
for every $\alpha \in X(\epsilon)$, $h$ satisfying $0 < |h| < \eta$ and isometric embedding $\gamma$ from $[0, \overline{z, \alpha}+\eta]$ to $Y$ with $\gamma (0)=z$ and $\gamma (\overline{z, \alpha})=\alpha$.
By Corollary \ref{131}, there exists Borel set $\hat{X}(\epsilon) \subset X(\epsilon)$ such that $\upsilon (X(\epsilon) \setminus \hat{X}(\epsilon))=0$ and that  
\[\lim_{t \rightarrow 0}\frac{1}{\upsilon(B_t(\alpha))}\int_{B_t(\alpha)}|\langle dr_z, df_{\infty}\rangle -\langle dr_z, df_{\infty}\rangle (\alpha)|d\upsilon = 0\]
for every $\alpha \in \hat{X}(\epsilon)$.
For every $\alpha \in \hat{X}(\epsilon)$, there exists $r(\alpha) > 0$ such that  
\[\frac{1}{\upsilon(B_t(\alpha))}\int_{B_t(\alpha)}|\langle dr_z, df_{\infty}\rangle -\langle dr_z, df_{\infty}\rangle (\alpha)|d\upsilon < \epsilon\]
for every $0 < t < r(x)$.
We put $l= \eta^{-1/4}$.
By an argument similar to the proof of Proposition \ref{1}, for every $i$, there exists a compact subset $K_i \subset B_{R-\epsilon}(x_i)$ such that 
\[\frac{\underline{\mathrm{vol}}(B_{R-\epsilon}(x_i) \setminus K_i)}{\underline{\mathrm{vol}}\,B_{R-\epsilon}(x_i)} \le \Psi(l^{-1};n, R,L)\]
and that 
\[\frac{1}{\underline{\mathrm{vol}}\,B_t(w)}\int_{B_t(w)}|\mathrm{Hess}_{f_i}|^2d\underline{\mathrm{vol}}\le l\]
for every $w \in K_i$ and $0 < t < \epsilon/100$.
Without loss of generality, we can assume that there exists a compact set $K_{\infty} \subset \overline{B}_R(x_{\infty})$ such that 
$K_i \rightarrow K_{\infty}$.
We put $W(\epsilon) = K_{\infty} \cap X(\epsilon)$.
By Proposition \ref{sup}, we have 
\[\frac{\upsilon(W(\epsilon))}{\upsilon(B_R(x_{\infty}))}\ge 1- \Psi(\epsilon; n, R, L).\]
We fix $\alpha \in W(\epsilon)$, $0 < t << \min \{\eta, r(\alpha)\}$ and an isometric embedding $\gamma$ from $[0, \overline{z, \alpha}+ \eta]$ to $Y$ satisfying $\gamma(0)=z$ and $\gamma (\overline{z, \alpha})=\alpha$.
We take   
$\alpha_i \in K_i$ satisfying $\alpha_i \rightarrow \alpha$.
We define a Borel function $F_i$ on $B_t(\alpha_i) \setminus (C_{z_i} \cup \{z_i\} )$ by 
\[F_i(\beta)=\frac{f_i \circ \gamma_{\beta}(\overline{z_i, \beta}-\eta^2)-f_i(\beta)}{-\eta^2}.\]
Here $\gamma_{\beta}$ is the minimal geodesic from $z_i$ to $\beta$.
By an argument similar to the proof of Claim \ref{115},  we have
\begin{align}
&\frac{1}{\underline{\mathrm{vol}}\,B_{t}(\alpha_i)}\int_{B_{t}(\alpha_i)}|\langle df_i, dr_{z_i}\rangle -F_i|d\underline{\mathrm{vol}} \\
& \le \eta^2 \frac{C(n)}{\underline{\mathrm{vol}}\,B_{10t}(\alpha_i)}\int_{B_{10t}(\alpha_i)}|\mathrm{Hess}_{f_i}|^2d\underline{\mathrm{vol}} \le \eta^2 C(n)l  \le \Psi(\epsilon;n)
\end{align}
for every $i$.
We take $i_0$ satisfying that $\epsilon_i << t$ for every $i \ge i_0$.
For every $i \ge i_0$ and $\beta_i \in B_t(\alpha_i)$, 
we remark that $\overline{\phi_i(\beta_i), \alpha} \le t + \epsilon_i \le \eta^3$.
Then, since
\[\overline{z, \phi_i(\gamma_{\beta_i}(\overline{z_i, \beta_i}-\eta^2))}^{\eta^{-2}d_Y}+
\overline{\phi_i(\gamma_{\beta_i}(\overline{z_i, \beta_i}-\eta^2)), \phi_i(\beta_i)}^{\eta^{-2}d_Y}-\overline{z, \phi_i(\beta_i)}^{\eta^{-2}d_Y} < 3\epsilon_i,\]
we have
\[\overline{z, \phi_i(\gamma_{\beta_i}(\overline{z_i, \beta_i}-\eta^2))}^{\eta^{-2}d_Y}+
\overline{\phi_i(\gamma_{\beta_i}(\overline{z_i, \beta_i}-\eta^2)), \alpha}^{\eta^{-2}d_Y}-\overline{z, \alpha}^{\eta^{-2}d_Y} < 5\eta.\]
Similarly, we have
\[\overline{z, \phi_i(\gamma_{\beta_i}(\overline{z_i, \beta_i}-\eta^2))}^{\eta^{-2}d_Y}+
\overline{\phi_i(\gamma_{\beta_i}(\overline{z_i, \beta_i}-\eta^2)), \gamma(\overline{z, \alpha}+\eta)}^{\eta^{-2}d_Y}-\overline{z, \gamma(\overline{z, \alpha}+\eta)}^{\eta^{-2}d_Y} < 5\eta,\]
\[\overline{\phi_i(\gamma_{\beta_i}(\overline{z_i, \beta_i}-\eta^2)), \gamma(\overline{z, \alpha}+\eta)}^{\eta^{-2}d_Y}\ge \eta^{-1}-\eta, \]
\[\overline{\phi_i(\gamma_{\beta_i}(\overline{z_i, \beta_i}-\eta^2)), z}^{\eta^{-2}d_Y}\ge \eta^{-1}-\eta\]
and
\[\overline{\phi_i(\gamma_{\beta_i}(\overline{z_i, \beta_i}-\eta^2)), \alpha}^{\eta^{-2}d_Y}= 1 \pm 5\eta.\]
Therefore, by splitting theorem, we have 
\[\overline{\phi_i(\gamma_{\beta_i}(\overline{z_i, \beta_i}-\eta^2)), \gamma(\overline{z, \alpha}-\eta^2)}^{\eta^{-2}d_Y}\le \Psi(\eta;n).\]
Thus we have 
\begin{align}
\frac{f_i(\gamma_{\beta_i}(\overline{z_i, \beta_i}-\eta^2))-f_i(\beta_i)}{-\eta^2}&=\frac{f_{\infty}(\phi_i(\gamma_{\beta_i}(\overline{z_i, \beta_i}-\eta^2)))-f_{\infty}(\phi_i(\beta_i))}{-\eta^2} \pm \frac{\epsilon_i}{\eta^2} \\
&=\frac{f_{\infty}(\gamma(\overline{z, \alpha}-\eta^2)))-f_{\infty}(\alpha)}{-\eta^2} \pm \Psi(\eta ;n, L) \\
&=\langle dr_z, df_{\infty}\rangle (\alpha) \pm \Psi(\eta ;n, L).
\end{align}
Especially,  we have 
\[\frac{1}{\underline{\mathrm{vol}}\,B_{t}(\alpha_i)}\int_{B_{t}(\alpha_i)}|F_i-\langle dr_z, df_{\infty}\rangle (\alpha)|d\underline{\mathrm{vol}} \le \Psi(\eta;n, L)\]
for $i \ge i_0$.
Therefore if we put $W= \bigcap _{N_1 \in \mathbf{N}}(\bigcup_{N_2 \ge N_1}W(N_2^{-1}))$, then $\upsilon (B_R(x_{\infty}) \setminus W)=0$,  $\{\langle dr_{z_i}, df_i\rangle \}$ has infinitesimal constant convergence property to $\langle dr_{w}, df_{\infty}\rangle$ at every $w \in W$. 
\end{proof}

\begin{remark}\label{cc}
We shall introduce the following important method to get some uniformly Hessian estimates by using cut-off functions with good properties by Cheeger-Colding:
Let $(M, m, \underline{\mathrm{vol}})$ be a pointed connected $n$-dimensional complete Riemannian manifold with renormalized measure satisfying $\mathrm{Ric}_{M} \ge -(n-1)$, $R$ a positive number and $f$ a $C^2$-function on $B_R(m)$.
We assume that there exists $L \ge 1$ such that 
\[|\nabla f|_{L^{\infty}(B_R(m))} + \frac{1}{\underline{\mathrm{vol}}\,B_R(m)}\int_{B_R(m)}|\Delta f|^2d\underline{\mathrm{vol}} \le L\]
Then, we have 
\[\frac{1}{\underline{\mathrm{vol}}\,B_r(m)}\int_{B_r(m)}|\mathrm{Hess}_{f}|^2d\underline{\mathrm{vol}} < C(n, r, R, L)\]
for every $0 < r < R$.
The proof is as follows.
By standard smoothing argument, without loss of generality, we can assume that $f$ is a smooth function.
There exists a smooth function $\phi$ on $M$ such that $0 \le \phi \le 1$, 
$\phi|_{B_r(m)}=1$, $ \mathrm{supp} \phi \subset B_R(m)$,  $|\nabla \phi| \le C(n, r, R)$ and $|\Delta \phi|\le C(n, r, R)$ 
(see for instance \cite[Theorem $8. 16$]{ch1}). 
By Bochner's formula, we have 
\[-\frac{1}{2} \Delta |\nabla (\phi f)|^2 \ge  |\mathrm{Hess}_{\phi f}|^2-\langle \nabla \Delta (\phi f), \nabla (\phi f)\rangle -(n-1)|\nabla (\phi f)|^2. \]
Thus, we have 
\begin{align}
&\frac{1}{\underline{\mathrm{vol}}\,B_r(m)}\int_{B_r(m)}|\mathrm{Hess}_{f}|^2d\underline{\mathrm{vol}}\\
&\le \frac{C(n, r, R)}{\underline{\mathrm{vol}}\,B_R(m)}\int_{B_R(m)}|\mathrm{Hess}_{\phi f}|^2d\underline{\mathrm{vol}} \\
&\le \frac{C(n, r, R)}{\underline{\mathrm{vol}}\,B_R(m)}\int_{B_R(m)}\left(\Delta (\phi f)\right)^2 d\underline{\mathrm{vol}}+C(n, R, L) \\
&\le  \frac{2C(n, r, R)}{\underline{\mathrm{vol}}\,B_R(m)}\int_{B_R(m)}(f\Delta \phi)^2
+(\phi \Delta f)^2
+|\langle \nabla f, \nabla \phi \rangle |^2 d\underline{\mathrm{vol}} + C(n, R, L) \\
&\le C(n, r, R, L).
\end{align}
This observation performs a crucial role to study limit functions of harmonic functions. 
\end{remark}
The following proposition follows from Lemma \ref{17} directly.
\begin{proposition}\label{10102}
Let $\{(M_i, m_i, \underline{\mathrm{vol}})\}$ be a sequence of pointed connected $n$-dimensional complete Riemannian manifolds with renormalized measure satisfying $\mathrm{Ric}_{M_i} \ge -(n-1)$, $(Y, y, \upsilon)$ be a Ricci limit space of $\{(M_i, m_i, \underline{\mathrm{vol}})\}_i$.
Then for every $w^1, w^2 \in Y$, $z \in Y \setminus (C_{w^1} \cup C_{w^2} \cup \{w^1, w^2\})$ and $w_i^j \rightarrow w^j \in Y(j =1, 2)$, the sequence $\{\langle dr_{w^1_i}, dr_{w^2_i}\rangle \}$ has infinitesimal constant convergence property to $\langle dr_{w^1_{\infty}}, dr_{w^2_{\infty}}\rangle$ at $z$.
\end{proposition}
\subsection{Infinitesimal convergence property}
In this subsection, we will give a notion of \textit{infinitesimal convergence property} and its fundamental properties.
\begin{definition}[Infinitesimal convergence property]
Let $R$ be a positive number, $w$ a point in $B_R(x_{\infty})$ and $f_i$ a Borel function on $B_R(x_i) (1 \le i \le \infty)$ satisfying
$\sup _i |f_i|_{L^{\infty}(B_R(x_i))} + |f_{\infty}|_{L^{\infty}(B_R(x_{\infty}))} < \infty$.
We say that $\{f_i\}_i$ has \textit{infinitesimal convergence property to $f_{\infty}$ at} $w$ if  
for every $\epsilon > 0$, there exists $r > 0$ such that   
\[\limsup _{i \rightarrow \infty}\left| \frac{1}{\upsilon_i(B_t(w_i))}\int_{B_t(w_i)}f_id\upsilon_i - \frac{1}{\upsilon_{\infty}(B_t(w))}\int_{B_t(w)}f_{\infty}d\upsilon_{\infty} \right|\le \epsilon\]
for every $0 < t < r$ and $w_i \rightarrow w$.
\end{definition}
It is clear that if the sequence $\{f_i\}_i$ has infinitesimal constant convergence property to $f_{\infty}$ at $w$, then $\{f_i\}_i$ has infinitesimal convergence property to $f_{\infty}$ at $w$.
We skip the proof of the next proposition because it is not difficult.
\begin{proposition}[Linearlity of infinitesimal convergence property]\label{101020}
Let $R$ be a positive number, $a_i$, $b_i$, $c_i$, $d_i$ Borel functions on $B_R(x_i) (1 \le i \le \infty)$, $w$ a point in $B_R(x_{\infty})$.
We assume that $\sup_i (|a_i| + |b_i| + |c_i| + |d_i|)_{L^{\infty}(B_R(x_i))} < \infty$ and that 
$\{a_i\}_i$, $\{b_i\}_i$ have infinitesimal constant convergence property to $a_{\infty}, b_{\infty}$ at $w$, respectively and 
$\{c_i\}_i$, $\{d_i\}_i$ have infinitesimal convergence property to $c_{\infty}, d_{\infty}$ at $w$, respectively.
Then $\{a_ic_i + b_i d_i\}$ has infinitesimal convergence property to $a_{\infty}c_{\infty}+b_{\infty}d_{\infty}$ at $w$.
\end{proposition}
The next proposition follows from an argument similar to the proof of Proposition \ref{sup}:
\begin{proposition}\label{sup1}
Let $R$ be a positive number, $K_i$  a Borel subset of $\overline{B}_R(x_i)$ and $f_i$ a nonnegative valued Borel function on $\overline{B}_R(x_i) (1 \le i \le \infty)$ satisfying
$\sup _i |f_i|_{L^{\infty}(B_R(x_i))} + |f_{\infty}|_{L^{\infty}(B_R(x_{\infty}))} < \infty$.
We assume that $K_{\infty}$ is compact, $\limsup_{i \rightarrow \infty}K_i \subset K_{\infty}$ and that
for a.e. $w \in K_{\infty}$,
$\{f_i\}$ has infinitesimal convergence property to $f_{\infty}$ at $w$.
Then we have 
\[\limsup_{i \rightarrow \infty}\int_{K_i}f_i d\upsilon_i \le \int_{K_{\infty}}f_{\infty}d\upsilon_{\infty}.\]
\end{proposition}
We shall state a fundamental result for  infinitesimal convergence property:
\begin{proposition}\label{10103}
Let $R$ be a positive number, $K_i$  a Borel subset of $\overline{B}_R(x_i)$ and $\{f_i\}_i$ a Borel function on $\overline{B}_R(x_i) (1 \le i \le \infty)$ satisfying
$\sup _i |f_i|_{L^{\infty}(B_R(x_i))} + |f_{\infty}|_{L^{\infty}(B_R(x_{\infty}))} < \infty$.
We assume that $K_{\infty}$ is compact, $\limsup_{i \rightarrow \infty}K_i \subset K_{\infty}$ and that for a.e. $w \in K_{\infty}$, $\{1_{K_i}\}_i$ and
$\{f_i\}_i$ have infinitesimal convergence property to $1_{K_{\infty}}$, $f_{\infty}$ at $w$, respectively.
Then, we have
\[\lim_{i \rightarrow \infty}\int_{K_i}f_i d\upsilon_i=\int_{K_{\infty}}f_{\infty}d\upsilon_{\infty}.\]
\end{proposition}
\begin{proof}
We fix $\epsilon >0$.
We take $L \ge 1$ satisfying $\sup_i |f_i|_{L^{\infty}} + |f_{\infty}| + \upsilon_{\infty}(B_R(x_{\infty}))< L$.
There exists a Borel subset $\hat{K}_{\infty} \subset K_{\infty}$ satisfying the following properties: 
For every $w \in \hat{K}_{\infty}$, there exists $t_w >0$ such that $\overline{B}_{10t_w}(w) \subset B_R(x)$ and that 

\[\limsup_{i \rightarrow \infty}\left| \frac{1}{\upsilon_i(B_s(w_i))}\int_{B_s(w_i)}f_id\upsilon_i - \frac{1}{\upsilon_{\infty}(B_s(w))}\int_{B_s(w)}f_{\infty}d\upsilon_{\infty}\right| <\epsilon, \]
\[\frac{\upsilon_{\infty}(B_s(w) \cap K_{\infty})}{\upsilon_{\infty}(B_s(w))} \ge 1- \epsilon \]
and 
\[\limsup_{i \rightarrow \infty}\left| \frac{1}{\upsilon_i(B_s(w_i))}\int_{B_s(w_i)}1_{K_i}d\upsilon_i - \frac{1}{\upsilon_{\infty}(B_s(w))}\int_{B_s(w)}1_{K_{\infty}}d\upsilon_{\infty} \right| < \epsilon\]
for every $0 < s < t_w$ and $w_i \rightarrow w$.
By Lemma \ref{cov}, there exists a pairwise disjoint collection $\{\overline{B}_{r_i}(x_i)\}_i$ such that $x_i \in K_{\infty}$, $r_i << t_{x_i}$,
and that
$K_{\infty} \setminus \bigcup _{i=1}^N\overline{B}_{r_i}(x_i)\subset \bigcup_{i=N+1}^{\infty}\overline{B}_{5r_i}(x_i)$
for every $N$.
We take $N$ satisfying 
$\sum_{i=N+1}^{\infty}\upsilon_{\infty}(B_{r_i}(x_i)) < \epsilon.$
Then, we have 
$\sum_{i=N+1}^{\infty}\upsilon_{\infty}(B_{5r_i}(x_i)) < 2^{5K(1)}\epsilon.$
We take $x_i(j) \in Z_j$ satisfying $x_i(j) \rightarrow x_i$.
Then we have 
\begin{align}
\int_{K_{\infty}}f_{\infty}d\upsilon_{\infty}&=\sum_{i=1}^N\int_{B_{r_i}(x_i)\cap K_{\infty}}f_{\infty}d\upsilon_{\infty} \pm \int_{\bigcup_{i=N+1}^{\infty}\overline{B}_{5r_i}(x_i)}|f_{\infty}|d\upsilon_{\infty} \\
&=\sum_{i=1}^N\int_{B_{r_i}(x_i)}f_{\infty}d\upsilon_{\infty} \pm \Psi(\epsilon; K(1), L) \\
&=\sum_{i=1}^N\int_{B_{r_i}(x_i(j))}f_{j}d\upsilon_{j} \pm \Psi(\epsilon; K(1), L) \\
&=\sum_{i=1}^N\int_{B_{r_i}(x_i(j))\cap K_j}f_{j}d\upsilon_{j} \pm \Psi(\epsilon; K(1), L) \\
&=\int_{K_j} f_j d\upsilon_j \pm \left(\int_{K_j \setminus \bigcup_{i=1}^N\overline{B}_{r_i}(x_i(j))}|f_j|d\upsilon_j + \Psi(\epsilon; K(1),  L)\right).
\end{align}
for every sufficiently large $j$.
On the other hand, by Proposition \ref{compact2}, Proposition \ref{compact3} and Proposition \ref{sup}, we have 
\begin{align}
\limsup_{j \rightarrow \infty}\int_{K_j \setminus \bigcup_{i=1}^N\overline{B}_{r_i}(x_i(j))}|f_j|d\upsilon_j &\le
L \limsup_{j \rightarrow \infty}\upsilon_j(K_j \setminus \bigcup_{i=1}^NB_{r_i}(x_i(j))) \\
&\le L\upsilon_{\infty}(K_{\infty} \setminus \bigcup_{i=1}^NB_{r_i}(x_i)) \\
&\le \Psi(\epsilon; K(1), L). 
\end{align}
Therefore, we have the assertion.
\end{proof}
\begin{remark}Proposition \ref{lap} also follows from Example \ref{ex1}, \ref{ex3} and Proposition \ref{10103} directly.
\end{remark}
Next corollary follows from Proposition \ref{10103} directly.
\begin{corollary}\label{ann}
Let $R,  r_i$ be  positive numbers, $N$ a positive integer, $\{z_j\}_{1 \le j \le N}$ points in $Y$ and $f_i$ a Borel function on $B_R(x_j) (1 \le i \le \infty)$ satisfying
$\sup _i |f_i|_{L^{\infty}(B_R(x_i))} + |f_{\infty}|_{L^{\infty}(B_R(x_{\infty}))} < \infty$.
We assume that for a.e. $w \in B_R(x_{\infty}) \setminus \bigcup_{i=1}^{N}B_{r_i}(z_i)$, 
$\{f_i\}_i$ have infinitesimal convergence property to $f_{\infty}$ at $w$.
Then, we have
\[\lim_{j \rightarrow \infty}\int_{B_R(x_{j}) \setminus \bigcup_{i=1}^{N}B_{r_i}(z_i(j))}f_j d\upsilon_j=\int_{B_R(x_{\infty}) \setminus \bigcup_{i=1}^{N}B_{r_i}(z_i)}f_{\infty}d\upsilon_{\infty}\]
for every $z_i(j) \rightarrow z_i$.
\end{corollary}
We end this subsection by giving the following proposition:
\begin{proposition}\label{lpoij}
Let $A_i$ be a Borel subset of $B_R(x_i)$ and $w \in \mathrm{Leb}A_{\infty}$.
We assume that $\{1_{A_i}\}_i$ has infinitesimal convergence property to $1_{A_{\infty}}$ at $w$.
Then  $\{1_{A_i}\}$ has infinitesimal constant convergence property to $1_{A_{\infty}}$ at $w$.
\end{proposition}
\begin{proof}
We fix $\epsilon >0$ and take a sequence $w_i \rightarrow w$.
There exists $r > 0$ such that 
\[\frac{\upsilon_{\infty}(B_t(w) \cap A_{\infty})}{\upsilon_{\infty}(B_t(w))} \ge 1-\epsilon \]
and
\[\limsup_{i \rightarrow \infty}\left| \frac{1}{\upsilon_i(B_t(w_i))}\int_{B_t(w_i)}1_{A_i}d\upsilon_i
-\frac{1}{\upsilon_{\infty}(B_t(w_{\infty}))}
\int_{B_t(w_{\infty})}1_{A_{\infty}}d\upsilon_{\infty} \right| < \epsilon\]
for every $0 < t < r$.
We fix $ 0 < t < r$. 
Then we have 
\begin{align}
&\frac{1}{\upsilon_i(B_t(w_i))}\int_{B_t(w_i)}\left| 1_{A_i}
-\frac{1}{\upsilon_{\infty}(B_t(w_{\infty}))}
\int_{B_t(w_{\infty})}1_{A_{\infty}}d\upsilon_{\infty}\right| d\upsilon_i \\
&\le \frac{1}{\upsilon_i(B_t(w_i))}\int_{B_t(w_i)}\left| 1_{A_i}
-\frac{1}{\upsilon_{\infty}(B_t(w_{\infty}))}
\int_{B_t(w_{i})}1_{A_{i}}d\upsilon_{\infty}\right| d\upsilon_i + \epsilon \\
&=\frac{1}{\upsilon_i(B_t(w_i))}\int_{B_t(w_i)}\left| 1_{A_i}
-\frac{\upsilon_i(A_i)}{\upsilon_i(B_t(w_i))}\right| d\upsilon_i + \epsilon \\
&=\frac{1}{\upsilon_i(B_t(w_i))}\int_{A_i}\frac{\upsilon_i(B_t(w_i) \setminus A_i)}{\upsilon_i(B_t(w_i))}d\upsilon_i + 
\frac{1}{\upsilon_i(B_t(w_i))}\int_{B_t(w_i) \setminus A_i}\frac{\upsilon_i(A_i)}{\upsilon_i(B_t(w_i))}d\upsilon_i + \epsilon \\
&\le 2\frac{\upsilon_i(B_t(w_i) \setminus A_i)}{\upsilon_i(B_t(w_i))}+\epsilon < 3\epsilon+2\epsilon < 5\epsilon.
\end{align}
for every sufficiently large $i$.
Similarly, we have 
\begin{align}
&\frac{1}{\upsilon_{\infty}(B_t(w_{\infty}))}\int_{B_t(w_{\infty})}\left| 1_{A_{\infty}}
-\frac{1}{\upsilon_{i}(B_t(w_{i}))}
\int_{B_t(w_{i})}1_{A_{i}}d\upsilon_{i}\right| d\upsilon_{\infty} \\
&\le \frac{1}{\upsilon_{\infty}(B_t(w_{\infty}))}\int_{B_t(w_{\infty})}\left| 1_{A_{\infty}}
-\frac{1}{\upsilon_{\infty}(B_t(w_{\infty}))}
\int_{B_t(w_{\infty})}1_{A_{\infty}}d\upsilon_{\infty}\right| d\upsilon_{\infty} + \epsilon \\
&\le 2\frac{\upsilon_{\infty}(B_t(w_{\infty}) \setminus A_{\infty})}{\upsilon_{\infty}(B_t(w_{\infty}))}+\epsilon < 3\epsilon
\end{align}
for every sufficiently large $i$.
Thus, we have the assertion.
\end{proof}
\subsection{Convergence of differential of Lipschitz functions}
The purpose of this subsection is to give a definition of convergence: $df_i \rightarrow df_{\infty}$.
See Definition \ref{063} or Definition \ref{Lip}.
Throughout this subsection, we fix the following situation:
Let $\{(M_i, m_i, \underline{\mathrm{vol}})\}_i$ be a sequence of pointed, connected $n$-dimensional complete Riemannian manifolds with 
renormalized measure satisfying 
$\mathrm{Ric}_{M_i} \ge -(n-1)$,
$(Y, y, \upsilon)$ a Ricci limit space of $\{(M_i, m_i, \underline{\mathrm{vol}})\}_i$, $R$ a positive number, 
$x_i$ a point in $M_i$, $x_{\infty}$ a point in $Y$,
$f_i$ a Lipschitz function on $B_R(x_i)$ and $f_{\infty}$ a Lipschitz function on $B_R(x_{\infty})$.
We assume that $\sup_i (\mathbf{Lip}f_i + |f_i|_{L^{\infty}}) < \infty$ and that $x_i \rightarrow x_{\infty}$.
\

For $w \in B_R(x_{\infty})$, we say that $f_i$ \textit{converges to} $f_{\infty}$ at $w$ if 
$f_i(w_i) \rightarrow f_{\infty}(w)$ holds for every $w_i \rightarrow w$.
We denote it by $f_i \rightarrow f_{\infty}$ at $w$.
It is easy to check that the following conditions are equivalent:
\begin{enumerate}
\item $\{f_i\}$ has infinitesimal convergence property to $f_{\infty}$ at $w$.
\item $f_i \rightarrow f_{\infty}$ at $w$.
\item $\{f_i\}$  has infinitesimal constant convergence property to $f_{\infty}$ at $w$.
\end{enumerate}  
We shall consider a convergence of energy of Lipschitz functions.
See also \cite[Corollary $10. 17$]{ch2}.
\begin{definition}[Infinitesimal upper semicontinuity of energy]
We say that $\{f_i\}_i$ has \textit{infinitesimal upper semicontinuity of energy to $f_{\infty}$ at} $w \in B_R(x_{\infty})$ if 
for every $\epsilon > 0$ and $w_i \rightarrow w$, there exists $r > 0$ such that  
\[\limsup_{i \rightarrow \infty}\frac{1}{\underline{\mathrm{vol}}\,B_t(w_i)}\int_{B_t(w_i)}(\mathrm{Lip}f_i)^2d\underline{\mathrm{vol}}
\le \frac{1}{\upsilon (B_t(w))}\int_{B_t(w)}(\mathrm{Lip}f_{\infty})^2d\upsilon + \epsilon \]
for every $0 < t < r$. 
\end{definition}
By the definition, if $\{(\mathrm{Lip}f_i)^2\}_i$ has infinitesimal convergence property to $(\mathrm{Lip}f_{\infty})^2$ at $w$, then $\{f_i\}_i$ has infinitesimal upper semicontinuity of energy to $f_{\infty}$ at $w$.
Next, we shall give a definition of convergence of differential of Lipschitz functions:
\begin{definition}[Convergence of differential of Lipschitz functions]\label{Lip}
We say that $df_i$ converges to $df_{\infty}$ at $w \in B_R(x_{\infty})$ if 
$\{\langle dr_{z_i}, df_i \rangle\}_i$ has infinitesimal convergence property to $\langle df_{\infty}, dg_{\infty}\rangle$ at $w$ for every $z_i \rightarrow z \in Y$ and  $\{f_i\}_i$ has infinitesimal upper semicontinuity of energy to $f_{\infty}$ at $w$.
Then we denote it by $df_i \rightarrow df_{\infty}$ at $w$. Moreover, for a subset $A$ of $B_R(x_{\infty})$, if $f_i \rightarrow f_{\infty}$ and $df_i \rightarrow df_{\infty}$ at every $a \in A$, then we denote it by $(f_i, df_i) \rightarrow (f_{\infty}, df_{\infty})$ on $A$. 
\end{definition}
\begin{proposition}\label{10104}
For every $w_i \rightarrow w \in Y$, we have $(r_{w_i}, dr_{w_i}) \rightarrow (r_w, dr_w)$ on $Y$.
\end{proposition}
\begin{proof}
It follows from Proposition \ref{10102} and Proposition \ref{10103} directly.
\end{proof}
The following theorem is the main result in this subsection:
\begin{theorem}\label{10105}
Let $g_i$ be a Lipschitz function on $B_R(x_i)$ and $A$ a Borel subset of $B_R(x_{\infty})$.
We assume that $df_i \rightarrow df_{\infty}$  and $dg_i \rightarrow dg_{\infty}$ on $A$.
Then, for a.e. $w \in A$, the sequence $\{\langle df_i, dg_i\rangle\}_i$ 
has infinitesimal constant convergence property to $\langle df_{\infty}, dg_{\infty}\rangle$ at $w$.
\end{theorem}
\begin{proof}
By Theorem \ref{7} and Lemma \ref{23}, there exist a collection of Borel set $A_j \subset A \setminus \{x_{\infty}\}$, positive integers $1 \le k_j \le n$
and points $x_1^j, \dots, x_{k_j}^j \in Y$ satisfying the following properties:
\begin{enumerate}
\item $\upsilon (A \setminus \bigcup_{j=1}^{\infty}A_j)=0$.
\item $A_j \subset Y \setminus \bigcup_{l=1}^{k_j}(C_{x_l^j} \cup \{x_l^j\})$. 
\item For every $w \in A_j$, there exists $a_1^j, \dots, a_{k_j}^j, b_1^j, \dots, b_{k_j}^j \in \mathbf{R}$ such that 
\[\lim_{r \rightarrow 0}\frac{1}{\upsilon (B_r(w))}\int_{B_r(w)}\left| df_{\infty}- d\left( \sum_{l=1}^{k_j}a_l^jr_{x_l^j}\right) \right|^2 + \left| dg_{\infty}- d\left( \sum_{l=1}^{k_j}b_l^jr_{x_l^j}\right) \right|^2d\upsilon = 0.\]
\end{enumerate}
We take $w \in A_j$ and $a_1^j, \dots, a_{k_j}^j, b_1^j, \dots, b_{k_j}^j \in \mathbf{R}$ satisfying equalities above.
We also take $L \ge 1$ satisfying $\sup_i (\mathbf{Lip}f_i + \mathbf{Lip}g_i) + \sum_{l=1}^{k_j}((a_l^j)^2 + (b_l^j)^2)\le L$.
There exists $\tau > 0$ such that $w \in \bigcup_{l=1}^{k_j}(D_{x_l^j}^{\tau} \setminus B_{\tau}(x_l^j))$.
We also take sequences $x_l^j(i) \rightarrow x_l^j$ and $w_i \rightarrow w$.
We fix $\epsilon > 0$ satisfying $\epsilon << \tau$.
Then, there exists $0 < r << \epsilon$ such that  
\[\frac{1}{\upsilon (B_t(w))}\int_{B_t(w)}\left| df_{\infty}- d\left( \sum_{l=1}^{k_j}a_l^jr_{x_l^j}\right) \right|^2 + \left| dg_{\infty}- d\left( \sum_{l=1}^{k_j}b_l^jr_{x_l^j}\right) \right|^2d\upsilon \le \epsilon, \]
\[\limsup_{i \rightarrow \infty}\frac{1}{\underline{\mathrm{vol}}\,B_t(w_i)}\int_{B_t(w_i)}(\mathrm{Lip}f_i)^2d\underline{\mathrm{vol}}
\le \frac{1}{\upsilon (B_t(w))}\int_{B_t(w)}(\mathrm{Lip}f_{\infty})^2d\upsilon + \epsilon, \]
\[\limsup_{i \rightarrow \infty}\frac{1}{\underline{\mathrm{vol}}\,B_t(w_i)}\int_{B_t(w_i)}(\mathrm{Lip}g_i)^2d\underline{\mathrm{vol}}
\le \frac{1}{\upsilon (B_t(w))}\int_{B_t(w)}(\mathrm{Lip}g_{\infty})^2d\upsilon + \epsilon, \]
\[\limsup_{i \rightarrow \infty}\left| \frac{1}{\underline{\mathrm{vol}}\,B_t(w_i)}\int_{B_t(w_i)} \langle df_i, dr_{x_l^j(i)}\rangle
d\underline{\mathrm{vol}}-\frac{1}{\upsilon (B_t(w))}\int_{B_t(w)}\langle df_{\infty}, dr_{x_l^j}\rangle d\upsilon \right|<\epsilon\]
and
\[\limsup_{i \rightarrow \infty}\left| \frac{1}{\underline{\mathrm{vol}}\,B_t(w_i)}\int_{B_t(w_i)} \langle dg_i, dr_{x_l^j(i)}\rangle
d\underline{\mathrm{vol}}-\frac{1}{\upsilon (B_t(w))}\int_{B_t(w)}\langle dg_{\infty}, dr_{x_l^j}\rangle d\upsilon \right|<\epsilon\]
for every $l$ and $0 < t < r$
We fix $ 0 < t < r$ below.
Thus, by Lemma \ref{17}, we have 
\[\frac{1}{\upsilon (B_t(w))}\int_{B_t(w)}\left| \langle df_{\infty}, dg_{\infty}\rangle-
\frac{1}{\upsilon (B_t(w))}\int_{B_t(w)}\left\langle d\left(\sum_{l=1}^{k_j}a_l^jr_{x_l^j}\right), d\left(\sum_{l=1}^{k_j}b_l^jr_{x_l^j}\right)\right\rangle d\upsilon \right| d\upsilon \le \Psi(\epsilon; L)\]
and
\begin{align}
&\frac{1}{\upsilon (B_t(w))}\int_{B_t(w)}\left| \langle df_{\infty}, dg_{\infty}\rangle -\frac{1}{\upsilon (B_t(w))}\int_{B_t(w)}\langle df_{\infty}, dg_{\infty}\rangle d\upsilon
\right| d\upsilon \\
&=\frac{1}{\upsilon (B_t(w))}\int_{B_t(w)}\biggl| \left\langle d\left(\sum_{l=1}^{k_j}a_l^jr_{x_l^j}\right), d\left(\sum_{l=1}^{k_j}b_l^jr_{x_l^j}\right)\right\rangle \\
&\ \ -\frac{1}{\upsilon (B_t(w))}\int_{B_t(w)}\left\langle d\left(\sum_{l=1}^{k_j}a_l^jr_{x_l^j}\right), d\left(\sum_{l=1}^{k_j}b_l^jr_{x_l^j}\right)\right\rangle d\upsilon \biggl| d\upsilon \pm \Psi(\epsilon; n, L) \\
&= \Psi(\epsilon; n, L).
\end{align}
On the other hand, for every sufficiently large $i$, we have 
\begin{align}
&\frac{1}{\underline{\mathrm{vol}}\,B_t(w_i)}\int_{B_t(w_i)}\left|df_i - d\left(\sum_{l=1}^{k_j}a_l^jr_{x_l^j(i)}\right)\right|^2d\underline{\mathrm{vol}} \\
&= \frac{1}{\underline{\mathrm{vol}}\,B_t(w_i)}\int_{B_t(w_i)}|df_i|^2d\underline{\mathrm{vol}}-
\sum_{l=1}^{k_j}\frac{a_l^j}{\underline{\mathrm{vol}}\,B_t(w_i)}\int_{B_t(w_i)}\langle df_i, dr_{x_l^j(i)}\rangle d\underline{\mathrm{vol}} \\
&+ \sum_{l, \hat{l}} \frac{a_l^j a_{\hat{l}}^j}{\underline{\mathrm{vol}}\,B_t(w_i)}\int_{B_t(w_i)}\langle dr_{x_l^j(i)}, dr_{x_{\hat{l}}^j(i)}\rangle d\underline{\mathrm{vol}} \\
&\le \frac{1}{\upsilon (B_t(w))}\int_{B_t(w)}|df_{\infty}|^2d\upsilon -\sum_{l=1}^k\frac{a_l^j}{\upsilon (B_t(w))}\int_{B_t(w)}\langle df_{\infty}, dr_{x_l^j}\rangle d\upsilon \\
&+ \sum_{l, \hat{l}} \frac{a_l^j a_{\hat{l}}^j}{\upsilon (B_t(w))}\int_{B_t(w)}\langle dr_{x_l^j}, dr_{x_{\hat{l}}^j}\rangle d\upsilon + \Psi(\epsilon; n, L)\\
&= \frac{1}{\upsilon (B_t(w))}\int_{B_t(w)}\left| df_{\infty}- d\left( \sum_{l=1}^{k_j}a_l^jr_{x_l^j}\right) \right|^2d\upsilon + \Psi(\epsilon; n, L) \le \Psi(\epsilon; n, L).
\end{align}
Similarly, we have 
\[\frac{1}{\underline{\mathrm{vol}}\,B_t(w_i)}\int_{B_t(w_i)}\left|dg_i - d\left(\sum_{l=1}^{k_j}b_l^jr_{x_l^j(i)}\right)\right|^2d\underline{\mathrm{vol}}\le \Psi(\epsilon; n, L)\]
for every sufficiently large $i$.
Especially, we have 
\begin{align}
&\frac{1}{\underline{\mathrm{vol}}\,B_t(w_i)}\int_{B_t(w_i)}\left| \langle df_i, dg_i\rangle - \frac{1}{\underline{\mathrm{vol}}\,B_t(w_i)}\int_{B_t(w_i)}
\left\langle d\left(\sum_{l=1}^{k_j}a_l^jr_{x_l^j(i)}\right), d\left(\sum_{l=1}^{k_j}b_l^jr_{x_l^j(i)}\right)\right\rangle d\underline{\mathrm{vol}}\right| d\underline{\mathrm{vol}} \\  
& \le \Psi(\epsilon; n, L).
\end{align}
Therefore, we have the assertion.
\end{proof}
We remark that Theorem \ref{473} follows from Theorem \ref{10105} directly.
\begin{corollary}\label{10106}
Let $\Omega$ be a non-empty open subset of $B_R(x_{\infty})$.
We assume that for a.e. $w \in \Omega$, $df_i \rightarrow df_{\infty}$ at $w$.
Then $df_i \rightarrow df_{\infty}$ on $\Omega$.
\end{corollary}
\begin{proof}
The assertion follows from Example \ref{ex3}, Proposition \ref{10103} and Theorem \ref{10105}. 
\end{proof}
\begin{corollary}\label{10107}
Let $g_i$ be a Lipschitz function on $B_R(x_i)$ satisfying $\sup_i (\mathbf{Lip}g_i + |g_i|_{L^{\infty}}) < \infty$ and 
$A$ a Borel subset of $B_R(x_{\infty})$.
We assume that $(f_i, df_i) \rightarrow (f_{\infty}, df_{\infty})$ and $(g_i, dg_i) \rightarrow (g_{\infty}, dg_{\infty})$ on $A$.
Then, there exists a Borel subset $\hat{A}$ of $A$ such that 
$\upsilon(A \setminus \hat{A})=0$ and that $(f_i + g_i, d(f_i + g_i)) \rightarrow (f_{\infty} + g_{\infty}, d(f_{\infty} + g_{\infty}))$ and $(f_ig_i, d(f_ig_i)) \rightarrow (f_{\infty}g_{\infty}, d(f_{\infty}g_{\infty}))$ on $\hat{A}$.
\end{corollary}
\begin{proof}
By Theorem \ref{10105},
there exists a Borel subset $\hat{A}$ of $A$ such that $\upsilon(A \setminus \hat{A})=0$ and that 
$\{|df_i|^2\}_i, \{\langle df_i, dg_i\rangle \}_i$ and $\{|dg_i|^2\}_i$ have infinitesimal constant convergence property to 
$|df_{\infty}|^2, \langle df_{\infty}, dg_{\infty}\rangle$ and $|dg_{\infty}|^2$ on $\hat{A}$, respectively.
Since $|d(f_ig_i)|^2=f_i^2|dg_i|^2+2f_ig_i\langle df_i, dg_i\rangle +g_i|df_i|^2$, by Proposition \ref{10101}, we have,
$\{|d(f_ig_i)|^2\}_i$ has infinitesimal constant convergence property to $f_{\infty}^2|dg_{\infty}|^2+2f_{\infty}g_{\infty}\langle df_{\infty}, dg_{\infty}\rangle +g_{\infty}^2|df_{\infty}|^2=|d(f_{\infty}g_{\infty})|^2$ on $\hat{A}$.
On the other hand, since $d(f_ig_i)=g_idf_i+f_idg_i$, by Proposition \ref{101020}, 
for every $z_i \rightarrow z$, 
we have, $\{\langle dr_{z_i}, d(f_ig_i)\rangle\}_i$ has infinitesimal convergence property to $g_{\infty}\langle dr_{z_{\infty}}, df_{\infty}\rangle +f_{\infty}\langle dr_{z_{\infty}}, dg_{\infty}\rangle = \langle dr_{z_{\infty}}, d(f_{\infty}g_{\infty})\rangle$ on 
$\hat{A}$.
Therefore we have $(f_ig_i, d(f_ig_i)) \rightarrow (f_{\infty}g_{\infty}, d(f_{\infty}g_{\infty}))$ on $\hat{A}$.
Similarly,  we have $(f_i + g_i, d(f_i + g_i)) \rightarrow (f_{\infty} + g_{\infty}, d(f_i + g_i))$ on $\hat{A}$.
\end{proof}
\begin{corollary}\label{10108}
Let $K_i$ be a Borel subset of $\overline{B}_R(x_i)$ and $g_i$ a Lipschitz function on $B_R(x_i)$  satisfying $\sup_i (\mathbf{Lip}g_i + |g_i|_{L^{\infty}}) < \infty$.
We assume that  $K_{\infty}$ is compact, $\limsup_{i \rightarrow \infty} \subset K_{\infty}$ and that for a.e. $w \in K_{\infty}$, $1_{K_i}$ has infinitesimal convergence property to $1_{K_{\infty}}$ at $w$, $dg_i \rightarrow df_{\infty}$ and $df_i \rightarrow df_{\infty}$ at $w$.
Then for every sequence of continuous functions $F_i$ on $\mathbf{R}$ satisfying that $F_i$ converges to $F_{\infty}$ in the sense of compact uniformly topology, we have 
\[\lim_{i \rightarrow \infty}\int_{K_i}F_i(|df_i-dg_i|)d\underline{\mathrm{vol}}=F_{\infty}(0)\upsilon(K_{\infty}).\]
\end{corollary}
\begin{proof}
The assertion follows from Proposition \ref{10101}, Proposition \ref{lpoij} and Theorem \ref{10105}.
\end{proof}
\begin{remark}
By several arguments in section $3$ and the proof of Theorem \ref{10105}, we can also prove the following: 
If $\{f_i\}_i$ satisfies,
\begin{enumerate}
\item $\{f_i\}_i$ has infinitesimal upper semicontinuity of energy to $f_{\infty}$ at every $\alpha \in B_R(x_{\infty})$,
\item there exists a dense subset $A$ of $B_R(x_{\infty})$ and a Borel subset $\hat{A}$ of $B_R(x_{\infty})$ such that 
$\upsilon (B_R(x_{\infty})\setminus \hat{A})=0$ and that for every $w \in A$ and $w_i \rightarrow w$,
$\{\langle dr_{w_i}, df_i\rangle \}_i$ has infinitesimal convergence property to $\langle dr_{w}, df_{\infty} \rangle$ at 
every $\alpha \in \hat{A}$,
\end{enumerate}
then, $df_i \rightarrow df_{\infty}$ on $B_R(x_{\infty})$.
\end{remark}
\begin{remark}
Similarly,
for a sequence of Ricci limit spaces $\{(Y_i, y_i, \upsilon_i)\}_i$ and a sequence of Lipschitz function $f_i$ on $B_R(y_i)$, 
we can also define a notion of convergence: $df_i \rightarrow df_{\infty}$ and prove several properties as above.
\end{remark}
\begin{remark}
For fixed Ricci limit space $(Y, y, \upsilon)$, a sequence of Lipschitz functions $f_i$ on $B_R(y)$ satisfying $\sup_i \mathbf{Lip}f_i <\infty$, we have, $df_i \rightarrow df_{\infty}$ on $B_R(y)$ (in the sense of the convergence $(Y, y, \upsilon) \stackrel{(\mathrm{id}_Y, R_i, \epsilon_i)}{\rightarrow} (Y, y, \upsilon)$) if and only if 
$|\mathrm{Lip}(f_i-f_{\infty})|_{L^2(B_R(y))} \rightarrow 0$.
We shall check it.
By Corollary \ref{10108}, it suffices to check that `if' part.
We assume that $|\mathrm{Lip}(f_i-f_{\infty})|_{L^2(B_R(y))} \rightarrow 0$.
Then, especially, for every $w \in B_R(y)$, $\{f_i\}_i$ has infinitesimal upper semicontinuity of energy to $f_{\infty}$ at $w$.
On the other hand, by Proposition \ref{10104}, we have
\[\lim_{i \rightarrow \infty}\int_{B_R(y)}|dr_{x_i}-dr_{x_{\infty}}|^2d\upsilon=0\]
for $x_i \rightarrow x_{\infty} \in Y$.
Therefore, $\{\langle dr_{x_i}, df_i\rangle \}$ has infinitesimal convergence property to $\langle dr_{x_{\infty}}, df_{\infty}\rangle$ at every $w \in B_R(y)$.
Thus, $df_i \rightarrow df_{\infty}$ on $B_R(y)$.
\end{remark}
We will give a sufficient condition to satisfy infinitesimal upper semicontinuity of energy in the next subsection.
See Proposition \ref{energy}.
\subsection{Approximation theorem}
Throughout this subsection, we shall use the following notation (same to previous subsection):
Let $\{(M_i, m_i, \underline{\mathrm{vol}})\}_i$ be a sequence of pointed, connected $n$-dimensional complete Riemannian manifolds with 
renormalized measure satisfying 
$\mathrm{Ric}_{M_i} \ge -(n-1)$,
$(Y, y, \upsilon)$ a Ricci limit space of $\{(M_i, m_i, \underline{\mathrm{vol}})\}$, $R$ a positive number, 
$x_i$ a point in $M_i$, $x_{\infty}$ a point in $Y$ satisfying $(M_i, m_i, x_i, \underline{\mathrm{vol}}) \stackrel{(\phi_i, R_i, \epsilon_i)}{\rightarrow} (Y, y, x_{\infty}, \upsilon)$.
The purpose in this subsection is to give an approximation theorem (Theorem \ref{app}).
Roughly speaking,
it means that for given Lipschitz function on $B_R(x_{\infty})$, there exists a sequence of Lipschitz function on $B_R(x_i)$ approximating the function
in the sense of the topology: $(f_i, df_i) \rightarrow (f_{\infty}, df_{\infty})$.
\begin{theorem}[Approximation theorem]\label{app}
Let $L, R$ be positive numbers, $f_{\infty}$ a $L$-Lipschitz function on $\overline{B}_R(x_{\infty})$,
$A_i$ a Borel subset of $\overline{B}_R(x_i)$, $A_{\infty}$ a compact subset of $\overline{B}_R(x_{\infty})$ and $f_i$ a $L$-Lipschitz function
on $A_i$.
We assume that $\limsup_{i \rightarrow \infty}A_i \subset A_{\infty}$ and that $f_{\infty}|_{A_{\infty}}$ is an extension of $\{f_i\}_i$ asymptotically.
Then, for every $\epsilon > 0$, there exist an open set $\Omega_{\epsilon} \subset B_R(x_{\infty}) \setminus A_{\infty}$, $C(n, L)$-Lipschitz function $f_{\infty}^{\epsilon}$ on $B_R(x_{\infty})$ and a sequence of 
$C(n, L)$-Lipschitz function $f_i^{\epsilon}$ on $B_R(x_i)$ such that $(f_i^{\epsilon}, df_i^{\epsilon}) \rightarrow (f_{\infty}^{\epsilon}, df_{\infty}^{\epsilon})$ on $\Omega_{\epsilon}$, $f_{\infty}^{\epsilon}|_{A_{\infty}}=f|_{A_{\infty}}$, 
$f_i^{\epsilon}|_{A_i}=f_i|_{A_i}$ and that 
\[\frac{\upsilon (B_R(x_{\infty}) \setminus (\Omega_{\epsilon} \cup A_{\infty}))}{\upsilon(B_R(x_{\infty}))} + |f_{\infty}- f_{\infty}^{\epsilon}|_{L^{\infty}(B_R(x_{\infty}))} + \frac{1}{\upsilon_{\infty}(B_R(x_{\infty}))}\int_{B_R(x_{\infty})} |df_{\infty}^{\epsilon}-df_{\infty}|^2d\upsilon< \epsilon. \]
\end{theorem}
\begin{proof}
We fix sufficiently small $\epsilon > 0$ and $\xi > 0$. (We will decide $\xi$ later.)
By Lemma \ref{6} and (the proof of) Theorem \ref{7}, there exist a (pairwise disjoint) collection of Borel set $E_j \subset B_R(x_{\infty})$, positive numbers $\tau_j >0$,
positive integers $1 \le k_j \le n$ and points $x_1^j, \dots, x_{k_j}^j \in Y$ satisfying following properties:
\begin{enumerate}
\item $\upsilon_{\infty}(B_R(x_{\infty}) \setminus \bigcup_j E_j)=0$.
\item $E_j \subset \bigcap_{l=1}^{k_j}(\mathcal{D}_{x_l^j}^{\tau_j} \setminus B_{\tau_j}(x_l^j))$.
\item For every $w \in E_j$, 
\[\langle dr_{x_l^j}, dr_{x_{\hat{l}}^j}\rangle (w)=\lim_{r \rightarrow 0}\frac{1}{\upsilon (B_r(w))}\int_{B_r(w)}\langle dr_{x_l^j}, dr_{x_{\hat{l}}^j}\rangle d\upsilon
= \delta_{l, \hat{l}} \pm \epsilon \]
\item For every $w \in E_j$, there exist $a_1^j(w), \dots, a_{k_j}^j(w) \in \mathbf{R}$ such that
\[\lim_{r \rightarrow 0}\frac{1}{\upsilon (B_r(w))}\int_{B_r(w)}\left| df-d\left( \sum_{l=1}^{k_j}a_l^j(w)r_{x_l^j}\right) \right|^2 d\upsilon =0.\]
\end{enumerate}
For every $w \in E_j$, there exists $0< r_w <<\tau_j$ such that $\overline{B}_t(w) \subset B_R(x_{\infty})$ and 
\[\frac{1}{\upsilon (B_t(w))}\int_{B_t(w)}\left| df-d\left( \sum_{l=1}^{k_j}a_l^jr_{x_l^j}\right) \right|^2 d\upsilon < \epsilon \]
for every $ 0 < t < r_w$.
We put $X= \bigcup_{j=1}^{\infty}(E_j \setminus \overline{B}_{5\xi}(A_{\infty}))$.
By Proposition \ref{cov}, there exists a pairwise disjoint collection $\{\overline{B}_{r_i}(z_i)\}_i \subset B_{R}(x_{\infty})$ such that
$z_i \in X$, $r_i << \min \{r_{z_i}, \epsilon, \xi \}$ and
$X \setminus \bigcup_{i=1}^N\overline{B}_{r_i}(z_i) \subset \bigcup_{i=N+1}^{\infty}\overline{B}_{5r_i}(z_i)$ for every $N$.
For every $i$, we take $l(i)$ satisfying $z_i \in E_{l(i)}$.
We fix $N$ satisfying 
$\sum_{i=N+1}^{\infty}\upsilon (B_{r_i}(z_i)) < \epsilon.$
We take sequences $z_i(j) \rightarrow z_i$ and $x_m^l(j) \rightarrow x_m^l$.
We define a function $F_i^j$ on $B_{r_i}(z_i(j))$ and a function $F_i$ on $B_{r_i}(z_i)$ by 
\[F_i^j = \sum_{m=1}^{k_{l(i)}}a_m^{l(i)}r_{x_m^{l(i)}(j)}+C_{i}, \ F_i = \sum_{m=1}^{k_{l(i)}}a_m^{l(i)}r_{x_m^{l(i)}}+C_{i}. \]
Here $C_{i}$ is the constant satisfying $F_i(z_i)=f_{\infty}(z_i)$.
\begin{claim}\label{1a}
We have $\mathbf{Lip}F_i^j + \mathbf{Lip}F_i \le C(n, L)$ for every $i, j$.
\end{claim}
The proof is as follows:
Since
\begin{align}
|df_{\infty}(z_i)|^2 &= \sum_{s, t} a_s^{l(i)} a_t^{l(i)} \langle dr_{x_s^{l(i)}}, r_{x_t^{l(i)}}\rangle (z_i) \\
&=\sum_{s, t}a_s^{l(i)} a_t^{l(i)}(\delta_{s, t} \pm \epsilon) \\
&=(1 \pm \epsilon)\sum_{s=1}^{l(i)}(a_s^{l(i)})^2 \pm \epsilon \sum_{s \neq t} |a_s^{l(i)}| |a_t^{l(i)}| \\
&=(1 \pm \epsilon)\sum_{s=1}^{l(i)}(a_s^{l(i)})^2 \pm \Psi(\epsilon; n)\sum_{s=1}^{l(i)}(a_s^{l(i)})^2 \\
&=(1 \pm \Psi(\epsilon; n))\sum_{s=1}^{l(i)}(a_s^{l(i)})^2
\end{align}
and $|df_{\infty}|(z_i) \le L$, we have
\[\sum_{m=1}^{k_i}(a_m^{l(i)})^2 \le L^2 + \Psi(\epsilon; n, L).\]
Therefore we have Claim \ref{1a}.
\

Since $\{\overline{B}_{r_i}(z_i(j))\}_{1 \le i \le N}$ are pairwise disjoint for every sufficiently large $j$, we define a function $F_j$ on $\bigcup_{m=1}^N\overline{B}_{(1-\xi)r_i}(z_i(j))$ and 
a function $F_{\infty}$ on $\bigcup_{m=1}^N\overline{B}_{(1-\xi)r_i}(z_i)$ by
$F_j|_{B_{(1-\xi)r_i}(z_i(j))}=F_j^i|_{B_{(1-\xi)r_i}(z_i(j))}, F_{\infty}|_{B_{(1-\xi)r_i}(z_i)}=F_j|_{B_{(1-\xi)r_i}(z_i)}.$
\begin{claim}\label{2a}
We have $\mathbf{Lip}F_j, \mathbf{Lip}F_{\infty} \le C(n, L) + \xi^{-1}\Psi(\epsilon;n, L)$ for every sufficiently large $j$.
\end{claim}
The proof is as follows.
By Claim \ref{1a}, for every $i, j$, we have $\mathbf{Lip}(F_j|_{\overline{B}_{(1-\xi)r_i}(z_i(j))}) + \mathbf{Lip}(F_{\infty}|_{\overline{B}_{(1-\xi)r_i}(z_i)}) \le C(n, L)$.
There exists $j_0$ such that $\epsilon_j << \min \{ \xi r_1, _{\cdots}, \xi r_N\}$ for every $j \ge j_0$.
We fix $j \ge j_0$, $1 \le l < m \le N$, $w_l(j) \in \overline{B}_{(1-\xi)r_l}(z_l(j))$ and $w_m(j) \in \overline{B}_{(1-\xi)r_m}(z_m(j))$.
Since $\overline{B}_{r_l}(z_l(j)) \cap \overline{B}_{r_m}(z_m(j)) = \emptyset$, by taking $\alpha(j) \in \partial B_{r_l}(z_l)$ satisfying $\overline{w_l(j), \alpha(j)}+ \overline{\alpha(j), w_m(j)}=\overline{w_l(j), w_m(j)}$, we have
$\overline{w_l(j), w_m(j)} \ge \overline{w_l(j), \alpha(j)}\ge \xi r_l.$
Similarly, we have $\overline{w_l(j), w_m(j)} \ge \xi r_m$.
Thus, we have $\overline{w_l(j), w_m(j)} \ge \xi (r_l + r_m)/2$.
On the other hand, since
\[\frac{1}{\upsilon (B_{10r_l}(z_l))}\int_{B_{10r_l}(z_l)}\left| \mathrm{Lip}\left(f_{\infty}-\sum_{s=1}^{k_l}a_s^{k_l}r_{x_s^{k_l}}\right) \right|^2d\upsilon < \epsilon,\]
by segment inequality on limit spaces (\cite[Theorem $2.6$]{ch-co3}), 
there exist $\hat{z}_l, \hat{\phi_j(w_l(j))} \in B_{r_l}(z_l)$ and a minimal geodesic $\gamma $ from $\hat{z}_l$ to $\hat{\phi_j(w_l(j))}$ such that
$\overline{z_l, \hat{z}_l} + \overline{\phi_j(w_l(j)), \hat{\phi_j(w_l(j))}} < \Psi(\epsilon;n)r_l$ and that 
\[\int_0^{\overline{\hat{z}_l,\hat{\phi_j(w_l(j))}}}\mathrm{Lip}\left( f_{\infty}-\sum_{s=1}^{k_l}a_s^{k_l}r_{x_s^{k_l}}\right)(\gamma (t))dt < \Psi(\epsilon;n)r_l. \]
Therefore we have
\begin{align}
&\left|f_{\infty}(\hat{z}_l)-\sum_{s=1}^{k_l}a_s^{k_l}r_{x_s^{k_l}}(\hat{z}_l)-\left(f_{\infty}(\hat{\phi_j(z_l(j))})-
\sum_{s=1}^{k_l}a_s^{k_l}r_{x_s^{k_l}}(\hat{\phi_j(z_l(j))})\right) \right| \\
&\le \int_0^{\overline{\hat{z}_l,\hat{\phi_j(w_l(j))}}}\mathrm{Lip}\left( f_{\infty}-\sum_{s=1}^{k_l}a_s^{k_l}r_{x_s^{k_l}}\right)(\gamma (t))dt < \Psi(\epsilon;n)r_l.
\end{align}
Thus 
\[\left|f_{\infty}(z_l)-\sum_{s=1}^{k_l}a_s^{k_l}r_{x_s^{k_l}}(z_l)-\left(f_{\infty}(\phi_j(z_l(j)))-
\sum_{s=1}^{k_l}a_s^{k_l}r_{x_s^{k_l}}(\phi_j(z_l(j)))\right) \right| \le \Psi(\epsilon;n, L)r_l.\]
Especially, we have $|F_j(w_l(j))-f_{\infty}\circ \phi_j(w_l(j))| \le \Psi(\epsilon;n, L)r_l.$  
Similarly, we have $|F_j(w_m(j))-f_{\infty}\circ \phi_j(w_m(j))| \le \Psi(\epsilon;n, L)r_m$ and $|F_{\infty}-f_{\infty}| \le \Psi(\epsilon;n, L)r_l$ on $\overline{B}_{(1-\xi)r_l}(z_l)$.
Therefore we have 
\begin{align}
|F_j(w_l(j))-F_j(w_m(j))| &\le |f_{\infty}\circ \phi_j(w_l(j))-f_{\infty} \circ \phi_j(w_l(j))| + \Psi(\epsilon;n, L)(r_l+ r_m) \\
&\le L \overline{\phi_j(w_l(j)), \phi_j(w_m(j))} + \Psi(\epsilon;n, L)(r_l+ r_m) \\
&\le L(\overline{w_l(j), w_m(j)}+\epsilon_j) + \Psi(\epsilon;n, L)(r_l+ r_m) \\
&\le L\overline{w_l(j), w_m(j)} + \Psi(\epsilon;n, L)(r_l+ r_m) \\
&\le (L+ \xi^{-1}\Psi(\epsilon;n, L))\overline{w_l(j), w_m(j)}.
\end{align}
Thus, by Claim \ref{1a}, we have $\mathbf{Lip}F_j \le C(n, L) + \xi^{-1}\Psi(\epsilon;n, L)$.
Similarly, we have $\mathbf{Lip}F_{\infty} \le C(n, L) + \xi^{-1}\Psi(\epsilon;n, L)$.
Therefore we have Claim \ref{2a}.
\
\begin{claim}\label{3a}
For every sufficiently large $j$, we have $\bigcup_{i=1}^N\overline{B}_{(1-\xi)r_i}(z_i(j)) \subset M_i \setminus B_{2\xi}(A_i)$ and 
$\bigcup_{i=1}^N\overline{B}_{(1-\xi)r_i}(z_i) \subset Y \setminus B_{2\xi}(A_{\infty})$.
\end{claim}
Because, by the definition, we have $\bigcup_{i=1}^N\overline{B}_{r_i}(z_i) \subset Y \setminus B_{2\xi}(A_{\infty})$.
On the other hand, by the assumption, there exists $i_0$ such that for every $i \ge i_0$, we have 
$\phi_i(A_i) \subset B_{\xi}(A_{\infty})$ and $\epsilon_i << \min_{1 \le j \le N} \{\xi r_j\}$.
Thus, since $\phi_i(\bigcup_{i=1}^N\overline{B}_{(1-\xi)r_i}(z_i(j))) \subset \bigcup_{i=1}^N\overline{B}_{r_i}(z_i) \subset Y \setminus B_{4\xi}(A_{\infty})$ for every $i \ge i_0$, we have Claim \ref{3a}.
\

On the other hand, we  remark the following claim:
\begin{claim}\label{4a}
We have 
\[\lim_{i \rightarrow \infty}\sup_{A_i}|f_i-f_{\infty}\circ \phi_i|=0.\]
\end{claim}
The proof is done by a contradiction.
We assume that the assertion is false.
Then, there exist $\tau >0$, a subsequence $\{n(i)\}$ of $\mathbf{N}$ and $\alpha_i \in A_{n(i)}$ such that  $|f_{n(i)}(\alpha_i)-f_{\infty}\circ \phi_{n(i)}(\alpha_i)| > \tau$.
Without loss of generality, we can assume that there exists $\alpha_{\infty} \in Y$ such that $\phi_{n(i)}(\alpha_i) \rightarrow \alpha_{\infty}$.
Thus, $\liminf_{i \rightarrow \infty}|f_{n(i)}(\alpha_i)-f_{\infty}(\alpha_{\infty})|\ge \tau$.
On the other hand, by the assumption, we have $\alpha_{\infty} \in \overline{A_{\infty}}=A_{\infty}$.
Since $f_{\infty}|_{A_{\infty}}$ is an extension of $\{f_i\}$ asymptotically, this is a cotradiction. Therefore we have Claim \ref{4a}.
\

We put $W_j = \bigcup_{m=1}^NB_{(1-\xi)r_i}(z_i(j))$ and $W_{\infty}= \bigcup_{m=1}^NB_{(1-\xi)r_i}(z_i)$.
By Claim \ref{3a},
we can define a Lipschitz function $G_j$ on $W_j \cup A_j$ and a Lipschitz function $G_{\infty}$
on $W_{\infty} \cup A_{\infty}$ by
$G_j|_{W_j}=F_j|_{W_j}$, $G_j|_{A_j}=f_j$, $G_{\infty}|_{W_{\infty}}=F_{\infty}|_{W_{\infty}}$ and $G_{\infty}|_{A_{\infty}}=f_{\infty}|_{A_{\infty}}$.
\begin{claim}\label{5a}
We have $\mathbf{Lip}G_j, \mathbf{Lip}G_{\infty} \le C(n, L) + \xi^{-1}\Psi(\epsilon;n, L)$ for every sufficiently large $j$.
\end{claim}
The proof is as follows.
We put $\xi_j = \sup_{A_j}|f_j-f_{\infty}\circ \phi_j|$.
Then by the proof of Claim \ref{2a}, there exists $j_0$ such that for every $j \ge j_0$,
$\alpha_j \in \overline{B}_{(1-\xi)r_i}(z_i(j))$ and $\beta_j \in A_j$, we have
\begin{align}
|G_j(\alpha_j)-G_j(\beta_j)| &= |F_j(\alpha_j)-f_j(\beta_j)| \\
&\le |F_{\infty}\circ \phi_j(\alpha_j)-f_{\infty}\circ \phi_j(\beta_j)| + \Psi(\epsilon;n, L)r_i + \xi_j \\
&\le |f_{\infty}\circ \phi_j(\alpha_j)-f_{\infty}\circ \phi_j(\beta_j)| + \Psi(\epsilon;n, L)r_i + \xi_j \\
&\le L \overline{\phi_j(\alpha_j), \phi_j(\beta_j)} + \Psi(\epsilon;n, L)r_i \\
&\le L (\overline{\alpha_j, \beta_j} + \epsilon_j) + \Psi(\epsilon;n, L)\xi \\
&\le (L + \Psi(\epsilon;n, L))\overline{\alpha_j, \beta_j}.
\end{align}
Therefore, by Claim \ref{2a}, we have $\mathbf{Lip}G_j \le C(n, L) + \xi^{-1}\Psi(\epsilon;n, L)$ for every sufficiently large $j$.
Similarly, we have $\mathbf{Lip}G_{\infty} \le C(n, L) + \xi^{-1}\Psi(\epsilon;n, L)$.
Thus, we have Claim \ref{5a}.

For $\Psi= \Psi (\epsilon; n, L)$ in Claim \ref{5a}, we put $\xi= \sqrt{\Psi}$.
We take a Lipschitz function $f_j^{\epsilon}$ on $M_i$ and a Lipschitz function $f_{\infty}^{\epsilon}$ on $Y$ satisfying that $\mathbf{Lip}f_j^{\epsilon} = \mathbf{Lip}G_j$, $\mathbf{Lip}f_{\infty}^{\epsilon}=\mathbf{Lip}G_{\infty}$, $f_j^{\epsilon}|_{W_j \cup A_j}=F_j|_{W_j \cup A_j}$ and $f_{\infty}^{\epsilon}|_{W_{\infty} \cup A_{\infty}}=F_{\infty}|_{W_{\infty} \cup A_{\infty}}$.
We put $\Omega_{\epsilon}= W_{\infty}$.
Then, by the definition, Proposition \ref{10104} and Corollary \ref{10107}, we have $(f_i^{\epsilon}, df_i^{\epsilon}) \rightarrow (f_{\infty}^{\epsilon}, df_{\infty}^{\epsilon})$ on $\Omega_{\epsilon}$.
We have
\begin{align}
\int_{B_R(x_{\infty})}|df_{\infty}-df_{\infty}^{\epsilon}|^2d\upsilon &\le \int_{X \setminus \overline{B}_{5\xi}(A_{\infty})} |df_{\infty}-df_{\infty}^{\epsilon}|^2d\upsilon + \int_{\overline{B}_{5\xi}(A_{\infty})}|df_{\infty}-df_{\infty}^{\epsilon}|^2d\upsilon \\
&\le \sum_{i=1}^N\int_{B_{(1-\xi)r_i}(z_i)}|df_{\infty}-df_{\infty}^{\epsilon}|^2d\upsilon \\
&\pm \left( 5L^2\upsilon(B_{5\xi}(A_{\infty})\setminus A_{\infty}) + \int_{A_{\infty}}|df_{\infty}^{\epsilon}-df_{\infty}|^2d\upsilon + \Psi(\epsilon; n, L) \right) \\
&\le \sum_{i=1}^N \epsilon \upsilon (B_{(1-\xi)r_i}(z_i)) \pm \left( 5L^2\upsilon(B_{5\xi}(A_{\infty})\setminus A_{\infty}) + \Psi(\epsilon; n, L) \right)  \\
&\le \epsilon \upsilon (B_R(x_{\infty})) \pm \left( 5L^2\upsilon(B_{5\xi}(A_{\infty})\setminus A_{\infty}) + \Psi(\epsilon; n, L) \right)
\end{align}
and
\begin{align}
\upsilon (B_R(x_{\infty}) \setminus (\Omega_{\epsilon} \cup A_{\infty})) &\le \upsilon (X \setminus (\Omega_{\epsilon} \cup A_{\infty}))
+ \upsilon (\overline{B}_{\xi}(A_{\infty}) \setminus A_{\infty}) \\
&\le \sum_{i=N+1}^{\infty}\upsilon (B_{5r_i}(z_i)) + \upsilon (\overline{B}_{\xi}(A_{\infty}) \setminus A_{\infty}) \\
& \le C(n) \epsilon + \upsilon (\overline{B}_{\xi}(A_{\infty}) \setminus A_{\infty}).
\end{align}
Since $A_{\infty}$ is compact, we remark that $\lim_{r \rightarrow 0}\upsilon(B_r(A_{\infty}) \setminus A_{\infty})=0$.
We put $ \tau(r) = \upsilon(B_r(A_{\infty}) \setminus A_{\infty})$.
On the other hand, by the proof of Claim \ref{2a}, we have $|f_{\infty}^{\epsilon}-f_{\infty}| < \Psi(\epsilon;n, L)$ on $\Omega_{\epsilon} \cup A_{\infty}$.
For every $w \in B_R(x_{\infty})$, there exists $\hat{w} \in \Omega_{\epsilon} \cup A_{\infty}$ such that $\overline{w, \hat{w}} < \Psi(\epsilon, \tau(5\xi);n, L, \upsilon(B_R(x_{\infty})))$.
Therefore, we have $|f_{\infty}^{\epsilon}(w)-f_{\infty}(w)| \le |f_{\infty}^{\epsilon}(\hat{w})-f_{\infty}(\hat{w})|+ \Psi(\epsilon, \tau(5\xi);n, L, \upsilon(B_R(x_{\infty}))) \le \Psi(\epsilon, \tau(5\xi);n, L, \upsilon(B_R(x_{\infty})))$. Thus, we have $|f_{\infty}^{\epsilon}-f_{\infty}| < \Psi(\epsilon, \tau(5\xi);n, L, \upsilon(B_R(x_{\infty})))$ on $B_R(x_{\infty})$.
Therefore, 
we have the assertion.
\end{proof}
As a corollary of Theorem \ref{app}, we shall give a sufficient condition to satisfy infinitesimal upper semicontinuity of energy:
\begin{proposition}\label{energy}
Let $R$ be a positive number, $f_i$ a $C^2$-function on $B_R(x_i) (i \in \mathbf{N})$, $f_{\infty}$ a Lipschitz function on $\overline{B}_R(x_{\infty})$. Assume that 
\[\sup_i \left(\mathbf{Lip}f_i  + \int_{B_R(x_i))}|\Delta f_i|d \underline{\mathrm{vol}} \right)< \infty\]
 and $f_i \rightarrow f_{\infty}$ on $B_R(x_{\infty})$.
Then, we have 
\[\limsup_{i \rightarrow \infty}\int_{B_R(x_i)}(\mathrm{Lip}f_i)^2d\underline{\mathrm{vol}} \le \int_{B_R(x_{\infty})}(\mathrm{Lip}f_{\infty})^2d\upsilon.\]
Especially, the sequence $\{f_i\}_i$ has infinitesimal upper semicontinuity of energy to $f_{\infty}$ at every $w \in B_R(x_{\infty})$.
\end{proposition}
\begin{proof}
We put $g_i = \Delta f_i$.
First, we shall remark  the following:
\begin{claim}\label{1f}
For every Lipschitz function $k$ on $B_R(x_i)$ satisfying $\mathrm{supp}k \subset B_R(x_i)$, we have 
\[\int_{B_R(x_i)}|d(f_i + k)|^2d\underline{\mathrm{vol}}-2\int_{B_R(x_i)}g_i(f_i+k)d\underline{\mathrm{vol}} \ge 
\int_{B_R(x_i)}|df_i|^2d\underline{\mathrm{vol}}-2\int_{B_R(x_i)}g_if_id\underline{\mathrm{vol}}.\]
\end{claim}
Because, since
\begin{align*}
\int_{B_R(x_i)}|d(f_i + k)|^2d\underline{\mathrm{vol}}-2\int_{B_R(x_i)}g_i(f_i+k)d\underline{\mathrm{vol}}&= 
\int_{B_R(x_i)}|df_i|^2d\underline{\mathrm{vol}}-2\int_{B_R(x_i)}g_if_id\underline{\mathrm{vol}} \\
&+ \int_{B_R(x_i)}|dk|^2d\underline{\mathrm{vol}},
\end{align*}
we have Claim \ref{1f}.
\

We fix $\epsilon > 0$ and take $L \ge 1$ satisfying 
\[\sup_i \left(\mathbf{Lip}f_i + |f_i|_{L^{\infty}}+
\int_{B_R(x_i)}|g_{i}|d\underline{\mathrm{vol}}\right)<L.\]
Since $\limsup_{i \rightarrow }A_{R-\epsilon, R}(x_i) \subset A_{R-\epsilon, R}(x_{\infty})$, by Theorem \ref{app}, there exist a $C(n, L)$-Lipschitz function $f_{\infty}^{\epsilon}$ on $B_R(x_{\infty})$,
a $C(n, L)$-Lipschitz function $f_i^{\epsilon}$ on $B_R(x_{i})$ and an open set $\Omega_{\epsilon} \subset B_R(x_{\infty}) \setminus A_{R-\epsilon, R}(x_{\infty})$ such that $f_{\infty}^{\epsilon}|_{A_{R-\epsilon, R}(x_{\infty})}= f_{\infty}|_{A_{R-\epsilon, R}(x_{\infty})}$, $f_{i}^{\epsilon}|_{A_{R-\epsilon, R}(x_{i})}= f_{i}|_{A_{R-\epsilon, R}(x_{i})}$, $(f_i^{\epsilon}, df_{i}^{\epsilon}) \rightarrow 
(f_{\infty}^{\epsilon}, df_{\infty}^{\epsilon})$ on $\Omega_{\epsilon}$ and that 
\begin{align*}
&\frac{\upsilon \left(B_R(x_{\infty}) \setminus (\Omega_{\epsilon} \cup A_{R-\epsilon, R}(x_{\infty}))\right)}{\upsilon(B_R(x_{\infty}))} + |f_{\infty}- f_{\infty}^{\epsilon}|_{L^{\infty}(B_R(x_{\infty}))} + \frac{1}{\upsilon_{\infty}(B_R(x_{\infty}))}\int_{B_R(x_{\infty})} |df_{\infty}^{\epsilon}-df_{\infty}|^2d\upsilon \\
&< \epsilon. 
\end{align*}
By Claim \ref{1f}, we have 
\[\int_{B_R(x_i)}|df_i^{\epsilon}|^2d\underline{\mathrm{vol}}-2\int_{B_R(x_i)}g_i f_i^{\epsilon}d\underline{\mathrm{vol}} \ge 
\int_{B_R(x_i)}|df_i|^2d\underline{\mathrm{vol}}-2\int_{B_R(x_i)}g_if_id\underline{\mathrm{vol}}.\]
By Proposition \ref{cov}, without loss of generality, we can assume that there exists a pairwise disjoint finite collection $\{\overline{B}_{r_i}(z_i)\}_{1 \le i \le N }$
such that $\Omega_{\epsilon}= \bigcup_{i=1}^NB_{r_i}(z_i)$.
We take a sequence $z_i(j) \rightarrow z_i$. We put $\Omega_{\epsilon}(j)=  \bigcup_{i=1}^NB_{r_i}(z_i(j))$.
Since $\underline{\mathrm{vol}} (\Omega_{\epsilon}(j) \cup A_{R-\epsilon, R}(x_{i}))/\underline{\mathrm{vol}}\,B_R(x_i) \ge 1- \epsilon$ for 
every sufficiently large $j$, by Proposition \ref{10103}, we have 
\[\left| \int_{B_R(x_j)}|df_j^{\epsilon}|^2d\underline{\mathrm{vol}}- \int_{B_R(x_{\infty})}|df_{\infty}|^2d\upsilon \right| 
< \Psi(\epsilon;n, L, R)\upsilon (B_R(x_{\infty})).\]
On the other hand, since $\sup _{B_R(x_j)}|f_j^{\epsilon}-f_j| \le C(n, R, L)\sup _{\Omega_{\epsilon}(j)}|f_j^{\epsilon}-f_j|$ and  $\limsup_{j \rightarrow \infty} \sup _{\Omega_{\epsilon}(j)}|f_j^{\epsilon}-f_j| \le \sup_{\Omega_{\epsilon}}|f_{\infty}^{\epsilon}-f_{\infty}|$, we have 
\[\left| \int_{B_R(x_j)}g_j f_j^{\epsilon}d\underline{\mathrm{vol}}-\int_{B_R(x_{j})}g_{j} f_{j}d\underline{\mathrm{vol}} \right|
\le \sup_{B_R(x_j)}|f_j^{\epsilon}-f_j|\int_{B_R(x_j)}|g_j|d\underline{\mathrm{vol}}
\le \Psi(\epsilon;n, R, L)\upsilon (B_R(x_{\infty}))\]
for every sufficiently large $j$.
Therefore, by Proposition \ref{10103}, we have 
\[\limsup_{i \rightarrow \infty} \int_{B_R(x_i)}|df_i|^2d\underline{\mathrm{vol}} 
\le \int_{B_R(x_{\infty})}|df_{\infty}^{\epsilon}|^2d\upsilon + \Psi(\epsilon;n, L, R)\upsilon (B_R(x_{\infty})).\]
Thus, we have 
\[\limsup_{i \rightarrow \infty} \int_{B_R(x_i)}|df_i|^2d\underline{\mathrm{vol}} \le \int_{B_R(x_{\infty})}|df_{\infty}|^2d\upsilon + \Psi(\epsilon;n, L, R)\upsilon (B_R(x_{\infty})).\]
By letting $\epsilon \rightarrow 0$, we have the assertion.
\end{proof}
Next corollary follows from Remark \ref{cc} and Proposition \ref{energy} directly:
\begin{corollary}\label{suffconv}
Let $R$ be a positive number, $f_i$ a $C^2$-function on $B_R(x_i)$ and $f_{\infty}$ Lipschitz functions on $B_R(x_{\infty})$.
Assume that 
\[\sup_i \left(\mathbf{Lip}f_i + \int_{B_R(x_i))}|\Delta f_i|^2d \underline{\mathrm{vol}} \right)< \infty\]
and
$f_i \rightarrow f_{\infty}$ on $B_R(x_{\infty})$.
Then, we have $(f_i, df_i) \rightarrow (f_{\infty}, df_{\infty})$ on $B_R(x_{\infty})$.
\end{corollary}
\begin{corollary}\label{convergence}
Let $R$ be a positive number, $f_i$ a $C^2$-function on $B_R(x_i)$ and $f_{\infty}$ a Lipschitz function on $B_R(x_{\infty})$ satisfying
$\sup_{i}(\mathbf{Lip}f_i +  |\Delta f_i|_{L^{\infty}(B_R(x_i))})< \infty$.
We assume that $f_i \rightarrow f_{\infty}$ on $B_R(x_{\infty})$
 and that there exists a $L^{\infty}$-function $g_{\infty}$ on $B_R(x_{\infty})$ such that  
$\{\Delta f_i\}_i$ has infinitesimal convergence property to $g_{\infty}$ at a.e. $w \in B_R(x_{\infty})$.
Then,  for every Lipschitz function $k_{\infty}$ satisfying $\mathrm{supp}k_{\infty} \subset B_R(x_{\infty})$, we have 
\[\int_{B_R(x_{\infty})}\langle df_{\infty}, dk_{\infty}\rangle d\upsilon = \int_{B_R(x_{\infty})}k_{\infty}g_{\infty}d\upsilon.\]
\end{corollary}
\begin{proof}
By Corollary \ref{suffconv}, we have
$(f_i, df_i) \rightarrow (f_{\infty}, df_{\infty})$ on $B_R(x_{\infty})$.
We take $L \ge 1$ satisfying $\sup_{i}(\mathbf{Lip}f_i + |f_i|_{L^{\infty}} + |\Delta f_i|_{L^{\infty}})<L$.
We put $r = \sup_{w \in \mathrm{supp}k_{\infty}}\overline{x_{\infty}, w}$ and $g_i = \Delta f_i$.
By compactness of $\mathrm{supp}k_{\infty}$, we have $r < R$.
We fix $\epsilon > 0$ satisfying $\epsilon < R-r$.
By Theorem \ref{app}, there exist a $C(n, L)$-Lipschitz function $k_{\infty}^{\epsilon}$ on $B_R(x_{\infty})$,
a $C(n, L)$-Lipschitz function $k_i^{\epsilon}$ on $B_R(x_{i})$ and an open set $\Omega_{\epsilon} \subset B_R(x_{\infty}) \setminus A_{R-\epsilon, R}(x_{\infty})$ such that $k_{\infty}^{\epsilon}|_{A_{R-\epsilon, R}(x_{\infty})}=0$, $k_{i}^{\epsilon}|_{A_{R-\epsilon, R}(x_{i})}= 0$, $(k_i^{\epsilon}, dk_{i}^{\epsilon}) \rightarrow 
(k_{\infty}^{\epsilon}, dk_{\infty}^{\epsilon})$ on $\Omega_{\epsilon}$ and that 
\begin{align*}
&\frac{\upsilon \left(B_R(x_{\infty}) \setminus (\Omega_{\epsilon} \cup A_{R-\epsilon, R}(x_{\infty}))\right)}{\upsilon(B_R(x_{\infty}))} + |k_{\infty}- k_{\infty}^{\epsilon}|_{L^{\infty}(B_R(x_{\infty}))} + \frac{1}{\upsilon_{\infty}(B_R(x_{\infty}))}\int_{B_R(x_{\infty})} |dk_{\infty}^{\epsilon}-dk_{\infty}|^2d\upsilon \\
&< \epsilon. 
\end{align*}
By Proposition \ref{101020}, $\{k_i^{\epsilon}g_i\}_i$ has infinitesimal convergence property to $k_{\infty}^{\epsilon}g_{\infty}$ at a.e. $w \in \Omega_{\epsilon}$.
By an argument similar to the proof of Proposition \ref{energy}, and Proposition \ref{10103}, we have
\begin{align*}
&\left| \int_{B_R(x_i)}\langle df_i, dk_i^{\epsilon}\rangle d\underline{\mathrm{vol}} - \int_{B_R(x_{\infty})}\langle df_{\infty}, dk_{\infty}^{\epsilon}\rangle d\upsilon \right| 
+ \left| \int_{B_R(x_i)}g_i k_i^{\epsilon}d\underline{\mathrm{vol}}-\int_{B_R(x_{\infty})}g_{\infty} k_{\infty}^{\epsilon}d\upsilon \right| \\
&< \Psi(\epsilon;n, L, R)\upsilon (B_R(x_{\infty}))
\end{align*}
for every sufficiently large $i$.
Since 
\[\int_{B_R(x_i)}\langle df_i, dk_i^{\epsilon}\rangle d\underline{\mathrm{vol}} = \int_{B_R(x_i)}g_i k_i^{\epsilon}d\underline{\mathrm{vol}},\]
we have 
\[\int_{B_R(x_{\infty})}\langle df_{\infty}, dk_{\infty}\rangle d\upsilon 
=\int_{B_R(x_{\infty})}g_{\infty} k_{\infty}d\upsilon \pm \Psi(\epsilon;n, L, R)\upsilon (B_R(x_{\infty})).
\]
By letting $\epsilon \rightarrow 0$, we have the assertion.
\end{proof}
The following corollary follows from Corollary \ref{suffconv} and \ref{convergence} directly.
See also \cite{di2}.
\begin{corollary}\label{har}
Let $R$ be a positive number, $f_i$ a harmonic function on $B_R(x_i)$ and $f_{\infty}$ a Lipschitz function on $B_R(x_{\infty})$ 
satisfying $\sup_{i} \mathbf{Lip}f_i< \infty$.
We assume that $f_i \rightarrow f_{\infty}$ on $B_R(x_{\infty})$.
Then, we have $(f_i, df_i) \rightarrow (f_{\infty}, df_{\infty})$ on $B_R(x_{\infty})$.
Moreover, for every Lipschitz function $k_{\infty}$ satisfying $\mathrm{supp}k_{\infty} \subset B_R(x_{\infty})$, we have 
\[\int_{B_R(x_{\infty})}\langle df_{\infty}, dk_{\infty} \rangle d\upsilon =0.\]
Especially $f_{\infty}$ is a harmonic function on $B_R(x_{\infty})$.
\end{corollary}
\section{Harmonic functions on asymptotic cones}
In this section, we will give several applications of results in section $4$ to harmonic functions on asymptotic cones of manifolds with nonnegative Ricci curvature and Euclidean volume growth via Colding-Minicozzi theory \cite{co-mi1, co-mi2, co-mi3, co-mi4, co-mi5, co-mi6}
for harmonic functions on manifolds.
Throughout this section, we will always assume that dimensions of all manifolds are greater than $2$.
\subsection{Convergence of frequency functions}
Throughout this section $5$, we fix an $n$-dimensional complete Riemannian manifolds $M$ satisfying $\mathrm{Ric}_{M}\ge 0$ and \textit{Euclidean volume growth condition}:
\[\lim_{R \rightarrow \infty}\frac{\mathrm{vol}^{g_M}\,B_R(m)}{R^n}>0.\]
Here $m$ is a point in $M$ and $g_M$ is the Riemannian metric of $M$.
We remark that by Bishop-Gromov volume comparison theorem, the limit above always exists and does not 
depend on choice of $m$.
We denote the limit by $V_M^{g_M}=\lim _{R \to \infty}\mathrm{vol}^{g_M}B_R(m)/R^n$.
It is easy to check that $V_M^{r^{-2}g_M}= V_M^{g_M}$ for $r>0$.
Therefore we shall use the notaiton: $V_M = V_M^{g_M}$.
We fix a point $m \in M$ below.
Then the global Green's function $G^{g_M}(m, x)$ on $M$ with singularity at $m$ exists. See \cite{sc-ya}.
First, we shall introduce an important result about asymptotic behavior of $G^{g_M}$ by 
Colding-Minicozzi:
\begin{theorem}[Colding-Minicozzi, \cite{co-mi4}]\label{green}
We have 
\[\lim_{\overline{m,x} \rightarrow \infty} \frac{G^{g_M}(m,x)}{\overline{m,x}^{2-n}}= \frac{\mathrm{vol}\,B_1(0_n)}{V_M}.\]
\end{theorem} 
By the definition of Green's function, we have
\[G^{r^{-2}g_M}(m, x)= \frac{G^{g_M}(m, x)}{r^{2-n}}.\]
It is known that there exists $C_1 > 1$ such that 
$\overline{m, x}^{2-n}\le G^{g_M}(m, x) \le C_1 \overline{m, x}^{2-n}$
for every $m \neq x$.
We define a smooth function $b_m^{g_M}$ on $M \setminus \{m\}$ by
\[b_m^{g_M}(x) = \left(\frac{V_M}{\mathrm{vol}\,B_1(0_n)}G^{g_M}(m, x)\right)^{\frac{1}{2-n}}.\]
Thus we have $b_m^{r^{-2}g_M}= b_m^{g_m}/r$.
We shall use the notation $b^{g_M} = b_m^{g_M}$ simply.
Then we have 
\[\left(\frac{V_M}{\mathrm{vol}\,B_1(0_n)}\right)^{2-n}\overline{m, y}^{r^{-2}g_M}\le
b^{r^{-2}g_M}(y)\le \left(\frac{C_1V_M}{\mathrm{vol}\,B_1(0_n)}\right)^{2-n}\overline{m, y}^{r^{-2}g_M}\]
for every $r > 0$.
We put $b^{g_M}(m) =0$.
It is easy to check 
\[\nabla^{g_M}b^{g_M}\ = \frac{V_M}{(2-n)\mathrm{vol}\,B_1(0_n)}(b^{g_M})^{n-1}\nabla^{g_M}G^{g_M}(m, \cdot).\]
On the other hand, for every $\epsilon >0$, there exists $R(\epsilon) > 0$ such that  
\[\int_{b^{g_M} \le R}||\nabla b^{g_M}|^2-1|^2d\mathrm{vol} \le \epsilon \mathrm{vol}(\{b^{g_M} \le R\}),\]
\[\int_{b^{g_M} \le R}|\mathrm{Hess}_{(b^{g_M})^2}-2g_M|^2d\mathrm{vol} \le \epsilon \mathrm{vol}(\{b^{g_M} \le R\})\]
for every $R > R(\epsilon)$
and that
\[\left|\frac{b^{g_M}(x)}{\overline{m, x}^{g_M}}-1\right| < \epsilon\]
for every $x \in M \setminus B_{R(\epsilon)}(m)$. 
See $(2. 23), (2. 24)$ and $(2. 25)$ in \cite{co-mi3} or section $4$ in \cite{co-mi4} for proofs of these results.
\begin{lemma}\label{001}
We have 
\[\lim_{R \rightarrow \infty}\frac{\mathrm{vol}(\{b^{g_M} \le R\})}{\mathrm{vol}\,B_R^{g_M}(m)}=1\]
\end{lemma}
\begin{proof}
For every $0 < \epsilon <1$, we take $R(\epsilon) > 0$ as above.
We put 
\[\hat{R}(\epsilon)= \left(\frac{C_1V_M}{\mathrm{vol}\,B_1(0_n)}\right)^{2-n}R(\epsilon)+R(\epsilon).\]
We take $R > \hat{R}(\epsilon)$.
First, we shall show  $B_R(m) \subset \{b^{g_M} \le (1 + \epsilon)R\}$.
We take $y \in B_R(m)$.
By the definition of $b^{g_M}$, if $y =m$, then $y \in \{b^{g_M} \le (1 + \epsilon)R\}$.
If $y \neq m$ and $\overline{m, y} \le R(\epsilon)$,
then we have
\[b^{g_M}(y)  \le \left(\frac{C_1V_M}{\mathrm{vol}\,B_1(0_n)}\right)^{2-n}\overline{m, y} 
\le\left(\frac{C_1V_M}{\mathrm{vol}\,B_1(0_n)}\right)^{2-n} R(\epsilon) 
\le \hat{R}(\epsilon) \le R.\]
Especially, we have  $y \in \{b^{g_M} \le (1 + \epsilon)R\}$.
On the other hand, by the definition of $R(\epsilon)$, if $\overline{m,y} > R(\epsilon)$,
then $|b^{g_M}(y)-\overline{m, y}| < \epsilon \overline{m,y}$.
Especially, we have $b^{g_M}(y) \le (1  + \epsilon)\overline{m,y} < (1+\epsilon)R$.
Thus, we have $B_R(m) \subset \{b^{g_M} \le (1 + \epsilon)R\}$.
Next, we shall show $\{b^{g_M} \le (1 + \epsilon)R\} \subset B_{\frac{1+\epsilon}{1-\epsilon}R}(m)$.
We take $x \in \{b^{g_M} \le (1 + \epsilon)R\}$ satisfying $\overline{m,x} \ge R(\epsilon)$.
Then, we have $(1-\epsilon)\overline{m,x} \le b^{g_M}(x) \le (1 + \epsilon)R$.
Thus, we have $\{b^{g_M} \le (1 + \epsilon)R\} \subset B_{\frac{1+\epsilon}{1-\epsilon}R}(m)$.
Therefore, we have $B_{\frac{R}{1+\epsilon}}(m) \subset \{b^{g_M} \le R\} \subset B_{\frac{R}{1-\epsilon}}(m)$ for every $R > 2\hat{R}(\epsilon)$. 
Since 
\[\lim_{R \rightarrow \infty}\frac{\mathrm{vol}\,B_{\frac{R}{1-\epsilon}}(m)}{\mathrm{vol}\,B_{\frac{R}{1+\epsilon}}(m)}=
\left(\frac{1+\epsilon}{1-\epsilon}\right)^n,\]
we have the assertion.
\end{proof}
We shall define frequency functions for harmonic functions on $M$.
For $R > 0$, $0 < r < R$ and a harmonic function $u$ on $\{b^{g_M}<R\}$, we put
\[I_u^{g_M}(r)=r^{1-n}\int_{b^{g_M}=r}u^2|\nabla^{g_M}b^{g_M}|d\mathrm{vol}^{g_M}_{n-1},\]
\[D_u^{g_M}(r)=r^{2-n}\int_{b^{g_M}\le r}|\nabla^{g_M}u|^2d\mathrm{vol}^{g_M}\]
and
\[F_u^{g_M}(r)=r^{3-n}\int_{b^{g_M}=r}\left|\frac{\partial u}{\partial n}\right|^2|\nabla b^{g_M}|d\mathrm{vol}^{g_M}_{n-1}.\]
Here $n$ is the unit outer vector of $\{b^{g_M}=r\}$, $\mathrm{vol}^{g_M}_{n-1}$ is the $(n-1)$-dimensional Hausdorff measure with respect to the Riemannian metric $g_M$.
Moreover, we put
\[U_u^{g_M}(r)= \frac{D_u^{g_M}(r)}{I_{u}^{g_M}(r)} \ \mathrm{if}\ I_{u}^{g_M}(r)\neq 0,\]
\[U_u^{g_M}(r)=0 \ \mathrm{if} \ I_{u}^{g_M}(r)=0\]
and call the function $U_u^{g_M}$ on $(0, R)$ \textit{frequency function for $u$}.
We remark that the critical set of $b^{g_M}$ has codimension two.
See \cite{cheng2}, \cite{har-sim} or \cite[Remark $2.11$]{co-mi3}.
By maximum principle on manifolds, $U_u^{g_M}(r)=0$ for some $0 < r < R $ if and only if $u$ is a constant function.
The following fundamental properties of functions above are given in \cite{co-mi3}: 
\[D_u^{g_M}(r) \le \left(\frac{r}{s}\right)^{2-n}D_u^{g_M}(s),\]
\[\frac{dI_u^{g_M}}{dr}=2\frac{D_u^{g_M}(r)}{r}, \]
\[I_u^{g_M}(s)=\exp \left(2\int_r^s\frac{U_{u}^{g_M}(t)}{t}dt\right)I_u^{g_M}(r)\]
for $r <s$ (see $(2. 10)$, $(2. 12)$, $(2. 13)$ and $(2. 14)$ in \cite{co-mi3}).
For every $\tau, r > 0$, $R > r\tau$ and harmonic function $u$ on $\{b^{g_M}<R\}$, 
we put $u_{\tau}= u/\tau$. Then we have $D_{u_{\tau}}^{\tau^{-2}g_M}(r) = {\tau}^{-2}D_u^{g_M}(r\tau),
I_{u_{\tau}}^{\tau^{-2}g_M}(r) = {\tau}^{-2}I_u^{g_M}(r\tau),
F_{u_{\tau}}^{\tau^{-2}g_M}(r) = \tau^{-2}F_u^{g_M}(r\tau)$ and 
$U_{u_{\tau}}^{\tau^{-2}g_M}(r) = U_u^{g_M}(r\tau).$
\

We shall recall the definition of \textit{asymptotic cone (or tangent cone at infinity) of $M$} by Cheeger-Colding:
\begin{definition}[Asymptotic cone]\label{asymp}
For pointed proper geodesic space $(M_{\infty}, m_{\infty})$, we say that
\textit{$(M_{\infty}, m_{\infty})$ is an asymptotic cone (or tangent cone at infinity) of $M$} if 
there exists a sequence $R_i \rightarrow \infty$ such that $(M, m, R_i^{-1}d_M) \rightarrow (M_{\infty}, m_{\infty})$.
\end{definition}
We fix an asymptotic cone $(M_{\infty}, m_{\infty})$ of $M$ and a sequence $R_i \rightarrow \infty$ satisfying 
$(M, m, R_i^{-1}d_M) \rightarrow (M_{\infty}, m_{\infty})$ in this subsection below.
We remark that by \cite[Theorem $5. 9$]{ch-co1}, we have $(M, m, R_i^{-1}d_M, \mathrm{vol}^{R_i^{-2}g_M}) \rightarrow (M_{\infty}, m_{\infty}, H^n)$.
We shall introduce an important result for asymptotic cones by Cheeger-Colding:
\begin{theorem}[Cheeger-Colding, \cite{ch-co}]\label{metric}
With same notation as above, there exists a compact geodesic space $X$ such that $\mathrm{diam}X \le \pi$ and $(M_{\infty}, m_{\infty})=(C(X), p)$.
\end{theorem}
See   \cite[Theorem $9. 79$]{ch1} or \cite{ch-co} for the proof.
We fix $X$ as in Theorem \ref{metric}.
For $R>0$, $0<r<R$ and Lipschitz function $u$ on $\overline{B}_R(p)$ satisfying that $u$ is harmonic on $B_R(p)$, we put 
\[I_u(r)=r^{1-n}\int_{\partial B_r(p)}u^2dH^{n-1}\]
and
\[D_u(r)=r^{2-n}\int_{B_r(p)}|du|^2dH^n.\]
Moreover, we put
\[U_u(r)=\frac{D_u(r)}{I_u(r)} \ \mathrm{if}\ I_u(r) \neq 0\]
and
\[U_u(r)=0 \ \mathrm{if}\ I_u(r)=0.\]
We also remark that by Proposition \ref{0005}, the function
\[F_u(r)=r^{3-n}\int_{\partial B_r(p)}\langle dr_p, du\rangle ^2dH^{n-1}.\]
is well defined for a.e. $r \in (0, R)$.
\begin{remark}\label{549760}
We remark the following:
Let $R$ be a positive number, $u_i$ a harmonic function on $B_{RR_i}^{g_M}(m)$.
Assume that $\sup_i|(u_i)_{R_i}|_{L^{\infty}(B_r^{R_i^{-2}g_M}(m))} < \infty$ for every $0<r<R$.
Then we have $\sup_i \mathbf{Lip}\left((u_i)_{R_i}|_{B_r^{R_i^{-2}g_M}(m)}\right)<  \infty$ for every $0<r<R$.
The proof is as follows.
We fix $\hat{r}$ satisfying $r < \hat{r} <R$.
Since $\overline{B}_r(p)$ is convex, it is not difficult to see that 
there exists $i_0$ such that for every $i \ge i_0$, $x_1(i), x_2(i) \in \overline{B}_r^{R_i^{-2}g_M}(m)$ and
geodesic $\gamma_i$ from $x_1(i)$ to $x_2(i)$, we have $\mathrm{Image}\gamma_i \subset \overline{B}_{\hat{r}}^{R_i^{-2}g_M}(m)$.
Therefore, by Cheng-Yau's gradient estimate, we have 
$\limsup_{i \rightarrow \infty} \mathbf{Lip}\left((u_i)_{R_i}|_{B_r^{R_i^{-2}g_M}(m)}\right)<  \infty$ for every $0<r<R$.
Thus we have the assertion.
\end{remark}
\begin{proposition}\label{002}
Let $R$ be a positive number, $u_i$ a harmonic function on $B_{RR_i}^{g_M}(m)$ and $u_{\infty}$ a Lipschitz function on $B_R(p)$.
We assume that $\sup_i|(u_i)_{R_i}|_{L^{\infty}(B_r^{R_i^{-2}g_M}(m))} < \infty$ and $(u_i)_{R_i} \rightarrow u_{\infty}$ on $B_r(p)$ for every $0<r<R$.
Then, for every $0<r<s<R$, we have 
\[\lim_{i \rightarrow \infty}\sup_{t \in [r,s]}\left|D_{(u_i)_{R_i}}^{R_i^{-2}g_M}(t)-D_{u_{\infty}}(t)\right| =0\]
and
\[\lim_{i \rightarrow \infty}\sup_{t \in [r,s]}\left|I_{(u_i)_{R_i}}^{R_i^{-2}g_M}(t)-I_{u_{\infty}}(t)\right| =0\]
\end{proposition}
\begin{proof}
We fix $0< \hat{r} <r < s < \hat{s} < R$.
We take $L \ge 1$ such that $|u_{\infty}|_{L^{\infty}(B_{\hat{s}}(x_{\infty}))} + \mathbf{Lip}u_{\infty} \le L.$
We fix $\epsilon >0$ satisfying $\epsilon <<\min \{\hat{r}, R-\hat{s}\}$.
Then, by the proof of Lemma \ref{001}, there exists $R_1(\epsilon)>1$ such that  
\[B_{(1-\epsilon^2)R}^{g_M}(m) \subset \{b^{g_M} \le R \} \subset B_{(1+\epsilon^2)R}^{g_M}(m) \]
and 
\[\int_{b^{g_M}\le R}\left||\nabla ^{g_M}b^{g_M}|^2-1\right|^2 \le \epsilon^8 \mathrm{vol}\{b^{g_M}\le R\}\]
for every $R > R_1(\epsilon)$.
Especially, by Cauchy-Schwartz inequality, we have
\[\int_{b^{g_M}\le R}\left||\nabla ^{g_M}b^{g_M}|^2-1\right| \le \epsilon^4 \mathrm{vol}\{b^{g_M}\le R\}\]
and
\[\int_{b^{g_M}\le R}\left||\nabla ^{g_M}b^{g_M}|-1\right| \le \epsilon^2 \mathrm{vol}\{b^{g_M}\le R\}.\]
For $0 < t <R$, we put 
\[F_i(t)=\int_{b^{R_i^{-2}g_M}\le t}(u_i)_{R_i}^2|\nabla ^{R_i^{-2}g_M}b^{R_i^{-2}g_M}|^2d\mathrm{vol}^{R_i^{-2}g_M}.\]
Then, we have 
\[\frac{dF_i}{dt}(t)=\int_{b^{R_i^{-2}g_M}=t}(u_i)_{R_i}^2|\nabla ^{R_i^{-2}g_M}b^{R_i^{-2}g_M}|d\mathrm{vol}_{n-1}^{R_i^{-2}g_M}=
I_{(u_i)_{R_i}}^{R_i^{-2}g_M}(t)t^{n-1}.\]
Thus, we have
\begin{align*}
\frac{d^2F_i}{dt^2}(t)&= 2t^{n-1}\frac{D_{(u_i)_{R_i}}^{R_i^{-2}g_M}(t)}{t} + (n-1)I_{(u_i)_{R_i}}^{R_i^{-2}g_M}(t)t^{n-2}\\
&=2\int_{b^{R_i^{-2}g_{M}}\le t}|\nabla^{R_i^{-2}g_M}(u_i)_{R_i}|^2d\mathrm{vol}^{R_i^{-2}g_M} + \frac{(n-1)}{t}
\int_{b^{R_i^{-2}g_M}=t}(u_i)_{R_i}^2|\nabla ^{R_i^{-2}g_M}b^{R_i^{-2}g_M}|^2d\mathrm{vol}^{R_i^{-2}g_M}_{n-1}.
\end{align*}
On the other hand, in general, for every $C^2$-function $f$ on $\mathbf{R}$, we have
\[f(t)=f(a)+(t-a)f'(a)-\int_a^t(s-t)f''(s)ds\]
for every $a, s, t \in \mathbf{R}$.
Therefore, for every $0 < t < R$, we have
\begin{align*}
&\left|\frac{F_i(t+\epsilon)-F_i(t)}{\epsilon}-\int_{b^{R_i^{-2}g_M}=t}(u_i)_{R_i}^2|\nabla ^{R_i^{-2}g_M}b^{R_i^{-2}g_M}|d\mathrm{vol}^{R_i^{-2}g_M}\right| \\
&\le \int_t^{t+\epsilon}2\int_{b^{R_i^{-2}g_M}\le a}|\nabla ^{R_i^{-2}g_M}(u_i)_{R_i}|^2d\mathrm{vol}^{R_i^{-2}g_M}da \\
& \ \ \ +(n-1)\int_t^{t+\epsilon}a^{-1}\int_{b^{R_i^{-2}g_M}=a}(u_i)_{R_i}^2 |\nabla ^{R_i^{-2}g_M}b^{R_i^{-2}g_M}|d\mathrm{vol}^{R_i^{-2}g_M}da \\
&\le 2\epsilon \int_{b^{R_i^{-2}g_M}\le t+\epsilon}|\nabla ^{R_i^{-2}g_M}(u_i)_{R_i}|^2d\mathrm{vol}^{R_i^{-2}g_M}
+\frac{n-1}{t}\int_{t\le b^{g_M} \le t+\epsilon}(u_i)_{R_i}^2|\nabla ^{R_i^{-2}g_M}b^{R_i^{-2}g_M}|^2d\mathrm{vol}^{R_i^{-2}g_M}.
\end{align*}
By Proposition \ref{99990}, there exists $i_0 \in \mathbf{N}$ such that $R_i\hat{r} \ge 10R_1(\epsilon)$,
$|(u_i)_{R_i}|_{L^{\infty}(B_{\hat{s}}^{R_i^{-2}g_M}(m))}\le 10L$ and 
\[\sup_{a \in [0, R]}\left|\mathrm{vol}^{R_i^{-2}g_M}\,B_a^{R_i^{-2}g_M}(m)-H^n(B_a(p))\right|< \epsilon^2\]
for every $i \ge i_0$.
Then, by Cheng-Yau's gradient estimate, for every $i \ge i_0$ and $r < t < s$, 
\begin{align}
\int_{b^{R_i^{-2}g_M} \le t + \epsilon}|\nabla ^{R_i^{-2}g_M}(u_i)_{R_i}|^2d\mathrm{vol}^{R_i^{-2}g_M} &\le 
\int_{B_{(1+\epsilon)(t+\epsilon)}^{R_i^{-2}g_M}(m)}|\nabla ^{R_i^{-2}g_M}(u_i)_{R_i}|^2d\mathrm{vol}^{R_i^{-2}g_M} \\
&\le C(n, L, R).
\end{align}
Here, we used $H^n(B_R(p))=R^nH^n(B_1(p))\le R^nC(n)$.
Moreover, we have
\begin{align}
&\int_{t \le b^{R_i^{-2}g_M} \le t + \epsilon}(u_i)_{R_i}^2|\nabla ^{R_i^{-2}g_M}b^{R_i^{-2}g_M}|^2d\mathrm{vol}^{R_i^{-2}g_M}\\
&\le \int_{t \le b^{R_i^{-2}g_M} \le t + \epsilon}(u_i)_{R_i}^2d\mathrm{vol}^{R_i^{-2}g_M} + \int_{t \le b^{R_i^{-2}g_M} \le t + \epsilon}(u_i)_{R_i}^2\left||\nabla ^{R_i^{-2}g_M}b^{R_i^{-2}g_M}|^2-1\right|d\mathrm{vol}^{R_i^{-2}g_M}\\
&\le \int_{t \le b^{R_i^{-2}g_M} \le t + \epsilon}(u_i)_{R_i}^2d\mathrm{vol}^{R_i^{-2}g_M} + 100L^2\mathrm{vol}^{R_i^{-2}g_M}\{t \le b^{R_i^{-2}g_M} \le t + \epsilon \} \\
&\le 200L^2 \mathrm{vol}^{R_i^{-2}g_M}\{t \le b^{R_i^{-2}g_M} \le t + \epsilon \} \\
&\le 200L^2\mathrm{vol}^{R_i^{-2}g_M}\,A_m^{R_i^{-2}g_M}\left((1-\epsilon^2)t, (1+\epsilon^2)(t+\epsilon)\right)\\
&\le 200L^2H^n\left(A_p\left((1-\epsilon^2)t,  (1+\epsilon^2)(t+\epsilon)\right)\right)+ 300L^2\epsilon^2.
\end{align}
On the other hand, we have
\begin{align}
\frac{F_i(t+\epsilon)-F_i(t)}{\epsilon}&=\frac{1}{\epsilon}\int_{t \le b^{R_i^{-2}g_M} \le t + \epsilon}(u_i)_{R_i}^2d\mathrm{vol}^{R_i^{-2}g_M} \\
&\  \ \ \pm \frac{1}{\epsilon}\int_{t \le b^{R_i^{-2}g_M}\le t + \epsilon}(u_i)_{R_i}^2\left||\nabla ^{R_i^{-2}g_M}b^{R_i^{-2}g_M}|^2-1\right|d\mathrm{vol}^{R_i^{-2}g_M},
\end{align}
and
\begin{align}
&\frac{1}{\epsilon}\int_{t \le b^{R_i^{-2}g_M} \le t + \epsilon}(u_i)_{R_i}^2\left||\nabla ^{R_i^{-2}g_M}b^{R_i^{-2}g_M}|^2-1\right|d\mathrm{vol}^{R_i^{-2}g_M}\\
&\le\frac{100L^2}{\epsilon}\int_{b^{R_i^{-2}g_M} \le t + \epsilon}
\left||\nabla ^{R_i^{-2}g_M}b^{R_i^{-2}g_M}|^2-1\right|d\mathrm{vol}^{R_i^{-2}g_M}\\
&\le \frac{100L^2}{\epsilon}\epsilon ^2 \mathrm{vol}^{R_i^{-2}g_M}\left(\{b^{R_i^{-2}g_M} \le t + \epsilon \}\right)\\
&\le 100L^2 \epsilon \frac{\mathrm{vol}^{g_M}B_{(1+\epsilon^2)(t+\epsilon)R_i}^{g_M}(m)}{R_i^n} \\
&\le \epsilon C(n, L, R).
\end{align}
We remark that 
\begin{align}
&\left|\int_{t \le b^{R_i^{-2}g_M} \le t + \epsilon}(u_i)_{R_i}^2d\mathrm{vol}^{R_i^{-2}g_M}-\int_{A_m^{R_i^{-2}g_M}(t, t+\epsilon)}
(u_i)_{R_i}^2d\mathrm{vol}^{R_i^{-2}g_M}\right|\\
&\le 100L^2 \mathrm{vol}^{R_i^{-2}g_M}\left(\{t \le b^{R_i^{-2}g_M} \le t + \epsilon\} \triangle 
A_m^{R_i^{-2}g_M}(t, t+\epsilon)\right).
\end{align}
Here $A \triangle B= (A \setminus B) \cup (B \setminus A)$.
\begin{claim}\label{003}
We have 
\begin{align}
&\{t \le b^{R_i^{-2}g_M} \le t + \epsilon\} \triangle 
A_m^{R_i^{-2}g_M}(t, t+\epsilon) \\
&\subset A_m^{R_i^{-2}g_M}\left((1-\epsilon^2)(t+\epsilon), (1 + \epsilon^2)(t+\epsilon)\right)
\cup A_m^{R_i^{-2}g_M}\left((1-\epsilon^2)t, (1 + \epsilon^2)t\right)
\end{align}
for every $i \ge i_0$ and $r < t < s$.
\end{claim}
The proof is as follows.
We put $A_i^{\epsilon}(t) = \{t \le b^{R_i^{-2}g_M} \le t + \epsilon\} \triangle 
A_m^{R_i^{-2}g_M}(t, t+\epsilon)$.
First, we take $y \in \{t \le b^{R_i^{-2}g_M} \le t + \epsilon/2\} \cap A_i^{\epsilon}(t)$.
Then we have 
$y \in B_{(1+\epsilon^2)(t+\epsilon/2)}^{R_i^{-2}g_M}(m)$.
Especially, we have 
\[\overline{m,y}^{R_i^{-2}g_M}\le (1+ \epsilon^2)(t+\frac{\epsilon}{2}) < t + \epsilon.\]
Since $y \in M \setminus A_m^{R_i^{-2}g_M}(t, t+\epsilon)$, we have $y \in B_t^{R_i^{-2}g_M}(m)$.
Thus, we have $\{t \le b^{R_i^{-2}g_M} \le t + \epsilon/2 \} \cap A_i^{\epsilon}(t) \subset B_t^{R_i^{-2}g_M}(m) \setminus B_{(1-\epsilon^2)t}^{R_i^{-2}g_M}(m)$.
Similarly, we have $\{t+\epsilon/2 \le b^{R_i^{-2}g_M} \le t + \epsilon\} \cap A_i^{\epsilon}(t) \subset 
B_{(1+\epsilon^2)(t+\epsilon)}^{R_i^{-2}g_M}(m) \setminus B_{t+\epsilon}^{R_i^{-2}g_M}(m)$.
Therefore, we have
\[\{t \le b^{R_i^{-2}g_M} \le t + \epsilon \} \cap A_i^{\epsilon}(t) \subset A_m^{R_i^{-2}g_M}((1-\epsilon^2)t, t) \cup 
A_m^{R_i^{-2}g_M}(t+\epsilon, (1+\epsilon^2)(t+\epsilon)).\]
Next, we take $x \in A_i^{\epsilon}(t) \cap A_m^{R_i^{-2}g_M}(t, t+\epsilon/2)$.
Then we have 
\[b^{R_i^{-2}g_M}(x) \le (1 + \epsilon^2)\overline{m,x}^{R_i^{-2}g_M} \le 
(1+\epsilon^2)(t + \epsilon/2)< t+\epsilon.\]
Since $x \in M \setminus \{t \le b^{R_i^{-2}g_M} \le t+\epsilon\}$, we have $b^{R_i^{-2}g_M}(x) < t$. Therefore, we have
$x \in B_{(1+\epsilon^2)t}^{R_i^{-2}g_M}(m)$. 
Thus, we have $A_m^{R_i^{-2}g_M}(t, t+\epsilon/2) \cap A_i^{\epsilon}(t) \subset A_m^{R_i^{-2}g_M}(t, (1 + \epsilon^2)t)$.
Similarly, we have $A_m^{R_i^{-2}g_M}(t+ \epsilon/2, t+\epsilon) \cap A_i^{\epsilon}(t) \subset 
A_m^{R_i^{-2}g_M}(t+\epsilon, (1+\epsilon^2)(t+\epsilon))$. 
Therefore we have Claim \ref{003}. 
\

By Claim \ref{003} and Bishop-Gromov volume comparison theorem, we have 
\begin{align}
&\epsilon^{-1}\mathrm{vol}^{R_i^{-2}g_M} \left(\{t \le b^{R_i^{-2}g_M} \le t + \epsilon\} \triangle 
A_m^{R_i^{-2}g_M}(t, t+\epsilon)\right) \\
& \le \epsilon^{-1} \mathrm{vol}^{R_i^{-2}g_M}\,\left(A_m^{R_i^{-2}g_M}\left((1-\epsilon^2)(t+\epsilon), (1 + \epsilon^2)(t+\epsilon)\right)\right)\\
& \ \ \ \ \ \ + \epsilon^{-1} \mathrm{vol}^{R_i^{-2}g_M}\,\left(A_m^{R_i^{-2}g_M}\left((1-\epsilon^2)t, (1+\epsilon^2)t)\right)\right) \\
& \le 3\epsilon^{-1}\epsilon^2 \mathrm{vol}^{R_i^{-2}g_M}_{n-1}\left(\partial B_{(1-\epsilon^2)(t+\epsilon)}^{R_i^{-2}g_M}(m) \setminus C_m\right)
+ 3\epsilon^{-1}\epsilon^2 \mathrm{vol}^{R_i^{-2}g_M}_{n-1}\left(\partial B_{(1-\epsilon^2)t}^{R_i^{-2}g_M}(m)\setminus C_m\right) \\
& \le 6\epsilon \mathrm{vol}\,\partial B_R(0_n).
\end{align}
Therefore we have 
\[\left|\int_{t \le b^{R_i^{-2}g_M} \le t + \epsilon}(u_i)_{R_i}^2d\mathrm{vol}^{R_i^{-2}g_M}-\int_{A_m^{R_i^{-2}g_M}(t, t+\epsilon)}
(u_i)_{R_i}^2d\mathrm{vol}^{R_i^{-2}g_M}\right|\le 600L^2 \epsilon \mathrm{vol}\,\partial B_R(0_n).\]
for every $i \ge i_0$ and $r < t < s$.
We take the canonical retraction, $\pi_t$ from $C(X)$ to $\overline{B}_t(p)$ for every $t >0$.
It is easy to check that $\pi_t$ is $1$-Lipschitz map. 
We put $u_{\infty}^t= (u_{\infty})^2 \circ \pi_t$.
We have $\mathbf{Lip}u_{\infty}^t \le \mathbf{Lip}(u_{\infty})^2$.
By Proposition \ref{0005}, we have 
\begin{align}
&\left|\int_t^{t+\epsilon}\int_{\partial B_a(p)}(u_{\infty})^2dH^{n-1}da-\int_t^{t + \epsilon}\int_{\partial B_a(p)}u_{\infty}^tdH^{n-1}da\right|\\
&\le \int_{A_p(t, t+\epsilon)}|(u_{\infty})^2-u_{\infty}^t|dH^n \\
&\le \mathbf{Lip}(u_{\infty})^2 \epsilon H^n(A_p(t, t+\epsilon)).
\end{align}
for every $r < t < s$.
On the other hand, 
\begin{align}
\int_t^{t+\epsilon}\int_{\partial B_a(p)}u_{\infty}^tdH^{n-1}da &=
\int_t^{t+\epsilon}\left(\frac{a}{t}\right)^{n-1}\int_{\partial B_t(p)}(u_{\infty})^2dH^{n-1}da \\
&=\int_{\partial B_t(p)}(u_{\infty})^2dH^{n-1}\int_t^{t+\epsilon}\left(\frac{a}{t}\right)^{n-1}da \\
&=I_{u_{\infty}}(t)\frac{(t+\epsilon)^n-t^n}{n} \\
&=I_{u_{\infty}}(t)(\epsilon t^{n-1} \pm \Psi(\epsilon; n, R)\epsilon).
\end{align}
Therefore we have 
\[\lim_{i \rightarrow \infty}\sup_{t \in [r,s]}\left|I_{(u_i)_{R_i}}^{R_i^{-2}g_M}(t)-I_{u_{\infty}}(t)\right| =0.\]
Next, we shall prove 
\[\lim_{i \rightarrow \infty}\sup_{t \in [r,s]}\left|D_{(u_i)_{R_i}}^{R_i^{-2}g_M}(t)-D_{u_{\infty}}(t)\right| =0.\]
We shall use same notations as above.
It is clear that 
\begin{align}
t^{2-n}\int_{B_{(1-\epsilon^2)t}^{R_i^{-2}g_M}(m)}|\nabla ^{R_i^{-2}g_M}(u_i)_{R_i}|^2d\mathrm{vol}^{R_i^{-2}g_M}
&\le D_{(u_i)_{R_i}}^{R_i^{-2}g_M}(t)\\
&\le t^{2-n}\int_{B_{(1+\epsilon^2)t}^{R_i^{-2}g_M}(m)}|\nabla ^{R_i^{-2}g_M}(u_i)_{R_i}|^2d\mathrm{vol}^{R_i^{-2}g_M}
\end{align}
for every $i \ge i_1$ and $r<t<s$.
On the other hand, we have
\begin{align}
&\int_{A_m^{R_i^{-2}g_M}\left((1-\epsilon^2)t, (1+\epsilon^2)t\right)}|\nabla ^{R_i^{-2}g_M}(u_i)_{R_i}|^2d\mathrm{vol}^{R_i^{-2}g_M}(m)\\
&\le C(n, L, R) \mathrm{vol}^{R_i^{-2}g_M}\,A_m^{R_i^{-2}g_M}((1-\epsilon^2)t, (1+\epsilon^2)t)) \\
&\le C(n, L, R) \left(H^n\left(A_p((1-\epsilon^2)t, (1+\epsilon^2)t)\right)+ \epsilon \right).
\end{align}
Therefore, by Theorem \ref{har}, we have the assertion.
\end{proof}
For every $0 < r < R$ and harmonic function $u$ on $\{b^{g_M}<R\}$, we put
\[E_u^{g_M}(r) = r^{2-n}\int_{b^{g_M}\le r}|\nabla ^{g_M}u|^2|\nabla ^{g_M}b^{g_M}|^2d\mathrm{vol}^{g_M}.\]
It is easy to check that for every $\tau, r, R > 0$ satisfying $R > r \tau$ and a harmonic function $u$ on $\{b^{g_M} < R\}$, we have 
$E_{u_{\tau}}^{\tau^{-2}g_M}(r)=\tau^{-2}E_u^{g_M}(\tau r)$.
By an argument similar to the proof of Proposition \ref{002} (or \cite[Proposition $3. 3$]{co-mi3}), we have the following:
\begin{proposition}\label{004}
With same assumption as in Lemma \ref{002}, we have
\[\lim_{i \rightarrow \infty}\sup_{t \in [r,s]}\left|E_{(u_i)_{R_i}}^{R_i^{-2}g_M}(t)-D_{u_{\infty}}(t)\right| =0\]
for every $0 < r < s <R$. 
\end{proposition} 
We shall introduce an important result \cite[Theorem $2.1$]{di1} by Ding:
\begin{theorem}[Ding, \cite{di1}]\label{005}
For every $0 < r < R$, all harmonic functions on $B_R(p)$ are Lipschitz on $B_r(p)$.
Moreover, for every $0 < r < s < R$ and harmonic function $v$ on $B_R(p)$, there exist
a subsequence $\{n(i)\}_i$ of $\mathbf{N}$ and
a sequence of harmonic functions $v_{n(i)}$
on $B_{s}^{R_{n(i)}^{-2}g_M}(m)$ such that $v_{n(i)} \rightarrow u_{\infty}$ on $B_r(x_{\infty})$.
\end{theorem}
\begin{proof}
We shall give an outline of the proof only.
First, we shall show that $u_{\infty}$ is Lipschitz function.
By  \cite[Proposition $5.1$]{KRS}, for every $u \in H_{1,2}(M_i)$ and $R>0$, we have 
\begin{align}
\int_M u(y)^2H^{R^{-2}g_M}(t, y, x)d\mathrm{vol}^{R^{-2}g_M}_y &\le 2t \int_M|d^{R^{-2}g_M}u|^2d\mathrm{vol}^{R^{-2}g_M}_y \\
& \ \ \ \ + \left(\int_{M}u(y)H^{R^{-2}g_M}(t, y, x)d\mathrm{vol}^{R^{-2}g_M}_y\right)^2
\end{align}
for a.e. $x \in M$.
Here. $H^{R^{-2}g_M}(t, y, x)$ is the heat kernel for rescaled manifold $(M, R^{-2}g_M)$.
By \cite[Theorem $5. 54$]{di2} and \cite[Lemma $10.3$]{ch2} (or Theorem \ref{app}), for  every $u \in \mathcal{K}(C(X))$, we have,
\[\int_{C(X)} u(y)^2H_{\infty}(t, y, x)dH^n(y) \le 2t \int_{C(X)}|du|^2dH^n(y) +
\left(\int_{C(X)}u(y)H_{\infty}(t, y, x)dH^n(y)\right)^2\]
for a.e. $x \in C(X)$.
Here $H_{\infty}$ is as in \cite[Theorem $5. 54$]{di2}.
Since $\mathcal{K}(C(X))$ is dense in $H_{1,2}(C(X))$, the inequality above holds for every $u \in H_{1, 2}(C(X))$.
Next, we fix $x \in X$ and $0< t < R$. Then, by Bishop-Gromov volume comparison theorem, it is easy to check that 
$H^n(B_t((1,x))) \ge C(n, V_M)t^n.$
For every $R > 0$, we define the map $\phi_R$ from $A_p(R-t, R+t)$ to $A_p(1-\frac{t}{R}, 1+\frac{t}{R})$ by
$\phi_R((\hat{t}, x)) = (\hat{t}/R, x)$.
Since  
$H^n(\phi_R(A))=R^{n}H^n(A)$
for every Borel subset $A \subset A_p(R-t, R+t)$, we have
\[H^n(B_t(R, x)) = R^{n}H^n(B_{\frac{t}{R}}(1, x)) \ge C(n, V_M)t^n.\]
Therefore, $(C(X), H^n)$ is Ahlfors $n$-regular metric measure space (see section $1$ in \cite{KRS}).
By \cite[Theorem $6.1$]{di2}, \cite[Theorem $6.20$]{di2} and \cite[Theorem $1.1$]{KRS}, $u_{\infty}$ is locally Lipschitz function on $B_R(p)$.
By convexity of $B_s(p)$ and the proof of \cite[Theorem $1.1$]{KRS}, $u_{\infty}$ is Lipschitz on $B_s(p)$.
Next, we shall take $L \ge 1$ satisfying $\mathbf{Lip}(u_{\infty}|_{B_s(p)}) + |u_{\infty}|_{L^{\infty}(B_s(p))} \le L$.
Without loss of generality, we can assume that there exists a sequence of Lipschitz functions $f_i$ on $\overline{B}_s^{R_i^{-2}g_M}(m)$ such that  $\mathbf{Lip}f_i + |f_i|_{L^{\infty}(B_s(p))} \le 10 L$ and $f_i \rightarrow u_{\infty}$ on $B_s(p)$.
We take a harmonic function $u_i$ on $B_s^{R_i^{-2}g_M}(m)$ such that 
\[u_i|_{\partial B_s^{R_i^{-2}g_M}(m)} = f_i|_{\partial B_s^{R_i^{-2}g_M}(m)}\]
in the sense of Perron's method for $f_i$.
We shall give a short review of Perron's method of subharmonic functions in this setting below.
See for instance section $2. 8$ in \cite{Gi}.
For $f \in C^0(B_s^{R_i^{-2}g_M}(m))$, we say that $f$ is subharmonic (superharmonic) in $B_s^{R_i^{-2}g_M}(m)$ if 
for every $w \in B_s^{R_i^{-2}g_M}(m)$, $r_1 > 0$ with $\overline{B}_{r_1}^{R_i^{-2}g_M}(w) \subset B_s^{R_i^{-2}g_M}(m)$,
and $h \in C^0(\overline{B}_{r_1}^{R_i^{-2}g_M}(w))$ satisfying $h|_{B_{r_1}^{R_i^{-2}g_M}(w)}$ is harmonic and $h|_{\partial B_{r_1}^{R_i^{-2}g_M}(w)}\le (\ge) f|_{\partial B_{r_1}^{R_i^{-2}g_M}(w)}$, we also have $h \le (\ge) f$ on $B_{r_1}^{R_i^{-2}g_M}(w)$.
For $g \in C^0(\overline{B}_s^{R_i^{-2}g_M}(m))$,
we say that $g$ is a subfunction relative to $f_i|_{B_s^{R_i^{-2}g_M}(m)}$ if $g|_{B_s^{R_i^{-2}g_M}(m)}$ is a subharmonic function and $g|_{\partial B_s^{R_i^{-2}g_M}(m)}
\le f_i|_{\partial B_s^{R_i^{-2}g_M}(m)}$.
We also say that $g$ is a superfunction relative to $f_i|_{B_s^{R_i^{-2}g_M}(m)}$ if $g|_{B_s^{R_i^{-2}g_M}(m)}$ is a superharmonic function and $g|_{\partial B_s^{R_i^{-2}g_M}(m)}
\ge f_i|_{\partial B_s^{R_i^{-2}g_M}(m)}$.
Let $S_{f_i}$ denote the set of subfunctions relative to $f_i|_{B_s^{R_i^{-2}g_M}(m)}$.
Then we put a function $u_i$ on $B_s^{R_i^{-2}g_M}(m)$ by
\[u_i(w)= \sup_{v \in S_{f_i}}v(w).\]
By an argument similar to the proof of  \cite[Theorem $2. 12$]{Gi}, it is easy to check that $u_i$ is harmonic on $B_s^{R_i^{-2}g_M}(m)$.

We fix $0< \tau < 3R$, $x \in \partial B_s(p)$ and $z \in \partial B_{2s}(p)$ satisfying $\overline{p,x}+
\overline{x,z}=\overline{p,z}$.
We take sequences $x(i) \in \partial B_s^{R_i^{-2}g_M}(m)$ and $z(i) \in \partial B_{2s}^{R_i^{-2}g_M}(m)$ such that
$x(i) \rightarrow x$ and $z(i) \rightarrow z$.
Then it is easy to check that for every $\alpha \in B_s(p)$, we have 
\[C_1(n, R)\overline{x, \alpha}^2 \le \overline{z, \alpha}-\overline{z, x} \le \overline{x, \alpha}.\]
We fix $\alpha \in B_r(p)$ and take a sequence of points $\alpha(i) \in  B_s^{R_i^{-2}g_M}(m)$ satisfying $\alpha(i) \rightarrow \alpha$.
We put $b^i= (r_{z(i)}^{R_i^{-2}g_M})^{2-n}-(r_{z(i)}^{R_i^{-2}g_M})^{2-n}(x(i))$ on $B_s^{R_i^{-2}g_M}(m)$.
By Laplacian comparison theorem on manifolds (or $(4.11)$ in \cite{ch1}), we have, a function $b^i$ is a superharmonic, a function $f_i(x(i)) + 100L\tau + C(n, L, R)b^i/\tau^2$ is a superfunction relative to $f_i|_{\partial B_s^{R_i^{-2}g_M}(m)}$ and a function $f_i(x(i)) - 100L\tau - C(n, L, R)b^i/\tau^2$ is a subfunction relative to $f_i|_{\partial B_s^{R_i^{-2}g_M}(m)}$ for every sufficiently large $i$.
By an argument similar to the proof of \cite[Lemma $2. 13$]{Gi}, we have 
\[|f_i(x(i))-u_i(\alpha(i))| \le C(n, R, L)\tau + \frac{C(n, R, L)}{\tau^2}\overline{x(i), \alpha(i)}^{R_i^{-2}g_M}\]
for  every sufficiently large $i$.
On the other hand, by Proposition \ref{Lips} and Corollary \ref{har}, we can assume that there exists a harmonic function $\hat{u}_{\infty}$ on
$B_s(p)$ such that $\hat{u}_{\infty}|_{B_{\hat{s}}(p)}$ is a Lipschitz function, $u_i \rightarrow u_{\infty}$ on $B_{\hat{s}}(p)$
for every $0 < \hat{s} <s$.
Thus we have  
\[|u_{\infty}(x)-\hat{u}_{\infty}(\alpha)| \le C(n, R, L)\tau + \frac{C(n, R, L)}{\tau^2}\overline{x, \alpha}\]
for every $\alpha \in B_s(p)$.
If we put $\tau = \overline{x, \alpha}^{1/3}$, then we have 
\[|u_{\infty}(x)-\hat{u}_{\infty}(\alpha)| \le C(n, R, L)\overline{x, \alpha}^{\frac{1}{3}}.\]
for every $x \in \partial B_s(p)$ and $\alpha \in B_s(p)$.
Since $\hat{u}_{\infty} \in H_{1,2}(B_{\hat{s}}(p))$ for every $0 < \hat{s} < s$, and $u_{\infty}$ is Lipschitz on $\overline{B}_s(p)$, by \cite[Cororally $6.6$]{sh} and an estimate above, we have $\sup_{B_s(p)}|u_{\infty}-\hat{u}_{\infty}|= \lim_{\hat{s} \to s}\left(\sup_{B_{\hat{s}}(p)}|u_{\infty}-\hat{u}_{\infty}|\right)=0$. 
Therefore, we have the assertion.
\end{proof}
We shall remark that the following:
\begin{corollary}\label{hahahaha}
Let $R$ be a positive number and $u_{\infty}, v_{\infty}$ harmonic functions on $B_R(p)$.
Then $u_{\infty}+v_{\infty}$ is a harmonic function on $B_R(p)$.
\end{corollary}
From now on, we shall replace most of many important statements in \cite{co-mi3} with statements on asymptotic cones:
\begin{proposition}\label{006}
For every $0 < r < s <R$ and harmonic function $u_{\infty}$ on $B_{R}(p)$, we have 
\[D_{u_{\infty}}(r)\le \left(\frac{r}{s}\right)^{2-n}D_{u_{\infty}}(s),\]
\[I_{u_{\infty}}(s)-I_{u_{\infty}}(r)=\int_r^s2\frac{D_{u_{\infty}}(t)}{t}dt.\]
Moreover, if $I_{u_{\infty}}(r)>0$, then we have
\[I_{u_{\infty}}(s)=\exp \left(2 \int_r^s\frac{U_{u_{\infty}}(t)}{t}dt\right)I_{u_{\infty}}(r).\]
\end{proposition}
\begin{proof}
By Theorem \ref{005}, without loss of generality, we can assume that the assumption of Proposition \ref{002} holds.
Since
\[D_{(u_i)_{R_i}}^{R_i^{-2}g_M}(r) \le \left(\frac{r}{s}\right)^{2-n}D_{(u_i)_{R_i}}^{R_i^{-2}g_M}(s),\]
by letting $i \rightarrow \infty$, we have the first assertion.
Similarly, since
\[I_{(u_i)_{R_i}}^{R_i^{-2}g_M}(s)-I_{(u_i)_{R_i}}^{R_i^{-2}g_M}(r)=\int_r^s2\frac{D_{(u_i)_{R_i}}^{R_i^{-2}g_M}(t)}{t}dt,\]
by letting $i \rightarrow \infty$ and dominated convergence theorem, we have the second assertion.
Especially, we remark that $I_{u_{\infty}}$ is a continuous function and that
a monotonicity $I_{u_{\infty}}(r) \le I_{u_{\infty}}(s)$ holds.
We shall prove the third assertion.
By Proposition \ref{002} and the monotonicity of $I_{u_{\infty}}$,
we have $\liminf_{i \to \infty}\left( \inf_{\alpha \in [r, s]}I_{(u_i)_{R_i}}^{R_i^{-2}g_M}(\alpha)\right)>0$. 
Therefore, by Cheng-Yau's gradient estimate, we have 
\[\limsup_{i \to \infty}\left(\sup_{\alpha \in [r, s]}U_{(u_i)_{R_i}}^{R_i^{-2}g_M}(\alpha)\right) < \infty.\]
On the other hand, since 
\[I_{(u_i)_{R_i}}^{R_i^{-2}g_M}(s)= \exp \left( 2\int_{r}^s\frac{U_{(u_i)_{R_i}}^{R_i^{-2}g_M}(t)}{t}dt\right)I_{(u_i)_{R_i}}^{R_i^{-2}g_M}(r),\]
by letting $i \rightarrow 0$, dominated convergence theorem and Proposition \ref{002}, we have the third assertion.
\end{proof}
\begin{corollary}\label{12453}
Let $r, R$ be positive numbers with $r<R$ and $u_{\infty}$ a harmonic function on $B_{R}(p)$.
If $U_{u_{\infty}}(r)=0$, then $u_{\infty}$ is a constant function on $B_r(p)$.
\end{corollary}
\begin{proof}
First, we assume $I_{u_{\infty}}(r)=0$.
Then, by Proposition \ref{006}, we have $D_{u_{\infty}}(t)=0$ for a.e. $0<t<r$.
Since $D_{u_{\infty}}$ is continuous, we have $D_{u_{\infty}}(r)=0$.
Thus, by Poincar\'e inequality on limit spaces, we have 
\[\frac{1}{\upsilon (B_r(p))}\int_{B_r(p)}\left| f- \frac{1}{\upsilon (B_r(p))}\int_{B_r(p)}fd\upsilon \right|d\upsilon \le C(n, R)r
\sqrt{\frac{1}{\upsilon (B_r(p))}\int_{B_r(p)}(\mathrm{Lip}f)^2d\upsilon}=0.\]
Since $f$ is Lipschitz on $B_r(p)$, $f$ is a constant function on $B_r(p)$.
Next, if $U_{u_{\infty}}(r)=0$ and $I_{u_{\infty}}(r)>0$, then, by the definition, we have $D_{u_{\infty}}(r)=0$.
Therefore, by an argumetnt above, we have the assertion in this case.
\end{proof}
The following corollary follows from Proposition \ref{006} and continuity of the function: $t \mapsto H^n(B_t(p))$, directly.
\begin{corollary}\label{007}
For every $R > 0$ and harmonic function $u_{\infty}$ on $B_{R}(p)$, the function $I_{u_{\infty}}$ is a $C^1$-function on $(0, R)$ and 
\[\frac{dI_{u_{\infty}}}{dt}(t)= \frac{2D_{u_{\infty}}(t)}{t}.\]
\end{corollary}
For every $0 < r < R$ and harmonic function $u$ on $B_R^{g_M}(m)$ satisfying $u \neq 0$, we put
\[W_u^{g_M}(r)=\frac{E_u^{g_M}(r)}{I_u^{g_M}(r)}\]
With same assumption of Lemma \ref{002}, if $u_{\infty}$ is not a constant function on $B_r(p)$, then, by Proposition \ref{002} and Proposition \ref{004}, we have 
\[\lim_{i \rightarrow \infty}W_{(u_i)_{R_i}}^{R_i^{-2}g_M}(r) = U_{u_{\infty}}(r).\]
\begin{proposition}\label{008}
For every $ 0< r< s <R$ and harmonic function $u_{\infty}$ on $B_{7R}(p)$, we have 
\[U_{u_{\infty}}(r) \le U_{u_{\infty}}(s).\]
\end{proposition}
\begin{proof}
By Theorem \ref{005}, 
there exists a sequence of harmonic functions $u_i$ on $B_{6RR_i}^{g_M}(m)$ such that 
$\sup_i\mathbf{Lip}u_i < \infty$ and $(u_i)_{R_i} \rightarrow u_{\infty}$ on $B_{6R}(p)$.
We fix $\epsilon > 0$.
Without loss of generality, we can assume that $U_{u_{\infty}}(r)>0$.
We shall use same notation as in \cite[Proposition $4.11$]{co-mi3}.
We put $\Omega_0 =s/r, \gamma =D_{u_{\infty}}(2s)/D_{u_{\infty}}(r) +1$.
Then we take $\hat{R}= R(m, \gamma, \epsilon, \Omega_0)$ as in \cite[Proposition $4.11$]{co-mi3}.
By Proposition \ref{002}, there exists $i_0$ such that  $R_ir > \hat{R}$ and  
\[\frac{D_{u_i}^{g_M}(2\Omega_0R_ir)}{D_{u_i}^{g_M}(R_ir)}=\frac{D_{u_i}^{g_M}(2R_is)}{D_{u_i}^{g_M}(R_ir)} = \frac{D_{(u_i)_{R_i}}^{R_i^{-2}g_M}(2s)}{D_{(u_i)_{R_i}}^{R_i^{-2}g_M}(r)}\le \gamma\]
for every $i \ge i_0$.
Then, by  \cite[Proposition $4.11$]{co-mi3}, we have 
\[\int_{R_ir}^{R_is}\frac{d\log W_{u_i}^{g_M}}{dt}dt \ge -\epsilon.\]
i.e. we have
\[\log W_{(u_i)}^{g_M}(R_is)-\log W_{(u_i)}^{g_M}(R_it) \ge -\epsilon.\]
Since $W_{(u_i)}^{g_M}(R_is)=W_{(u_i)_{R_i}}^{R_i^{-2}g_M}(s)$, by letting $i \rightarrow \infty$, we have 
\[\log U_{u_{\infty}}(s)-\log U_{u_{\infty}}(r) \ge -\epsilon.\]
Since $\epsilon $ is arbitrary, we have the assertion.
\end{proof}
\begin{remark}
Most of their results in \cite{co-mi3} are about global harmonic functions on manifolds. 
However, by the proof, their results in \cite{co-mi3} also hold for harmonic function on a big domain like one used in the proof of Proposition \ref{008}.
We will often use these facts below.  
\end{remark}
For $d \ge 0$, we put 
$\mathcal{H}^d(M_{\infty})=\{u_{\infty}:M_{\infty} \rightarrow \mathbf{R};$ $u_{\infty}$ is a harmonic function and there exists $C>1$ such that $|u_{\infty}(x)|\le
C(1 + \overline{m_{\infty}, x}^{d})$ for every $x \in M_{\infty}\}.$
\begin{proposition}\label{009}
We have $U_{u_{\infty}}(t) \le d$ for every $t>0$ and $u_{\infty} \in \mathcal{H}^d(M_{\infty})$.
\end{proposition}
\begin{proof}
This proof is done by a contradiction. 
We assume that there exist $\tau_0, s_0 > 0$ such that $U_{u_{\infty}}(s_0) \ge d_0 + \tau_0$.
By Proposition \ref{008}, we have  $U_{u_{\infty}}(s) \ge d+\tau_0$ for every $s \ge s_0$.
Since $u_{\infty} \in \mathcal{H}^d(M_{\infty})$, there exist $s_1 > s_0$ and $C > 1$ such that  
\[I_{u_{\infty}}(s) = s^{1-n}\int_{\partial B_s(p)}u_{\infty}^2dH^{n-1}\le s^{1-n}s^{2d}\mathrm{vol}\,\partial B_s(p)C \le Cs^{2d} \mathrm{vol}\,B_1(0_n)\]
for every $s \ge s_1$.
For $s > s_1$, by Proposition \ref{006}, we have 
\[Cs^{2d}\mathrm{vol}\,B_1(0_n) \ge \exp \left(\int_{s_1}^s\frac{2U_{u_{\infty}}(t)}{t}dt\right)I_{u_{\infty}}(s_1)
\ge \exp \left(\int_{s_1}^s \frac{2d+2\tau_0}{t}dt\right)I_{u_{\infty}}(s_1).\]
Therefore, we have 
\[2d + \frac{\log (C\mathrm{vol}\,B_1(0_n))}{\log s} \ge \frac{1}{\log s}\int_{s_1}^s\frac{2d+2\tau_0}{t}dt +\frac{\log I_{u_{\infty}}(s_1)}{\log s}.\]
By letting $s \rightarrow \infty$, we have $2d \ge 2d + 2\tau_0$. This is a contradiction. 
\end{proof}
\begin{proposition}\label{1234567890}
For every $0 < s < t < \alpha < R$ and harmonic function $u_{\infty}$ on $B_{7R}(p)$, we have 
\[I_{u_{\infty}}(t) \le \left(\frac{t}{s}\right)^{2U_{u_{\infty}}(\alpha)}I_{u_{\infty}}(s).\]
\end{proposition}
\begin{proof}
First, we assume that $u_{\infty}$ is not a constant function on $B_s(p)$.
By Theorem \ref{005}, 
there exists a sequence of harmonic functions $u_i$ on $B_{6RR_i}^{g_M}(m)$ such that 
$\sup_i\mathbf{Lip}u_i < \infty$ and $(u_i)_{R_i} \rightarrow u_{\infty}$ on $B_{6R}(p)$.
We fix $\epsilon >0$. 
By the assumption and Corollary \ref{12453}, there exists $0 < r <s$ such that $U_{u_{\infty}}(r)>0$.
We shall apply \cite[Corollary $4. 37$]{co-mi3}.
We put $\Omega_0=2\alpha/r$, $\Omega=\alpha/r$ and 
\[\gamma=\frac{D_{u_{\infty}}(2\Omega r)}{D_{u_{\infty}}(r)}+1.\]
We take $\hat{R}=R(m, \gamma, \epsilon, \Omega_0)$ as  in \cite[Corollary $4.37$]{co-mi3}.
There exists $i_0$ such that $R_i r > \hat{R}$ and that 
\[\frac{D_{(u_i)_{R_i}}^{R_i^{-2}g_M}(2\Omega r)}{D_{(u_i)_{R_i}}^{R_i^{-2}g_M}(r)} < \gamma \] 
for every $i \ge i_0$.
Thus, by  \cite[Corollary $4.37$]{co-mi3}, we have
\[I_{u_i}^{g_M}(R_it) \le \left(\frac{R_it}{R_is}\right)^{2(1+\epsilon)W_{u_i}^{g_M}(\Omega R_i r)}I_{u_i}^{g_M}(R_is).\]
Thus by letting $i \rightarrow \infty$, we have 
\[I_{u_{\infty}}(t) \le \left(\frac{t}{s}\right)^{2(1+\epsilon)U_{u_{\infty}}(\alpha)}I_{u_{\infty}}(s).\]
Since $\epsilon $ is arbitrary, we have the assertion.
Next we assume that $u_{\infty}$ is a constant function on $B_s(p)$.
We put $\hat{s}=\sup \{\beta \in [0, R]; u_{\infty}$ is a constant function on $B_{\beta}(p)\}$.
If $\hat{s} \ge t$, then, since $I_{u_{\infty}}(t)=I_{u_{\infty}}(s)$,
the assertion is clear.
We assume $\hat{s} < t$.
We take $\hat{s}< \tilde{s}<t$.
Then, by an argument above, we have
\[I_{u_{\infty}}(t) \le \left(\frac{t}{\tilde{s}}\right)^{2U_{u_{\infty}}(\alpha)}I_{u_{\infty}}(\tilde{s}).\]
By $s \le \hat{s}$, $I_{u_{\infty}}(s)=I_{u_{\infty}}(\hat{s})$ and letting $\tilde{s} \to \hat{s}$, we have the assertion. 
\end{proof}
\begin{corollary}\label{178907}
Let $s, R$ be positive numbers with $0 < s < R$ and $u_{\infty}$ a harmonic function on $B_{7R}(p)$.
Assume that $U_{u_{\infty}}(s)=0$.
Then $u_{\infty}$ is a constant function on $B_R(p)$.
\end{corollary}
\begin{proof}
First, we assume that $I_{u_{\infty}}(s)=0$.
Then, by Proposition \ref{1234567890}, we have $I_{u_{\infty}}(t)=0$ for every $s < t < R$.
Therefore, by Proposition \ref{12453}, we have the assertion.
Next, we assume that $I_{u_{\infty}}(s)>0$ and $U_{u_{\infty}}(s)=0$.
Then, we put $\hat{u}_{\infty}= u_{\infty}-u_{\infty}(p)$.
We remark that $\hat{u}_{\infty} \equiv 0$ on $B_s(p)$.
Since $I_{\hat{u}_{\infty}}(s)=0$, we have the assertion.
\end{proof}
\begin{proposition}\label{110}
Let $R$ be a positive number and $u_{\infty}$ a harmonic function  on $B_{7R}(p)$ with $u_{\infty}(p)=0$.
Assume that $u_{\infty}$ is not a constant function on $B_R(p)$.
Then, we have 
\[U_{u_{\infty}}(s)\ge 1\]
for every $0 < s < R$.
\end{proposition}
\begin{proof}
By Theorem \ref{005}, 
there exists a sequence of harmonic functions $u_i$ on $B_{6RR_i}^{g_M}(m)$ such that 
$\sup_i\mathbf{Lip}u_i < \infty$ and $(u_i)_{R_i} \rightarrow u_{\infty}$ on $B_{6R}(p)$.
Moreover, we can assume that $(u_i)_{R_i}(m)=0$.
We remark that by Proposition \ref{178907}, 
$U_{u_{\infty}}(r) > 0$ for every $0 < r < R$.
We fix a sufficiently small $\epsilon > 0$.
We shall apply \cite[Corollary $4. 40$]{co-mi3} and use same notation as in there.
We take $\Omega_L=\Omega_L(n, \epsilon) \ge 2$ as in \cite[Corollay $4. 40$]{co-mi3} (or \cite[Corollary $3.29$]{co-mi3}).
We put $\Omega_0 = 5\Omega_L$, $r = s/2(2\Omega_L )^2< s$ and
\[\gamma = \frac{D_{u_{\infty}}(2(2\Omega_L)^2r)}{D_{u_{\infty}}(r)}+1=\frac{D_{u_{\infty}}(s)}{D_{u_{\infty}}(r)} + 1.\]
We take $\hat{R}=R(m, \gamma, \epsilon, \Omega_0)$ as in \cite[Corollary $4.40$]{co-mi3}.
Then there exists $i_0$ such that $R_i r > \hat{R}$ and
\[\frac{D_{u_i}^{g_M}(2(2\Omega_L)^2R_ir)}{D_{u_i}^{g_M}(R_ir)}=\frac{D_{(u_i)_{R_i}}^{R_i^{-2}g_{M}}(2(2\Omega_L)^2r)}{D_{(u_i)_{R_i}}^{R_i^{-2}g_{M}}(r)} \le \gamma \]
for every $i \ge i_0$.
Then by \cite[Corollary $4. 40$]{co-mi3}, we have 
\[1-3\epsilon \le U_{u_i}^{g_M}(2\Omega_LR_ir) = U_{(u_i)_{R_i}}^{R_i^{-2}g_M}(2\Omega_Lr).\]
By letting $i \rightarrow \infty$, Proposition \ref{002} and Proposition \ref{008}, we have $1-3\epsilon \le U_{u_{\infty}}(2\Omega_Lr) \le U_{u_{\infty}}(s)$.
Since $\epsilon$ is arbitrary, we have the assertion.
\end{proof}
\begin{proposition}\label{009}
Let $r, s, R, \delta, d_0$ be positive numbers with $0 < r < s < R$ and $u_{\infty}$ a harmonic function on $B_{7R}(p)$.
We assume that $ U_{u_{\infty}}(s) \le d_0$, $u_{\infty}$ is not a constant function on $B_{R}(p)$ and
\[\left|\log \frac{U_{u_{\infty}}(s)}{U_{u_{\infty}}(r)}\right|<\delta.\]
Then, we have
\[\int_{A_p(r, s)}r_p^{-n}\left|r_p\langle dr_p, du_{\infty}\rangle -U_{u_{\infty}}(r_p)u_{\infty}\right|^2dH^n
\le \Psi(\delta;n,d_0)I_{u_{\infty}}(s)\]
\end{proposition}
\begin{proof}
By Theorem \ref{005}, 
there exists a sequence of harmonic functions $u_i$ on $B_{6RR_i}^{g_M}(m)$ such that 
$\sup_i \mathbf{Lip}u_i < \infty$ and $(u_i)_{R_i} \rightarrow u_{\infty}$ on $B_{6R}(p)$.
We shall apply  \cite[Proposition $4. 50$]{co-mi3}.
We put $\Omega_0 = 2s/r$, $\Omega=s/r$ and 
\[\gamma= \frac{D_{u_{\infty}}(2\Omega r)}{D_{u_{\infty}}(r)} + 1.\]
Then, by Proposition \ref{002}, there exists $i_0$ such that
\[\frac{D_{u_{\infty}}^{R_i^{-2}g_M}(2\Omega r)}{D_{u_{\infty}}^{R_i^{-2}g_M}(r)}\le \gamma, \] 
\[\max_{r \le t \le \Omega r}U_{(u_i)_{R_i}}^{R_i^{-2}g_M}(t) \le 2d_0 \]
and 
\[\left|\log \frac{U_{(u_i)_{R_i}}^{R_i^{-2}g_M}(\Omega r)}{U_{(u_i)_{R_i}}^{R_i^{-2}g_M}(r)}\right|\le \delta \]
for every $i \ge i_0$.
Thus, by \cite[Proposition $4.50$]{co-mi3}, we have,
\[\int_{rR_i \le b^{g_M} \le sR_i}(b^{g_M})^{-n}\left(b^{g_M}\frac{\partial u_i}{\partial n}-U_{u_i}^{g_M}(b^{g_M})|\nabla^{g_M}b^{g_M}|\right)^2d\mathrm{vol}^{g_M}\le \Psi(\delta;n, d_0)I_{u_i}^{g_M}(R_is)\]
for every sufficiently large $i$. 
On the other hand, by Cheng-Yau's gradient estimate, we have 
\begin{align}
|\nabla^{R_i^{-2}g_M}b^{R_i^{-2}g_M}|&=\frac{V_M}{(n-2)\mathrm{vol}\,B_1(0_n)}|b^{R_i^{-2}g_M}|^{n-1}|\nabla^{R_i^{-2}g_M}G^{R_i^{-2}g_M}(m, \cdot )| \\
&\le \frac{V_M}{(n-2)\mathrm{vol}\,B_1(0_n)}2(r_m^{R_i^{-2}g_M})^{n-1}C(n)(r_m^{R_i^{-2}g_M})^{-1}|G^{R_i^{-2}g_M}(m, \cdot )| \\
&\le C(n)(r_m^{R_i^{-2}g_M})^{-1}(r_m^{R_i^{-2}g_M})^{n-1}(r_m^{R_i^{-2}g_M})^{2-n} \\
&\le C(n). 
\end{align}
on $A_m^{R_i^{-2}g_M}(r, s)$ for every sufficiently large $i$.
Thus by Corollary \ref{har} and Theorem \ref{green}, we have $(b^{R_i^{-2}g_M}, db^{R_i^{-2}g_M}) \rightarrow (r_p, dr_p)$ on $A_p(r, s)$.
We also have
\begin{align}
&\int_{r \le b^{R_i^{-2}g_M} \le s}(b^{R_i^{-2}g_M})^{-n}\Biggl(b^{R_i^{-2}g_M}(R_i^{-2}g_M)\left(\nabla^{R_i^{-2}g_M}(u_i)_{R_i}, \nabla^{R_i^{-2}g_M}b^{R_i^{-2}g_M}\right) \\
& \ \ \ \ \ \ \ \ \ \ \ \ \ \ \ \ \ \ \ \ \ -U_{(u_i)_{R_i}}^{R_i^{-2}g_M}(b^{R_i^{-2}g_M})|\nabla^{R_i^{-2}g_M}b^{R_i^{-2}g_M}|^2\Biggl)^2d\mathrm{vol}^{R_i^{-2}g_M}\\
&=\int_{rR_i \le b^{g_M} \le sR_i}(b^{g_M})^{-n}|\nabla b^{g_M}|^2\left(b^{g_M}\frac{\partial u_i}{\partial n}-U_{u_i}^{g_M}(b^{g_M})|\nabla^{g_M}b^{g_M}|\right)^2d\mathrm{vol}^{g_M}\\
&\le C(n)\int_{rR_i \le b^{g_M} \le sR_i}(b^{g_M})^{-n}\left(b^{g_M}\frac{\partial u_i}{\partial n}-U_{u_i}^{g_M}(b^{g_M})|\nabla^{g_M}b^{g_M}|\right)^2d\mathrm{vol}^{g_M}\\
&\le \Psi(\delta;n, d_0)I_{u_i}^{g_M}(R_is)= \Psi(\delta;n, d_0)I_{(u_i)_{R_i}}^{R_i^{-2}g_M}(s)
\end{align}
for every sufficiently large $i$.
Therefore, by letting $i \rightarrow \infty$, Proposition \ref{10101} and Proposition \ref{002}, we have the assertion.
\end{proof}
The following corollary follows from Proposition \ref{009} directly.
\begin{corollary}\label{010}
Let $r, s, R$ be positive numbers with $r < s <R$ and $u_{\infty}$ be a harmonic function on $B_{7R}(p)$ with $u_{\infty}(p)=0$.
We assume that $U_{u_{\infty}}(r)=U_{u_{\infty}}(s)$. Then we have 
\[r_p(w)\langle du_{\infty}, dr_p\rangle (w)=U_{u_{\infty}}(s)u_{\infty}(w)\]
for a.e $w \in A_p(r, s)$.
\end{corollary}
\begin{proposition}
With same assumption as in Corollary \ref{010}, we have
\[u_{\infty}(\hat{t}, x) = \frac{u_{\infty}(t, x)}{t^C}\hat{t}^C\]
for every $r \le t \le \hat{t} \le s$ and $x \in X$.
Here $C =  U_{u_{\infty}}(r)$.
\end{proposition}
\begin{proof}
We define a Borel function $a$ on $A_p(r, s)$ by
\[a(t, x)=\limsup_{h \rightarrow 0}\frac{u_{\infty}(t+h, x)-u_{\infty}(t, x)}{h}.\]
By Theorem \ref{14} and Corollary \ref{010}, there exists a Borel set $A \subset A_p(r, s)$ such that 
$H^n(A_p(r, s) \setminus A)=0$ and that $\langle dr_p, du_{\infty}\rangle (z)=a(z)=Cu_{\infty}(z)/r_p(z)$ for every $z \in A$.
On the other hand, for $0<s \le r_0 \le s_0 \le s$, we put a bi-Lipschitz map $\phi(t, x)=(t, x)$ from $A_{p}(r_0, s_0)$ to $[r_0, s_0] \times X$.
Then we have $H^n([r_0, s_0] \times X \setminus \phi(A))=0$.
Therefore by Fubini's theorem, there exists a Borel set $\hat{X} \subset X$ such that $H^{n-1}(X \setminus \hat{X})=0$ and that 
 $H^1([r_0, s_0] \times \{x\} \setminus \phi(A))=0$ for every $x \in X$.
Thus we have $H^1(\phi^{-1}([r_0, s_0] \times \{x\} \setminus \phi(A)))=0$ for $x \in \hat{X}$.
For every $x \in \hat{X}$, by Rademacher's theorem for Lipschitz functions on $\mathbf{R}$, 
\begin{align}
u_{\infty}(s_0, x)-u_{\infty}(r_0, x)&=\int_{r_0}^{s_0}a(t, x)dt \\
&=\int_{r_p(\phi^{-1}([r_0, s_0] \times \{x\} \cap \phi(A)))}a(t, x)dt \\
&=\int_{r_p(\phi^{-1}([r_0, s_0] \times \{x\} \cap \phi(A)))}\frac{Cu_{\infty}(t, x)}{t}dt \\
&=\int_{r_0}^{s_0}\frac{Cu_{\infty}(t, x)}{t}dt.
\end{align}
For every $x \in X$, by taking a sequence $x_i \in \hat{X}$ satisfying $x_i \rightarrow x$ and dominated convergence theorem, 
we have
\[u_{\infty}(s_0, x)-u_{\infty}(r_0, x)=\int_{r_0}^{s_0}\frac{Cu_{\infty}(t, x)}{t}dt.\] 
Thus, for every $x \in X$, the map $f_x(\tilde{t})=u_{\infty}(\tilde{t}, x)$ on $[r, s]$ is $C^1$-function, we have 
\[\frac{df_x}{d\tilde{t}}(\tilde{t})=\frac{Cf_x(\tilde{t})}{\tilde{t}}.\]
Therefore, we have the assertion.
\end{proof}
\begin{proposition}\label{012}
Let $r, s, \delta, R, d_0$ be positive numbers with $0<r < s <R$, and $u_{\infty}, v_{\infty}$  harmonic functions on $B_{7R}(p)$.
We assume that $\max_{r \le t \le s}U_{v_{\infty}}(t) \le d_0$, $v_{\infty}$ is not a constant function on $B_{R}(p)$ and 
\[\left|\log \frac{U_{v_{\infty}}(s)}{U_{v_{\infty}}(r)}\right|<\delta.\]
Then, we have 
\begin{align*}
&\left|s_0^{1-n}\int_{\partial B_{s_0}(p)}u_{\infty}v_{\infty}d\upsilon-\exp \left(2\int_{r_0}^{s_0}\frac{U_{v_{\infty}}(\hat{s})}{\hat{s}}d\hat{s}\right)
r_0^{1-n}\int_{\partial B_{r_0}(p)}u_{\infty}v_{\infty}d\upsilon \right|^2 \\
&\le \Psi(\delta;n, d_0)\left(\frac{s_0}{r_0}\right)^{6d_0+3}I_{u_{\infty}}(s_0)I_{v_{\infty}}(s_0).
\end{align*}
for every $r \le r_0 \le s_0 \le s$.
\end{proposition}
\begin{proof}
By Theorem \ref{005}, 
there exists a sequence of harmonic functions $u_i, v_i$ on $B_{6RR_i}^{g_M}(m)$ such that 
$\sup_i(\mathbf{Lip}u_i + \mathbf{Lip}v_i) < \infty$, $(u_i)_{R_i} \rightarrow u_{\infty}$, 
$(v_i)_{R_i} \rightarrow v_{\infty}$ on $B_{6R}(p)$.
By the proof of Proposition \ref{009} (or  \cite[Proposition $4. 50$]{co-mi3}), there exists $i_0$ such that 
\[\int_{rR_i \le b^{g_M} \le sR_i}(b^{g_M})^{-n}\left(b^{g_M}\frac{\partial v_i}{\partial n}-U_{v_i}^{g_M}(b^{g_M})|\nabla^{g_M}b^{g_M}|\right)^2d\mathrm{vol}^{g_M}\le \Psi(\delta;n, d_0)I_{v_i}^{g_M}(R_is)\]
for every $i \ge i_0$. 
Thus, by  \cite[Corollary $5.24$]{co-mi3}, we have 
\begin{align*}
&\left|(R_is_0)^{1-n}\int_{b^{g_M}=R_is_0}u_iv_id\mathrm{vol}^{g_M}_{n-1}-\exp \left(2\int_{r_0R_i}^{s_0R_i}\frac{U_{v_i}^{g_M}(\hat{s})}{\hat{s}}d\hat{s}\right)
(R_ir_0)^{1-n}\int_{b^{g_M}=R_ir_0}u_iv_id\mathrm{vol}^{g_M}_{n-1}\right|^2 \\
&\le \Psi(\delta;n, d_0)\left(\frac{s_0}{r_0}\right)^{6d_0+3}I_{u_i}^{g_M}(R_is_0)I_{v_i}^{g_M}(R_is_0)
\end{align*}
for $i \ge i_0$.
By rescaling $R_i^{-2}g_M$, we have
\begin{align*}
&\Biggl|s_0^{1-n}\int_{b^{R_i^{-2}g_M}=s_0}(u_i)_{R_i}(v_i)_{R_i}d\mathrm{vol}^{R_i^{-2}g_M}_{n-1} \\
& \ \ \ \ \ \ \ \ \ \ -\exp \left(2\int_{r_0}^{s_0}\frac{U_{(v_i)_{R_i}}^{R_i^{-2}g_M}(\hat{s})}{\hat{s}}d\hat{s}\right)
r_0^{1-n}\int_{b^{R_i^{-2}g_M}=r_0}(u_i)_{R_i}(v_i)_{R_i}d\mathrm{vol}^{R_i^{-2}g_M}_{n-1}\Biggl|^2 \\
&\le \Psi(\delta;n, d_0)\left(\frac{s_0}{r_0}\right)^{6d_0+3}I_{(u_i)_{R_i}}^{R_i^{-2}g_M}(s_0)I_{(v_i)_{R_i}}^{R_i^{-2}g_M}(s_0).
\end{align*}
On the other hand, by Proposition \ref{002}, we have 
\begin{align*}
&\int_{b^{R_i^{-2}g_M}=s_0}(u_i)_{R_i}(v_i)_{R_i}d\mathrm{vol}^{R_i^{-2}g_M}_{n-1} \\
&=\frac{1}{2}\int_{b^{R_i^{-2}g_M}=s_0}\left((u_i)_{R_i}+(v_i)_{R_i}\right)^2d\mathrm{vol}^{R_i^{-2}g_M}_{n-1} -\frac{1}{2}
\int_{b^{R_i^{-2}g_M}=s_0}(u_i)_{R_i}^2d\mathrm{vol}^{R_i^{-2}g_M}_{n-1} \\
& \ \ \ \ \ \ \ \ \ \ \ -\frac{1}{2}\int_{b^{R_i^{-2}g_M}=s_0}(v_i)_{R_i}^2d\mathrm{vol}^{R_i^{-2}g_M}_{n-1} \\
& \stackrel{i \rightarrow \infty}{\rightarrow} \frac{1}{2}\int_{\partial B_{s_0}(p)}(u_{\infty}+v_{\infty})^2dH^{n-1} -\frac{1}{2}
\int_{\partial B_{s_0}(p)}u_{\infty}^2dH^{n-1} -\frac{1}{2}\int_{\partial B_{s_0}(p)}v_{\infty}^2dH^{n-1} \\
&=\int_{\partial B_{s_0}(p)}u_{\infty}v_{\infty}dH^{n-1}.
\end{align*}
Therefore we have the assertion.
\end{proof}
The following corollary follows from Proposition \ref{012} directly:
\begin{corollary}\label{013}
Let $r, s, R$ be positive numbers with
$0 < r < s <R$ and $u_{\infty}, v_{\infty}$ harmonic functions on $B_{7R}(p)$.
We assume that  $U_{v_{\infty}}(r)=U_{v_{\infty}}(s)$ and $v_{\infty}$ is not a constant function on $B_R(p)$.
Then, we have
\[s_0^{1-n}\int_{\partial B_{s_0}(p)}u_{\infty}v_{\infty}dH^{n-1}=\left(\frac{s_0}{r_0}\right)^{2C}
r_0^{1-n}\int_{\partial B_{r_0}(p)}u_{\infty}v_{\infty}dH^{n-1} \]
for every $ r \le r_0 \le s_0 \le s$.
Here $C=U_{v_{\infty}}(r)$.
\end{corollary}
Next proposition follows from Proposition \ref{110} directly: 
\begin{proposition}\label{015}
For every non-constant harmonic function $u_{\infty}$ on $C(X)$ with $u_{\infty}(p)=0$, we have 
\[\mathrm{ord}_0u_{\infty} \ge 1.\] 
\end{proposition}
\begin{proposition}\label{020}
With same assumption as in Lemma \ref{002}, for every $0 < r < s <R$, we have 
\[\lim_{i \rightarrow \infty}\int_r^sF_{(u_i)_{R_i}}^{R_i^{-2}g_M}(t)dt = \int_r^sF_{u_{\infty}}(t)dt.\]
\end{proposition}
\begin{proof}
Since $(b^{R_i^{-2}g_M}, db^{R_i^{-2}g_M}) \rightarrow (r_p, dr_p)$ on $A_p(r, s)$, by Corollary \ref{har}, we have
\begin{align*}
&\int_r^sF_{(u_i)_{R_i}}^{R_i^{-2}g_M}(t)dt \\
&=\int_r^s t^{3-n}\int_{b^{R_i^{-2}g_M}=t}(R^{-2}g_M)\left(\nabla^{R_i^{-2}g_M}(u_i)_{R_i}, \frac{\nabla^{R_i^{-2}g_M}b^{R_i^{-2}g_M}}{|\nabla^{R_i^{-2}g_M}b^{R_i^{-2}g_M}|}\right)^2|\nabla^{R_i^{-2}g_M}b^{R_i^{-2}g_M}|d\mathrm{vol}^{R_i^{-2}g_M}dt \\
&=\int_{r \le b^{R_i^{-2}g_M} \le s}(R^{-2}g_M)\left(\nabla^{R_i^{-2}g_M}(u_i)_{R_i}, \frac{\nabla^{R_i^{-2}g_M}b^{R_i^{-2}g_M}}{|\nabla^{R_i^{-2}g_M}b^{R_i^{-2}g_M}|}\right)^2|\nabla^{R_i^{-2}g_M}b^{R_i^{-2}g_M}|^2(b^{R_i^{-2}g_M})^{3-n}d\mathrm{vol}^{R_i^{-2}g_M} \\
&=\int_{r \le b^{R_i^{-2}g_M} \le s}(R^{-2}g_M)(\nabla^{R_i^{-2}g_M}(u_i)_{R_i},\nabla^{R_i^{-2}g_M}b^{R_i^{-2}g_M})^2(b^{R_i^{-2}g_M})^{3-n}d\mathrm{vol}^{R_i^{-2}g_M} \\
& \stackrel{i \rightarrow \infty}{\rightarrow} \int_{A_p(r, s)}r_p^{3-n}\langle du_{\infty}, dr_p\rangle ^2d\upsilon = \int_r^sF_{u_{\infty}}(t)dt.
\end{align*}
\end{proof}
\begin{proposition}\label{8890}
For every $0 < r < s <R$ and harmonic function $u_{\infty}$ on $B_R(p)$, we have 
\[D_{u_{\infty}}(s)-D_{u_{\infty}}(r)= \int_r^s\frac{2F_{u_{\infty}}(t)}{t}dt.\]
\end{proposition}
\begin{proof}
We can assume that the assumption of Proposition \ref{002} holds. By $(4.3)$ in \cite{co-mi3}, we have 
\begin{align*}
&E_{(u_i)_{R_i}}^{R_i^{-2}g_M}(s)-E_{(u_i)_{R_i}}^{R_i^{-2}g_M}(r) \\
&= \int_r^s\frac{2F_{(u_i)_{R_i}}^{R_i^{-2}g_M}(t)}{t}dt +\int_r^s\frac{2E_{(u_i)_{R_i}}^{R_i^{-2}g_M}(t)}{t}dt -\int_r^st^{1-n}\int_{b^{R_i^{-2}g_M}\le t}2|\nabla^{R_i^{-2}g_M}(u_i)_{R_i}|^2d\mathrm{vol}^{R_i^{-2}g_M}dt\\
& \pm \int_r^s t^{1-n}\int_{b^{R_i^{-2}g_M} \le t}\biggl|\mathrm{Hess}_{(b^{R_i^{-2}g_M})^2}^{R_i^{-2}g_M}\left(\nabla^{R_i^{-2}g_M}(u_i)_{R_i}, 
\nabla^{R_i^{-2}g_M}(u_i)_{R_i}\right) \\
& \ \ \ \ \ \ \ \ -2(R_i^{-2}g_M)\left(\nabla^{R_i^{-2}g_M}(u_i)_{R_i}, \nabla^{R_i^{-2}g_M}(u_i)_{R_i}\right)\biggl|d\mathrm{vol}^{R_i^{-2}g_M}dt.
\end{align*}
By Corollary \ref{har} and Theorem \ref{green}, we have 
\[\lim_{i \rightarrow \infty}\int_{b^{R_i^{-2}g_M} \le t}|d(u_i)_{R_i}|^2d\mathrm{vol}^{R_i^{-2}g_M}= \int_{B_t(p)}|du_{\infty}|^2dH^n.\]
By dominated convergence theorem, we have
\begin{align*}
\lim_{i \rightarrow \infty}\int_r^st^{1-n}\int_{b^{R_i^{-2}g_M}\le t}2|\nabla^{R_i^{-2}g_M}(u_i)_{R_i}|^2d\mathrm{vol}^{R_i^{-2}g_M}dt
&=\int_r^st^{1-n}\int_{B_t(p)}2|du_{\infty}|^2dH^2dt\\
&=\int_r^s\frac{2E_{u_{\infty}}(t)}{t}dt.
\end{align*}
On the other hand, we recall
\[\lim_{R \rightarrow \infty}\frac{1}{\mathrm{vol}^{g_M}
(\{b^{g_M} \le R\})}\int_{b^{g_M} \le R}|\mathrm{Hess}_{(b^{g_M})^2}-2g_M|d\mathrm{vol}^{g_M} = 0.\]
Thus we have 
\begin{align*}
& \lim_{i \rightarrow \infty}\int_{b^{R_i^{-2}g_M} \le t}\biggl|\mathrm{Hess}_{(b^{R_i^{-2}g_M})^2}^{R_i^{-2}g_M}\left(\nabla^{R_i^{-2}g_M}(u_i)_{R_i}, 
\nabla^{R_i^{-2}g_M}(u_i)_{R_i}\right) \\
& \ \ \ \ \ \ \ \ \ \ \ -2(R_i^{-2}g_M)\left(\nabla^{R_i^{-2}g_M}(u_i)_{R_i}, \nabla^{R_i^{-2}g_M}(u_i)_{R_i}\right)\biggl|d\mathrm{vol}^{R_i^{-2}g_M}dt=0.
\end{align*}
Therefore we have the assertion.
\end{proof}
We shall give a short review of important works by Ding in \cite{di1} and \cite{di2}.
By Corollary \ref{25}, $X$ is $H^{n-1}$-rectifiable. By  \cite[Lemma $4. 3$]{di2}, $(X, H^{n-1})$ satisfies weak Poincar\'e inequality of type $(1, 2)$ locally.
Thus, by section $4$ in \cite{ch2} (or section $6$ in \cite{ch-co3}) and Proposition \ref{9009}, we can define the cotangent bundle $T^*X$ of $X$.
We denote the differential section of a Lipschitz function $f$ on $X$ by $d_Xf: X \rightarrow T^*X$.
By \cite[Theorem $6. 25$]{ch-co3}, there exists a unique self-adjoint operator $\Delta_X$ on $L^2(X)$ such that 
\[\int_X \langle d_Xf, d_Xg\rangle dH=\int_Xf \Delta_X gdH^n \]
for every $f \in H_{1,2}(X)$ and $g \in \mathrm{Domain}(\Delta)$.
For every $i$, we take a $i$-th eigenfunction $\phi_i$ on $X$ and the $i$-th eigenvalue $\lambda_i \ge 0$, i.e. $\Delta_X\phi_i = \lambda_i\phi_i$ $(0 = \lambda_0 < \lambda_1 \le \lambda_2 \le \cdots)$.
We define the  nonnegative number $\alpha_i$ by satisfying $\lambda_i=\alpha_i(\alpha_i + n-2)$.
According to \cite{di1}, the function $v_i(r, x)= r^{\alpha_i}\phi_i(x)$ on $C(X)$ is a harmonic function on $C(X)$.
Actually, by  \cite[Theorem $4. 15$]{di2}, for every Lipschitz function $f \in \mathcal{K}(C(X) \setminus \{p\})$, we have 
\begin{align}
\int_{C(X)}\langle df, dv_i\rangle dH^n &= \int_0^{\infty}\int_{\partial B_r(p)}\biggl(-\alpha_i(\alpha_i-1)r^{\alpha_i-2}f\phi_i  \\
& \ \ \ \ \ \ \ \ \ \ \ \ \ \ - \frac{n-1}{r} \alpha_i r^{\alpha_i-1} + \frac{1}{r^2}\langle d_Xf, d_X\phi_i\rangle \biggl)dH^{n-1}dr \\
&= \int_0^{\infty}\int_{\partial B_r(p)}\biggl(-\alpha_i(\alpha_i-1)r^{\alpha_i-2}f\phi_i \\
& \ \ \ \ \ \ \ \ - (n-1) \alpha_i r^{\alpha_i-2}f\phi_i +\lambda_i r^{\alpha_i-2}f\phi_i\biggl)dH^{n-1}dr \\
&=0.
\end{align}
Thus, $v_i$ is a harmonic function on $C(X) \setminus \{p\}$.
Moreover, by \cite[Corollary $4. 25$]{di2}, $v_i$ is a harmonic function on $C(X)$.
By Theorem \ref{005}, $v_i$ is locally Lipschitz.
Especially, $\phi_i$ is Lipschitz.
Therefore, we have $\lambda_1 \ge n-1$ (see \cite[Corollary $2. 4$]{di1} and \cite[Corollary $2.5$]{di1}).
On the other hand, it is easy to check
\[U_{v_i}(s)=\alpha_i\]
for every $s > 0$.
We say that the function $v_i$\textit{ is a homogeneous harmonic function with growth $\alpha_i$}.
We shall prove that we can apply \cite[Theorem $4. 15$]{di2} for every $d \ge 0$ and $u_{\infty} \in \mathcal{H}^d(M_{\infty})$ below.
As a corollary, we will give the classification of harmonic functions with polynomial growth on asymptotic cones (see Theorem \ref{023}).
\

We put
\[\mathrm{ord}_{\infty}u_{\infty}=\lim_{r \rightarrow \infty}U_{u_{\infty}}(r), \ \mathrm{ord}_{0}u_{\infty}=\lim_{r \rightarrow 0}U_{u_{\infty}}(r)\]
for every harmonic function $u_{\infty}$ on $C(X)$.
By an argument similar to the proof of \cite[Lemma $1. 36$]{co-mi3}, we can prove the following proposition:
\begin{proposition}\label{0987654321}
For harmonic functions $u_{\infty}, v_{\infty}$ on $C(X)$, we have 
\[\mathrm{ord}_{\infty}(u_{\infty}+v_{\infty}) \le \max \{\mathrm{ord}_{\infty}u_{\infty}, \mathrm{ord}_{\infty}v_{\infty} \}.\]
\end{proposition}
\begin{definition}\label{014}
For harmonic functions $u_{\infty}, v_{\infty}$ on $C(X)$, we say that \textit{$u_{\infty}$ and $v_{\infty}$ are orthogonal} if 
\[\int_{\partial B_1(p)}u_{\infty}v_{\infty}d\upsilon=0.\]
\end{definition}
\begin{proposition}\label{021}
Let $u_{\infty}$ be a harmonic function on $C(X)$.
We assume that $\mathrm{ord}_{\infty}u_{\infty}=d < \infty$ and that 
$v$ and $u_{\infty}$ are orthogonal for every homogeneous harmonic function $v$ with growth $\alpha$ satisfying $\alpha <d$.
Then,  we have 
\[D_{u_{\infty}}(s) \ge \left(\frac{s}{r}\right)^{2d}D_{u_{\infty}}(r)\]
for every $0 < r < s < \infty$.
\end{proposition}
\begin{proof}
For every $i$, we take the $i$-th eigenvalue $\lambda_i$ of $\Delta_X$, a $i$-th eigenfunction $\phi_i$ of $\Delta_X$, the nonnegative number $\alpha_i$ satisfying $\lambda_i=\alpha_i(\alpha_i + n-2)$ and a 
homogeneous harmonic function $v_i(t, x)=r^{\alpha_i}\phi_i(x)$ with growth $\alpha_i$.
By Corollary \ref{013} anf the assumption, we have 
\[\int_{\partial B_t(p)}v_iu_{\infty}dH^{n-1}=0\]
for every $t > 0$ and $\alpha_i < d$.
We put $\lambda = d(d+n-2)$.
We remark that $\alpha_i < d$ holds if and only if $\lambda_i < \lambda$ holds.
We put $i_d = \max \{ i \in \mathbf{N}|\alpha_i < d\}$.
Thus, we have $\lambda_{i_d} < \lambda \le \lambda_{i_d+1}$.
We also remark 
\[\lambda_{i_d+1}=\inf \left\{ \frac{\int_X|d_Xu|^2dH^{n-1}}{\int_Xu^2dH^{n-1}}\Biggl|u \in H_{1,2}(X), \ u \neq 0, \ \int_X u\phi_jdH^{n-1}=0
\mathrm{\ for\ every}\ 1 \le j \le i_d \right\}. \]
Since the $k$-th eigenvalue $\lambda_k^t$ of $\Delta_{\partial B_t(p)}$ is equal to $t^{-2}\lambda_k$,
we have 
\[\frac{\int_{\partial B_t(p)}|d_{\partial B_t(p)}u_{\infty}|^2dH^{n-1}}{\int_{\partial B_t(p)}(u_{\infty})^2dH^{n-1}} \ge \frac{\lambda}{t^2}.\]
Here $d_{\partial B_t(p)}f$ is differential section: $d_{\partial B_t(p)}f: \partial B_t(p) \rightarrow T^* \partial B_t(p)$ of a Lipschitz function $f$ on $\partial B_t(p)$. 
On the other hand, by Theorem \ref{9} and Proposition \ref{0005}, for a.e. $t>0$, we have
$|du_{\infty}|^2(w)=(\langle dr_p, du_{\infty}\rangle (w))^2 + |d_{\partial B_r(p)}u_{\infty}|^2(w)$ for a.e. $w  \in \partial B_t(p)$.
Therefore, we have
\[\int_{\partial B_t(p)}(|du_{\infty}|^2-\langle dr_p, du_{\infty}\rangle ^2)dH^{n-1}\ge \frac{\lambda}{t^2}\int_{\partial B_t(p)}u_{\infty}^2dH^{n-1}\]
i.e.
\[t^{3-n}\int_{\partial B_t(p)}|du_{\infty}|^2dH^{n-1}-F_{u_{\infty}}(t)
\ge \lambda I_{u_{\infty}}(t)\]
for a.e. $t > 0$.
We shall use the notation: $f'=df/dt$ for locally Lipschitz functions $f$ on $\mathbf{R}$ below.
By Proposition \ref{8890},  $D_{u_{\infty}}$ is locally Lipschitz function on $(0, \infty)$.
By the definition of $D_{u_{\infty}}$, Proposition \ref{0005} and Rademacher's theorem for Lipschitz functions on $\mathbf{R}$, we have
\[D_{u_{\infty}}'(t)=(2-n)t^{1-n}\int_{B_t(p)}(\mathrm{Lip}u_{\infty})^2dH^n + t^{2-n}\int_{\partial B_t(p)}(\mathrm{Lip}u_{\infty})^2dH^{n-1}\]
for a.e. $t > 0$.
Therefore, we have 
\[tD_{u_{\infty}}'(t)-(2-n)D_{u_{\infty}}(t)-F_{u_{\infty}}(t) \ge \lambda I_{u_{\infty}}(t)\]
for a.e. $t > 0$.
On the other hand, by Proposition \ref{8890}, we have $D_{u_{\infty}}'(t)=2F_{u_{\infty}}(t)/t$ for every $t > 0$.
Therefore, we have 
\[\frac{t}{2}D_{u_{\infty}}'(t)-(2-n)D_{u_{\infty}}(t) \ge \lambda I_{u_{\infty}}(t)\]
for a.e. $t>0$.
Thus we have 
\[\frac{D_{u_{\infty}}'(t)}{D_{u_{\infty}}(t)}-\frac{2(2-n)}{t}\ge \frac{2\lambda I_{u_{\infty}}(t)}{tD_{u_{\infty}}(t)}\ge \frac{2\lambda}{dt}\]
for a.e. $t >0$.
Therefore, we have
\begin{align}
\frac{D_{u_{\infty}}'(t)}{D_{u_{\infty}}(t)}&\ge \frac{1}{t}\left(\frac{2\lambda}{d}+2(2-n)\right)\\
&= \frac{1}{t}\frac{2\lambda +4d-2nd}{d} \\
&=\frac{1}{t}\frac{2d(d+n-2)+4d-2nd}{d} \\
&=\frac{2d}{t}.
\end{align}
for a.e $t >0$.
By integrating the inequality above, we have the assertion.
\end{proof}
\begin{proposition}\label{eigen}
Let $g$ be a Lipschitz function on $X$ and $f$ a $C^2$-function on $\mathbf{R}_{>0}$. 
We assume that $f(1)=1$, $\lim_{r \rightarrow 0}f(r)=0$, $g \neq 0$ and that function $u(r, x) =f(r)g(x)$ on $C(X) \setminus \{p\}$ is locally Lipschitz and harmonic.
Then, there exists $\lambda \ge n-1$ such that $\Delta_Xg=\lambda g$ and that $f(r)=r^p$.
Here $p$ is the nonnegative number satisfying $\lambda=p(p+n-2)$. 
\end{proposition}
\begin{proof}
For every $i$, we take the $i$-th eigenvalue $\lambda_i$ of $\Delta_X$ and a $i$-th eigenfunction $\phi_i$ of $\Delta_X$.
We put $g= \sum_{i=1}^{\infty}a_i\phi_i$ in $H_{1,2}(X)$.
For every function $h$ on $X$, we shall define a function $h^r$ on $\partial B_r(p)$ by
$h^r(r, x)=h(x)$.
It is easy to check that $g^r=\sum_{i=1}^{\infty}a_i\phi_i^r$ in $H_{1,2}(\partial B_r(p))$.
We remark that $\Delta_{\partial B_r(p)}\phi_i^r = \lambda_i^r \phi_i^r$ and $\lambda_i^r=r^{-2} \lambda_i$.
By \cite[Theorem $4. 15$]{di2} and Corollary \ref{har}, for every Lipschitz function $\phi \in \mathcal{K}(C(X) \setminus \{p\})$, we have
\begin{align*}
0&=\int_{C(X)}\langle du,d\phi \rangle dH^n \\
&=\int_0^{\infty}\int_{\partial B_r(p)}\left(\phi \left(-\frac{d^2f}{dr^2}(r)g(x)-\frac{n-1}{r}\frac{df}{dr}(r)g(x)\right)
+\langle d_{\partial B_r(p)}\phi, d_{\partial B_r(p)}g^r\rangle f(r)\right)dH^{n-1}dr\\
&=\int_0^{\infty}\int_{\partial B_r(p)}\phi \left(-\frac{d^2f}{dr^2}(r)g(x)-\frac{n-1}{r}\frac{df}{dr}(r)g(x)
+f(r)\sum_{i=1}^{\infty}a_i\lambda_i^r\phi_i^r \right)dH^{n-1}dr.
\end{align*}
Especially, for every Lipschitz function $a \in \mathcal{K}(\mathbf{R}_{>0})$ and Lipschitz function $b$ on $X$, we have  
\[\int_0^{\infty}a(r)\int_{\partial B_r(p)}b(x)\left(-\frac{d^2f}{dr^2}(r)g(x)-\frac{n-1}{r}\frac{df}{dr}(r)g(x)
+f(r)\sum_{i=1}^{\infty}a_i\lambda_i^r\phi_i^r\right)dH^{n-1}dr = 0.\]
Since
\[\sum_{i=1}^{\infty} (\lambda_i^r)^2 a_i^2\int_{\partial B_r(p)}(\phi_i^r)^2dH^{n-1}=\int_{\partial B_r(p)}|d_{\partial B_r(p)}g^r|^2dH^{n-1} < \infty,\]
the function 
\[-\frac{d^2f}{dr^2}(r)g(x)-\frac{n-1}{r}\frac{df}{dr}(r)g(x)
+f(r)\sum_{i=1}^{\infty}a_i\lambda_i^r\phi_i^r\]
on $\partial B_r(p)$ is in $L^2(\partial B_r(p))$.
Since the space which consist of Lipschitz functions on $\partial B_r(p)$ is dence in $L^2(\partial B_r(p))$, we have
\begin{align*}
0 &= \int_0^{\infty}a(r)\int_{\partial B_r(p)}\left|-\frac{d^2f}{dr^2}(r)g(x)-\frac{n-1}{r}\frac{df}{dr}(r)g(x)
+f(r)\sum_{i=1}^{\infty}a_i\lambda_i^r\phi_i^r\right|^2dH^{n-1}dr \\
&= \int_0^{\infty}a(r)\int_{\partial B_r(p)}\left|-\frac{d^2f}{dr^2}(r)g(x)-\frac{n-1}{r}\frac{df}{dr}(r)g(x)
+\frac{f(r)}{r^2}\sum_{i=1}^{\infty}a_i\lambda_i\phi_i(x)\right|^2dH^{n-1}dr.
\end{align*}
On the other hand, it is easy to check that the function ( of $r$)
\[\int_{\partial B_r(p)}\left|-\frac{d^2f}{dr^2}(r)g(x)-\frac{n-1}{r}\frac{df}{dr}(r)g(x)
+\frac{f(r)}{r^2}\sum_{i=1}^{\infty}a_i\lambda_i\phi_i(x)\right|^2dH^{n-1}\]
is continuous.
Therefore for every $r > 0$, there exists $A(r) \subset X$ such that $H^{n-1}(X \setminus A(r))=0$ and
\[-\frac{d^2f}{dr^2}(r)g(x)-\frac{n-1}{r}\frac{df}{dr}(r)g(x)
+\frac{f(r)}{r^2}\sum_{i=1}^{\infty}a_i\lambda_i\phi_i(x)=0\]
for every $ x \in A(r)$.
We put 
\[\lambda = \frac{d^2f}{dr^2}(1)+(n-1)\frac{df}{dr}(1).\]
Then, for every Lipschitz function  $\phi$ on $X$, we have
\[\int_X \lambda g\phi dH^{n-1}=\int_X \phi\sum_{i=1}^{\infty} a_i\lambda_i\phi_idH^{n-1}=\int_X\langle d_X\phi, d_Xg\rangle dH^{n-1}.\]
Thus, $g$ is a $\lambda$-eigenfunction.
Therefore, by \cite[Corollary $2. 5$]{di1}, we have $\lambda \ge n-1$.
For every $r > 0$, we have
\begin{align*}
0&=-\frac{d^2f}{dr^2}(r)\int_Xg^2dH^{n-1}-\frac{n-1}{r}\frac{df}{dr}(r)\int_Xg^2dH^{n-1}
+\frac{f(r)}{r^2}\int_Xg\sum_{i=1}^{\infty}a_i\lambda_i\phi_i(x)dH^{n-1}\\
&=-\frac{d^2f}{dr^2}(r)\int_Xg^2dH^{n-1}-\frac{n-1}{r}\frac{df}{dr}(r)\int_Xg^2dH^{n-1}
+\frac{f(r)}{r^2}\int_X|d_Xg|^2dH^{n-1}\\
&=-\frac{d^2f}{dr^2}(r)\int_Xg^2dH^{n-1}-\frac{n-1}{r}\frac{df}{dr}(r)\int_Xg^2dH^{n-1}
+\frac{f(r)}{r^2}\lambda \int_Xg^2dH^{n-1}.
\end{align*} 
Thus, we have 
\[-\frac{d^2f}{dr^2}(r)-\frac{n-1}{r}\frac{df}{dr}(r)
+\frac{f(r)}{r^2}\lambda =0.\]
Therefore, we have the assertion.
\end{proof}
Next corollary follows from Proposition \ref{020} and Proposition \ref{eigen} directly:
\begin{corollary}\label{022}
Let $u_{\infty}$ be a nonconstant harmonic function on $C(X)$ with $u_{\infty}(p)=0$.
We assume that $\mathrm{ord}_0u_{\infty}=\mathrm{ord}_{\infty}u_{\infty}=d < \infty$.
Then, the function $g(x)=u_{\infty}(1,x)$ on $X$ is a $d(d+n-2)$-eigenfunction of $\Delta_X$.
Moreover, we have 
$u_{\infty}(r, x)=r^dg(x)$.
\end{corollary}
\begin{corollary}
Let $u_{\infty}$ be a nonconstant harmonic function on $C(X)$. 
We assume that $u_{\infty}(p)=0$, $\mathrm{ord}_{\infty}u=d < \infty$ and that $v$ and $u_{\infty}$ are orthogonal for every homogeneous harmonic function $v$ with growth $\alpha$ satisfying $\alpha < d$.
Then, the function $g(x)=u_{\infty}(1,x)$ on $X$ is a $d(d+n-2)$-eigenfunction of $\Delta_X$.
Moreover we have 
$u_{\infty}(r, x)=r^dg(x)$.
\end{corollary}
\begin{proof}
We fix $0 < r < s < \infty$. 
By Proposition \ref{008} we have $D_{u_{\infty}}(s)/D_{u_{\infty}}(r)
\le I_{u_{\infty}}(s)/I_{u_{\infty}}(r)$.
By Proposition \ref{1234567890}, we have
\[\frac{I_{u_{\infty}}(s)}{I_{u_{\infty}}(r)}\le \left(\frac{s}{r}\right)^{2d}.\]
On the other hand, by the assumption and Proposition \ref{021}, we have
\[\frac{D_{u_{\infty}}(s)}{D_{u_{\infty}}(r)} \ge \left(\frac{s}{r}\right)^{2d}.\]
Therefore, we have $U_{u_{\infty}}(s)=U_{u_{\infty}}(r)$.
By Corollary \ref{022}, we have the assertion.
\end{proof}
For every $i$, we denote the $i$-th eigenvalue of $\Delta_X$ by $\lambda_i(X)$, $(0=\lambda_0(X) < \lambda_1(X) \le \lambda_2(X) \le \cdots)$.
For $\lambda \ge 0$, we put $E_{\lambda}(X) = \mathrm{span}\{\phi_i; \Delta_X\phi_i = \lambda_i(X) \phi_i, \ \lambda_i \le \lambda\}$.
Then, by an argument similar to the proof of \cite[Theorem $1. 67$]{co-mi3}, we have the following main theorem in this subsection.
\begin{theorem}[Harmonic functions with polynomial growth on asymptotic cones]\label{023}
For every $d \ge 0$, we have 
\[\mathrm{dim}H^d(C(X))=\mathrm{dim}E_{d(d+n-2)}(X).\]
Especially, we have $\mathrm{dim}H^d(C(X)) < \infty$.
\end{theorem}
\subsection{Gromov-Hausdorff topology on moduli space of asymptotic cones.}
In this subsection, we will study the moduli space of asymptotic cones of a fixed nonnegatively Ricci curved manifold $M$ with Euclidean volume growth.
In general, asymptotic cones of $M$ are \textit{not} unique.
See \cite{ch-co1} and \cite{Per1} for such examples.
Therefore, we shall consider the moduli space of them: $\mathcal{M}_{M} = \{X$: compact geodesic space ; $(C(X), p)$ is an asymptotic cone of $M$ $\}$.
We define a topology on $\mathcal{M}_M$ by Gromov-Hausdorff distance $d_{GH}$.
On the other hand, if we put $\hat{\mathcal{M}}_M=\{(M_{\infty}, m_{\infty}): $ an asymptotic cone of $M \}$ and define a topology on $\hat{\mathcal{M}}_{M}$ by pointed Gromov-Hausdorff topology, then the canonical map $\pi(X)=(C(X), p)$ from $\mathcal{M}_M$ to
$\hat{\mathcal{M}}_M$ give a homeomorphism.
We remark that if a sequence of asymptotic cones $(M_{\infty}^i, m_{\infty}^i)$ of $M$ converges to some proper geodesic space $(M_{\infty}^{\infty}, m_{\infty}^{\infty})$, then $(M_{\infty}^{\infty}, m_{\infty}^{\infty})$ is also an asymptotic cone of $M$.
Therefore, by Proposition \ref{090},  $\hat{\mathcal{M}}_M$ is compact, especially, $\mathcal{M}_M$ is compact.
The main result in this subsection is the following theorem.
We can regard it as ``$\mathcal{M}_M$-version'' of \cite[$(0.4)$ Theorem]{fu} by Fukaya or \cite[Theorem $7.9$]{ch-co3} by Cheeger-Colding. 
\begin{theorem}\label{10001}
If $X_i$ converges to $X_{\infty}$ in $\mathcal{M}_M$, then 
$(X_i, H^{n-1})$ converges to $(X_{\infty}, H^{n-1})$.
Moreover, we have 
\[\lim_{i \rightarrow \infty} \lambda_k(X_i)= \lambda_k(X_{\infty})\]
foe every $k \ge 1$.
Here, $\lambda_k(X)$ is the $k$-th eigenvalue of the Laplacian $\Delta_X$ on $X \in \mathcal{M}_M$.
\end{theorem}
\begin{proof}
Let $x_i$ be a point in $X_i$ and $x_{\infty}$ a point in $X_{\infty}$ satisfying that $x_i \rightarrow x_{\infty}$.
We take $r > 0$ and $\epsilon >0$. We put
$A_{\epsilon}^r(x_i)=\{(t, x) \in C(X_i); x \in B_r(x_i), 1-\epsilon \le t \le 1+\epsilon.\}$.
By Proposition \ref{10103}, we have 
\[\lim_{i \rightarrow \infty}H^n(A_{\epsilon}^r(x_i))=H^n(A_{\epsilon}^r(x_{\infty})).\]
By Proposition \ref{0005}, we have
\[H^n(A_{\epsilon}^r(x_i))=\int_{1-\epsilon}^{1+\epsilon}H^{n-1}(\partial B_t(p_i) \cap A_{\epsilon}^r(x_i))dt=C(n)\epsilon 
H^{n-1}(B_r^{X_i}(x_i))\]
for every $1 \le i \le \infty$.
Here, $p_i$ is the pole of $C(X_i)$.
Thus, we have $(X_i, H^{n-1}) \rightarrow  (X_{\infty}, H^{n-1})$.
We shall give a proof of second assertion by induction for $k$.
We fix a subsequence $\{n(i)\}_i$ of $\mathbf{N}$.
We take a  Lipschitz function on $X_{n(i)}$ satisfying $f_1^{n(i)} \in E_{\lambda _1(X_{n(i)})}(X_{n(i)})$ and
\[\frac{1}{H^{n-1}(X_{n(i)})}\int_{X_{n(i)}}(f_1^{n(i)})^2dH^{n-1}=1.\]
By the definition, we have 
\[\frac{1}{H^{n-1}(X_{n(i)})}\int_{X_{n(i)}}|df_1^{n(i)}|^2dH^{n-1}=\lambda_1(X_{n(i)}).\]
We define a harmonic function $u_1^{n(i)}$ on $C(X_{n(i)})$ by 
$u_1^{n(i)}(r, x) = r^{\alpha_1^{n(i)}}f_1^{n(i)}(x)$. 
Here $\alpha_1^{n(i)}$ is the positive number satisfying $\lambda_1(X_{n(i)}) = \alpha_1^{n(i)}(\alpha_1^{n(i)}+n-2)$.
Since $\lambda_1(X_{n(i)}) \ge n-1$, we have $\alpha_1^{n(i)} \ge 1$.
Then, by Proposition \ref{9}, we have
\begin{align*}
&\int_{B_{7}(p_{n(i)})}(\mathrm{Lip}u_1^{n(i)})^2dH^n \\
&=\int_0^{7}\int_{\partial B_r(p_{n(i)})}(\alpha_1^{n(i)})^2(r^{\alpha_1^{n(i)}-1})^2
(f_1^{n(i)})^2 dH^{n-1}dr +\int_0^{7}\int_{\partial B_r(p_{n(i)})}r^{2\alpha_1^{n(i)}-2}|d_Xf_1^{n(i)}|^2dH^{n-1}dr\\
&=\int_0^{7}(\alpha_1^{n(i)})^2r^{2\alpha_1^{n(i)}-2}r^{n-1}H^{n-1}(X_{n(i)})dr +\int_0^{7}r^{2\alpha_1^{n(i)}+n-1-2}\lambda_1(X_{n(i)})H^{n-1}(X_{n(i)})dr \\
&=H^{n-1}(X_{n(i)})\left(\frac{7^{2\alpha_1^{n(i)}+n-2}(\alpha_1^{n(i)})^2}{2\alpha_1^{n(i)}+n-2} +\frac{7^{\alpha_1^{n(i)}+n}\lambda_1(X_{n(i)})}{2\alpha_1^{n(i)}+n-2}\right).
\end{align*}
By Li-Schoen's gradient estimate (Theorem  \ref{gradient2}) and Theorem \ref{005}, we have 
\[\mathbf{Lip}(u_1^{n(i)}|_{B_2(p_{n(i)})}) \le \frac{C(n)}{H^n(B_{7}(p_{n(i)}))}\int_{B_{7}(p_{n(i)})}(\mathrm{Lip}u_1^{n(i)})^2dH^n.\]
On the other hand, by Claim \ref{eig}, we have 
\[\lambda_j(X_{n(i)}) \le C(n) \left(\frac{j}{H^{n-1}(X)}\right)^{\frac{2}{n-1}}\]
for every $j$.
Thus, we have
\[\mathbf{Lip}(u_1^{n(i)}|_{B_{2}(p_{n(i)})})\le C(n, V_M).\]
By Proposition \ref{Lips}, there exist a subsequence of $\{m(i)\}_i$ of $\{n(i)\}_i$, a Lipschitz harmonic function $u_1^{\infty}$ on $B_{2}(p_{\infty})$ 
, a Lipschitz function $f_1^{\infty}$ on $X_{\infty}$ and a nonnegative real number $\alpha_1^{\infty}$ such that
$u_1^{m(i)} \rightarrow u_1^{\infty}$ on $B_2(p_{\infty})$, $f_1^{m(i)} \rightarrow f_1^{\infty}$ on $X_{\infty}$ and that $\alpha_1^{m(i)} \rightarrow \alpha_1^{\infty}$.
Thus, we have $u_1^{\infty}(r, x)=r^{\alpha_1^{\infty}}f_1^{\infty}(x)$ on $B_2(p_{\infty})$,
\[\lim_{i \rightarrow \infty} \int_{X_{m(i)}}(f_1^{m(i)})^2dH^{n-1}=\int_{X_{\infty}}(f_1^{\infty})^2dH^{n-1}.\]
On the other hand, by Proposition \ref{002} and Theorem \ref{005}, we have 
\begin{align*}
\lim_{i \rightarrow \infty}\int_{1-\epsilon}^1t^{3-n}\int_{\partial B_t(p_{m(i)})}|d_{\partial B_t(p_{m(i)})}u_1^{m(i)}|^2dH^{n-1}dt &=\lim_{i \rightarrow \infty}\left(\int_{1-\epsilon}^1tD_{u_1^{m(i)}}(t)dt-\int_{1-\epsilon}^1F_{u_1^{m(i)}}(t)dt\right)\\
&=\int_{1-\epsilon}^1tD_{u_1^{\infty}}(t)dt-\int_{1-\epsilon}^1F_{u_1^{\infty}}(t)dt \\
&=\int_{1-\epsilon}^1t^{3-n}\int_{\partial B_t(p_{\infty})}|d_{\partial B_t(p_{\infty})}u_1^{\infty}|^2dH^{n-1}dt
\end{align*}
for every $0 < \epsilon <1$.
Since $|d_{\partial B_t(p_{m(i)})}u_1^{m(i)}|^2=t^{2\alpha_1^{m(i)}-2}|d_{X(m(i))}u_1^{m(i)}|^2$, we have
\begin{align*}
\int_{1-\epsilon}^1t^{3-n}\int_{\partial B_t(p_{m(i)})}|d_{\partial B_t(p_{m(i)})}u_1^{m(i)}|^2dH^{n-1}dt 
&=\int_{1-\epsilon}^1t^{3-n}t^{2\alpha_1^{m(i)}-2}t^{n-1}\int_{X_{m(i)}}|d_{X_{m(i)}}u_1^{m(i)}|^2dH^{n-1}dt \\
&=\int_{1-\epsilon}^1t^{2\alpha_1^{m(i)}}\lambda_1(X_{m(i)})H^{n-1}(X_{m(i)})dt \\
&=\frac{1-(1-\epsilon)^{2\alpha_1^{m(i)}+1}}{2\alpha_1^{m(i)}+1}\lambda_1(X_{m(i)})H^{n-1}(X_{m(i)}).
\end{align*}
Similarly, we have 
\[\int_{1-\epsilon}^1t^{3-n}\int_{\partial B_t(p_{\infty})}|d_{\partial B_t(p_{\infty})}u_1^{\infty}|^2dH^{n-1}dt=
\frac{1-(1-\epsilon)^{2\alpha_1^{\infty}+1}}{2\alpha_1^{\infty}+1}\int_{X_{\infty}}|df_1^{\infty}|^2dH^{n-1}.\]
Therefore, we have 
\[\lim_{i \rightarrow \infty}\frac{1}{H^{n-1}(X_{m(i)})}\int_{X_{m(i)}}|df_1^{m(i)}|^2dH^{n-1}=\lim_{i \rightarrow \infty}\lambda_1(X_{m(i)})=\frac{1}{H^{n-1}(X_{\infty})}\int_{X_{\infty}}|df_1^{\infty}|^2dH^{n-1}.\]
Therefore, since $\{n(i)\}_i$ is arbitrary, we have 
\[\liminf_{i \rightarrow \infty}\lambda_1(X_i)\ge \lambda_1(X_{\infty}).\]
On the other hand, by \cite[Theorem $7. 1$]{ch-co3}, we have 
\[\limsup_{i \rightarrow \infty}\lambda_1(X_i)\le \lambda_1(X_{\infty}).\]
Therefore we have
\[\lim_{i \rightarrow \infty}\lambda_1(X_i) = \lambda_1(X_{\infty}),\]
$f_1^{\infty}$ is $\lambda_1(X_{\infty})$-eigenfunction.

\

Next, we fix an integer $k \ge 2$.
We assume that 
\[\lim_{i \rightarrow \infty}\lambda_j(X_i)=\lambda_j(X_{\infty})\]
for every $1 \le j \le k-1$
and that for every subsequence $\{n(i)\}_i$ of $\mathbf{N}$, there exist a subsequence $\{m(i)\}_i$ of $\{n(i)\}_i$, $\lambda_j(X_{m(i)})$-eigenfunction $f_j^{m(i)}$ on $X_{m(i)}$ and
$\lambda_j(X_{\infty})$-eigenfunction $f_j^{\infty}$ on $X_{\infty}$ such that 
$f_j^{m(i)} \rightarrow f_j^{\infty}$ on $X_{\infty}$, $\mathbf{Lip}(f_j^{m(i)}|_{B_2(p_{m(i)})}) \le C(n, j,  V_M)$ for every $1 \le j \le k-1$ and that 
\[\frac{1}{H^{n-1}(X_{m(i)})}\int_{X_{m(i)}}f_l^{m(i)}f_j^{m(i)}dH^{n-1}=\delta_{jl}\]
for every $1 \le j \le l \le k-1$.
Especially, $\{f_j^{\infty}\}_{1 \le j \le k-1}$ are linearly independent in $L^2(X_{\infty})$. 
We fix a subsequence $\{n(i)\}_i$ of $\mathbf{N}$ and take a subsequece $\{m(i)\}_i$ of $\{n(i)\}_i$ as above.
We also take a $\lambda_k(X_{m(i)})$-eigenfunction $f_k^{m(i)}$ such that 
\[\frac{1}{H^{n-1}(X_{m(i)})}\int_{X_{m(i)}}(f_k^{m(i)})^2dH^{n-1}=1.\]
We define
a harmonic function $u_k^{m(i)}$ on $C(X_{m(i)})$ by 
$u_k^{m(i)}(r, x) = r^{\alpha_k^{m(i)}}f_k^{m(i)}(x)$. 
Here $\alpha_k^{m(i)}$ is the positive number  satisfying $\alpha_k^{m(i)}(\alpha_k^{m(i)}+n-2)=\lambda_k(X_{m(i)})$.

By Proposition \ref{Lips} and an argument similar to one of the case $k=1$, we can assume that 
there exist a locally Lipschitz harmonic function $u_k^{\infty}$ on $C(X_{\infty})$, a Lipschitz function $f_k^{\infty}$ on $X_{\infty}$ and a nonnegative number $\alpha_k^{\infty}$
such that
$\mathbf{Lip}(u_k^{m(i)}|_{B_2(p_{m(i)})})\le C(n, k, V_M)$,
$\mathbf{Lip}f_k^{m(i)} \le C(n, k, V_M)$, $u_k^{m(i)} \rightarrow u_k^{\infty}$ on $C(X_{\infty})$, 
$f_k^{m(i)} \rightarrow f_k^{\infty}$ on $X_{\infty}$ and $\alpha_k^{m(i)} \rightarrow \alpha_k^{\infty}$.
Thus, we have $u_k^{\infty}(r, x) = r^{\alpha_k^{\infty}}f_k^{\infty}(x)$.
By an argument similar to one of the case $k=1$, we have 
\[\lim_{i \rightarrow \infty} \int_{X_{m(i)}}|df_k^{m(i)}|^2dH^{n-1}=\int_{X_{\infty}}|df_k^{\infty}|^2dH^{n-1}.\]
On the other hand, by Proposition \ref{10103}, 
\[\lim_{i \rightarrow \infty}\int_{X_{m(i)}}f_j^{m(i)}f_l^{m(i)}dH^{n-1}=\int_{X_{\infty}}f_j^{\infty}f_l^{\infty}dH^{n-1}.\]
for every $1 \le j \le l \le k$.
Thus, we have $f_k^{\infty} \in (\mathrm{span}\{f_1^{\infty}, _{\cdots}, f_{k-1}^{\infty}\})^{\perp}$ and $f_k^{\infty} \neq 0$.
Therefore, by min-max principle, we have 
\[\lambda_k(X_{\infty}) \le \frac{\int_{X_{\infty}}|df_k^{\infty}|^2dH^{n-1}}{\int_{X_{\infty}}(f_k^{\infty})^2dH^{n-1}}.\]
Since $\{n(i)\}_i$ is arbitrary, we have 
\[\lambda_k(X_{\infty}) \le \liminf_{i \rightarrow \infty}\lambda_k(X_i).\]
On the other hand, by  \cite[Theorem $7. 1$]{ch-co3}, we have 
\[\limsup_{i \rightarrow \infty}\lambda_k(X_i) \le \lambda_k(X_{\infty}).\]
Therefore, we have
\[\lim_{i \rightarrow \infty}\lambda_k(X_i) = \lambda_k(X_{\infty}),\]
$f_k^{\infty}$ is a $\lambda_k(X_{\infty})$-eigenfunction.
Thus, by induction, we have the assertion.
\end{proof}
\begin{remark}
By the proof of Theorem \ref{10001}, with same assumption as in Theorem \ref{10001},  
if a sequence of $\lambda_k(X_i)$-eigenfunction $f^i_k$ on $X_i$ converges to some Lipschitz function $f_k^{\infty}$ on $X_{\infty}$,
then $f^{\infty}_k$ is also a $\lambda_k(X_{\infty})$-eigenfunction.
\end{remark}

\subsection{Asymptotic behavior of spaces of harmonic functions on asymptotic cones}
In this subsection, we shall give a \textit{Weyl type asymptotic formula} for harmonic functions on asymptotic cones of a fixed nonnegatively 
Ricci curved manifold $M$ with Euclidean volume growth, as in \cite{co-mi6} by Colding-Minicozzi. See \cite[Theorem $0. 26$]{co-mi6}, \cite[Proposition $6. 1$]{co-mi6} and Corollary \ref{coweyl}.
On asymptotic cones of such manifolds, we can give a Weyl type \textit{two-sided bound} asymptotic formula.
See Theorem \ref{weyl}.
\begin{proposition}\label{20001}
For every $n$-dimensional complete Riemannian manifold $M$ with $\mathrm{Ric}_M \ge 0$ and $V_M >0$, $(M_{\infty}, m_{\infty}) \in \hat{\mathcal{M}}_M$ and $d >0$,
we have $\mathrm{dim}H^d(M_{\infty}) \le C(n)d^{n-1}$.
Moreover, for every $V>0$, there exists $d(V, n) > 1$ such that for every 
$n$-dimensional complete Riemannian manifold $M$ with $\mathrm{Ric}_M \ge 0$ and $V_M \ge V$, $d > d(V, n)$ and $(M_{\infty}, m_{\infty}) \in \hat{\mathcal{M}}_M$, we have
\[\mathrm{dim}H^d(M_{\infty})\le C(n)V_Md^{n-1}.\] 
\end{proposition}
\begin{proof}
This follows from proofs of \cite[Proposition $3.1$]{co-mi6}, \cite[Proposition $6. 1$]{co-mi6} and Theorem \ref{023}.
We shall introduce important ideas used in proofs of their propositions and give an outline of a proof of our assertion only.
We fix $V>0$,
an $n$-dimensional complete Riemannian manifold $M$ with $\mathrm{Ric}_M \ge 0$ and $V_M \ge V$
and $(M_{\infty}, m_{\infty}) \in \hat{\mathcal{M}}_{M}$. 
There exists a compact geodesic space $X$ such that $(M_{\infty}, m_{\infty})=(C(X), p)$. 
We take $d_1=d_1(n) \ge 1 $ satisfying that $d(d+n-2)\le 2d^2$ for every $d\ge d_1$.
We take an $i$-th eigenfunction $u_i$ of $\Delta_X$ and the $i$-th eigenvalue $\lambda_i(X)$ of $\Delta_X$ satisfying 
\[\int_Xu_iu_jdH^{n-1}=\delta_{ij}.\] 
We put $N_d= \max \{l \in \mathbf{N}; \lambda_l(X) \le d(d+n-2)\}$.
Then, we have 
\[\int_{X}|du_i|^2dH^{n-1}=\lambda_i(X)\le d(d+n-2)\]
for every $1 \le i \le N_d$.
On the other hand, by the proof of  \cite[Proposition $6.1$]{co-mi6} (and Proposition \ref{9009}), there exists $d_2=d_2(n, V_M)\ge d_1$ such that 
for every $d \ge d_2$ and $\{x_i\}_{1 \le i \le l}$ which is a maximal $1/d$-separated subset of $X$, we have
\[l\le C(n)V_Md^{n-1}.\]
We fix $C>1$ and $d\ge d_2$. 
(We will decide $C$ depending only on $n$ later.) 
Let $\{x_j\}_{1\le j \le l}$ be a maximal $1/(Cd)$-separated subset of $X$.
We put $\mathcal{V}= \mathrm{span} \{u_i; 1 \le i \le N_d\}$.
We define a linear map $\mathcal{M}$ from $\mathcal{V}$ to $\mathbf{R}^{l}$ by
\[\mathcal{M}(v)=\left(\int_{B_{2/Cd}(x_1)}vdH^{n-1}, _{\cdots}, \int_{B_{2/Cd}(x_l)}vdH^{n-1}\right).\]
We put $\mathcal{K}=\mathrm{Ker}\mathcal{M}$.
Let $w_1, _{\cdots}, w_k$ be an $L^2(X)$-orthonormal basis of $\mathcal{K}$.
We take $w_{k+1}, _{\cdots}, w_{N_d} \in \mathcal{V}$ satisfying that $\{w_i\}_{1 \le i \le N_d}$ are an $L^2(X)$-orthonormal basis of $\mathcal{V}$.
By Poincar\' e inequality on $X$ (see  \cite[Lemma $4. 3$]{di2}), we have
\[\int_{B_{2/Cd}(x_i)}w_j^2dH^{n-1}\le \frac{C(n)}{(Cd)^2}\int_{B_{2/Cd}(x_i)}|dw_j|^2dH^{n-1}\]
for every $1 \le j \le k$ and $1 \le i \le l$.
Therefore, we have
\[1 \le \sum_{i=1}^l\int_{B_{2/Cd}(x_i)}w_j^2dH^{n-1}\le \frac{C(n)}{(Cd)^2}\sum_{i=1}^l\int_{B_{2/Cd}(x_i)}|dw_j|^2dH^{n-1}
\le \frac{C(n)}{(Cd)^2}\int_{X}|dw_j|^2dH^{n-1}\]
for $1 \le j \le k$.
Thus we have 
\[k \le \frac{C(n)}{(Cd)^2}\sum_{j=1}^k\int_{X}|dw_j|^2dH^{n-1} \le \frac{C(n)}{(Cd)^2}\sum_{j=1}^{N_d}\int_{X}|dw_j|^2dH^{n-1}
\le \frac{C(n)}{(Cd)^2}2d^2N_d\le \frac{C(n)}{C^2}N_d.\]
We put $C=\sqrt{2C(n)}$ for $C(n)$ as above.
Then we have $k \le N_d/2$.
Since $N_d= k + \mathrm{dim}(\mathrm{Image}\,\mathcal{M})$, we have $N_d \le 2l \le C(n)V_Md^{n-1}$.
On the other hand, by Theorem \ref{023}, we have $\mathrm{dim}H^{d}(M_{\infty})\le N_d$.
Therefore, we have the assertion.
\end{proof}
\begin{proposition}\label{20002}
For every $V>0$, there exists $d(V, n) > 1$ such that 
\[\mathrm{dim}H^d(M_{\infty})\ge C(n)V_Md^{n-1}\]
holds for every $n$-dimensional complete Riemannian manifold $M$ with $\mathrm{Ric}_M \ge 0$ and $V_M \ge V$, $d > d(V, n)$ and $(M_{\infty}, m_{\infty}) \in \hat{\mathcal{M}}_M$.
\end{proposition}
\begin{proof}
We fix $V>0$, an $n$-dimensional complete Riemannian manifold $M$ with $\mathrm{Ric}_M \ge 0$ and $V_M \ge V$, and $X \in \mathcal{M}_M$.
First, we remark the following.
This follows from Proposition \ref{9009}, directly.
\begin{claim}\label{3003}
Let $\epsilon$ be a positive number, $k$ a positive integer and $\{x_i\}_{1 \le i \le k}$ points in $X$.
We assume that $\{x_i\}_{1 \le i \le k}$ are an $\epsilon$-separated subset of $X$.
Then we have $k \le C(n)/\epsilon^{n-1}$.
\end{claim}
We shall give an upper bound of the first eigenvalue for Dirichlet problem on each balls:
\begin{claim}\label{4004}
We have 
\[\inf_{k \in \mathcal{K}(B_r(x)), k \neq 0}\frac{\int_{B_r(x)}|d_Xk|^2dH^{n-1}}{\int_{B_r(x)}k^2dH^{n-1}}\le \frac{C(n)}{r^2}\]
for every $x \in X$ and $0 < r \le \pi$.
\end{claim}
The proof is as follows.
We define a Lipschitz function $k$ on $X$ by $k(w)=\max \{r/2 - \overline{x,w}, 0\}$.
By the definition, we have $k \in \mathcal{K}(B_r(x))$,
\[\int_{B_r(x)}|d_Xk|^2dH^{n-1}=H^{n-1}(B_{\frac{r}{2}}(x))\]
and
\[\int_{B_r(x)}k^2dH^{n-1} \ge \int_{B_{\frac{r}{4}(x)}(x)}k^2dH^{n-1}\ge \int_{B_{\frac{r}{4}}(x)}\frac{r^2}{16}dH^{n-1} \ge \frac{r^2}{16}H^{n-1}(B_{\frac{r}{4}}(x)).\]
By Proposition \ref{9009}, we have
\[\frac{\int_{B_r(x)}|dk|^2dH^{n-1}}{\int_{B_r(x)}k^2dH^{n-1}}
\le \frac{16}{r^2}\frac{H^{n-1}(B_r(x))}{H^{n-1}(B_{\frac{r}{4}}(x))} \le \frac{C(n)}{r^2}.\]
Thus, we have Claim \ref{4004}.
\begin{claim}\label{3300}
We have
\[\limsup_{r \rightarrow 0}\frac{H^{n-1}(B_r(x))}{r^{n-1}}\le C(n)\]
for every $x \in X$.
\end{claim}
The proof is as follows.
For every sufficiently small $r > 0$, we put $A=\{(s, w) \in C(X); 1-r \le s \le 1+r, w \in B_r(x)\}$.
By Proposition \ref{0005}, we have
\begin{align}
H^n(B_{5r}(1, x)) &= \int_{1-r}^{1+r}H^{n-1}(\partial B_t(p) \cap B_{5r}(1, x))dt \\
& \ge \int_{1-r}^{1+r}H^{n-1}(\partial B_t(p) \cap A)dt \\
& \ge C(n) r H^{n-1}(B_r(x)).
\end{align}
Therefore, by Bishop-Gromov volume comparison theorem on limit spaces, we have Claim \ref{3300}.
\begin{claim}\label{eig}
We have 
\[\lambda_d(X) \le C(n) \left(\frac{d}{H^{n-1}(X)}\right)^{\frac{2}{n-1}}\]
for every $d \ge 1$.
\end{claim}
The proof is as follows.
We fix $0 < C < 1$.
(We will decide $C$ depending only on $n$ later.)
We put 
\[\epsilon=C \left(\frac{H^{n-1}(X)}{d}\right)^{\frac{1}{n-1}} \]
and take maximum $\epsilon$-separated subset $\{x_i\}_{1 \le i \le k}$ of $X$. 
By Claim \ref{3003}, we have $k \le C(n)/\epsilon^{n-1} \le C(n)d^{n-1}/(C^{n-1}H^{n-1}(X))$.
On the other hand, we have
\[\sum_{i=1}^kH^{n-1}(B_{2\epsilon}(x_i)) \ge H^{n-1}(X).\]
By Claim \ref{3300} and Proposition \ref{9009}, we have $H^{n-1}(B_{5\epsilon}(x_i))\le C(n) \epsilon^{n-1}$.
Thus, we have
\[H^{n-1}(X) \le \sum_{i=1}^kH^{n-1}(B_{2\epsilon}(x_i)) \le k C(n) \epsilon^{n-1}.\]
Therefore, we have
\[k \ge \frac{C_1(n)H^{n-1}(X)}{\epsilon^{n-1}}=\frac{C_1(n)}{C^{n-1}}\frac{H^{n-1}(X)d}{H^{n-1}(X)}\ge \frac{C_1(n)}{C^{n-1}}d.\]
Here $C_1(n)$ is a sufficiently small positive constant depending only on $n$.
We define $C$ by $C = C_1(n)^{1/(n-1)}$. 
Then, we have $k \ge d$.
By Claim \ref{4004}, for every $1 \le i \le k$, there exists $\phi_i \in \mathcal{K}(B_{\epsilon/10}(x_i))$ such that $\phi_i \neq 0$ and 
\[\frac{\int_{B_{\epsilon/10}(x_i)}|d\phi_i|^2dH^{n-1}}{\int_{B_{\epsilon/10}(x_i)}(\phi_i)^2dH^{n-1}}\le \frac{C(n)}{\epsilon^2}.\]
Since $\{B_{\epsilon/10}(x_i)\}_i$ are pairwise disjoint,  $\{\phi_i\}_i$ are linearly independent in $L^2(X)$.
Then, for every $a_1, _{\cdots}, a_k \in \mathbf{R}$ satisfying $\sum_{i=1}^k (a_i)^2 \neq 0$, we have  
\begin{align}
\int_X \left|d\left(\sum_{i=1}^k a_i \phi_i\right)\right|^2dH^{n-1} &= \sum_{i=1}^k \int_X|d(a_i\phi_i)|^2dH^{n-1} \\
&\le \sum_{i=1}^k \frac{C(n)}{\epsilon^2}\int_X (a_i\phi_i)^2dH^{n-1} \\
&=\frac{C(n)}{\epsilon^2}\int_X\left|\sum_{i=1}^k a_i \phi_i\right|^2dH^{n-1}.
\end{align}
Thus, by min-max principle, we have $\lambda_k(X) \le C(n)/\epsilon^2$.
Therefore, we have 
\[\lambda_d(X) \le \lambda_k(X) \le \frac{C(n)}{\epsilon^2} \le C(n) \left(\frac{d}{H^{n-1}(X)}\right)^{\frac{2}{n-1}}.\] 
Thus, we have Claim \ref{eig}.
\

The assertion follows from Claim \ref{eig} and Theorem \ref{023}.
\end{proof}
The main result in this subsection is the following:
\begin{theorem}[Weyl type asymptotic formula on asymptotic cones]\label{weyl}
For every $V > 0$, there exists $d(n, V) \ge 1$ such that 
\[C(n)^{-1}V_Md^{n-1}\le \mathrm{dim}H^d(M_{\infty})\le C(n)V_Md^{n-1}\]
holds for every $n$-dimensional complete Riemannian manifold $M$ with $\mathrm{Ric}_M \ge 0$ and $V_M \ge V$, $d \ge d(n, V)$ and $(M_{\infty}, m_{\infty}) \in \hat{\mathcal{M}}_M$.
\end{theorem}
\begin{proof}
It follows from Proposition \ref{20001} and \ref{20002} directly.
\end{proof}
\subsection{A dimension comparison theorem and Liouville type theorem}
In this subsection, we shall give a comparison theorem for dimensions between a space of harmonic functions on 
a fixed nonnegatively Ricci curved manifold with Euclidean volume growth,
and one on an asymptotic cone of the manifold (Theorem \ref{com} below).
Essential tools to prove it are \cite[Lemma $3. 1$]{co-mi2} (or \cite[Lemma $7. 1$]{co-mi3}) and several properties of frequency functions on asymptotic cones given in section $5$. 
As a corollary, we will give a Liouville type theorem on the manifold. See Corollary \ref{Liouville}.
First, we shall introduce an important mean value inequality for subharmonic functions
on nonnegatively Ricci curved manifolds by Li-Schoen:
\begin{theorem}[Li-Schoen, \cite{li-sh}]\label{gradient2}
Let $M$ be a complete $n$-dimensional Riemannian manifold with $\mathrm{Ric}_M \ge 0$, $m$ a point in $M$ and $R$ a positive number.
Then for every nonnegative subharmonic function $f$ on $B_{3R/2}(m)$, we have 
\[\sup_{B_R(m)}f \le \frac{C(n)}{\mathrm{vol}\,B_{\frac{3R}{2}}(m)}\int_{B_{\frac{3R}{2}}(m)}fd\mathrm{vol}.\]
\end{theorem}
We remark that if $\mathrm{Ric}_M \ge 0$, then, by Bochner's formula, for every harmonic function $h$ on $B_R(m)$, we have, $|\nabla h|^2$ is a subharmonic function.
We fix an $n$-dimensional complete Riemannian manifold $M$ with $\mathrm{Ric}_M \ge 0$ and $V_M > 0$ below.
\begin{theorem}\label{com}
For every $d \ge 0$, $\epsilon > 0$ and nonnegative integers $k \le \mathrm{dim}H^d(M)-1$ and $0 \le l \le k$, 
, there exists $(M_{\infty}, m_{\infty}) \in \hat{\mathcal{M}}_M$ such that 
\[l \le \mathrm{dim}H^{\frac{k}{k-l+1}\left(d-1 + \frac{n}{2}\right)+1-\frac{n}{2}+\epsilon}(M_{\infty}) -1.\]
\end{theorem}
\begin{proof}
Without loss of generality, we can assume that $k \ge 1$ and $l \ge 1$.
We take linearly independent harmonic functions $u_1, u_2, _{\cdots}, u_k \in H^d(M)$ satisfying $u_i(m)=0$.
We put
\[J_r(u_i, u_j)=\int_{b^{g_M} \le r}\langle du_i, du_j\rangle d\mathrm{vol}^{g_M}\]
for every $r > 0$.
We define $u_i=\sum_{j=1}^{i-1}\lambda_{ji}(r)u_j +w_{i, r}$ by satisfying $J_r(w_{i, r}, w_{j, r})=0$ for $i \neq j$ and
put 
\[f_i(r)=\int_{b^{g_M} \le r}|dw_{i,r}|^2d\mathrm{vol}^{g_M}.\]
\begin{claim}\label{100002}
We have the following:
\begin{enumerate}
\item There exists $K >0$ such that  $f_i(r) \le K(r^{2d-2+n}+1)$ for every $i =1, _{\cdots}, k$ and $r>0$.
\item $f_i > 0$.
\item $f_i(r) \le f_i(s)$ for $r \le s$.
\item $f_i$ is a barrier for $t^{n-2}D_{w_{i,s}}^{g_M}(t)$ at every $s > 0$. 
Here, for functions $g, h$ on $\mathbf{R}$ and a real number $r \in \mathbf{R}$, we say that $f$ is a barrier for $g$ at $r$ if $f(r)=g(r)$ and $f(s)\le g(s)$ for $s<r$.
(see also \cite[Definition $4. 6$]{co-mi2}).
\end{enumerate}
\end{claim}
By the trivial monotonicity of $t^{n-2}D_u^{g_M}(t)$ and an argument similar to the proof of \cite[Proposition $8. 6$]{co-mi3} (or \cite[Proposition $4.7$]{co-mi2}), we have Claim \ref{100002}.
\

We put $\lambda = \frac{k}{k-l+1}$.
By \cite[Lemma $3. 1$]{co-mi2}, for every $N \in \mathbf{N}_{\ge 2}$, there exist a subsequence $\{m(N, i)\}_{i \in \mathbf{N}}$ of $\mathbf{N}$ and a pairwise distinct integers
$\alpha_1^N, _{\cdots}, \alpha_{l}^N \in \{1, _{\cdots}, k\}$ such that $f_j(N^{m(N, i)+1}) \le 2N^{\lambda(2d-2+n)}f_j(N^{m(N, i)})$
for every $j \in \{\alpha_1^N, _{\cdots}, \alpha_{k}^N\}$
and $i \in \mathbf{N}$.
By Claim \ref{100002}, we have 
\[\frac{f_j(N^{m(N, i)+1})}{f_j(N^{m(N, i)})}\ge \frac{(N^{m(N, i)+1})^{n-2}D_{w_{j, N^{m(N, i)+1}}}^{g_M}(N^{m(N, i)+1})}{(N^{m(N, i)})^{n-2}D_{w_{j, N^{m(N, i)+1}}}^{g_M}(N^{m(N, i)})}.\]
Thus, we have
\[\frac{D_{w_{j, N^{m(N, i)+1}}}^{g_M}(N^{m(N, i)+1})}{D_{w_{j, N^{m(N, i)+1}}}^{g_M}(N^{m(N, i)})}\le 2N^{\lambda (2d-2+n)+2-n}.\]
We define a harmonic function $w_j^{N, i}$ on $B_{N/10}^{(N^{m(N, i)})^{-2}g_M}(m)$  by
\begin{align}
w_j^{N, i}(w)&=w_{j, N^{m(N, i)+1}} \\
& \times \left(N^{m(N, i)} \sqrt{\frac{1}{\mathrm{vol}^{g_M}(\{b^{g_M} \le N^{m(N, i)}\})}
\int_{b^{g_M} \le N^{m(N, i)}}|dw_{j, N^{m(N, i)+1}}|^2d\mathrm{vol}^{g_M}}\right)^{-1}.
\end{align}
We assume that $N$ is sufficiently large below.
Then, for $x_1, x_2 \in B_{N/10}^{(N^{m(N, i)})^{-2}g_M}(m)$, by Li-Schoen's gradient estimate, we have
\begin{align}
&|w_j^{N, i}(x_1)-w_j^{N,i}(x_2)| \\
&\le \sup_{B_{N^{m(N, i)}\frac{N}{5}}(m)}|\nabla w_{j, N^{m(n, i)+1}}|\overline{x_1, x_2}^{g_M}\\
& \times \left(N^{m(N, i)} \sqrt{\frac{1}{\mathrm{vol}^{g_M}(\{b^{g_M} \le N^{m(N, i)}\})}
\int_{b^{g_M} \le N^{m(N, i)}}|dw_{j, N^{m(N, i)+1}}|^2d\mathrm{vol}^{g_M}}\right)^{-1}\\
&\le C(n)\sqrt{\frac{1}{\mathrm{vol}^{g_M}(\{b^{g_M} \le N^{m(N, i)}\frac{2N}{3}\})}
\int_{b^{g_M} \le N^{m(N, i)}\frac{2N}{3}}|dw_{j, N^{m(N, i)+1}}|^2d\mathrm{vol}^{g_M}}\\
& \times \left(\sqrt{\frac{1}{\mathrm{vol}^{g_M}(\{b^{g_M} \le N^{m(N, i)}\})}
\int_{b^{g_M} \le N^{m(N, i)}}|dw_{j, N^{m(N, i)+1}}|^2d\mathrm{vol}^{g_M}}\right)^{-1} \times \overline{x_1, x_2}^{(N^{m(N, i)})^{-2}g_M}\\
&\le C(n) N^{\lambda (d-1 + n/2)+1-n/2}\overline{x_1, x_2}^{(N^{m(N, i)})^{-2}g_M}.
\end{align}
By Proposition \ref{Lips} and compactness of $\mathcal{M}_M$, without loss of generality, we can assume that there exist $X_N \in \mathcal{M}_M$ and Lipschitz functions $w_j^{N, \infty}$ on $B_{N/10}(p_N)$
such that $(M, m, (N^{m(N, i)})^{-1}d_M, w_j^{N, i}) \rightarrow (C(X_N), p_N, w_j^{N, \infty})$
$(j \in \{\alpha_1, \cdots, \alpha_l\})$. 
On the other hand, 
\begin{align}
&\frac{1}{\mathrm{vol}^{(N^{m(N, i)})^{-2}g_M}\, B^{(N^{m(N, i)})^{-2}g_M}_1(m) } \\
& \times \int_{B^{(N^{m(N, i)})^{-2}g_M}_1(m)}|d^{(N^{m(N, i)})^{-2}g_M}w_j^{N, i}|^2d\mathrm{vol}^{(N^{m(N, i)})^{-2}g_M} \\
&= \frac{1}{\mathrm{vol}^{g_M}\,B_{N^{m(N, i)}}(m)}\int_{B_{N^{m(N, i)}}(m)}|dw_{j, N^{m(N, i)+1}}|^2(N^{m(N, i)})^2d\mathrm{vol}^{g_M}\\
& \times \left(N^{2m(N, i)} \frac{1}{\mathrm{vol}^{g_M}(\{b^{g_M} \le N^{m(N, i)}\})}
\int_{b^{g_M} \le N^{m(N, i)}}|dw_{j, N^{m(N, i)+1}}|^2d\mathrm{vol}^{g_M}\right)^{-1} \\
&= 1 \pm \Psi(i^{-1}; n, N).
\end{align}
By Corollary \ref{har} and Theorem \ref{green}, we have 
\[\frac{1}{H^n(B_1(p_N))}\int_{B_1(p_N)}|dw_j^{N, \infty}|^2dH^n=1.\]
Similarly, we have 
\[\int_{B_1(p_N)}\langle dw_i^{N, \infty}, dw_j^{N, \infty}\rangle dH^n=0 \]
for $i \neq j$.
Therefore, $\{w_j^{N, \infty}\}_j$ are linearly independent harmonic functions.
For convenience, we  shall change the notation: $\{\alpha_1^N, _{\cdots}, \alpha_{l}^N\}=\{1, _{\cdots}, l\}$.
By Proposition \ref{008}, we have 
\[ \frac{I_{w_j^{N, \infty}}(\frac{N}{100})}{I_{w_j^{N, \infty}}(1)}=\frac{U_{w_j^{N, \infty}}(1)}{U_{w_j^{N, \infty}}(\frac{N}{100})}
\frac{D_{w_j^{N, \infty}}(\frac{N}{100})}{D_{w_j^{N, \infty}}(1)}
 \le \frac{D_{w_j^{N, \infty}}(\frac{N}{100})}{D_{w_j^{N, \infty}}(1)} \le 2N^{\lambda (2d-2+n)+2-n}.\]
Thus, by Proposition \ref{006}, we have 
\[\exp \int_1^{\frac{N}{100}}2\frac{U_{w_j^{N, \infty}}(t)}{t}dt \le 2N^{\lambda (2d-2+n)+2-n}.\]
We take $1 \le \hat{l} < N/100$. Since
\[\exp \int_{\hat{l}}^{\frac{N}{100}}2\frac{U_{w_j^{N, \infty}}(t)}{t}dt \le 2N^{\lambda (2d-2+n) +2-n},\]
by Proposition \ref{008}, we have 
\[\left(\frac{N}{100\hat{l}}\right)^{2U_{w_j^{N, \infty}}(\hat{l})}\le 2N ^{\lambda (2d-2+n) +2-n}\]
i.e.
\[2U_{w_j^{N, \infty}}(\hat{l}) \le \frac{\log N}{\log N - \log (100\hat{l})}+ \frac{\log N}{\log N - \log (100\hat{l})}(\lambda (2d-2+n)+2-n).\]
Therefore, for every $\hat{l} \ge 1$, there exists $N_{\hat{l}}$ such that $U_{w_j^{N, \infty}}(a)\le \lambda (d-1+n/2)+1 -n/2+ \epsilon$
for every $N \ge N_{\hat{l}}$ and $1 \le a \le \hat{l}$.
We take $x_1 \in B_{\frac{\hat{l}}{10}}(p_N)$.
By Li-Schoen's gradient estimate and Theorem \ref{005}, we have
\begin{align}
\mathrm{Lip}w_j^{N, \infty}(x_1) &\le C(n) \sqrt{\frac{1}{H^n(B_{\hat{l}}(p_N))}\int_{B_{\hat{l}}(p_N)}(\mathrm{Lip}w_j^{N, \infty})^2
dH^n} \\
&\le C(n, V_M) \sqrt{\hat{l}^{-n}\int_{B_{\hat{l}}(p_N)}|dw_j^{N, \infty}|^2dH^n} \\
&\le C(n, V_M, \lambda, d) \sqrt{\hat{l}^{-1-n}\int_{\partial B_{\hat{l}}(p_N)}|w_j^{N, \infty}|^2dH^n} \\
&\le C(n, V_M, \lambda, d) \hat{l}^{-1} \sqrt{\frac{1}{H^{n-1}(\partial B_{\hat{l}}(p_N))}\int_{\partial B_{\hat{l}}(p_N)}|w_j^{N, \infty}|^2dH^n}.
\end{align}
On the other hand, by Proposition \ref{008}, we have
\begin{align}
I_{w_j^{N, \infty}}(\hat{l})&=\exp \left(\int_1^{\hat{l}}\frac{2U_{w_j^{N, \infty}}(t)}{t}dt\right)I_{w_j^{N, \infty}}(1) \\
& \le \exp \left(\int_1^{\hat{l}}\frac{\lambda (2d-2+n) + 2-n + 2\epsilon}{t}dt\right)I_{w_j^{N, \infty}}(1) \\
& \le \hat{l}^{\lambda (2d-2+n) + 2-n + 2\epsilon}I_{w_j^{N, \infty}}(1)
\end{align}
for $N \ge N_l$.
By Proposition \ref{110}, we have
\[0 \le I_{w_j^{N, \infty}}(1) \le I_{w_j^{N, \infty}}(1) U_{w_j^{N, \infty}}(1) \le D_{w_j^{N, \infty}}(1) =1.\]
Thus, we have $I_{w_j^{N, \infty}}(\hat{l}) \le \hat{l}^{\lambda (2d-2+n)+2 -n + 2\epsilon}$.
Therefore, we have 
\[\mathbf{Lip}\left(w_j^{N, \infty}|_{B_{\frac{\hat{l}}{10}}(p_N)}\right) \le C(n, V_M, \lambda, d) \hat{l}^{\lambda (d-1+n/2) -n/2+\epsilon}.\]
By Proposition \ref{Lips} and compactness of $\mathcal{M}_M$, we can assume that there exist $X_{\infty} \in \mathcal{M}_M$ and locally Lipschitz harmonic functions $w_j^{\infty} \in H^{\lambda (d-1+n/2) + 1 -n/2+ \epsilon}(C(X_{\infty}))$ such that $X_N \rightarrow X_{\infty}$ and that $w_j^{N, \infty} \rightarrow  w_j^{\infty}$ on $B_R(p_{\infty})$ for every $R>0$.
By Corollary \ref{har}, we have 
\[\frac{1}{H^n(B_1(p_{\infty}))}\int_{B_1(p_{\infty})}\langle dw_j^{\infty}, dw_i^{\infty}\rangle dH^n=\delta_{ij}.\]
Especially, $\{w_j^{\infty}\}_j$ are linearly independent nonconstant harmonic functions.
Therefore we have the assertion.
\end{proof}
As a corollary of Theorem \ref{com}, we have the following result by Colding-Minicozzi:
\begin{corollary}[Colding-Minicozzi, \cite{co-mi6}]\label{coweyl}
For every $V>0$, there exists $d(V, n) > 1$ such that 
\[\mathrm{dim}H^d(M)\le C(n)V_Md^{n-1}\] 
for every $d > d(V, n)$ and 
$n$-dimensional complete Riemannian manifold $M$ with $\mathrm{Ric}_M \ge 0$ and $V_M \ge V$.
\end{corollary}
\begin{proof}
By taking $k=[(\mathrm{dim}H^d(M)-1)/2]$ as in Theorem \ref{com}, the assertion follows from Theorem \ref{weyl} and Theorem \ref{com} directly.  
Here $[a]= \inf \{l \in \mathbf{Z}; a \le l\}$ for every $a \in \mathbf{R}$.
\end{proof}
We put $\lambda_1 = \inf \{\lambda_1(X); X \in \mathcal{M}_M\}$ and 
 define $d_1 \ge 1$ by
\[d_1= \frac{-(n-1)+\sqrt{(n-2)^2+4\lambda_1}}{2}.\]
By Theorem \ref{10001}, we have the following:
\begin{enumerate}
\item $H^d(M_{\infty})=\{$Constant functions$\}$ for every $(M_{\infty}, m_{\infty}) \in \hat{\mathcal{M}}_M$ and $0< d < d_1$.
\item $H^{d_1}(\hat{M}_{\infty}) \neq \{$Constant functions$\}$ for some $(\hat{M}_{\infty}, \hat{m}_{\infty}) \in \hat{\mathcal{M}}_M$.
\end{enumerate}
\begin{corollary}[Liouville type theorem]\label{Liouville}
We have $H^d(M)= \{$Constant functions$\}$ for every $0< d < d_1$.
\end{corollary}
\begin{proof}
We assume that the assertion is false.
We take $\epsilon > 0$ satisfying $\epsilon < d_1-d$.
By taking $k=l=1$ as in Theorem \ref{com}, there exists $(M_{\infty}, m_{\infty}) \in \hat{\mathcal{M}}_M$ such that $2 \le \mathrm{dim}H^{d + \epsilon}(M_{\infty})$.
This is a contradiction.
\end{proof}
Finally, we end this subsection by showing the following.
See also \cite[Conjecture $0.9$]{co-mi4}.
\begin{corollary}\label{freqen}
Let $d$ be a positive number and $u \in H^d(M)$.
Then we have 
\[\liminf_{t \rightarrow \infty}\left(\sup_{s \in K}U_u^{g_M}(ts)\right) \le d\]
for every compact set $K \subset (0, \infty)$.
\end{corollary}
\begin{proof}
Assume that $u$ is not a constant.
By the proof of Theorem \ref{com}, for every $\epsilon > 0$, there exist sequences of positive numbers $\{R_i\}_i$, $\{\hat{R}_i\}_i$, an asymptotic cone $(M_{\infty}, m_{\infty}) \in \hat{\mathcal{M}}_M$ and a nonconstant harmonic function $u_{\infty} \in H^{d+\epsilon}(M_{\infty})$
such that $R_i \rightarrow \infty$, $\hat{R}_i \rightarrow \infty$, $(M, m, R_i^{-1}d_{M}) \rightarrow (M_{\infty}, m_{\infty})$, $\sup_{i}\mathbf{Lip}^{R_i^{-1}d_{M}}\left((u)_{\hat{R}_i}|_{B_R^{R_i^{-1}d_{M}}(m_i)}\right)<\infty$ for every $R>0$ and that $(u)_{\hat{R}_{i}}(x_i) \rightarrow u_{\infty}(x_{\infty})$ for 
every sequence $x_i \rightarrow x_{\infty}$ with respect to the convergence $(M, m, R_i^{-1}d_{M}) \rightarrow (M_{\infty}, m_{\infty})$.
By the definition of $U_u^{g_M}(t)$, we have $U_{(u)_{\hat{R}_i}}^{R_i^{-2}g_M}(s)=U_{u}^{R_i^{-2}g_M}(s)=U_u^{g_M}(R_is)$ for every $s > 0$.
Thus, since $\lim_{i \rightarrow \infty}\left(\sup_{s \in K}|U_{(u)_{\hat{R}_i}}^{R_i^{-2}g_M}(s)-U_{u_{\infty}}(s)|\right)=0$ and $U_{u_{\infty}}\le d+\epsilon$, we have $\liminf_{t \rightarrow \infty}\left(\sup_{s \in K}U_u^{g_M}(ts)\right) \le d+\epsilon$.
Therefore, we have the assertion.

\end{proof}
\section{Stability of lower bounds on Ricci curvature via Laplacian comparison theorem}
In this section, as an application of Theorem \ref{app}, we shall establish Laplacian comparison theorem on Ricci limit spaces.
For $H \in \mathbf{R}$, we define a smooth function $\underline{k}_H$ on $\mathbf{R}$ by 
\[\underline{k}''_H(r) + H\underline{k}''_H(r) = 0, \ \ \underline{k}(0)=0, \ \ \underline{k}'_H(0)=1.\]
Here $f' = df/dr$ for every differentiable function $f$ on $\mathbf{R}$.
We remark the following:
\begin{enumerate}
\item (Laplacian comparison theorem on manifolds). For every $n$-dimensional complete Riemannian manifold $M$ with $\mathrm{Ric}_M \ge H(n-1)$ and point $p \in M$, we have
\[\Delta r_p(x) \ge -(n-1)\frac{\underline{k}'_H(\overline{p,x})}{\underline{k}_H(\overline{p,x})}\] 
for every $x \in M \setminus (C_p \cup \{p\})$.
\item For the $n$-dimensional space form $\underline{M}_H^n$ whose sectional curvature is equal to $H$ and every point $\underline{p} \in \underline{M}_H$,   
we have 
\[\Delta r_{\underline{p}}(x)=-(n-1)\frac{\underline{k}'_H(\overline{p,x})}{\underline{k}_H(\overline{p,x})}\]
for every $x \in \underline{M}_H \setminus (C_{\underline{p}} \cup \{\underline{p}\})$.
\item If an $n$-dimensional complete Riemannian manifold $M$ satisfies that 
\[\Delta r_p(x) \ge -(n-1)\frac{\underline{k}'_H(\overline{p,x})}{\underline{k}_H(\overline{p,x})}\]
for every $p \in M$ and $x \in M \setminus (C_p \cup \{p\})$, then we have $\mathrm{Ric}_{M} \ge H(n-1)$.
\end{enumerate}
See for instance \cite{ch1}, \cite{ch-co1}, \cite{KS1}, \cite{oh} and \cite{Vi2}.
The following theorem is the main result in this subsection.
This formulation is given in \cite{KS1} by Kuwae-Shioya on weighted Alexandrov spaces.
\begin{theorem}[Laplacian comparison theorem]\label{26}
Let $H$ be a real number, $(Y, y, \upsilon)$ a $(n, H)$-Ricci limit space $(n \ge 2)$,
$x$ a point in $Y$ and $R$ a positive number and $f$ a nonnegative valued Lipschitz function on $B_R(x)$.
Then, we have 
\[\int_{B_R(x)}\langle df, dr_x\rangle d \upsilon \ge -(n-1)\int_{B_R(x)}\frac{\underline{k}'_{H}(\overline{x,w})}{\underline{k}_H(\overline{x, w})}f(w) d \upsilon.\]
\end{theorem}
\begin{proof}
Let $(M_i, m_i, \underline{\mathrm{vol}}) \rightarrow (Y, y, \upsilon)$ with $\mathrm{Ric}_{M_i}\ge H_i(n-1)$ satisfying $H_i \rightarrow H$.
We take $L \ge 1$ and $x(j) \in M_j$ satisfying $|f|_{L^{\infty}(B_R(x))}+\mathbf{Lip}f + \upsilon (B_R(x))\le L$
and $x(j) \rightarrow x$.
First, we assume that $\mathrm{supp}f \cap (\{x \} \cup \partial B_{\pi/\sqrt{H}}(x)) = \emptyset$. 
Here, if $H \le 0$, then $\partial B_{\pi/\sqrt{H}}(x)) = \emptyset$.
Then there exists $\tau > 0$ such that
$\mathrm{supp}f \cap \overline{B}_{\tau}\left(\{x \} \cup \partial B_{\pi/\sqrt{H}}(x)\right) = \emptyset$.
By Theorem \ref{app}, for every $\epsilon > 0$, there exist an open set $\Omega^{\epsilon} \subset B_R(x) \setminus \overline{B}_{\tau}\left(\{x \} \cup \partial B_{\pi/\sqrt{H}}(x)\right)$, $2L$-Lipschitz function $f^{\epsilon}$ on $B_R(x)$ and  
a sequence of $2L$-Lipschitz function $f_i^{\epsilon}$ on $B_R(x_i)$ such that 
$\mathrm{supp}f^{\epsilon} \cap \overline{B}_{\tau}\left(\{x \} \cup \partial B_{\pi/\sqrt{H}}(x)\right) = \emptyset$,
$\mathrm{supp}f^{\epsilon}_i \cap \overline{B}_{\tau}\left(\{x(i) \} \cup \partial B_{\pi/\sqrt{H}}(x(i))\right) = \emptyset$, $(f_i^{\epsilon}, df_i^{\epsilon}) \rightarrow (f^{\epsilon}, df^{\epsilon})$ on $\Omega^{\epsilon}$ and 
\[\frac{\upsilon\left(\Omega^{\epsilon} \cup \overline{B}_{\tau}\left(\{x \} \cup \partial B_{\pi/\sqrt{H}}(x)\right)\right)}{\upsilon(B_R(x))}+
|f-f^{\epsilon}|_{L^{\infty}(B_R(x))} + \frac{1}{\upsilon(B_R(x))}\int_{B_R(x)}|df-df^{\epsilon}|^2d\upsilon < \epsilon.\]
By Proposition \ref{cov}, we can assume that there exists a finite pairwise disjoint collection $\{\overline{B}_{r_i}(x_i)\}_{1 \le i \le N}$ 
such that $\Omega^{\epsilon}= \bigcup_{i=1}^NB_{r_i}(x_i)$.
We take $x_i(j) \in M_j$ satisfying $x_i(j) \rightarrow x_i$.
Then, by Proposition \ref{10103}, we have 
\begin{align}
\int_{B_R(x(j))}\langle df_j^{\epsilon}, dr_{x(j)}\rangle d\underline{\mathrm{vol}}&=\int_{B_R(x(j)) \setminus \overline{B}_{\tau}\left(\{x(j) \} \cup \partial B_{\pi/\sqrt{H}}(x(j))\right)}\langle df_j^{\epsilon}, dr_{x(j)}\rangle d\underline{\mathrm{vol}}\\
&=\sum_{i=1}^N\int_{B_{r_i}(x_i(j))}\langle df_j^{\epsilon}, dr_{x(j)}\rangle d\underline{\mathrm{vol}} \pm \Psi(\epsilon;n,L, H) \\
&=\sum_{i=1}^N\int_{B_{r_i}(x_i(j))}\langle df^{\epsilon}, dr_x \rangle d\upsilon \pm \Psi(\epsilon;n, L, H)\\
&=\int_{B_R(x) \setminus \overline{B}_{\tau}\left(\{x \} \cup \partial B_{\pi/\sqrt{H}}(x)\right)}\langle df^{\epsilon}, dr_x\rangle d\upsilon \pm \Psi(\epsilon;n, L, H)\\
&=\int_{B_R(x)}\langle df^{\epsilon}, dr_x \rangle d\upsilon \pm \Psi(\epsilon;n, L, H)\\
&=\int_{B_R(x)}\langle df, dr_s\rangle d\upsilon \pm \Psi(\epsilon;n, L, H)
\end{align}
for every sufficiently large $j$.
On the other hand, for every $i$, there exists a Lipschitz function $\psi_i$ on $M_i$ such that $0 \le \psi_i \le 1$, 
$\psi_i|_{\overline{B}_{\tau/2}\left(\{x \} \cup \partial B_{\pi/\sqrt{H}}(x)\right)}=0$, $\psi_i|_{M_i \setminus \overline{B}_{\tau}\left(\{x \} \cup \partial B_{\pi/\sqrt{H}}(x)\right)}=1$ and $\mathbf{Lip}\psi_i \le C(n, \tau)$.
Since $f_i^{\epsilon}+\Psi(\epsilon;n, L, H) \ge 0$ on $B_R(x(i))$ for every sufficiently large $i$, we have 
$f_i^{\epsilon}+\Psi(\epsilon;n, L, H)\psi_i \ge 0$ on $B_R(x(i))$.
Therefore by Proposition \ref{10103} and Corollary \ref{mmmmm}, we have 
\begin{align}
&\int_{B_R(x(i))}\langle d\left(f_i^{\epsilon}+ \Psi(\epsilon;n, L, H)\psi_i\right), dr_{x(i)}\rangle d\underline{\mathrm{vol}} \\
&\ge -(n-1)\int_{B_R(x(i))}
\frac{\underline{k}'_{H_i}(\overline{x(i),w})}{\underline{k}_{H_i}(\overline{x(i), w})}(f^{\epsilon}_i+\Psi(\epsilon;n,L, H)\psi_i)d\underline{\mathrm{vol}} \\
&\ge -(n-1)\int_{B_R(x(i))}
\frac{\underline{k}'_{H_i}(\overline{x(i),w})}{\underline{k}_{H_i}(\overline{x(i), w})}f^{\epsilon}_id\underline{\mathrm{vol}} -\Psi(\epsilon;n,L, H)\int_{B_R(x(i))}
\left|\frac{\underline{k}'_{H_i}(\overline{x(i),w})}{\underline{k}_{H_i}(\overline{x(i), w})}\psi_i\right|d\underline{\mathrm{vol}}\\
&\ge -(n-1)\int_{B_R(x(i))}
\frac{\underline{k}'_{H_i}(\overline{x(i),w})}{\underline{k}_{H_i}(\overline{x(i), w})}f^{\epsilon}_id\underline{\mathrm{vol}} -\Psi(\epsilon;n, L, H, \tau, R)\\
&= -(n-1)\int_{B_R(x)}\frac{\underline{k}'_{H}(\overline{x,w})}{\underline{k}_{H}(\overline{x, w})}f^{\epsilon}d\upsilon -\Psi(\epsilon;n, L, H, \tau, R)\\
&= -(n-1)\int_{B_R(x)}\frac{\underline{k}'_{H}(\overline{x,w})}{\underline{k}_{H}(\overline{x, w})}fd\upsilon -\Psi(\epsilon;n, L, H, \tau, R)
\end{align}
for every sufficiently large $i$.
Since
\[\int_{B_R(x(i))}|df^{\epsilon}_i-d\left(f^{\epsilon}_i+\Psi(\epsilon; n, L, H)\psi_i\right)|d\underline{\mathrm{vol}} \le \Psi(\epsilon;n, L, H, \tau),\]
we have
\[\int_{B_R(x)}\langle df, dr_x\rangle d \upsilon \ge -(n-1)\int_{B_R(x)}\frac{\underline{k}'_{H}(\overline{x,w})}{\underline{k}_H(\overline{x, w})}f(w) d \upsilon - \Psi(\epsilon;n, L, H, \tau, R).\]
By letting $\epsilon \rightarrow 0$, we have the assertion of the case $\mathrm{supp}f \cap (\{x \} \cup \partial B_{\pi/\sqrt{H}}(x)) = \emptyset$.
\

Next, we shall discuss the assertion of the case $\mathrm{supp}f \cap (\{x \} \cup \partial B_{\pi/\sqrt{H}}(x)) \neq \emptyset$.
We assume that  $H \le 0$ and $\liminf_{r \rightarrow 0}\upsilon(B_r(x))/r=0$.
We take a sequence of positive numbers $\{r_i\}_i$ satisfying $r_i \rightarrow 0$ and 
$\lim_{i \rightarrow \infty}\upsilon(B_{r_i}(x))/r_i=0.$
We also take a Lipschitz function $\phi_i$ on $Y$ satisfying $\phi_i|_{B_{r_i/2}(x)}=1$, $0 \le \phi_i \le 1$, 
$\mathrm{supp}\phi_i \subset B_{r_i}(x)$ and $\mathbf{Lip}\phi_i \le C(n)/r_i$.
We fix $\epsilon > 0$.
Then we have 
\begin{align}
\left|\int_Y \langle df, dr_x\rangle d\upsilon-\int_Y \langle d(1-\phi_i)f, dr_x\rangle d\upsilon \right |&\le \int_Y|d(\phi_if)|d\upsilon  \\
&=\int_{B_{r_i}(x)}|d(\phi_if)|d\upsilon \\
&\le \frac{C(n, L)}{r_i}\upsilon(B_{r_i}(x)).
\end{align}
On the other hand, since $\underline{k}'_H/\underline{k}_H \ge 0$, we have 
\begin{align}
\int_Y \langle d(1-\phi_i)f, dr_x\rangle d\upsilon & \ge -(n-1)\int_Y\frac{\underline{k}_H'(\overline{x,w})}{\underline{k}_H(\overline{x,w})}(1-\phi_i)fd\upsilon \\
&\ge -(n-1)\int_Y\frac{\underline{k}_H'(\overline{x,w})}{\underline{k}_H(\overline{x,w})}f(w)d\upsilon.
\end{align}
Thus, by letting $i \rightarrow \infty$, we have the assertion of the case 
$H \le 0$ and $\liminf_{r \rightarrow 0}\upsilon(B_r(x))/r=0$.
\

Next, we shall discuss the assertion the case $H \le 0$, $\liminf_{r \rightarrow 0}\upsilon(B_r(x))/r > 0$ and $f(x)=0$.
We take a sequence of positive numbers $\{r_i\}_i$ satisfying $r_i \rightarrow 0$.
We also take $\phi_i$ as above.
Then we have 
\begin{align}
\left|\int_Y \langle df, dr_x\rangle d\upsilon-\int_Y \langle d(1-\phi_i)f, dr_x\rangle d\upsilon \right|&\le \int_Y|d(\phi_if)|d\upsilon \\
&=\int_{B_{r_i}(x)}|d(\phi_if)|d\upsilon \\
&=\int_{B_{r_i}(x)}|fd\phi_i + \phi_i df|d\upsilon \\
&\le \int_{B_{r_i}(x)}|f||d\phi_i|d\upsilon + \mathbf{Lip}f\upsilon(B_{r_i}(x)) \\
&\le r_i \mathbf{Lip}f\frac{\upsilon(B_{r_i}(x))}{r_i} + L\upsilon(B_{r_i}(x)) \\
&=2L\upsilon(B_{r_i}(x)).
\end{align} 
Therefore, we have the assertion of the case 
$H \le 0$, $\liminf_{r \rightarrow 0}\upsilon(B_r(x))/r > 0$ and $f(x)=0$.
\

We shall discuss the case $H \le 0$, $\liminf_{r \rightarrow 0}\upsilon(B_r(x))/r > 0$ and $f(x) > 0$.
Then we remark the following:
\begin{claim}\label{coco}
We have
\[\liminf_{r \rightarrow 0}\upsilon_{-1}(\partial B_r(x) \setminus C_x) > 0.\]
\end{claim}
The proof is as follows.
For every sufficiently small $r >0$, there exists an isometric embedding $\gamma$ from $[0, 3r]$ to $Y$ satisfying 
$\gamma(0)=x$.
We put $x_r = \gamma (5r/2)$.
Then we have
\[\upsilon(B_{3r}(x) \setminus B_{2r}(x))\ge \upsilon(B_{\frac{r}{100}}(x_r)) \ge C(n, H)\upsilon(B_r(x)).\]
By \cite[Theorem $4.6$]{ho}, we have 
\begin{align}
\upsilon_{-1}(\partial B_r(x)\setminus C_x) &\ge \upsilon_{-1}\left(\partial B_r(x) \cap C_x(B_{3r}(x) \setminus B_{2r}(x))\right) \\
&\ge \frac{C(n, H)\upsilon (B_{3r}(x)\setminus B_{2r}(x))}{\mathrm{vol}\,B_{3r}(\underline{p})-\mathrm{vol}\,B_{2r}(\underline{p})}\mathrm{vol}_{n-1}\,\partial B_r(\underline{p}) \\
&\ge C(n, H)\frac{\upsilon(B_r(x))}{r^n}r^{n-1}\ge C(n)\frac{\upsilon(B_r(x))}{r}.
\end{align}
Therefore, we have Claim \ref{coco}.

\
By the assumption, there exist $r_0 > 0$ and $\tau_0 > 0$ such that $f(w) > \tau_0$ for every $w \in B_{r_0}(x)$.
Thus, by  \cite[Theorem $5. 2$]{ho}, we have 
\begin{align}
\int_{Y}\frac{k'_H(\overline{x, w})}{k_H(\overline{x, w})}f(w)d\upsilon &\ge C(n, r_0, H,  \tau_0) \int_{B_{r_0}(x)}\frac{1}{r_x(w)}d\upsilon \\
&\ge C(n, r_0, H, \tau_0)\int_0^{r_0}\int_{\partial B_r(x) \setminus C_x}\frac{1}{r}d\upsilon_{-1}dr \\
&= C(n, r_0, H, \tau_0) \int_0^{r_0}\frac{\upsilon_{-1}(\partial B_r(x) \setminus C_x)}{r}dr = \infty
\end{align}
Therefore, we have the assertion of the case $H \le 0$, $\liminf_{r \rightarrow 0}\upsilon(B_r(x))/r > 0$ and $f(x) > 0$.
\

Finally, we shall discuss the assertion of the case $H>0$.
By rescaling, without loss of generality, we can assume that $H=1$.
If $R < \pi$, then we can prove the assertion by an argument similar to one above.
Therefore, we assume that $R=\pi$ and $\partial B_{\pi}(x)\neq \emptyset$ below.
Then, by \cite{ch-co} (or \cite{ho2}), we have $Y = \mathbf{S}^0 * \partial B_{\pi/2}(x)$.
Here, for every metric space $X$, we define a distance on $[0, \pi] \times X/\{0, \pi\} \times X$ by
\[\overline{(t_1, x_1), (t_2, x_2)}= \arccos(\cos t_1 \cos t_2 + \sin t_1 \sin t_2 \cos \min \{ \overline{x_1, x_2},\pi \}),\]
$\mathbf{S}^0*X$ denote this metric space. 
We take $z \in \partial B_{\pi}(x)$.
By Bishop-Gromov volume comparison theorem for $\upsilon$,  we have 
\[\frac{\upsilon(B_r(x))}{\upsilon(Y)} = \frac{\upsilon(Y \setminus B_{\pi-r}(z))}{\upsilon(Y)}=1-\frac{\upsilon(B_{\pi-r}(x))}{\upsilon(Y)} \le 1-\frac{\mathrm{vol}\,B_{\pi-r}(\underline{p})}{\mathrm{vol}\,B_{\pi}(\underline{p})} =\frac{\mathrm{vol}\,B_r(\underline{p})}{\mathrm{vol}\,\mathbf{S}^n}\]
for every $0 < r \le \pi/2$.
On the other hand, by Bishop-Gromov volume comparison theorem, since $\upsilon(B_r(x))/\upsilon(Y) \ge \mathrm{vol}\,B_r(\underline{p})/\mathrm{vol}\,\mathbf{S}^n$,
we have 
\[\frac{\upsilon(B_r(x))}{\upsilon(Y)} = \frac{\mathrm{vol}\,B_r(\underline{p})}{\mathrm{vol}\,\mathbf{S}^n}.\]
Similarly, we have $\upsilon(B_r(z))/\upsilon(Y) = \mathrm{vol}\,B_r(\underline{p})/\mathrm{vol}\,\mathbf{S}^n$.
Especially, we have 
\[\lim_{r \rightarrow 0}\frac{\upsilon(B_r(x))}{\omega_nr^n}=\lim_{r \rightarrow 0}\frac{\upsilon(B_r(z))}{\omega_nr^n}=
\frac{\upsilon(Y)}{\mathrm{vol}\,\mathbf{S}^n}.\]
Since $\underline{k}_1'(r)/\underline{k}_1(r) \ge 0$ for every $ 0 < r \le \pi/2$, by \cite[Theorem $4. 2$]{ho} and \cite[Theorem $5. 2$]{ho}, we have
\begin{align}
\int_{B_{\frac{\pi}{2}}(x)}\frac{\underline{k}_1'(\overline{x, w})}{\underline{k}_1(\overline{x, w})}d\upsilon 
&\le \int_0^{\frac{\pi}{2}}\int_{\partial B_t(x) \setminus C_x}\frac{C(n)}{r_x}d\upsilon_{-1}dt \\
&=C(n)\int_0^{\frac{\pi}{2}}\frac{\upsilon_{-1}(\partial B_t(x) \setminus C_x)}{t}dt \\
&\le C(n)\int_0^{\frac{\pi}{2}}\frac{\upsilon(B_t(x))}{t}\frac{\mathrm{vol}_{n-1}\,\partial B_t(\underline{p})}{\mathrm{vol}\,B_t(\underline{p})}dt\\
& \le C(n)\int_0^{\frac{\pi}{2}}\frac{\upsilon(B_t(x))}{t^2}dt \le C(n).
\end{align}
We remark that $C_z=\{x\}$ and $C_x=\{z\}$. 
Similarly, we have 
\begin{align}
\int_{M \setminus B_{\frac{\pi}{2}}(x)}\left|\frac{\underline{k}_1'(\overline{x, w})}{\underline{k}_1(\overline{x, w})}\right|d\upsilon &= \int_{B_{\frac{\pi}{2}}(z)}\left|\frac{\underline{k}_1'(\overline{x, w})}{\underline{k}_1(\overline{x, w})}\right|d\upsilon \\
&\le C(n) \int_{0}^{\frac{\pi}{2}}\int_{\partial B_t(z)}\frac{1}{t}d\upsilon_{-1}dt \\
&\le C(n) \int_0^{\frac{\pi}{2}}\frac{\upsilon_{-1}(\partial B_t(z) \setminus C_z)}{t}dt \le C(n).
\end{align}
We take $r_i > 0$ satisfying $r_i \rightarrow 0$ and $\phi_i$ as above.
We also take Lipschitz functions $\hat{\phi}_i$ on $Y$ satisfying $0 \le \hat{\phi}_i \le 1$, $\hat{\phi}_i|_{B_{r_i/2}(z)}=1$, 
$\mathrm{supp}\hat{\phi}_i \subset B_{r_i}(z)$ and $\mathbf{Lip}\hat{\phi}_i \le C(n)/r_i$.
Then we have 
\begin{align}
\left|\int_Y \langle df, dr_x\rangle d\upsilon-\int_Y \langle d(1-\phi_i)(1-\hat{\phi}_i)f, dr_x\rangle d\upsilon \right|
&\le \int_Y|d(f-(1-\phi_i)(1-\hat{\phi}_i)f)|d\upsilon \\
&=\int_{B_{r_i}(x)}|d(\phi_if)|d\upsilon + \int_{B_{r_i}(z)}|d(\hat{\phi}_if)|d\upsilon \\
&\le \mathbf{Lip}f\frac{\upsilon(B_{r_i}(x))}{r_i}+ \mathbf{Lip}f\frac{\upsilon(B_{r_i}(z))}{r_i} \\
&\stackrel{i \rightarrow \infty}{\rightarrow} 0.
\end{align}
On the other hand, by dominated convergence theorem, we have 
\begin{align}
&\int_Y \langle d(1-\phi_i)(1-\hat{\phi}_i)f, dr_x \rangle d\upsilon \\
&\ge -(n-1)\int_Y\frac{\underline{k}_1'(\overline{x, w})}{\underline{k}_1(\overline{x, w})}(1-\phi_i)(1-\hat{\phi}_i)f(w)d\upsilon \\
&\ge -(n-1)\int_Y\frac{\underline{k}_1'(\overline{x, w})}{\underline{k}_1(\overline{x, w})}f(w)d\upsilon -(n-1)\int_Y\left|\frac{\underline{k}_1'(\overline{x, w})}{\underline{k}_1(\overline{x, w})}\right||(1-\phi_i)(1-\hat{\phi}_i)f(w)-f(w)|d\upsilon \\
&\stackrel{i \rightarrow \infty}{\rightarrow} -(n-1)\int_Y\frac{\underline{k}_1'(r_x(w))}{\underline{k}_1(r_x(w))}f(w)d\upsilon.
\end{align}
Therefore we have the assertion.
\end{proof}
We end this section by giving a corollary of Theorem \ref{26}.
The corollary is well known in the setting of metric measure spaces.
See for instance \cite{oh, St1, St2, lo, lo-vi, Vi1, Vi2}.
We will give a new proof via Laplacian comparison theorem on Ricci limit spaces:
\begin{corollary}\label{87593}
Let $\{H_i\}_{i=1, 2, _{\cdots}, \infty}$ be a sequence of real numbers, $\{(M_i, m_i)\}_{i \in \mathbf{N}}$ a sequence of pointed $n$-dimensional complete Riemannian manifolds with $\mathrm{Ric}_{M_i}\ge H_i(n-1)$ and $(M_{\infty}, m_{\infty})$ a pointed $n$-dimensional complete Riemannian manifold $(n \ge 2)$.
We assume that $H_i \rightarrow H_{\infty}$ and $(M_i, m_i) \rightarrow (M_{\infty}, m_{\infty})$.
Then we have $\mathrm{Ric}_{M_{\infty}} \ge H_{\infty}(n-1)$.
\end{corollary}
\begin{proof}
By \cite[Theorem $5. 9$]{ch-co}, we have $(M_i, m_i, H^n) \rightarrow (M_{\infty}, m_{\infty}, H^n)$.
Then, by Theorem \ref{26}, we have, $\Delta r_x(w) \ge -(n-1)\underline{k}'_{H_{\infty}}(\overline{x, w})/\underline{k}_{H_{\infty}}(\overline{x, w})$ for every $x \in M_{\infty}$ and $w \in M_{\infty} \setminus (C_x \cup \{x\})$.
Therefore, we have the assertion.
\end{proof}
\section{Appendix}
\subsection{Infinitesimal doubling condition and Lebesgue set}
In this subsection, we shall study metric measure spaces satisfying a good property (Definition \ref{infinitesimal}).
On such metric measure spaces, we can construct an outer measure associated to the measure and give several properties about it.
Especially, we will define Lebesgue set and give 
several properties of the set (see Corollary \ref{131} and Proposition \ref{133}).
\begin{definition}\label{infinitesimal}
Let $(Z, \upsilon)$ be a metric measure space, $A$ a Borel subset of $Z$ and $C \ge 1$.
We say that $(Z, \upsilon)$ \textit{satisfies infinitesimal doubling condition on $A$ with doubling constant $C$} if the following properties hold:
\begin{enumerate}
\item $\upsilon(K) < \infty$ for every bounded Borel subset $K$ of $A$.
\item For every $z \in A$, there exists $r > 0$ such that  
\[\upsilon(\overline{B}_{2s}(z)) \le C\upsilon(\overline{B}_s(z))\]
for every $0 < s < r$.
\end{enumerate}
\end{definition}
We shall give an example:
\begin{example}
Let $(Y, y, \upsilon)$ be a Ricci limit space, $x$ a point in $Y$, $R$ a positive number satisfying 
$\partial B_R(x) \setminus C_x \neq \emptyset$.
Then, the metric measure space $(\partial B_R(x), \upsilon_{-1})$ satisfies
infinitesimal doubling condition on $\partial B_R(x) \setminus C_x$.
In fact, we have 
\[\limsup_{r \rightarrow 0}\frac{\upsilon_{-1}(\partial B_R(x) \cap \overline{B}_{2r}(z))}{\upsilon_{-1}(\partial B_R(x) \cap B_r(z) \setminus C_x)} \le C(n)\]
for every $z \in \partial B_R(x) \setminus C_x $.
This follows from \cite[Corollary $4. 7$]{ho} and \cite[Theorem $5. 2$]{ho}.
\end{example}
We fix a metric measure space $(Z, \upsilon)$ and a Borel subset $A$ of $Z$ satisfying that 
$(Z, \upsilon)$ satisfies infinitesimal doubling condition on $A$ with doubling constant $C \ge 1$ below.
For every $\delta > 0$ and $\hat{A} \subset Z$, we put
\[\upsilon^*_{\delta}(\hat{A})= \inf \left\{\sum_{i=1}^{\infty}\upsilon (B_{r_{\lambda}}(x_{\lambda}));\ 0 \le r_{i}<\delta, \ \hat{A} \subset \bigcup_{i=1}^{\infty}B_{r_{i}}(x_{i}) \right\}\]
and define an outer measure $\upsilon^*$ on $Z$ by 
\[\upsilon^*(\hat{A})=\lim_{\delta \rightarrow 0}\upsilon^*_{\delta}(\hat{A}).\]
We also put $\mathcal{M}= \{\hat{A} \in 2^Z; \upsilon^*(B \cap \hat{A}) + \upsilon^*(B \setminus \hat{A}) \le \upsilon^*(B)$ for every $B \in 2^Z\}$.
We shall recall that $(Z, \mathcal{M}, \upsilon^*)$ is a complete measure space and that $ \mathcal{B}(Z) = \{B \in 2^Z; B$ is a Borel subset of $Z$ $\} \subset \mathcal{M}$.
See for instance chapter $1$ in \cite{Si}.
By the definition, we have $\upsilon(\hat{A}) \le \upsilon^*(\hat{A})$ for every Borel subset $\hat{A}$ of $Z$.
\begin{proposition}\label{borel}
We have $\upsilon^*(\hat{A})=\upsilon(\hat{A})$ for every Borel subset $\hat{A} \subset A$.
\end{proposition}
\begin{proof}
Without loss of generality, we can assume that $\upsilon (\hat{A}) < \infty$.
We fix $\epsilon, \delta > 0$.
There exists an open set $O \subset Z$ such that $\hat{A} \subset O$ and $\upsilon(O \setminus \hat{A})<\epsilon$.
For every $a \in \hat{A}$, there exists $r_a > 0$ such that $\overline{B}_{r_a}(a) \subset O$ and that $\upsilon(B_{2r}(a)) \le C \upsilon(B_r(a))$ for every $0 < r < r_a$.
By Proposition \ref{cov}, there exists a pairwise disjoint collection $\{\overline{B}_{r_i}(a_i)\}$ such that
$a_i \in \hat{A}$, $r_i < \min \{\delta, r_{a_i} \}/100$ and 
$\hat{A} \setminus \bigcup_{i=1}^N\overline{B}_{r_i}(a_i) \subset \bigcup_{i=N+1}^{\infty}\overline{B}_{5r_i}(a_i)$
for every $N$.
Since $\upsilon(O) < \infty$, there exists $N$ such that 
$\sum_{i=N+1}^{\infty}\upsilon(\overline{B}_{r_i}(a_i))< \epsilon.$
Then we have
\begin{align}
\upsilon^*_{\delta}(\hat{A})&\le \sum_{i=1}^N\upsilon(\overline{B}_{r_i}(a_i)) + \sum_{i=N+1}^{\infty}\upsilon(\overline{B}_{5r_i}(a_i)) \\
&\le \sum_{i=1}^N\upsilon(\overline{B}_{r_i}(a_i)) + \sum_{i=N+1}^{\infty}C^3\upsilon(\overline{B}_{r_i}(a_i)) \\
&\le \upsilon(O) + C^3\epsilon \le \upsilon(\hat{A})+(1 + C^3)\epsilon.
\end{align}
By letting $\delta \rightarrow 0$ and $\epsilon \rightarrow 0$, we have the assertion.
\end{proof}
The following corollary is a fundamental property for a relation to Hausdorff measure on metric measure spaces satisfying infinitesimal doubling condition.
\begin{corollary}\label{132}
Assume that there exists $\alpha \ge 0$ such that
$\upsilon$ is Ahlfors $\alpha$-regular at every $z \in A$.
Then, $\upsilon$ and $H^{\alpha}$ are mutually absolutely continuous on $A$.
\end{corollary}
\begin{proof}
For every $i \in \mathbf{N}$, we put $A_i=\{a \in A; i^{-1}r^{\alpha}\le \upsilon(B_r(a))\le ir^{\alpha}$ for every $0<r < i^{-1}\}$.
Let $D$ be a Borel subset of $A$.
First, we assume that $H^{\alpha}(D)=0$.
Then,  we have $H^{\alpha}(D \cap A_i) \le H^{\alpha}(D)=0$ for every $i$.
We fix $i$.
Then, for every positive numbers $\epsilon, \delta$ satisfying $\epsilon, \delta << i^{-1}$, there exists a countable collection $\{B_{r_j}(x_j)\}_j$ such that $r_j < \delta$, $x_j \in D \cap A_i$ and $\sum_{j}r_j^{\alpha}<\epsilon$. 
Thus, we have $\sum_j\upsilon(B_{r_j}(x_j))<\Psi(\epsilon ;i)$.
Therefore, we have $\upsilon^*(D \cap A_i)=0$.
Since $(Z, \upsilon^*, \mathcal{M})$ is a complete measure space, we have $\upsilon^*(D)=0$.
Especially, we have $\upsilon(D)=0$.
Next, we assume that $\upsilon(D)=0$.
By Proposition \ref{borel}, we have $\upsilon^*(D \cap A_i) \le \upsilon^*(D) = \upsilon(D)=0$ for every $i$.
Then, by an argument similar to that above, we have $H^{\alpha}(D \cap A_i)=0$.
Thus, we have $H^{\alpha}(D)=0$.
\end{proof}
For subset $\hat{A} \subset Z$, let $\mathrm{Leb}\,\hat{A}$, denote the set of points, $a \in A$, such that for every $\epsilon > 0$, there 
exists $r > 0$ such that $\upsilon^*(\overline{B}_s(a) \cap \hat{A}) \ge (1-\epsilon) \upsilon(\overline{B}_s(a))$
for every $0 < s < r$.
We call $\mathrm{Leb}\,\hat{A}$ \textit{Lebesgue set of $\hat{A}$}.
\begin{proposition}\label{130}
We have
\[\upsilon^*(\hat{A} \setminus \mathrm{Leb}\,\hat{A}) = 0\]
for every Borel subset $\hat{A}$ of $A$.  
\end{proposition}
\begin{proof}
We fix $z \in Z$ and $\epsilon> 0$.
For $\tau > 0$ and $N \in \mathbf{N}$, let $\hat{A}_{\tau, N}$, denote the set of points, $a \in \hat{A} \cap B_N(z)$, such that there exists a
sequence of positive numbers $r_i > 0$ such that $r_i \to 0$ and that $\upsilon^*(\overline{B}_{r_i}(a) \cap \hat{A}) \le (1-\tau) \upsilon(\overline{B}_{r_i}(a))$ holds for every $i$.
We remark that $\upsilon^*(\hat{A}_{\tau, N}) \le \upsilon^*(\hat{A} \cap B_N(z)) = \upsilon(\hat{A} \cap B_N(z)) < \infty$.
Thus, by the definition of $\upsilon^*$, there exists a countable collection $\{B_{s_{i}}(x_{i})\}_i$ such that $\hat{A}_{\tau, N} \subset \bigcup_{i=1}^{\infty}B_{s_{i}}(x_{i})$ and $|\upsilon^*(A_{\tau, N})-\sum_{i=1}^{\infty}\upsilon(B_{s_{i}}(x_{i}))|<\epsilon$.
We put $O= B_N(z) \cap \bigcup_{i=1}^{\infty}B_{s_{i}}(x_{i})$.
By the definition of $\hat{A}_{\tau, N}$ and Proposition \ref{cov}, there exists a pairwise disjoint collection $\{\overline{B}_{r_i}(a_i)\}_i$ such that $a_i \in \hat{A}_{\tau, N}$,
$\upsilon(\overline{B}_{2r_i}(a_i))\le C\upsilon (\overline{B}_{r_i}(a_i))$,
$\overline{B}_{100r_i}(a_i) \subset O$,
$\upsilon(\overline{B}_{r_i}(a_i) \cap \hat{A})\le (1-\tau)\upsilon(\overline{B}_{r_i}(a_i))$ for every $i$, 
and 
$\hat{A}_{\tau, N} \setminus \bigcup_{i=1}^{\hat{N}}\overline{B}_{r_i}(a_i)\subset \bigcup_{i=\hat{N}+1}^{\infty}\overline{B}_{5r_i}(a_i)$
for every $\hat{N}$.
We take $\hat{N}$ satisfying 
$\sum_{i=\hat{N}+1}^{\infty}\upsilon(\overline{B}_{r_i}(a_i))< \epsilon.$
Then we have 
\begin{align}
\upsilon^*(\hat{A}_{\tau, N}) &\le \sum_{i=1}^{\hat{N}}\upsilon^*(\hat{A}_{\tau, N} \cap \overline{B}_{r_i}(a_i))+ \sum_{i=\hat{N}+1}^{\infty}\upsilon(\overline{B}_{5r_i}(a_i)) \\
& \le \sum_{i=1}^{\hat{N}}\upsilon(\hat{A} \cap \overline{B}_{r_i}(a_i)) +C^3\sum_{i=\hat{N}+1}^{\infty}\upsilon(\overline{B}_{r_i}(a_i)) \\
& \le (1-\tau)\sum_{i=1}^{\hat{N}}\upsilon(\overline{B}_{r_i}(a_i)) + \epsilon C^3 \\
& \le (1-\tau)\upsilon(O) + \epsilon C^3 \\
& \le (1-\tau)\sum_{i=1}^{\infty}\upsilon(B_{s_i}(x_i)) + \epsilon C^3 \\
& \le (1-\tau)(\upsilon^*(\hat{A}_{\tau, N})+\epsilon) + \epsilon C^3.
\end{align}
By letting $\epsilon \rightarrow 0$, we have $\upsilon^*(\hat{A}_{\tau, N})=0$.
Thus, we have $\hat{A}_{\tau, N} \in \mathcal{M}$ and $\upsilon^*(\hat{A} \setminus \mathrm{Leb}\,\hat{A})=\upsilon^*(\bigcup_{\tau > 0, N \in \mathbf{N}}\hat{A}_{\tau, N})=0$.
\end{proof}
By Proposition \ref{130}, we remark $\mathrm{Leb}(\mathrm{Leb}(\hat{A})) = \mathrm{Leb}(\hat{A})$ for every Borel subset $\hat{A} \subset A$.
\begin{corollary}[Lebesgue differentiation theorem for locally bounded functions]\label{131}
Let $f$ be a Borel function $f$ on $Z$ satisfying that $f$ is locally bounded at every $a \in A$.
Then, there exists a Borel subset $\hat{A}$ of $A$ such that $\upsilon (A \setminus \hat{A})=0$ and that for every $a \in \hat{A}$ and $\epsilon > 0$, there exists $r > 0$ such that 
\[-\epsilon \upsilon(\overline{B}_s(x))\le \int_{\overline{B}_s(x)}|f-f(x)|d\upsilon \le \epsilon \upsilon (\overline{B}_s(x))\]
for every $0 < s < r$.
\end{corollary}
\begin{proof}
We fix $\epsilon > 0$ and $z \in A$.
For every $N \in \mathbf{N}$, by Lusin's theorem, there exists a compact set $K_{\epsilon, N} \subset A \cap B_N(z)$
such that $\upsilon( A \cap B_N(z) \setminus K_{\epsilon, N})<\epsilon$ and that $f$ is continuous on $K_{\epsilon, N}$.
We put $\hat{K}_{\epsilon, N}= \mathrm{Leb}\,K_{\epsilon, N}$.
Then, it is easy to check that for every $x \in \hat{K}_{\epsilon ,N}$ and $\epsilon > 0$, there exists $r > 0$ such that 
\[-\epsilon \upsilon(\overline{B}_s(x))\le \int_{\overline{B}_s(x)}|f-f(x)|d\upsilon \le \epsilon \upsilon (\overline{B}_s(x))\]
for every $0 < s < r$.
Therefore, we have the assertion.
\end{proof}
We end this subsection by giving a fundamental property of Lebesgue sets for Lipschitz functions on metric measure spaces satisfying \textit{doubling condition}:
\begin{proposition}\label{133}
Assume that the following properties hold: 
\begin{enumerate}
\item $0 < \upsilon(B_r(z))$ for every $z \in Z$ and $r > 0$
\item There exist $r_0 > 0$ and $C > 1$ such that  
\[\upsilon(B_{2r}(z)) \le C \upsilon(B_{r}(z))\]
for every $z \in Z$ and $ 0 < r < r_0$.
\end{enumerate}
Then, for every Lipschitz function $f$ on $Z$ and Borel subset $A$ of $Z$, we have 
$\mathrm{Lip}f(a) = \mathrm{Lip}(f|_A)(a)$ and $lipf(a)=lip(f|_A)(a)$ for every $a \in \mathrm{Leb}(A)$.
\end{proposition}
\begin{proof}
Without loss of generality, we can assume that $a$ is not isolated point.
There exists a sequence $a_i \in Z \setminus \{a\}$ such that $a_i \rightarrow a$ and that $|f(a_i)-f(a)|/\overline{a_i, a} \rightarrow \mathrm{Lip}f(a)$.
By the assumption, for every sufficiently large $i$, there exists $\hat{a_i} \in A$ such that $\overline{a_i \hat{a}_i} \le \Psi(\overline{a, a_i}; C)\overline{a, a_i}$.
Especially we have $\hat{a}_i \neq a$, i.e. $a$ is not an isolated point in $A$.
It is easy to check 
\[\lim_{i \rightarrow \infty}\frac{|f(a)-f(a_i)|}{\overline{a, a_i}}=\lim_{i \rightarrow \infty}\frac{|f(a)-f(\hat{a}_i)|}{\overline{a, \hat{a}_i}}.\]
Therefore, we have $\mathrm{Lip}f(a) \le \mathrm{Lip}(f|_A)(a)$.
Thus we have the first assertion.
Similarly, we have the second assertion.
\end{proof}
\subsection{A proof of Claim \ref{111}}
In this subsection, we shall give a proof of Claim \ref{111}.
We define a function $\pi_1$ on $\mathbf{R}^k$ by $\pi_1((x_1, _{\cdots}, x_k)) = x_1$.
Then, by the definition,  we have
\[sl_1-\mathrm{Leb}A = \left\{ a = (a_1, _{\cdots}, a_k) \in A; \liminf _{r \rightarrow 0} \frac{H^{k-1}(\overline{B}_r(a) \cap A \cap 
\pi_1^{-1}(\pi_1(a)))}{\omega_{k-1}r^{k-1}} =1\right\}.\]
We define a function $f_r^A$ on $\mathbf{R}^k$ by $f_r^A(x) = H^{k-1}\left(\overline{B}_r(x) \cap A \cap \pi_1^{-1}(\pi_1(x))\right)1_A(x).$
First, we assume that $A$ is compact.
\begin{claim}\label{401}
The function $f_r^A$ is upper semi-continuous.
Especially, $f_r^A$ is $H^k$-measurable function.
\end{claim}
\begin{proof}
Let $x_{\infty}$ be a point in $A$ and $\{x_i\}_i$ a sequence of points in $\mathbf{R}^k$ satisfying $x_i \rightarrow x_{\infty}$.
It suffices to check that $\limsup_{i \to \infty}f_r^A(x_i) \le f_r^A(x_{\infty})$.
Without loss of generality, we can assume 
that $x_j \in A$ for every sufficiently large $j$.
We fix $\delta > 0$.
We take a subsequence $\{n(i)\}_{i \in \mathbf{N}}$ of $\mathbf{N}$ such that 
\[\lim_{j \rightarrow \infty}H^{k-1}\left(\overline{B}_r(x_{n(j)}) \cap A \cap \pi_1^{-1}(\pi_1(x_{n(j)}))\right)
= \limsup_{i \rightarrow \infty}H^{k-1}\left(\overline{B}_r(x_i) \cap A \cap \pi_1^{-1}(\pi_1(x_i))\right).\] 
On the other hand, since a sequence of compact set $\{\overline{B}_r(x_{n(j)}) \cap A \cap \pi_1^{-1}(\pi_1(x_{n(j)}))\}$ is precompact with
respect to the Hausdroff distance on $\mathbf{R}^k$.
Thus, without loss of generality, we can assume that there exists a compact subset $K_{\infty}$ of $\mathbf{R}^k$ such that 
$\overline{B}_r(x_{n(j)}) \cap A \cap \pi_1^{-1}(\pi_1(x_{n(j)}))$ converges to $K_{\infty}$ in the sense of Hausdorff distance on $\mathbf{R}^k$.
Then, it is easy to check $K_{\infty} \subset \overline{B}_r(x_{\infty}) \cap A \cap \pi_1^{-1}(\pi_1(x_{\infty}))$.
There exists a finite collection $\{B_{r_i}(y_i)\}_{i = 1, _{\cdots}, N}$ such that 
$r_i << \delta$,  
$\overline{B}_{r}(x_{\infty}) \cap A \cap \pi_1^{-1}(\pi_1(x_{\infty})) \subset \bigcup_{i=1}^N B_{r_i}(y_i)$
and 
\[\left|H^{k-1}(\overline{B}_{r}(x_{\infty}) \cap A \cap \pi_1^{-1}(\pi_1(x_{\infty}))-\sum_{i=1}^N\omega_{k-1}r_i^{k-1}\right| < \delta.\]
Since $\overline{B}_{r}(x_{\infty}) \cap A \cap \pi_1^{-1}(\pi_1(x_{\infty}))$ is compact, there exists $\tau_0 > 0$ such that
$B_{\tau_0}(\overline{B}_{r}(x_{\infty}) \cap A \cap \pi_1^{-1}(\pi_1(x_{\infty}))) \subset \bigcup_{i=1}^N B_{r_i}(y_i)$.
Since  
$\overline{B}_{r}(x_{n(j)}) \cap A \cap \pi_1^{-1}(\pi_1(x_{n(j)})) \subset B_{\tau_0}(K_{\infty})$
for every sufficiently large $j$,
we have $\overline{B}_{r}(x_{n(j)}) \cap A \cap \pi_1^{-1}(\pi_1(x_{n(j)})) \subset \bigcup_{i=1}^N B_{r_i}(y_i)$.
Thus, we have 
\begin{align}
H^{k-1}(\overline{B}_r(x_{n(j)}) \cap A \cap \pi_1^{-1}(\pi_1(x_{n(j)})) &\le 
\sum_{i = 1}^N H^{k-1}(\overline{B}_r(y_i) \cap \pi_1^{-1}(\pi_1(x_{n(j)})) \\
& \le \sum_{i=1}^N\omega_{k-1}r^{k-1} \\
& \le H^{k-1}(\overline{B}_{r}(x_{\infty}) \cap A \cap \pi_1^{-1}(\pi_1(x_{\infty})) ) + \delta
\end{align}
for every sufficiently large $j$.
Therefore, we have Claim \ref{401}.
\end{proof}
By Claim \ref{401}, we have the statement $1$ in Claim \ref{111}.
The statement $2$ follows from Lebesgue differentiation theorem on Euclidean spaces.
Finally,
by Fubini's theorem, we have 
\[H^k(A \setminus sl_1-\mathrm{Leb}A) = \int_{\mathbf{R}}H^{k-1}(A \cap \{t\} \times \mathbf{R}^{k-1} \setminus sl_1-\mathrm{Leb}A)dt
=0.\]
Thus, we have the statement $3$.
Therefore, we have Claim \ref{111} if $A$ is compact.
\

We shall give a proof of Claim \ref{111} in general case.
We fix $R >0$.
There exists a sequence of compact sets $K_i \subset B_R(0_k) \cap A$ such that $H^k(B_R(0_k) \cap A \setminus K_i) \rightarrow 0 (i \rightarrow \infty)$.
By the definition, we have $sl_1-\mathrm{Leb}K_i \subset  sl_1-\mathrm{Leb}(B_R(0_k) \cap A)$.
As an outer measure, we have 
\begin{align}
H^k(B_R(0_k) \cap A \setminus sl_1-\mathrm{Leb}(B_R(0_k) \cap A))
&\le H^k(B_R(0_k) \cap A \setminus sl_1-\mathrm{Leb}K_i) \\
&\le H^k(B_R(0_k) \cap A \setminus K_i)+ H^k(K_i \setminus sl_1-\mathrm{leb}K_i) \\
&\stackrel{i \to \infty}{\rightarrow} 0
\end{align}
Thus, $sl_1-\mathrm{leb}(B_R(0) \cap A)$ is a $H^k$-measurable set.
Since $sl_1-\mathrm{Leb}A = \bigcup_{R >0}sl_1-\mathrm{Leb}(A \cap B_R(0))$,
we have the statement $1$ in Claim \ref{111}.
By Lebesgue differentiation theorem and Fubini's theorem, we have statements $2$ and $3$.
Thus, we have Claim \ref{111}.
\subsection{Distributional Laplacian comparison theorem on manifolds}
Our aim in this subsection is to state distributional Laplacian comparison theorem on manifolds we want to use in section $6$.
It is Corollary \ref{mmmmm}.
Throughout this subsection, we fix a positive number $R >0$ and $(M, m)$ be a pointed $n$-dimensional complete Riemannian manifold $(n \ge 2)$.
We put 
$C^{\infty}(\overline{B}_R(m)) = \{ f \in C^0(\overline{B}_R(m));$ there exist an open subset $U$ of $M$ and 
a smooth function $g$ on $U$ such that $\overline{B}_R(m) \subset U$ and  $g|_{\overline{B}_R(m)}= f \}.$ 
We define a linear functional $\Delta_R^{\mathrm{dist}}r_m$ on $C^{\infty}(\overline{B}_R(m))$ by
\[\Delta_R^{\mathrm{dist}}r_m(f)=\int_{B_R(m)}\langle dr_m, df \rangle d\mathrm{vol}.\]
\begin{proposition}\label{501}
There exists unique Radon measure $\upsilon_{R,m}^{\mathrm{sing}}$ on $\overline{B}_R(m)$ satisfying the following properties:
\begin{enumerate}
\item A smooth function $\Delta r_m$ on $B_R(m) \setminus (C_m \cup \{m\})$ is in $L^1(B_R(m))$.
\item $\mathrm{supp}(\upsilon_{R,m}^{\mathrm{sing}}) \subset C_m \cap \overline{B}_R(m)$
\item For every $f \in C^{\infty}(\overline{B}_R(m))$, we have
\[\Delta_R^{\mathrm{dist}}r_m(f)= \int_{B_R(m)}f\Delta r_m d\mathrm{vol} +\int_{\partial B_R(m) \setminus C_m}fd\mathrm{vol}_{n-1}
+ \int_{\overline{B}_R(m)} f d\upsilon_{R,x}^{\mathrm{sing}}.\]
\item  We have
\[\int_{B_R(x)}|\Delta r_x|d\mathrm{vol} + \upsilon_{R,m}^{\mathrm{sing}}(B_R(x)) + \mathrm{vol}_{n-1}(\partial B_R(x) \setminus C_x)
= -2\int_{B_R(x) \cap \{\Delta r_x < 0\}}\Delta r_x.\]
\end{enumerate}
\end{proposition}
\begin{proof}
We put $S_mM=\{ u \in T_mM; |u|=1\}$ and define $t(u) > 0$ as the supremum of $t \in (0, \infty)$ such that $\exp _msu|_{[0, t]}$ is a minimal geodesic segment from $m$ to 
$\exp_mtu$ 
for $u \in S_mM$. 
We also define a continuous function $\phi_R$ on $S_mM$ by $\phi_R(u)=\min \{t(u), R\}$.
We take a sequence of $C^{\infty}$-functions $\{\phi_R^j\}_j$ on $S_mM$ and a sequence of open sets $O_i \subset S_mM$ satisfying the following properties:
\begin{enumerate}
\item $\overline{O}_i \subset O_{i+1}$, $\bigcup_{i=1}^{\infty} O_i = \{u \in S_mM; t(u) > R\}$, $\phi_R -j^{-1}\le \phi_R^j \le \phi_R$
and $\phi_R^j(u) < t(u)$.
\item For every $i$, there exists $l$ such that $\phi_R^j|_{O_i}=R$ for every $j \ge l$.
\end{enumerate}
\begin{remark}
There exists $\{\phi_R^j\}_j$ and $\{O_i\}_i$ as above.
We shall explain it below.
We take a sequence of $C^{\infty}$-functions $\psi_R^j$ on $S_mM$ satisfying $|\psi_R^j-\phi_R|_{L^{\infty}(S_mM)} \rightarrow 0$.
Without loss of generality, we can assume that $\psi_R^j<\phi_R$.
We take a sequence of open subsets $O_i$ of $S_mM$ satisfying $\overline{O}_i \subset O_{i+1}$ and $\bigcup_{i=1}^{\infty} O_i = \{u \in S_mM; t(u) > R\}$.
We put $O= \{u \in S_mM; t(u) > R\}$.
We take a $C^{\infty}$-function $\phi_i$ on $S_mM$ satisfying $0 \le \phi_i \le 1$,
$\phi_i|_{O_i}=1$ and $ \mathrm{supp}\phi_i \subset O_{i+1}$.
We define a $C^{\infty}$-function $\phi_R^{i,j}$ on $S_xM$ by $\phi_R^{i,j}(u)=(1-\phi_i(u))\psi_R^j(u) + \phi_i(u)R$.
Then, we have 
$\phi_R^{i, j}(u) = \psi_R^j(u) < \phi_R(u) = t(u)$  for every $u \in S_mM \setminus O_{i+1}$.
and $\phi_R^{i,j}(u)=R$ for every $u \in O_i$.
For every $i$, there exists $j_0(i)$ such that $|\psi_R^j - \phi_R|_{L^{\infty}(S_mM)}<i^{-1}$ for every $j \ge j_0(i)$.
We put $\phi_R^i=\phi_R^{i, j_0(i)}$.
Then, we have $\phi_R^i|_{O_i}=R$ and
$\phi_R^i(u) = (1-\phi_i(u))\psi_R^{j_0(i)}(u) + \phi_i(u)R \le (1-\phi_i(u))\phi_R(u) + \phi_i(u)R \le (1-\phi_i(u))\phi_R(u) + \phi_i(u)\phi_R(u) =\phi_R(u) < t(u)$
for every $u \in O$ and 
$\phi_R^i(u)  = (1-\phi_i(u))\psi_R^{j_0(i)}(u) + \phi_i(u)R =(1-\phi_i(u))\psi_R^{j_0(i)}(u) \le \psi_R^{j_0(i)}(u) <\phi_R(u) =t(u)$
for every $u \in S_xM \setminus O$.
Therefore, we have $\phi_R^i(u) \le \phi_R(u)$ and 
$\phi_R^i(u) < t(u)$ for every $u \in S_mM$.
Since $\phi_R|_{O_{i+1}}=R$, we have 
$|\phi_R^i(u)-\phi_R(u)| =|(1-\phi_i(u))\psi_R^{j_0(i)}(u)+\phi_i(u)R-\phi_R(u)|\le (1-\phi_i(u))|\phi_R^{j_0(i)}(u)-\phi_R(u)|+\phi_i(u)|R-\phi_R(u)| \le i^{-1}$ 
for every $u \in O_{i+1}$.
On the other hand, since $\phi_i|_{S_mM \setminus O_{i+1}}=0$, by an argument similar to one above, we have 
$|\phi_R^i(u)-\phi_R(u)|\le i^{-1}$ for every $u \in S_mM \setminus O_{i+1}$.
Thus, we get an existence of sequences $\{\phi_R^j\}_j$ and $\{O_i\}_i$ as above.
\end{remark}

We define an open subset $V_R^j$ of $M$ by $V_R^j = \{\exp_m tu \in M; u \in S_mM, 0\le t < \phi_R^j(u)\}$.
\begin{claim}\label{11112}
We have 
$\partial V_R^j =\{\exp_mtu \in M; u \in S_mM, t=\phi_R^j(u)\}$
for every $j$ satisfying $j^{-1} < \overline{m, C_m}$.
\end{claim} 
The proof is as follows.
We take $w \in \partial V_R^j$.
By the definition, there exist $u_i \in S_xM$ and $0 \le t_i < \phi_R^j$ such that $w_i= \exp_xt_iu_i \rightarrow w$.
By the compactness of $S_mM$, we can assume that there exist $t \in [0,R]$ and $u \in S_mM$ such that 
$t_i \rightarrow t$ and $u_i \rightarrow u$.
Thus, we have $w= \lim_{i \rightarrow \infty}\exp_mt_iu_i=\exp_mtu$.
Since $t_i < \phi_R^j(u_i)$, we have $t \le \phi_R^j(u) < t(u)$.
Thus, we have $w \in M \setminus C_x$.
If $t<\phi_R^j(u)$, then by the continuity of $\phi_R^j$, there exists $\tau>0$ such that $\hat{t} < \phi_R^j(\hat{u})$
for every $\hat{t}$ and $\hat{u} \in S_mM$
 satisfying $|\hat{t}-t|<\tau$ and $\overline{\hat{u}, u} < \tau$.
Thus, we have $\exp_xtu \in M \setminus \partial V_R^j$.
This is a contradiction.
Therefore, we have $\phi_R^j(u) \ge t$.
Similarly,   
if $t>\phi_R^j(u)$, by the continuity of $\phi_R^j$, there exists $\tau>0$ such that $\phi_R^j(\hat{u}) < \hat{t} < t(\hat{u})$
for every $\hat{t}$ and $\hat{u} \in S_mM$
 satisfying $|\hat{t}-t|<\tau$ and $\overline{\hat{u}, u} < \tau$.
Thus, we have $\exp_mtu \in M \setminus \partial V_R^j$.
This is a contradiction.
Therefore, we have 
$\partial V_R^j \subset \{\exp_mtu \in M; u \in S_mM, t=\phi_R^j(u)\}.$
On the other hand, for every $u \in S_mM$, we take a increasing sequence $0 < t_i < \phi_R^j(u)$ such that $t_i \rightarrow \phi_R^j(u)$.
Since $\exp_m\phi_R^j(u)u = \lim_{i \rightarrow \infty}\exp_mt_i u \in \overline{V}_R^j$ and $\exp_m^{-1}|_{M \setminus C_m}$ gives diffeomorphism to the image, we have $\exp_m\phi_R^j(u)u \in M \setminus V_R^j$.
Especially, we have $\exp_m\phi_R^j(u)u \in \partial V_R^j$.
Therefore, we have Claim \ref{11112}.
\

It is easy to check that $\partial V_R^j$ 
is a compact $(n-1)$-dimensional $C^{\infty}$-Riemannian submanifold of $M$ and is diffeomorphic to $\mathbf{S}^{n-1}$ for every $j$ satisfying $j^{-1}<\overline{m, C_m}$.
Especially, $\overline{V}_R^j$ is a compact $n$-dimensional $C^{\infty}$-Riemannian submanifold with $C^{\infty}$-boundary.
\begin{claim}\label{098901}
We have $\langle \nabla r_m, N_w\rangle (w) \ge 0$
for every $j$ satisfying $j^{-1} < \overline{m, C_m}$, and $w \in \partial V_R^j$.
Here $N_w$ is the unit outer normal vector of $\overline{V}_R^j$ at $w$., 
\end{claim}
Because, since $N_w$ is outer vector,  we have 
$\langle N_w, \gamma'(0)\rangle \le 0$ for the minimal geodesic $\gamma$ from $w$ to $m$.
Thus, we have Claim \ref{098901}.
\

For every $j$ satisfying $j^{-1}<\overline{m,  C_m}$, we define open subsets $A_+^{R, j}, A_-^{R, j}$ of $V_R^j$ by 
$A_+^{R, j}=\{w \in V_R^j \setminus \{m\}; \Delta r_m(w) > 0\}$ and
$A_-^{R, j}=\{w \in V_R^j \setminus \{m\}; \Delta r_m(w) < 0\}$.
\begin{claim}\label{49287}
We have 
\[\int_{A_+^{R, j}}\Delta r_m d\mathrm{vol} \le -\int_{A_-^{R, j}}\Delta r_m d\mathrm{vol} < \infty.\]
\end{claim}
The proof is as follows.
We put $\theta (s, u) = s^{n-1}  \sqrt{\mathrm{det}(g_{ij}|_{\exp _{m}su}))}$ for $u \in S_mM$ and $0 < s < t(u)$.
Here, $g_{ij} = g (\partial / \partial x_i, \partial / \partial x_j)$ for a normal coordinate $(x_1, x_2, _{\cdots}, x_n)$ around $m$.
By rescaling, without loss of generality, we can assume that  $\mathrm{Ric}_M \ge -(n-1)$ on $B_{100R}(x)$.
Then, we have
\begin{align}
-\int_{A_-^{R, j}}\Delta r_x d\mathrm{vol} &\le \int_{B_R(x)}(n-1)\frac{\cosh \overline{x,w}}{\sinh \overline{x,w}}d\mathrm{vol} \\
&= \int_{S_mM}\int_0^{\min \{t(u), R\}}(n-1)\frac{\cosh t}{\sinh t}\theta (t, u)dtdu\\
&\le \int_{S_mM}\int_0^{R}(n-1)\frac{\cosh t}{\sinh t}\sinh^{n-1}tdtdu\\
&\le \int_{S_mM}\int_0^R(n-1)\cosh t\sinh^{n-2}tdtdu < \infty.
\end{align}
Since
\[\Delta r_m (w) = -\frac{n-1}{\overline{m,w}}+O(\overline{m,w})\]
for every $w$ satisfying that $\overline{m,w}$ is sufficiently small, 
we have 
\begin{align}
\int_{B_{\tau}(x)}|\Delta r_x|d\mathrm{vol} &\le \int_{B_{\tau}(x)}\left(\frac{n-1}{\overline{x,w}}+1\right)d\mathrm{vol} \\
&=\int_{S_xM}\int_0^{\tau}\frac{n-1}{t}\theta (t, u)dtdu+\mathrm{vol}\,B_{\tau}(x) \\
&\le C(n)\int_{S_xM}\int_0^{\tau}\frac{n-1}{t}t^{n-1}dtdu + \mathrm{vol}\,B_{\tau}(x)\\
& \stackrel{\tau \rightarrow 0}{\rightarrow} 0.
\end{align}
Therefore, we have 
\[\int_{V_R^j}\Delta r_x d\mathrm{vol} = \lim_{\tau \rightarrow 0}\int_{V_R^j \setminus B_{\tau}(x)}\Delta r_xd\mathrm{vol}.\]
Thus, by divergence formula and Claim \ref{098901}, we have 
\begin{align}
\int_{V_R^j}\Delta r_x d\mathrm{vol} &= -\int_{\partial V_R^j}\langle \nabla r_x, N_w \rangle d\mathrm{vol}_{n-1}
-\lim_{\tau \rightarrow 0}\int_{\partial B_{\tau}(x)}\langle -\nabla r_x, \nabla r_x\rangle d\mathrm{vol}_{n-1} \\
&= -\int_{\partial V_R^j}\langle \nabla r_x, N_w \rangle d\mathrm{vol}_{n-1} \le 0.
\end{align}
Thus, we have
\[\int_{A_+^{R, j}}\Delta r_x d\mathrm{vol} + \int_{A_-^{R, j}}\Delta r_x d\mathrm{vol} \le 0.\]
Therefore we have Claim \ref{49287}.
\

Next claim follows from Claim \ref{49287} directly: 
\begin{claim}\label{LL}
We have 
\[\int_{V_R^j}|\Delta r_x|d\mathrm{vol} \le -2 \int_{A_-^{R, j}}\Delta r_xd\mathrm{vol} < \infty.\]
Especially, $\Delta r_x \in L^1(B_R(x))$.
\end{claim}
Therefore, for $f \in C^{\infty}(\overline{B}_R(x))$, we have
\begin{align}
\Delta_R^{\mathrm{dist}}r_m(f) &=\lim_{j \rightarrow \infty, \tau \rightarrow 0}\int_{V_R^j \setminus B_{\tau}(m)}\langle df, dr_m\rangle d\mathrm{vol} \\
&=\lim_{j \rightarrow \infty, \tau \rightarrow 0}\biggl(\int_{V_R^j \setminus B_{\tau}(m)}f\Delta r_xd\mathrm{vol}+\int_{\partial V_R^j}\langle dr_m, N_w\rangle f
d\mathrm{vol}_{n-1} \\
&\ \ \ \ \ -\int_{\partial B_{\tau}(m)}\langle dr_m, dr_m\rangle d\mathrm{vol}_{n-1}\biggl) \\
&=\int_{B_R(m)}f\Delta r_md\mathrm{vol}+\lim_{j \rightarrow \infty}\int_{\partial V_R^j}\langle dr_m, N_w \rangle fd\mathrm{vol}_{n-1}.
\end{align}
\begin{claim}
For every $w \in M$, we have 
\[\lim_{j \rightarrow \infty}1_{\partial V_R^j \cap \partial B_R(x)}(w) = 1_{\partial B_R(x) \setminus C_x}(w).\]
\end{claim}
The proof is as follows. We take $w \in M$.
\begin{enumerate}
\item The case $w \in \partial B_R(m) \setminus C_x$. Then, there exists $u \in S_mM$ such that $R < t(u)$ and $w = \exp_m Ru$.
By the definition of $\phi_R^j$, we have $\phi_R^j(u)=R$ for every sufficiently large $j$.
Thus, by Claim \ref{11112}, we have $w = \exp_x \phi_R^j(u)u \in \partial V_R^j$.
Therefore, we have $\lim_{j \rightarrow \infty}1_{\partial V_R^j \cap \partial B_R(m)}(w) = 1_{\partial B_R(m) \setminus C_m}(w)$.
\item The case $w \in (M \setminus (\partial B_R(m) \setminus C_m)) \cap \partial B_R(m)$.
Then $w \in C_x$.
By Claim \ref{11112}, since $\overline{V}_R^j \cap C_x = \emptyset$, we have 
$\lim_{j \rightarrow \infty}1_{\partial V_R^j \cap \partial B_R(m)}(w)=0= 1_{\partial B_R(m) \setminus C_m}(w).$
\item The case $w \in (M \setminus (\partial B_R(m) \setminus C_m)) \setminus \partial B_R(m)$.
Then, we have $w \in M \setminus (\partial V_R^j \cap \partial B_R(m))$ for every $j$.
Especially, $\lim_{j \rightarrow \infty}1_{\partial V_R^j \cap \partial B_R(m)}(w)=0=1_{\partial B_R(m) \setminus C_m}(w)$.
\end{enumerate}
Thus, we have Claim \ref{LL}.
\

Then, since $\langle \nabla r_m, N_w\rangle f(w)1_{\partial V_R^j \cap \partial B_R(m)}(w) \rightarrow f(w)$ for every $w \in \partial B_R(m) \setminus C_m$, by dominated convergence theorem, we have 
\begin{align}
\int_{\partial V_R^j \cap \partial B_R(m)}\langle \nabla r_m, N_w\rangle fd\mathrm{vol}_{n-1}&=\int_{\partial B_R(m) \setminus C_m}
\langle \nabla r_m, N_w\rangle f(w)1_{\partial V_R^j \cap \partial B_R(m)}(w)d\mathrm{vol}_{n-1} \\
& \stackrel{j \rightarrow \infty}{\rightarrow} \int_{\partial B_R(m) \setminus C_m}fd\mathrm{vol}_{n-1}.
\end{align}
Therefore, we have 
\begin{align}
&\lim_{j \rightarrow \infty}\int_{\partial V_R^j}\langle \nabla r_m, N_w\rangle f(w)d\mathrm{vol}_{n-1}\\
&=\lim_{j \rightarrow \infty}
\left(\int_{\partial V_R^j \cap \partial B_R(m)}\langle \nabla r_m, N_w\rangle f(w)d\mathrm{vol}_{n-1}+\int_{\partial V_R^j \setminus \partial B_R(m)}\langle \nabla r_m, N_w\rangle f(w)d\mathrm{vol}_{n-1}\right) \\
&=\int_{\partial B_R(m) \setminus C_x}fd\mathrm{vol}_{n-1}+ \lim_{j \rightarrow \infty}\int_{\partial V_R^j \setminus \partial B_R(m)}\langle \nabla r_x, N_w\rangle f(w)d\mathrm{vol}_{n-1}.
\end{align}
We define a linear functional $\Phi$ on $C_c^{\infty}(B_{2R}(m))$ by 
\[\Phi (f) = \lim_{j \rightarrow \infty}\int_{\partial V_R^j \setminus \partial B_R(m)}\langle \nabla r_m, N_w\rangle f(w)d\mathrm{vol}_{n-1}.\]
By Claim \ref{098901}, if $f \ge 0$, then $\Phi(f) \ge 0$.
Therefore, by Riesz's theorem, there exists a Radon measure $\upsilon_{R,m}^{\mathrm{sing}}$ on $B_{2R}(m)$ such that 
\[\Phi(f)=\int_{B_{2R}(m)}fd\upsilon_{R,m}^{\mathrm{sing}}.\]
for every $f \in C_c^{\infty}(B_{2R}(m))$.
\begin{claim}\label{plokij}
We have $\mathrm{supp}(\upsilon_{R,m}^{\mathrm{sing}}) \subset \overline{B}_R(m)$, i.e.
for every Borel set $A \subset B_{2R}(m) \setminus \overline{B}_R(m)$, we have 
$\upsilon_{R,m}^{\mathrm{sing}}(A)=0.$
\end{claim}
The proof is as follows.
Since $\upsilon_{R,x}^{\mathrm{sing}}$ is a Radon measure, without loss of generality, we can assume that $A$ is compact.
We take $\tau > 0$ satisfying $\tau << \min \{ \overline{A, B_R(x)}, \overline{A, \partial B_{2R}(x)}\}$.
We also take $\phi \in C_c^{\infty}(B_{2R}(x))$ satisfying $\phi|_A=1$, $0 \le \phi \le 1$, $\mathrm{supp}\phi \subset B_{\tau}(A)$.
Since $\phi |_{\overline{B}_R(x)}=0$, by the definition of $\Phi$, we have $\Phi(\phi)=0$.
On the other hand,
\[\upsilon_{R,x}^{\mathrm{sing}}(A) \le \int_{B_{2R}(x)}\phi d\upsilon_{R,x}^{\mathrm{sing}} = \Phi(\phi) =0.\]
Thus, we have Claim \ref{plokij}.
\

Since $V_R^j \subset \overline{B}_R(m)$, if $f_1, f_2 \in C_c^{\infty}(B_{2R}(m))$ satisfies 
$f_1|_{\overline{B}_R(m)}=f_2|_{\overline{B}_R(m)}$, then we have 
\[\int_{\partial V_R^j \setminus \partial B_R(m)}\langle \nabla r_m, N_w\rangle f_1(w)d\mathrm{vol}_{n-1}
=\int_{\partial V_R^j \setminus \partial B_R(m)}\langle \nabla r_m, N_w\rangle f_2(w)d\mathrm{vol}_{n-1}\]
for every $j$.
Especially, we have $\Phi(f_1)=\Phi(f_2)$.
By the definition, for every $f \in C^{\infty}(\overline{B}_R(m))$, there exists $F \in C_c^{\infty}(B_{2R}(m))$ such that  $F|_{\overline{B}_R(m)}=f$.
If we put $\Phi(f)=\Phi(F)$,
then, $\Phi(f)$ does not depend on the choice of $F$. 
Thus for $f \in C^{\infty}(\overline{B}_R(m))$, $\Phi(f)$ is well defined, we have,
\begin{align}
\Phi(f) = \Phi(F) &= \lim_{j \rightarrow \infty}\int_{\partial V_R^j \setminus \partial B_R(m)}\langle \nabla r_m, N_w\rangle F(w)d\mathrm{vol}_{n-1}\\
&=\lim_{j \rightarrow \infty}\int_{\partial V_R^j \setminus \partial B_R(m)}\langle \nabla r_m, N_w\rangle f(w)d\mathrm{vol}_{n-1}
\end{align}
and
\[\Phi(f)=\Phi(F) =\int_{B_{2R}(m)}Fd\upsilon_{R,m}^{\mathrm{sing}} =\int_{\overline{B}_R(m)}fd\upsilon_{R,m}^{\mathrm{sing}}.\]
Therefore, we have 
\[\Delta_R^{\mathrm{dist}}r_m(f)=\int_{B_R(m)}f\Delta r_md\mathrm{vol}+ \int_{\partial B_R(m) \setminus C_m}fd\mathrm{vol}_{n-1}
+\int_{\overline{B}_R(m)}fd\upsilon_{R,m}^{\mathrm{sing}}\]
for every $f \in C^{\infty}(\overline{B}_R(x))$.
By taking $f=1$ and the definition of $\Delta_R^{\mathrm{dist}}r_m$, we have 
\[0=\Delta_R^{\mathrm{dist}}r_m(1)=\int_{B_R(m)}\Delta r_md\mathrm{vol} +\mathrm{vol}_{n-1}(\partial B_R(m) \setminus C_m)+
\upsilon_{R,m}^{\mathrm{sing}}(\overline{B}_R(m)).\]
Thus, we have 
\[\upsilon_{R,x}^{\mathrm{sing}}(\overline{B}_r(x))=-\int_{B_R(x)}\Delta r_xd\mathrm{vol}- \mathrm{vol}_{n-1}(\partial B_R(x) \setminus C_x).\]
Especially, we have 
\begin{align}
\int_{B_R(m)}|\Delta r_m|d\mathrm{vol}+\mathrm{vol}(\partial B_R(m) \setminus C_m)+\upsilon_{R,m}^{\mathrm{sing}}(\overline{B}_R(m))
&=\int_{B_R(m)}|\Delta r_m|d\mathrm{vol}-\int_{B_R(m)}\Delta r_md\mathrm{vol} \\
&=-2\int_{B_R(m) \cap \{\Delta r_m < 0\}}\Delta r_md\mathrm{vol}.
\end{align}
\begin{claim}\label{last}
We have $\mathrm{supp}(\upsilon_{R,m}^{\mathrm{sing}}) \subset C_m \cap \overline{B}_R(m)$.
\end{claim}
The proos is as follows.
First, we shall prove $\mathrm{supp}(\upsilon_{R,m}^{\mathrm{sing}}) \subset \partial B_R(m) \cup C_m$.
It suffices to check that $\upsilon_{r,m}^{\mathrm{sing}}(A)=0$ for every compact set $A \subset \overline{B}_R(m)$ satisfying 
$A \cap (\partial B_R(m) \cup C_m) = \emptyset$.
We take $\tau >0$ satisfying $\tau << \overline{A, C_m \cup \partial B_R(m)}$.
We also take $\phi \in C_c^{\infty}(B_{2R}(m))$ satisfying $0 \le \phi \le 1$, $\phi|_A=1$ and $\mathrm{supp}\phi \subset B_{\tau}(A)$.
Then, we have 
\[\upsilon_{R,m}^{\mathrm{sing}}(A) \le \int_{B_{2R}(m)}\phi d\upsilon_{R,m}^{\mathrm{sing}} = \lim_{j \rightarrow \infty}
\int_{\partial V_R^j \setminus \partial B_R(m)}\langle \nabla r_m, N_w\rangle \phi(w)d\mathrm{vol}_{n-1}.\]
We take $j$ satisfying $j^{-1}< \frac{\tau}{100}$, and $w \in \partial V_R^j \setminus \partial B_R(m)$.
By Claim \ref{11112}, there exists $u \in S_xM$ such that $w = \exp_x \phi_R^j(u)u$.
Since $\phi_R(u)-j^{-1}\le \phi_R^j(u)\le \phi_R(u)$, if $\phi_R(u) =t(u)$, then, 
since $\overline{w, C_m}\le j^{-1} < \frac{\tau}{100},$
we have $\phi(w)=0$.
On the other hand, if $\phi_R(u)=R$, then, since
$\overline{w, \partial B_R(m)} \le j^{-1} < \frac{\tau}{100}$, 
we have $\phi(w)=0$.
Therefore, we have $\phi |_{\partial V_R^j \setminus \partial B_R(m)}=0$.
Thus, we have $\upsilon_{r,m}^{\mathrm{sing}}(A)=0$.
Finally, we shall prove $\mathrm{supp}(\upsilon_{R,m}^{\mathrm{sing}}) \subset C_m \cap \overline{B}_R(m)$.
It suffices to check that $\upsilon_{r,m}^{\mathrm{sing}}(\partial B_R(m) \setminus C_m)=0$.
Since $\overline{O}_i$ is compact, there exists a sequence of nonincresing sequence $\tau_i > 0$
such that $\tau_i \rightarrow 0$ and $t(u) > R +\tau_i$ for every $u \in O_i$.
We put $U_i=\{\exp_mtu; u \in O_i, R-\tau_i < t < R + \tau_i\}$ and 
$V_i=\{\exp_mtu;; u \in O_i, R-\tau_{i+1}/2<t<R+\tau_{i+1}/2\}$.
Since $\overline{O}_i \subset O_{i+1}$, we have $\overline{V}_i \subset U_{i+1}$.
We take $\phi_i \in C_c^{\infty}(B_{2R}(m))$ satisfying $\phi_i|_{\overline{V}_i}=1$, $0 \le \phi_i \le 1$ and $\mathrm{supp}\phi_i \subset 
U_{i+1}$.
We fix $i$.
Then, since 
$U_i \cap \partial V_R^j \subset U_i \cap \partial B_R(m)$
for every sufficiently large $j$,
we have 
$\mathrm{supp}\phi_i \cap (\partial V_R^j \setminus \partial B_R(m)) \subset U_{i+1} \cap (\partial V_R^j \setminus \partial B_R(m))= \emptyset$
for every sufficiently large $j$.
Thus, we have
\begin{align}
\upsilon_{R,m}^{\mathrm{sing}}(\partial B_R(m) \cap V_i) &\le \int_{B_{2R}(m)}\phi_i d\upsilon_{R,m}^{\mathrm{sing}} \\
&= \lim_{j \rightarrow \infty}\int_{\partial V_R^j \setminus \partial B_R(m)}\langle \nabla r_m, N_w\rangle \phi_i(w)d\mathrm{vol}_{n-1} \\
&=0. 
\end{align}
By letting $i \rightarrow \infty$, we have $\upsilon_{R,x}^{\mathrm{sing}}(\partial B_R(x) \setminus C_x)=0$.
Therefore, we have Claim \ref{last}.
\

Thus, we have the assertion.
\end{proof}
The following corollary is used in the proof of Theorem \ref{26}.
See also \cite[Theorem $4. 1$]{ch1}.
\begin{corollary}\label{mmmmm}
Let $H$ be a real number, $(M, m)$ a pointed complete $n$-dimensional $(n \ge 2)$ Riemannian manifold with $\mathrm{Ric}_M \ge (n-1)H$, $R$ a positive number and 
$f$ a nonnegative valued Lipschitz function on $\overline{B}_R(m)$.
Then, we have 
\[\int_{B_R(m)}\langle df, dr_m\rangle d\mathrm{vol} \ge -(n-1)\int_{B_R(m)}\frac{\underline{k}'_H(\overline{m,w})}{\underline{k}_H(\overline{m, w})}f(w)d\mathrm{vol}.\]
\end{corollary}
\subsection{Co-area formula for distance functions}
In this subsection, we shall give several measure theoretical properties on non-collapsing Euclidean cones.
For example, we will prove co-area formula for distance functions (see Proposition \ref{0005}).
Throughout this subsection, we fix an $(n, -1)$-Ricci limit space $(n \ge 2)$ $(Y, y, \upsilon)$ and assume that the following properties hold:
\begin{enumerate}
\item There exists a compact geodesic space $X$ such that $\mathrm{diam}X \le \pi$ and $(Y, y)=(C(X), p)$.
\item $\mathrm{dim}_HX=n-1$. Here, $\mathrm{dim}_HX$ is the Hausdorff dimension of $X$.
\end{enumerate}
Then by \cite[Theorem $5. 9$]{ch-co1}, there exists $C>0$ such that $\upsilon=CH^n$.
First, we shall recall definitions of lower dimensional Hausdorff measures associated to $\upsilon$ and standard (spherical) Hausdorff measures (see section $2$ in \cite{ch-co2}).
For convenience, we will use the notaion below: $r^{-\alpha}\upsilon(B_r(x))=0$ for every $x \in Y$ and $\alpha \ge 0$ if $r=0$.
For $\alpha \in \mathbf{R}_{\ge0}$, $\delta > 0$ and a set $A \subset Y$, we put
\[(\upsilon_{-\alpha})_{\delta}(A) = \inf \left\{ \sum_{i =1}^{\infty}r_i^{-\alpha}\upsilon ( B_{r_i}(x_i)); \  x_i \in Y, \  0 \le r_i < \delta,  \ 
A \subset \bigcup_{i =1}^{\infty}B_{r_i}(x_i)\right\},\]
\[(H^{\alpha})_{\delta}(A) = \inf \left\{ \sum_{i=1}^{\infty}\omega_{\alpha}r_i^{\alpha}; \ x_i \in Y, \  0 \le r_i < \delta, \ 
A \subset \bigcup_{i=1}^{\infty}B_{r_i}(x_i)\right\}\]
and
\[\upsilon_{-\alpha}(A) = \lim_{\delta \rightarrow 0}(\upsilon_{-\alpha})_{\delta}(A),\ \ 
H^{\alpha}(A) = \lim_{\delta \rightarrow 0}(H^{\alpha})_{\delta}(A).\]
For a subset $A \subset \{1\} \times X \subset C(X)$. we also put
\[(\upsilon_{-\alpha})_{X, \delta}(A) = \left\{ \sum_{i=1}^{\infty}r_i^{-\alpha}\upsilon ( B_{r_i}(x_i)); \ x_i \in \{1\} \times X, \ 0\le r_i < \delta, \ 
A \subset \bigcup_{i=1}^{\infty}B_{r_i}(x_i)\right\},\]
\[(H^{\alpha})_{X, \delta}(A) = \left\{ \sum_{i=1}^{\infty}\omega_{\alpha}r_i^{\alpha}; \  x_i \in \{1\} \times X, \ 0\le r_i < \delta, \   
A \subset \bigcup_{i=1}^{\infty}B_{r_i}(x_i)\right\}\]
and 
\[(\upsilon_{-\alpha})_X(A) = \lim_{\delta \rightarrow 0}(\upsilon_{-\alpha})_{\delta}(A), \ \ H^{\alpha}_X(A) = \lim_{\delta \rightarrow 0}(H^{\alpha})_{\delta}(A).\]
We remark that 
$\upsilon_{-\alpha}(A) \le (\upsilon_{-\alpha})_X(A)$, $H^{\alpha}(A) \le H^{\alpha}_X(A)$
for every subset $A \subset \{1\} \times X$
and that if we define a map $\phi$ from $(X, d_X) \rightarrow (\{1\} \times X, d_{C(X)})$ by $\phi (x)=(1, x)$, then 
$H^{n-1}(A)=H^{n-1}_X(\phi(A))$ for every $A \subset X$.
\begin{lemma}\label{0001}
We have $\upsilon_{-1}(A) = (\upsilon_{-1})_X(A)$ for every Borel set $A \subset X$.
\end{lemma}
\begin{proof}
We fix sufficiently small positive numbers $\delta, \epsilon > 0$.
By definition, there exists $\{B_{r_i}(x_i)\}_i$ such that $0 \le r_i < \delta$, $x_i = (t_i, w_i) \in C(X)= \mathbf{R}_{\ge0} \times X /\{0\} \times X$ and 
$\left|(\upsilon_{1})_{\delta}(A)-\sum_{i=1}^{\infty}r_i^{-1}\upsilon(B_{r_i}(x_i))\right| < \epsilon.$
Without loss of generality, we can assume that $B_{r_i}(x_i) \cap A \neq \emptyset$ for every $i$.
We put $y_i = (1, w_i) \in  C(X)$ and $\hat{y}_i = (1, w_i) \in (\mathbf{R} \times X, \sqrt{d_{\mathbf{R}^2}+d_X^2})$.
It is easy to check that  the map $\Phi_i(s, z)=(s, z)$ from $B_{5r_i}(x_i)$ to $\mathbf{R} \times X$ gives $(1 \pm \Psi(\delta))$-bi-Lipschitz equivalent to the image.
Therefore, we have $B_{r_i}(x_i) \cap (\{1\} \times X) \subset B_{(1 + \Psi(\delta))\sqrt{r_i^2-\overline{x_i, y_i}^2}}(y_i)$.
On the other hand, since $|t_i-1| \le \delta$, a map $\hat{\Phi}_i(t, w)=(t+t_i-1, w)$ from $B_{(1+\Psi(\delta))r_i}(\hat{y}_i)$ to $C(X)$ gives $(1\pm \Psi(\delta))$-bi-Lipschitz equivalent to the image.
By $\hat{\Phi}_i(\hat{y}_i)=x_i$, we have $\mathrm{Image}\,\hat{\Phi} \subset B_{(1+\Psi(\delta))r_i}(x_i)$.
Therefore, we have 
$H^n(B_{(1 + \Psi(\delta))r_i}(y_i)) \le (1+\Psi(\delta))H^n(B_{(1+\Psi(\delta))r_i}(x_i))\le (1 + \Psi(\delta))H^n(B_{r_i}(x_i)).$
Thus, since $\upsilon = CH^n$, we have 
\begin{align}
(\upsilon_{-1})_{X, (1 + \Psi(\delta))\delta}(A) &\le \sum_{i=1}^{\infty} ((1 + \Psi(\delta))r_i)^{-1}CH^n(B_{(1+\Psi(\delta))r_i}(y_i))\\
&\le (1 + \Psi(\delta))\sum_{i=1}^{\infty} r_i^{-1}CH^n(B_{r_i}(x_i)) \\
&\le (1 + \Psi(\delta))((\upsilon_{-1})_{X, \delta}(A) + \epsilon).
\end{align}
By letting $\epsilon \rightarrow 0$ and $\delta \rightarrow 0$, we have the assertion.
\end{proof}
Similarly, we have the following lemma:
\begin{lemma}\label{0002}
We have $H_X^{n-1}(A) = H^{n-1}(A)$ for every Borel set $A \subset \{1\} \times X$.
\end{lemma}
We shall remark the following:
By Bishop-Gromov volume comparison theorem for $\upsilon$, there exists $V>1$ such that 
$V^{-1} \le \lim_{r \rightarrow 0}\upsilon(B_r(x))/\omega_n r^n \le V$
for every $x \in B_2(p)$.
On the other hand, since $\upsilon=CH^n$, we have $\lim_{r \rightarrow 0}\upsilon(B_r((t,w)))/\omega_n r^n = \lim_{r \rightarrow 0}\upsilon(B_r((s,w)))/\omega_n r^n$ for every $0 < s < t < \infty$ and $w \in X$.
By these facts and Corollary \ref{132}, it is easy to check that there exists $C_1>1$ such that $C_1^{-1}\upsilon_{-1}(A) \le H^{n-1}(A) \le C_1 \upsilon_{-1}(A)$ for every Borel subset $A$ of $C(X)$.
\begin{lemma}\label{0004}
The product measure $H^1 \times H^{n-1}$ on $\mathbf{R} \times X$ is equal to $H^n$.
\end{lemma}
\begin{proof}
It suffices to check that $H^n([0, a] \times A) = a H^{n-1}(A)$ for every Borel subset $A$ os $X$ and $a > 0$.
By Corollary \ref{25}, there exists a Borel subset $\hat{X}$ of $X$ such that the following properties hold:
\begin{enumerate}
\item $H^{n-1}(X \setminus \hat{X})=0$.
\item For every $x \in \hat{X}$ and $\epsilon >0$, there exist $r_x^{\epsilon} > 0$ such that for every $0 < r < r_x^{\epsilon}$,  there exist a compact set $C_r^x \subset \overline{B}_r(x)$ and a Lipschitz $\phi_r^x$ from $C_r^x$ to $\mathbf{R}^{n-1}$
 such that 
\[\frac{H^{n-1}(\overline{B}_r(x) \setminus C_r^x)}{H^{n-1}(\overline{B}_r(x))}\le \epsilon\]
and that $\phi_r^x$ gives $(1 \pm \epsilon)$-bi-Lipschitz equivalent to the image.
\end{enumerate}
For every $x \in \hat{X}$ and $\epsilon > 0$, by Fubini's theorem, we have 
\begin{align}
H^n([0, a] \times C_r^x) &=(1 \pm \epsilon)H^n([0,a] \times \phi_r^x(C_r^x)) \\
&=(1 \pm \epsilon)aH^{n-1}(\phi_r^x(C_r^x))\\
&=(1 \pm \epsilon)aH^{n-1}(C_r^x)\\
&=(1 \pm \epsilon)aH^{n-1}(\overline{B}_r(x))
\end{align}
for every sufficiently small $r >0$.
On the other hand, by the proof of \cite[Lemma $5.2$]{ho2}, we have
$H^n([0, a] \times \hat{A}) \le C(n) a H^{n-1}(\hat{A})$  for every $\hat{A} \subset X$.
Thus, we have 
\[\lim_{r \rightarrow 0}\frac{H^n([0,a] \times \overline{B}_r(x))}{a H^{n-1}(\overline{B}_r(x))}=1\]
for every $x \in \hat{X}$.
Therefore, there exists a Borel set $\hat{A} \subset A$ such that $H^{n-1}(A \setminus \hat{A})=0$
\[\lim_{r \rightarrow 0}\frac{H^n([0,a] \times \overline{B}_r(x))}{a H^{n-1}(\overline{B}_r(x))}=1\]
and 
\[\lim_{r \rightarrow 0}\frac{H^{n-1}(A \cap \overline{B}_r(x))}{H^{n-1}(\overline{B}_r(x))}=1\]
for every $x \in \hat{A}$.
We remark that $H^n([0, a] \times (A \setminus \hat{A}))\le C(n)a H^{n-1}(A \setminus \hat{A})=0$.
We fix a sufficiently small $\epsilon >0$.
By Proposition \ref{cov}, there exists a pairwise disjoint collection $\{\overline{B}_{r_i}(x_i)\}_{i \in \mathbf{N}}$ such that 
$x_i \in \hat{A}$, $r_i < \epsilon$,
$\hat{A} \setminus \bigcup_{i=1}^N\overline{B}_{r_i}(x_i) \subset \bigcup_{i=N+1}^{\infty}\overline{B}_{5r_i}(x_i)$
for every $N \in \mathbf{N}$ and 
\[\left|\frac{H^n([0,a] \times \overline{B}_r(x_i))}{a H^{n-1}(\overline{B}_r(x_i))}-1\right|+\left|\frac{H^{n-1}(A \cap \overline{B}_r(x_i))}{H^{n-1}(\overline{B}_r(x_i))}-1\right|<\epsilon \]
for every $0 < r < r_i$.
We take $N$ satisfying
$\sum_{i=N+1}^{\infty}H^{n-1}(\overline{B}_{r_i}(x_i)) < \epsilon.$
Then, we have 
\begin{align}
H^n([0,a] \times \hat{A})& \le \sum_{i=1}^NH^n([0,a] \times \overline{B}_{r_i}(x_i))+
\sum_{i=N+1}^{\infty}H^n([0,a] \times \overline{B}_{5r_i}(x_i)) \\
&\le \sum_{i=1}^NH^n([0,a] \times \overline{B}_{r_i}(x_i))+ a C(n)
\sum_{i=N+1}^{\infty}H^{n-1}(\overline{B}_{5r_i}(x_i)) \\
&\le \sum_{i=1}^NH^n([0,a] \times \overline{B}_{r_i}(x_i))+ \Psi(\epsilon;n,a, C_1)\\
&\le a \sum_{i=1}^NH^{n-1}(\overline{B}_{r_i}(x_i)) + \Psi(\epsilon;n, a, C_1)\\
& \le a(1+\epsilon)(H^{n-1}(\hat{A}) + \epsilon)+\Psi(\epsilon;n, a, C_1).
\end{align}
Therefore, we have 
\[H^n([0,a] \times A) \le a H^{n-1}(A).\]
On the other hand, we have
\begin{align*}
aH^{n-1}(A) = a\left(\sum_{i=1}^NH^{n-1}(\overline{B}_{r_i}(x_i) ) + \Psi(\epsilon; n, C_1)\right)
&\le(1 +\epsilon) \sum_{i=1}^NH^n([0,a] \times \overline{B}_{r_i}(x_i)) + \Psi(\epsilon;n, a, C_1)
\end{align*}
and
\begin{align*}
\frac{H^n([0,a] \times (\overline{B}_{r_i}(x_i) \setminus A))}{H^n([0, a] \times \overline{B}_{r_i}(x_i))}\le C(n)(1 + \epsilon )\frac{aH^{n-1}(\overline{B}_{r_i}(x_i) \setminus A)}{a H^{n-1}(\overline{B}_{r_i}(x_i))}\le \Psi(\epsilon;n).
\end{align*}
Therefore, we have 
\begin{align}
aH^{n-1}(A) &\le (1+\epsilon)\sum_{i=1}^NH^n([0,a] \times \overline{B}_{r_i}(x_i))+ \Psi(\epsilon;n, a, C_1) \\
&\le ( 1 + \Psi(\epsilon; n))\sum_{i=1}^NH^n\left(([0,a] \times \overline{B}_{r_i}(x_i)) \cap A\right)+ \Psi(\epsilon;n, a, C_1) \\
&\le ( 1 + \Psi(\epsilon; n))H^n([0,a] \times A) + \Psi(\epsilon;n, a, C_1).
\end{align}
Therefore, we have 
\[a H^{n-1}(A) \le H^n([0,a] \times A).\]
Thus, we have the assertion.
\end{proof}
\begin{proposition}[Co-area formula for distance functions on non-collapsing Euclidean cones]\label{0005}
We have
\[\int_{C(X)}fdH^n=\int_0^{\infty}\int_{\partial B_t(p)}fdH^{n-1}dt\]
for every $f \in L^1(C(X))$.
\end{proposition}
\begin{proof}
By \cite[Theorem $5.2$]{ho} and $C_1\upsilon_{-1} \le H^{n-1}\le C_1\upsilon_{-1}$, it suffices to check that
\[\lim_{r \rightarrow 0}\frac{1}{H^n(B_r(x))}\int_0^{\infty}H^{n-1}(\partial B_t(p) \cap \overline{B}_r(x))dt=1\]
for every $x \in C(X) \setminus \{p\}$.
We put $R = \overline{p, x} > 0$ and fix sufficiently small $r >0$.
Then, since a map $\Phi(t, w)=(t, w)$ from $B_r(x)$ to $\mathbf{R} \times X$ gives $(1 \pm \Psi(r))$-bi-Lipschitz equivalent to the image, we
have
\[B_{(1-\Psi(r))r}(\Phi(x)) \subset \Phi(B_r(x)) \subset B_{(1 + \Psi(r))r}(\Phi(x)).\]
On the other hand, by Lemma \ref{0004} and Fubini's Theorem, we have
\[H^n\left(B_{(1+\Psi(r))r}(\Phi(x))\right)=\int_{R-(1+\Psi(r))r}^{R+(1+\Psi(r))r}H^{n-1}\left((\{t\} \times X) \cap B_{(1+\Psi(r))r}(\Phi(x))\right)dt.\]
Since $\Phi(\partial B_t(p) \cap B_r(x)) \subset  (\{t\} \times X) \cap B_{(1+\Psi(r))r}(\Phi(x))$, we have
\[H^n\left(B_{(1+\Psi(r))r}(\Phi(x))\right) \ge (1-\Psi(r;n))\int_{R-(1+\Psi(r))r}^{R+(1+\Psi(r))r}H^{n-1}(\partial B_t(p) \cap B_r(x))dt.\]
Therefore, we have 
\[1 \ge \limsup_{r \rightarrow 0}\frac{1}{H^n(B_r(x))}\int_0^{\infty}H^{n-1}(\partial B_t(p) \cap B_r(x))dt.\]
Similarly, we have 
\[1 \le \liminf_{r \rightarrow 0}\frac{1}{H^n(B_r(x))}\int_0^{\infty}H^{n-1}(\partial B_t(p) \cap B_r(x))dt.\]
Therefore, we have the assertion.
\end{proof}
\begin{proposition}\label{0003}
We have $\upsilon_{-1}(A) = C(n) CH^{n-1}(A)$ for every Borel set $A \subset \{1\} \times X$.
\end{proposition}
\begin{proof}
By \cite{co3}, we have
\[\lim _{r \rightarrow 0}\frac{H^n(B_r(z))}{\omega_n r^n} = 1\]
for every $z \in \mathcal{R}_n(Y)$.
Since $\mathcal{R}_n(Y) \cap (\{1\} \times X)= \{1\} \times \mathcal{R}_{n-1}(X)$,
by Proposition \ref{0005}, we have 
$H^{n-1}(X \setminus \mathcal{R}_{n-1}(X))=0$.
We fix $\epsilon, \delta, \tau > 0$.
We put 
\[A_{\tau}=\left\{a \in A \cap \mathcal{R}_{n-1}(X);\ \left|\frac{H^n(B_r(a))}{\omega_nr^n}-1\right|<\epsilon \ \mathrm{for\ every}\ 0<r\le \tau \right\}.\]
By the definition of $\upsilon_{-1}$, there exists $\{B_{r_i}(x_i)\}_i$ such that $x_i \in A_{\tau}$, $r_i < \min\{\delta, \tau\}$ and 
$|\upsilon_{-1}(A_{\tau})-\sum_{i=1}^{\infty}r_i^{-1}\upsilon(B_{r_i}(x_i))|<\epsilon$.
Thus, we have 
\begin{align}
(H^{n-1})_{\delta}(A_{\tau})&\le \sum_{i=1}^{\infty}\omega_{n-1}r_i^{n-1}\\
&\le \sum_{i=1}^{\infty}\frac{\omega_{n-1}}{\omega_n}r_i^{-1}(1+\epsilon)H^n(B_{r_i}(x_i)) \\
&= \sum_{i=1}^{\infty}\frac{\omega_{n-1}}{\omega_n}(1+\epsilon)r_i^{-1}C^{-1}\upsilon(B_{r_i}(x_i))\\
&\le \sum_{i=1}^{\infty}\frac{\omega_{n-1}}{\omega_n}(1+\epsilon)C^{-1}(\upsilon_{-1}(A_{\tau})+\epsilon).
\end{align}
By letting $\delta \rightarrow 0$, $\tau \rightarrow 0$ and $\epsilon \rightarrow 0$, we have  
\[CH^{n-1}(A) \le \frac{\omega_{n-1}}{\omega_{n}}\upsilon_{-1}(A).\]
\begin{claim}\label{0006}
There exists a Borel subset $Z$ of $\{1\} \times X$ such that $H^{n-1}((\{1\} \times X) \setminus Z) = 0$,
\[\lim_{r \rightarrow 0}\frac{H^{n-1}(\overline{B}_r(z) \cap (\{1\} \times X))}{\omega_{n-1}r^{n-1}}=1\]
for every $z \in Z$.
\end{claim}
The proof is as follows.
Let $x$ be a point in $X$ and $\{r_i\}_i$ a sequence of positive numbers satisfying $r_i \rightarrow 0$.
We assume that there exists a tangent cone $(T_xX, 0_x)$ of $X$ at $x$ such that 
$(X, x, r_i^{-1}d_X) \rightarrow (T_xX, 0_x)$. 
By \cite[Claim $4. 5$]{ho2} and \cite[Theorem $5. 9$]{ch-co1}, we have $(C(X), r_i^{-1}d_{C(X)}, (1, x), H^n) \rightarrow (\mathbf{R} \times T_xX, (0, 0_x), H^n)$.
Moreover, 
By the $H^{n-1}$-rectifiability of $T_xX$ (Corollary \ref{25}) and an argument similar to the proof of Lemma \ref{0004}, we have $H^1 \times H^{n-1}=H^n$ on $\mathbf{R} \times T_xX$.
Since a sequence of compact sets $[-1, 1] \times \overline{B}_1^{r_i^{-1}d_X}(x) \subset C(X)$ converges to $[-1, 1] \times \overline{B}_1(0_x)$, by Proposition \ref{sup} and Proposition \ref{10103}, we have 
\[\lim_{i \rightarrow \infty}H^n([-1, 1] \times \overline{B}_1^{r_i^{-1}d_X}(x)) = H^n([-1, 1] \times \overline{B}_1(0_x)).\]
By Proposition \ref{0005}, we have $H^n([-1, 1] \times \overline{B}_1^{r_i^{-1}d_X}(x))=2H^{n-1}(\overline{B}_1^{r_i^{-1}d_X}(x))$.
Especially, we have 
\[\lim_{i \rightarrow \infty}H^{n-1}(\overline{B}_1^{r_i^{-1}d_X}(x))=H^{n-1}(\overline{B}_1(0_x)).\]
Therefore, if we put $Z=\mathcal{R}_{n}(Y)\cap (\{1\} \times X)$, then we have Claim \ref{0006}.
\

We put $W = \mathrm{Leb}(A \cap Z)$ with respect to the measure $H^{n-1}$.
By Proposition \ref{cov}, there exists a pairwise disjoint collection $\{\overline{B}_{r_i}(a_i)\}_i$ such that 
$a_i \in W$, $r_i < \delta/100$, 
$W \setminus \bigcup_{i=1}^N\overline{B}_{r_i}(a_i) \subset \bigcup_{i=N+1}^{\infty}\overline{B}_{5r_i}(a_i)$
for every $N$
and 
\[\left|\frac{H^n(B_{r_i}(a_i))}{\omega_nr_i^n}-1\right|+\left|\frac{H^{n-1}(\overline{B}_{r_i}(a_i) \cap W)}{\omega_{n-1}r_i^{n-1}}-1\right|<\epsilon\]
for every $i$.
We take $N$ satisfying
$\sum_{i=N+1}^{\infty}H^{n-1}(\overline{B}_{r_i}(a_i) \cap W) < \epsilon.$
Therefore,  we have 
$\sum_{N+1}^{\infty}H^{n-1}(\overline{B}_{5r_i}(a_i) \cap W) < \Psi(\epsilon; n, C_1).$
Then, by the assumption, we have 
$\sum_{i=N+1}^{\infty} \omega_{n-1}r_i^{n-1} \le \Psi(\epsilon; n, C_1).$
Therefore, we have
\begin{align}
(\upsilon_{-1})_{\delta}(W)&\le \sum_{i=1}^Nr_i^{-1}\upsilon(\overline{B}_{r_i}(a_i)) + \sum_{i=N+1}^{\infty}(5r_i)^{-1}
\upsilon(\overline{B}_{5r_i}(a_i)) \\
&\le \sum_{i=1}^Nr_i^{-1}CH^n(\overline{B}_{r_i}(a_i)) + \sum_{i=N+1}^{\infty}C(n)Cr_i^{n-1} \\
&\le \sum_{i=1}^Nr_i^{-1}CH^n(\overline{B}_{r_i}(a_i)) + \Psi(\epsilon;n, C, C_1) \\
&\le \sum_{i=1}^NC\omega_nr_i^{n-1}(1+\epsilon)+\Psi(\epsilon;n, C, C_1) \\
&\le \frac{C\omega_n}{\omega_{n-1}}(1+\epsilon)\sum_{i=1}^NH^{n-1}(\overline{B}_{r_i}(a_i) \cap W) + \Psi(\epsilon;n, C, C_1)\\
&\le \frac{C\omega_n}{\omega_{n-1}}(1 + \epsilon)H^{n-1}(W)+\Psi(\epsilon;n, C, C_1).
\end{align}
By letting $\delta \rightarrow 0$ and $\epsilon \rightarrow 0$, we have 
\[\upsilon_{-1}(A) \le \frac{C\omega_n}{\omega_{n-1}}H^{n-1}(A).\]
Thus, we have the assertion.
\end{proof}
We end this subsection by giving a proof of the following proposition:
\begin{proposition}\label{9009}
We have 
\[H^{n-1}(B_t(x))\le C(n) \frac{t^{n-1}}{s^{n-1}}H^{n-1}(B_s(x))\]
for every $0 < s < t \le \pi$ and $x \in X$.
\end{proposition}
\begin{proof}
We remark that there exists $C_2 >1$ such that for every metric space $\hat{X}$,
a bi-Lipschitz map $f_{\hat{X}}(\hat{x})=(1, \hat{x})$ from $\hat{X}$ to $\{1\} \times \hat{X} \subset C(\hat{X})$ satisfies $\mathbf{Lip}f_{\hat{X}}
+ \mathbf{Lip}f_{\hat{X}}^{-1} \le C_2$.
Therefore, by \cite[Theorem $5.7$]{ho} and Proposition \ref{0005}, we have
\begin{align}
H^{n-1}(B_t(x))&\le C(n)H^{n-1}(B_{C_2t}(1, x) \cap (\{1\} \times X))\\
&=C(n)C^{-1}\upsilon_{-1}(B_{C_2t}(1, x) \cap (\{1\} \times X)) \\
&\le \frac{C(n)\upsilon \left(C_p(B_{C_2t}(1, x) \cap (\{1\} \times X)) \cap A_p(\max \{0, 1-C_2t\}, 1)\right)}{C \mathrm{vol}\,A_p(\max \{0, 1-C_2t\}, 1)} \\
&\le \frac{C(n)}{Ct} \upsilon (B_{5C_2t}(1, x))\\
&\le \frac{C(n)}{Ct} \frac{t^n}{s^n}\upsilon (B_{C_2^{-1}s}(1, x))\\
&\le C(n)\frac{t^{n-1}}{s^n}\int_{\max\{0, 1-C_2^{-1}s\}}^{1+C_2^{-1}s}H^{n-1}(\partial B_r(p) \cap B_{C_2^{-1}s}(1, x))dr\\
&\le C(n)\frac{t^{n-1}}{s^n}\int_{\max\{0, 1-C_2^{-1}s\}}^{1 + C_2^{-1}s}r^{n-1}H^{n-1}( \partial B_1(p) \cap B_{C_2^{-1}s}(1, x))dr\\
&\le C(n)\frac{t^{n-1}}{s^n}sH^{n-1}( \partial B_1(p) \cap B_{C_2^{-1}s}(1, x))\\
&\le C(n)\frac{t^{n-1}}{s^{n-1}}H^{n-1}( \partial B_1(p) \cap B_{C_2^{-1}s}(1, x))\\
&\le C(n)\frac{t^{n-1}}{s^{n-1}}H^{n-1}(B_s(x)).
\end{align}
\end{proof}


\bigskip
\address{
Research Institute for 
Mathematical Sciences, \\
Kyoto University, \\
Kyoto 606-8502 \\
Japan
}
{honda@kurims.kyoto-u.ac.jp}


\begin{thebibliography}{99}
\bibitem{alm}
\textsc{F. Almgren, Jr.},
Almgren's big regularity paper, volume 1 of World Scientific Monograph Series in Mathematics. World Scientific Publishing Co. Inc., River Edge, NJ, 2000.
\bibitem{BKN}
\textsc{S. Bando, A. Kasue and H. Nakajima},
On a construction of coordinates at infinity on manifolds with fast curvature decay and maximal volume growth, Invent. Math. vol. 97 (1989), 313-349.
\bibitem{bubuiv}
\textsc{D. Burago, Y. Burago and S. Ivanov},
A course in metric geometry, Grad. Stud. Math. 33, American Mathematical Society, Providence, RI, 2001.
\bibitem{ch1}
\textsc{J. Cheeger}, 
Degeneration of Riemannian metrics under Ricci curvature bounds, Lezioni Fermiane, Scuola Normale Superiore, Pisa, 2001.
\bibitem{ch2}
\textsc{J. Cheeger}, 
Differentiability of Lipschitz functions on metric measure spaces, Geom. Funct. Anal. 9 (1999), 428-517.
\bibitem{ch-co}
\textsc{J. Cheeger and T. H. Colding}, 
Lower bounds on Ricci curvature and the almost rigidity of warped products, Ann. of Math. 144 (1996), 189-237.
\bibitem{ch-co1}
\textsc{J. Cheeger and T. H. Colding}, 
On the structure of spaces with Ricci curvature bounded below, I, J. Differential Geom. 45 (1997), 406-480.
\bibitem{ch-co2}
\textsc{J. Cheeger and T. H. Colding}, 
On the structure of spaces with Ricci curvature bounded below, II, J. Differential Geom. 54 (2000), 13-35.
\bibitem{ch-co3}
\textsc{J. Cheeger and T. H. Colding}, 
On the structure of spaces with Ricci curvature bounded below, III, J. Differential Geom. 54 (2000), 37-74.
\bibitem{ch-co-mi}
\textsc{J. Cheeger, T. H. Colding and W. P. Minicozzi II},
Linear growth harmonic functions on complete manifolds with nonnegative Ricci curvature, Geom. Funct. Anal. 5 (1995), 948-954.
\bibitem{cheng1}
\textsc{S. Y. Cheng},
Liouville theorem for harmonic maps, Proc. Sympos. Pure Math. 36 (1980), 147-151.
\bibitem{cheng2}
\textsc{S. Y. Cheng},
Eigenfunctions and nodal sets, Comment. Math. Helv. 51 (1976), 43-55.
\bibitem{ch-yau}
\textsc{S. Y. Cheng and S. T. Yau},
Differential equations on Riemannian manifolds and their geometric applications, Comm. Pure Appl. Math. 28 (1975), 333-354.
\bibitem{co1}
\textsc{T. H. Colding},
Shape of manifolds with positive Ricci curvature, Invent. Math. 124 (1996), 175-191.
\bibitem{co2}
\textsc{T. H. Colding}
Large manifolds with positive Ricci curvature, Invent. Math. 124 (1996),  193-214.
\bibitem{co3}
\textsc{T. H. Colding}, 
Ricci curvature and volume convergence, Ann. of Math. 145 (1997), 477-501.
\bibitem{co-mi1}
\textsc{T. H. Colding and W. P. Minicozzi II},
Generalized Liouville properties of manifolds, Math. Res. Letters 3 (1996), 723-729.
\bibitem{co-mi2}
\textsc{T. H. Colding and W. P. Minicozzi II},
Harmonic functions on manifolds, Ann. of Math.  146 (1997), 725-747.
\bibitem{co-mi3}
\textsc{T. H. Colding and W. P. Minicozzi II},
Harmonic functions with polynomial growth, J. Differential Geom. 46 (1997), 1-77.
\bibitem{co-mi4}
\textsc{T. H. Colding and W. P. Minicozzi II},
Large scale behavior of the kernel of Schr$\ddot{\mathrm{a}}$dinger operators, Amer. J. Math. 117 (1997), 1355-1398.
\bibitem{co-mi5}
\textsc{T. H. Colding and W. P. Minicozzi II},
Liouville theorems for harmonic sections and applications, Comm. Pure Appl. Math. 51 (1998), 113-138.
\bibitem{co-mi6}
\textsc{T. H. Colding and W. P. Minicozzi II},
Weyl type bounds for harmonic functions, Invent. Math. 131 (1998) 257-298.
\bibitem{di1}
\textsc{Y. Ding},
An existence theorem of harmonic functions with polynomial growth, Proc. Amer. Math. Soc. 132 (2004),  543-551.
\bibitem{di2}
\textsc{Y. Ding},
Heat kernels and Green's functions on limit spaces, Comm. Anal. Geom. 10 (2002),  475-514.
\bibitem{di3}
\textsc{Y. Ding},
The gradient of certain harmonic functions on manifolds of almost nonnegative Ricci curvature, Israel J. Math. 129 (2002), 241-251.
\bibitem{do-fe}
\textsc{H. Donnely and C. Fefferman},
Nodal domains and growth of harmonic functions on noncompact manifolds, J. Geom. Anal. 2 (1992), 79-93.
\bibitem{fed}
\textsc{H. Federer}, 
Geometric measure theory, Springer, Berlin-New York, 1969.
\bibitem{fu}
\textsc{K. Fukaya},
Collapsing of Riemannian manifolds and eigenvalues of the laplace operator, Invent. Math. 87 (1987), 517-547.
\bibitem{fu1}
\textsc{K. Fukaya},
Hausdorff convergence of Riemannian manifolds and its applications, Recent topics in differential and analytic geometry, 143-238, Adv. Stud. Pure Math. 18-I, Academic Press, Boston, MA, (1990).
\bibitem{fu}
\textsc{K. Fukaya},
Metric Riemannian geometry, Handbook of differential geometry, Vol. II, 189-313,
Elsevier/North-Holland, Amsterdam, 2006.
\bibitem{memo}
\textsc{K. Fukaya, A. Kasue, T. Sakai, T. Shioya, Y. Otsu and T. Yamaguchi}, 
Riemannian manifolds and its limit, (in Japanese) Memoire Math. Soc. Japan, (2004).
\bibitem{fuku}
\textsc{M. Fukushima},
Dirichlet forms and Markoff processes, North Holland (Amsterdam) 1980.
\bibitem{Gi}
\textsc{D. Gilbarg and N. Trudinger},
Elliptic partial differential equations of second order, Reprint of the 1998 edition, Classics in Mathematics, Springer-Verlag, Berlin (2001).
\bibitem{gri}
\textsc{A. A. Grigor'yan},
The heat equation on noncompact Riemannian manifolds, English translation in Math. USSR Sb. 72  (1992), 47-77.
\bibitem{gr}
\textsc{M. Gromov}, 
Metric Structures for Riemannian and Non-Riemannian Spaces, Birkhauser Boston Inc, Boston, MA, 1999, Based on the 1981 French original 
[MR 85e:53051], With appendices by M. Katz, P. Pansu, and S. Semmes, Translated from the French by Sean Michael Bates. 
\bibitem{ha}
\textsc{P. Hajlasz},
Sobolev spaces on an arbitrary metric space, J. Potential Anal. 5 (1995), 403-415.
\bibitem{ha-ko}
\textsc{P. Hajlasz and P. Koskela},
Sobolev meets Poincare, C. R. Acad. Sci. Paris 320 (1995), 1211-1215.
\bibitem{har-sim}
\textsc{R. Hardt and L. Simon},
Nodal sets for solutions of elliptic equations, J. Differential Geom. 30 (1989), 505-522.
\bibitem{he-ki-ma}
\textsc{J. Heinonen, T. Kilpel$\ddot{\mathrm{a}}$inen and O. Martio},
Nonlinear potential theory for degenerate elliptic equations, Clarendon Press (Oxford, Tokyo, New York) 1993.
\bibitem{he-ko1}
\textsc{J. Heinonen and P. Koskela},
Weighted Sobolev and Poincare inequarities and quasiregular mappings of polynomial type, Math. Scand. 77 (1995), 251-271.
\bibitem{he-ko2}
\textsc{J. Heinonen and P. Koskela},
Quasiconformal maps in metric spaces with controlled geometry, Acta Math. 181 (1998), 1-61.
\bibitem{ho}
\textsc{S. Honda},
Bishop-Gromov type inequality on Ricci limit spaces, J. Math. Soc. Japan, to appear.
\bibitem{ho1}
\textsc{S. Honda},
On Low Dimensional Ricci limit spaces, preprint. 
http://www.math.kyoto-u.ac.jp/preprint/preprint2008.html
\bibitem{ho2}
\textsc{S. Honda},
Ricci curvature and almost spherical multi-suspension, Tohoku Math. J. 61 (2009), 499-522.
\bibitem{ho5}
\textsc{S. Honda},
(in preparation).
\bibitem{ka1}
\textsc{A. Kasue},
Harmonic functions of polynomial growth on complete manifolds, Proc. Sympos. Pure Math., Amer. Math. Soc. Vol. 54, Part 1, (Ed. R. Green and S. T. Yau), 1993.
\bibitem{ka2}
\textsc{A. Kasue},
Harmonic functions of polynomial growth on complete manifolds II, J. Math. Soc. Japan. 47 (1995), 37-65.
\bibitem{ka-ku1}
\textsc{A. Kasue and H. Kumura},
Spectral convergence of Riemannian manifolds, Tohoku Math. J. 4 (1994), 147-179.
\bibitem{ka-ku2}
\textsc{A. Kasue and H. Kumura},
Spectral convergence of Riemannian manifolds II, Tohoku Math. J. 2 (1996), 71-120. 
\bibitem{ka-wa}
\textsc{A. Kasue and T. Washio},
Growth of equivariant harmonic maps and harmonic morphism, Osaka J. Math. 27 (1990), 899-928; Errata, Osaka J. Math. 29 (1992), 419-420.
\bibitem{KRS}
\textsc{P. Koskela, K. Rajala and N. Shanmugalingam}
Lipschitz continuity of Cheeger-harmonic functions in metric measure spaces, J. Funct. Anal. 202 (2003), 147-173.
\bibitem{Ku}
\textsc{K. Kuwae},
Maximum principles for subharmonic functions via local semi-Dirichlet forms, Canadian J. Math. 60 (2008), 822-874. 
\bibitem{KS1}
\textsc{K. Kuwae and T. Shioya},
A topological splitting theorem for weighted Alexandrov spaces, preprint.
http://arxiv.org/abs/0903.5150
\bibitem{KS2}
\textsc{K. Kuwae and T. Shioya},
Convergence of spectral structures: a functional analytic theory and its applications to spectral geometry, Comm. Anal. Geom. 11 (2003), 599-673.
\bibitem{KS3}
\textsc{K. Kuwae and T. Shioya},
Infinitesimal Bishop-Gromov condition for Alexandrov spaces, Proceedings of the 1st Math. Soc. Japan Seasonal Institute, Kyoto, 2008,(2009)
\bibitem{KS4}
\textsc{K. Kuwae and T. Shioya},
On generalized measure contraction property and energy functionals over Lipschitz maps, Potential Anal. 15 (2001), 105-121. 
\bibitem{li1}
\textsc{P. Li},
A lower bound for the first eigenvalues of Laplacian on a compact manifold, Indiana Univ. Math. J. 28 (1979), 1013-1019.
\bibitem{li2}
\textsc{P. Li},
Large time behavior of the heat equation on complete manifolds with non-negative Ricci curvature, Ann. of Math.  124 (1986), 1-21.
\bibitem{li3}
\textsc{P. Li},
Linear growth harmonic functions on K$\ddot{\mathrm{a}}$hler manifolds with nonnegative Ricci curvature, Math. Res. Lett. 2 (1995), 79-94.
\bibitem{li4}
\textsc{P. Li},
The theory of harmonic functions and its relation to geometry, Proc. Sympos. Pure Math., Amer. Math. Soc. Vol. 54, Part 1, (Ed. R. Green and S. T. Yau), 1993.
\bibitem{li-sh}
\textsc{P. Li and R. Schoen},
$L^p$ and mean value properties of subharmonic functions on Riemannian manifolds, Acta Math. 153 (1984), 279-301.
\bibitem{li-tam1}
\textsc{P. Li and L-F. Tam},
Linear growth harmonic functions on a complete Riemannian manifold, J. Differential Geom. 29 (1989), 421-425.
\bibitem{li-tam2}
\textsc{P. Li and L-F. Tam},
Green's functions, harmonic functions, and volume comparison, J. Differential Geom. 41 (1995), 277-318.
\bibitem{li-yau}
\textsc{P. Li and S. T. Yau},
Estimates of eigenvalues of a compact riemannian manifold, Proc. Sympos. Pure Math. , Amer. Math. Soc. 36 (1980), 205-239. 
\bibitem{lo}
\textsc{J. Lott},
Optimal transport and Ricci curvature for metric-measure spaces, in Surveys in
Differential Geometry, vol. XI, Metric and Comparison Geometry, eds. J. Cheeger and K.
Grove, International Press, Somerville, MA, p. 229-257 (2007).
\bibitem{lo-vi}
\textsc{J. Lott and C. Villani},
Ricci curvature for metric measure spaces via optimal transport, Ann. of Math. 169 (2009), 903-991.
\bibitem{men}
\textsc{X. Menguy}, 
Examples of nonpolar limit spaces, Amer. J. Math. 122 (2000), 927-937.
\bibitem{men1}
\textsc{X. Menguy},
Examples of strictly weakly regular points,  Geom. Funct. Anal. 11 (2001), 124-131.
\bibitem{men3}
\textsc{X. Menguy},
Examples with bounded diameter growth and infinite topological type, Duke Math. J. 102 (2000), 403-412.
\bibitem{men2}
\textsc{X. Menguy},
Noncollapsing examples with positive Ricci curvature and infinite topological type, Geom. Funct. Anal. 10 (2000), 600-627.
\bibitem{mor}
\textsc{C. B. Morrey Jr.},
Multiple integrals in the calculus of variations, Springer, New York, 1966.
\bibitem{oh}
\textsc{S. Ohta}
On the measure contraction property of metric measure spaces, Comment. Math. Helv. 82 (2007), 805-828.
\bibitem{Per1}
\textsc{G. Perelman},
A complete Riemannian manifold of positive Ricci curvature with with Euclidean volume growth and nonunique asymptotic cone, 
Comparison geometry, (Berkeley CA 1993-94), 165-166, Math. Sci. Res. Inst. Publ., 30, Cambridge Univ. Press, Cambridge, 1997.
\bibitem{Per2}
\textsc{G. Perelman},
Construction of manifolds of positive Ricci curvature and large Betti numbers, Comparison geometry, (Berkeley CA 1993-94), 157-163, Math. Sci. Res. Inst. Publ., 30, Cambridge Univ. Press, Cambridge, 1997.
\bibitem{Rei}
\textsc{E. R. Reifenberg},
Solution of the Plateau problem for $m$-dimensional surfaces of varying topological type, Acta Math. 104 (1962), 1-92.
\bibitem{sa1}
\textsc{L. Saloff-Coste},
Uniformly elliptic operators on Riemannian manifolds, J. Differential Geom. 36 (1992), 417-450.
\bibitem{sa2}
\textsc{L. Saloff-Coste},
A note on Poincar\'{e}, Sobolev, and Harnack inequalities, Int. Math. Res. Not. 2 (1992), 27-38.
\bibitem{sa-str}
\textsc{L. Saloff-Coste and D. W. Strook},
Op\'{e}rateurs uniform\'{e}ment sous-elliptiques sur les groupes de Lie, J. Funct. Anal. 98 (1991), 97- 121.
\bibitem{sc-ya}
\textsc{R. Schoen and S. T. Yau},
Lectures on Differential Geometry, International Press, 1995.
\bibitem{sh}
\textsc{N. Shanmugalingam},
Harmonic functions on metric spaces, Illinois J. Math. 45 (2001), 1021-1050.
\bibitem{Si}
\textsc{L. M. Simon},
Lectures on Geometric Measure Theory, Proc. of the Center for Mathematical Analysis 3, Australian National University, 1983.
\bibitem{sor1}
\textsc{C. Sormani},
Busemann functions on manifolds with lower Ricci curvature bounds and minimal volume growth,  J. Differential Geom. 48 (1998), 557-585.
\bibitem{sor}
\textsc{C. Sormani},
Harmonic functions on manifolds with nonnegative Ricci curvature and linear volume growth, Pacific J. Math. 192 (2000), 183-189.
\bibitem{sor3}
\textsc{C. Sormani},
Nonnegative Ricci curvature, small linear diameter growth, and finite generation of fundamental groups, J. Differential Geom. 54 (2000), 547-559.
\bibitem{sor2}
\textsc{C. Sormani},
The rigidity and almost rigidity of manifolds with lower bounds on Ricci curvature and minimal volume growth, Comm. Anal. Geom. 8 (2000), 159-212.
\bibitem{sor-wei1}
\textsc{C. Sormani and G. Wei},
Hausdorff convergence and universal covers,  Trans. Amer. Math.  Soc. 353 (2001), 3584-3602.
\bibitem{sor-wei}
\textsc{C. Sormani and G. Wei},
Universal covers for Hausdorff limits of noncompact spaces,  Trans. Amer. Math.  Soc. 356 (2004), 1233-1270.
\bibitem{St1}
\textsc{K.-T. Strum},
On the geometry of metric measure spaces, Acta Math. 196 (2006), 65-131.
\bibitem{St2}
\textsc{K.-T. Strum},
On the geometry of metric measure spaces II, Acta Math. 196 (2006), 133-177.
\bibitem{ti-yau1}
\textsc{G. Tian and S. T. Yau},
Complete K$\ddot{\mathrm{a}}$hler manifolds with zero Ricci curvature. I, J. Amer. Math. Soc. 3 (1990), 579-609.
\bibitem{ti-yau2}
\textsc{G. Tian and S. T. Yau},
Complete K$\ddot{\mathrm{a}}$hler manifolds with zero Ricci curvature. II, Invent. Math. 106 (1991), 27-60.
\bibitem{Vi1}
\textsc{C. Villani},
Topics optimal transpotation, American Mathematical Society, Providence, RI, 2003.
\bibitem{Vi2}
\textsc{C. Villani},
Optimal transport, old and new, Springer-Verlag, 2008.
\bibitem{wa}
\textsc{J. Wang},
Linear growth harmonic fucntions on complete manifolds, Comm. Anal. Geom. 3 (1995), 683-698.
\bibitem{yau1}
\textsc{S. T. Yau},
Harmonic functions on complete Riemannian manifolds, Comm. Pure and Appl. Math. 28 (1975), 201-228.
\bibitem{yau2}
\textsc{S. T. Yau},
Some function-theoretic properties of complete Riemannian manifold and their applications to geometry, Indiana Univ. Math. J. 25 (1976), 659-670.
\bibitem{Zha}
\textsc{L. Zhang},
On the generic eigenvalue flow of a family of metrics and its application, Comm. Anal. Geom. 7 (1999), 259-278.
\end{thebibliography}
\end{document}